\documentclass[11pt]{amsart}
\usepackage{amsmath} 
\usepackage{amssymb} 
\usepackage{amsthm}
\usepackage{amsfonts}            
\usepackage{mathrsfs}
\usepackage[all]{xy}
\usepackage{mathtools}
\usepackage{mathabx}
\usepackage{colonequals}
\usepackage[dvipsnames]{xcolor}
\usepackage[T1]{fontenc}

\usepackage{amsmath} 
\usepackage{amssymb} 
\usepackage{amsthm}
\usepackage{amsfonts}            
\usepackage{mathrsfs}
\usepackage{amsthm,amsfonts,amssymb,amsmath,amsxtra,mathtools}
\usepackage[all]{xy}
\usepackage{mathtools}
\usepackage{tikz-cd}
\tikzcdset{scale cd/.style={every label/.append style={scale=#1},
    cells={nodes={scale=#1}}}}
\usepackage{mathabx}
\usepackage{colonequals}
\usepackage{xcolor}
\usepackage[pagebackref]{hyperref} 
\hypersetup{  
	colorlinks=true,
	linkcolor=black,
	citecolor=black,
	urlcolor=magenta,
}

\usepackage{geometry}
\usepackage{setspace}

\makeatletter
\tikzset{
  edge node/.code={%
      \expandafter\def\expandafter\tikz@tonodes\expandafter{\tikz@tonodes #1}}}
\makeatother
\tikzset{
  subseteq/.style={
    draw=none,
    edge node={node [sloped, allow upside down, auto=false]{$\in$}}},
  Subseteq/.style={
    draw=none,
    every to/.append style={
      edge node={node [sloped, allow upside down, auto=false]{$\in$}}}
  }
}

\title{Intersections of Hecke correspondences on modular curves}

\author[Qiao He]{Qiao He}
\address{Department of Mathematics, 
	Columbia University,
	2990 Broadway,
	New York, NY 10027, USA}
\email{qh2275@columbia.edu}

\author[Baiqing Zhu]{Baiqing Zhu}
\address{Department of Mathematics, 
	Columbia University,
	2990 Broadway,
	New York, NY 10027, USA}
\email{bz2393@columbia.edu}

\date{\today}
\begin{document}
\theoremstyle{definition}        
\newtheorem{definition}{Definition}[subsection] 
\newtheorem{example}[definition]{Example}
\newtheorem{remark}[definition]{Remark}
\newtheorem{conjecture}[definition]{Conjecture}

\theoremstyle{plain}
\newtheorem{theorem}[definition]{Theorem}
\newtheorem{lemma}[definition]{Lemma}         
\newtheorem{proposition}[definition]{Proposition}
\newtheorem{corollary}[definition]{Corollary}

\newcommand{\ofb}{\mathcal{O}_{\Breve{F}}}
\newcommand{\of}{\mathcal{O}_{F}}
\newcommand{\zpb}{\Breve{\mathbb{Z}}_{p}}
\newcommand{\zp}{\mathbb{Z}_{p}}
\newcommand{\qpb}{\Breve{\mathbb{Q}}_{p}}
\newcommand{\qp}{\mathbb{Q}_{p}}
\newcommand{\Q}{\mathbb{Q}}
\newcommand{\Z}{\mathbb{Z}}
\newcommand{\X}{\mathbb{X}}
\newcommand{\F}{\mathbb{F}}
\newcommand{\fb}{\Breve{F}}
\newcommand{\testleftlong}{\longleftarrow\!\shortmid}
\newcommand{\B}{\mathbb{B}}
\newcommand{\N}{\mathcal{N}}
\newcommand{\M}{\mathcal{M}}
\newcommand{\Gp}{\Gamma_0(p)}
\newcommand{\univ}{\textup{univ}}
\newcommand{\st}{\textup{st}}
\newcommand{\OB}{\mathcal{O}_{\B}}
\newcommand{\crys}{\textup{crys}}
\newcommand{\rO}{\mathcal{O}}
\newcommand{\rN}{\widetilde{\N}}
\newcommand{\rNpr}{\widetilde{\N}^{\prime}}
\newcommand{\exc}{\textup{Exc}}
\newcommand{\bP}{\mathbb{P}}
\newcommand{\MFF}{\M^{\textup{F}\textup{F}}}
\newcommand{\MVV}{\M^{\textup{V}\textup{V}}}
\newcommand{\MVF}{\M^{\textup{V}\textup{F}}}
\newcommand{\MFV}{\M^{\textup{F}\textup{V}}}
\newcommand{\sz}{\widetilde{\mathcal{Z}}}
\newcommand{\Int}{\textup{Int}}
\newcommand{\den}{\textup{Den}}
\newcommand{\pden}{\textup{{Pden}}}
\newcommand{\FF}{\textup{FF}}
\newcommand{\FV}{\textup{FV}}
\newcommand{\VF}{\textup{VF}}
\newcommand{\VV}{\textup{VV}}
\newcommand{\lx}{\langle x\rangle}
\newcommand{\rep}{\textup{Rep}}
\newcommand{\prep}{\textup{PRep}}
\newcommand{\res}{\textup{res}}
\newcommand{\pr}{\textup{pr}}
\newcommand{\tc}{\textcolor{Cyan}}

\begin{abstract}
  We compute the arithmetic intersections of Hecke correspondences on the product of integral model of modular curve $\mathcal{X}_0(N)$ and relate it to the derivatives of certain Siegel Eisenstein series when $N$ is odd and squarefree. We prove this by establishing a precise identity between the arithmetic intersection numbers on the Rapoport--Zink space associated to $\mathcal{X}_0(N)^{2}$ and the derivatives of local representation densities of quadratic forms.
\end{abstract}
\maketitle

\pagenumbering{roman}
\tableofcontents
\newpage
\pagenumbering{arabic}

\section{Introduction}
\subsection{Background}
Let $\mathbb{H}=\{\tau=x+iy:x,y\in\mathbb{R}, y\neq0\}$ be the union of the upper and lower half plane. Let $Y_0(1)$ be the modular curve without level structure whose $\mathbb{C}$-points are given by the quotient $\textup{GL}_2(\mathbb{Z})\backslash\mathbb{H}$. Then $Y_0(1)\simeq\textup{Spec}\,\mathbb{C}[j]$ by the elliptic modular function $j=j(\tau)$. Let $X_0(1)$ be its compactification, then $X_0(1)\simeq\mathbb{P}_{\mathbb{C}}^{1}$.
\par
For a positive integer $m$, the $m$-th Hecke correspondence $T(m)$ is a divisor on the product $Y_0(1)^{2}$ which parameterizes degree $m$ isogenies between elliptic curves. Let $\overline{T}(m)$ be the closure of $T(m)$ in the compactification $X_0(1)^{2}$. Kronecker and Hurwitz showed that as a line bundle on $\mathbb{P}_{\mathbb{C}}^{1}\times\mathbb{P}_{\mathbb{C}}^{1}$,
\begin{equation*}
    \overline{T}(m)=\mathcal{O}(\sigma_1(m),\sigma_1(m))\coloneqq\textup{pr}_1^{\ast}\mathcal{O}(\sigma_1(m))\otimes\textup{pr}_2^{\ast}\mathcal{O}(\sigma_1(m)),
\end{equation*}
where $\sigma_1(m)=\sum\limits_{d\vert m}d$ and $\textup{pr}_{1},\textup{pr}_{2}$ are the two projections from $\mathbb{P}_{\mathbb{C}}^{1}\times\mathbb{P}_{\mathbb{C}}^{1}$ to its two factors. Notice that $\{\sigma_1(m)\}_{m\geq1}$ is the Fourier coefficients of the following non-holomorphic modular form, which is also a special value of a weight $2$ Eisenstein series of genus $1$ (cf. \cite[$\S$2.3]{Modularforms123}),
\begin{equation*}
    G_2(\tau)=\frac{1}{8\pi y}-\frac{1}{24}+\sum\limits_{m\geq1}\sigma_1(m)q^{m}, \tau=x+iy\,\,\textup{with}\,\,y\gg0\,\,\textup{and}\,\,q=\textup{exp}(2\pi i\tau).
\end{equation*}
\par
Hurwitz (cf. \cite[Proposition 2.4]{GK93}) also studied the intersection number of two divisors $T(m_1)$ and $T(m_2)$ on the affine plane $Y_0(1)^{2}\simeq\textup{Spec}\,\mathbb{C}[j,j^{\prime}]$. The two divisors intersect properly if and only if $m_1m_2$ is not a perfect square, and in this case the intersection $T(m_1)\cap T(m_2)$ parameterizes certain CM elliptic curves over $\mathbb{C}$. It turns out that the intersection number equals to the summation of certain Fourier coefficients of a weight 2 genus 2 Siegel Eisenstein series. This is a special case of the \emph{geometric Siegel-Weil formula} and yields a beautiful geometric proof of the Hurwitz class number formula.
\par
In this article, we consider the \textit{arithmetic intersection} between Hecke correspondences. Let $N$ be an odd square-free positive integer. Denote by $Y_0(N)$ the curve whose $\mathbb{C}$-points are given by the quotient $\Gamma_0(N)\backslash\mathbb{H}$, let $X_0(N)$ be the compactification of $Y_0(N)$. The curve $X_0(N)$ has a natural integral model $\mathcal{X}_0(N)$ over $\mathbb{Z}$ defined by Katz and Mazur \cite{KM85} and \cite{Ces17}. The product space $\mathcal{X}_0(N)^{2}\coloneqq\mathcal{X}_0(N)\times_{\mathbb{Z}}\mathcal{X}_0(N)$ is normal but not regular when $N\neq1$. We construct a regular integral model $\M_0(N)$ of $\mathcal{X}_0(N)^{2}$ by blowing up at certain supersingular points of $\mathcal{X}_0(N)^{2}$. We also define an ``integral'' Hecke correspondence $\mathcal{T}(m)$ on the model $\M_0(N)$. It will be a Cartier divisor on $\M_0(N)$ whose generic fiber over $\mathbb{C}$ equals to the $m$-th Hecke correspondence   on $X_0(N)$.  We study the intersection between three correspondences $\{\mathcal{T}(m_i)\}_{i=1,2,3}$ on $\M_0(N)$. The computation of the intersection numbers will be reduced to the intersections of ``local'' Hecke correspondences on the Rapoport--Zink space associated to the model $\M_0(N)$. We established an identity:
\begin{equation}
\left(\begin{array}{ll}
     \textup{Intersection numbers on}  \\
      \,\,\,\,\textup{Rapoport--Zink space}
\end{array}\right)=\left(\begin{array}{ll}
     \textup{Derivatives of local density polynomial} \\
       \,\,\,\,\,\,\,\,\,\,\,\,\,\,\,\,\textup{of quadratic lattices}
\end{array}\right)
    \label{KR}
\end{equation}
This finally builds up the bridge connecting the intersection numbers of Hecke correspondences and the summation of derivatives of certain Fourier coefficients of a weight 2 genus 3 Siegel Eisenstein series, which is a special case of the \emph{arithmetic Siegel-Weil formula} and belongs to an influential program initiated by Kudla \cite{Kud97Ann,Kudla2004}.
\par
The key result is the identity (\ref{KR}), as the global result follows easily by $p$-adic uniformization of the supersingular locus of $\M_0(N)$. The identity of the type in (\ref{KR}) are called \textit{Kudla--Rapoport conjecture}, which can be regarded as the nonarchimedean part of the arithmetic Siegel-Weil formula.

 The arithmetic Siegel--Weil formula was first established in the works of Kudla, Rapoport and Yang \cite{KRYtiny,KRYcomp,KRYbook} for GSpin Shimura varieties associated to $\mathrm{GSpin}(n-1,2)$ where $n=1,2$. The archimedean part was proved by the work of Garcia--Sankaran \cite{GS19} and Bruinier-Yang \cite{BY20}. The work of Gross--Keating \cite{GK93}  complemented by the ARGOS volume \cite{ARGOS} established the nonarchimedean part for the case $n=3$ with hyperspecial level. The case for general $n$ with hyperspecial level and odd $p$ was proved by Li and Zhang \cite{LZ22b}.

 In fact, the Kudla--Rapoport conjecture was originally formulated for unitary Rapoport--Zink space associated to $\mathrm{U}(n-1,1)$ with hyperspecial level \cite{KR11}. Terstiege \cite{Ter13a} solved the $n=3$ case. The work of Li--Zhang \cite{LZ22a} solved the general case of the original Kudla--Rapoport conjecture. The work of Li--Liu and Yao \cite{LL2,Yao} established an invariant for the exotic smooth models over ramified primes. 
 
 In the unitary case, there are also many works about Shimura varieties with bad reduction. However, because of the bad reduction, even the formulation of the formula was not clear for a long time. For Kr\"amer models over ramified primes, the works of Shi \cite{Shi1} and He--Shi--Yang \cite{HSY} proved the $n=2$ case. A general formulation was proposed and the case $n=3$ was proved in \cite{HSY3} and finally He--Li--Shi--Yang \cite{HLSY} proved the general case. Over unramified primes, Cho \cite{Cho} proposed a general conjecture for the maximal parahoric levels. In \cite{CHZ}, Cho--He--Zhang showed that a formulation similar to that of \cite{HSY3} is the same as the formulation of \cite{Cho} and reduced it to a strong version of Tate conjecture for certain Deligne-Lusztig varieties.

 Moreover, the unitary case has seen a number of other recent and thrilling developments, including formulas for degenerate terms by Bruinier--Howard \cite{BH} and Chen \cite{Chen} and the function field case by Feng--Yun--Zhang \cite{FYZ}.

 We refer to the left hand side of \eqref{KR} the geometric side, while the right hand side the analytic side. The Rapoport--Zink space appearing on the geometric side is determined by the quadratic lattice appearing on the analytic side. The case $N=1$, which locally corresponds to the rank 4 self-dual lattice $\left(\textup{M}_2(\mathbb{Z}_p),\textup{det}\right)$, is the classical work of Gross and Keating \cite{GK93}   complemented by the ARGOS volume \cite{ARGOS}. The question is widely open if the quadratic lattice is not self-dual, or equivalently, the Rapoport--Zink space has some non-hyperspecial level structures. The work of Kudla--Rapoport \cite{KRshimuracurve} and Sankaran--Shi--Yang \cite{SSY} proved the case when $L$ is of rank $3$ but not self-dual.  Our work essentially settled the case when $L$ is a specific rank 4 but not self-dual lattice, see (\ref{thelattice}) in the next section for the precise definition of the lattice.  
Since the GSpin Shimura varieties are not of PEL type, such problems are much more involved. In the current work, we make essential use of the concrete moduli intepretation of the Rapoport--Zink space which only exists for certain low dimensional special cases.

\subsection{Main results}
We only state the local version of our main result. Let $p$ be an odd prime.  Let $\F$ be the algebraic closure of the finite field $\mathbb{F}_p$ and $W=W(\mathbb{F})$.  We fix  a $p$-divisible group  $\mathbb{X}$  over $\F$ of dimension 1 and height 2. Let  $\B=\textup{End}^{\circ}(\mathbb{X})\coloneqq\textup{End}(\mathbb{X})\otimes_{\mathbb{Z}_p}\qp$ be the space of quasi-endomorphisms of $\mathbb{X}$, which is the unique quaternion division algebra over $\mathbb{Q}_p$. Let $q$ be the quadratic form on $\B$. Let $x_0\in\mathbb{B}$ be an element such that $\nu_p(q(x_0))=1$. We consider the following functor $\N(x_0)$: for a  a $W$-scheme $S$ such that $p$ is locally nilpotent in $\mathcal{O}_S$, the set $\N(x_0)(S)$ consists of tuples 
 \begin{equation*}
        \left(X_1\stackrel{\pi_1}\rightarrow X_1^{\prime}, (\rho_1,\rho_1^{\prime})\right), \left(X_2\stackrel{\pi_2}\rightarrow X_2^{\prime}, (\rho_2,\rho_2^{\prime})\right),
    \end{equation*}
where $\{X_i, X_i^{\prime}\}_{i=1,2}$ are deformations of $\mathbb{X}$ to $S$ via the height 0 quasi-isogenies $\rho_i$ and $\rho_i^{\prime}$ from $\X$ to $X_i$ and $X_i^{\prime}$, and $\pi_1,\pi_2$ are deformations of the quasi-isogeny $x_0$ under the framing morphisms $\rho_i$. The functor $\N(x_0)$ is represented by a formal scheme formally of finite type over $W$.
\par
Let 
\begin{equation}
   \left(X_1^{\textup{univ}}\stackrel{x_{0,1}^{\univ}}\longrightarrow X_{1}^{\prime\textup{\univ}},\left(\rho_1^{\univ},\rho_1^{\prime\univ}\right)\right),\,\,\,\,\left(X_2^{\textup{univ}}\stackrel{x_{0,2}^{\univ}}\longrightarrow X_{2}^{\prime\textup{\univ}},\left(\rho_2^{\univ},\rho_2^{\prime\univ}\right)\right)
\end{equation}
be the universal object over the formal scheme $\N(x_0)$. Let $x\in\B$ be a non-zero element and denote by $x^{\prime}$ the element $x_0\cdot x\cdot x_0^{-1}$. We have the following commutative diagram
\begin{equation*}
        \begin{tikzcd}
    {X_1^{\textup{univ}}}
    \arrow[d,"{x_{0,1}^{\univ}}"swap] \arrow[r, dashed, "{x}"]
    & {X_2^{\textup{univ}}}
    \arrow[d,"{x_{0,2}^{\univ}}"]
    \\
    {X_1^{\prime\textup{univ}}}
    \arrow[r, dashed, "{x^{\prime}}"]
    &{X_2^{\prime\textup{univ}}},
\end{tikzcd}
\end{equation*}
where the dotted arrows below $x$ and $x^{\prime}$ mean that they are quasi-isogenies. For a subset $H\subset\B$, define the \textit{special $\mathcal{Z}$-cycle} $\mathcal{Z}_{\N(x_0)}(H)\subset\N(x_0)$ (resp. \textit{special $\mathcal{Y}$-cycle} $\mathcal{Y}_{\N(x_0)}(H)\subset\N(x_0)$) to be the closed formal subscheme over which both $x$ and $x^{\prime}$ (resp. $x_0\cdot x$ and $x^{\prime}\cdot x_0$) deform to isogenies for all $x\in H$ (cf. Definition \ref{special-cycle-originalmodel}).
\par
The formal scheme $\N(x_0)$ is normal but not regular. In order to consider intersection problem, we need to construct a regular model first. Let $\pi:\M\rightarrow\N(x_0)$ be the blow up morphism along the unique closed $\F$-point of $\N(x_0)$. Then $\M$ is a 3-dimensional regular formal scheme. For a subset $H\subset\B$, denote by $\mathcal{Z}(H)=\pi^{\ast}\mathcal{Z}_{\N(x_0)}(H)$ (resp. $\mathcal{Y}(H)=\pi^{\ast}\mathcal{Y}_{\N(x_0)}(H)$) the direct pullback. When $H=\{x\}$ consists of only one element $x$, $\mathcal{Z}_{\mathcal{N}(x_0)}(x)$ and $\mathcal{Y}_{\mathcal{N}(x_0)}(x)$ are not Cartier divisors. However, we show that $\mathcal{Z}(x)$ and $\mathcal{Y}(x)$ are Cartier divisors on $\M$ (cf. Lemma \ref{lem: divisor} and Corollary \ref{cor: divisor}).
\par
Let $L\subset\B$ be a $\zp$-lattice of rank $3$, we now associate to $L$ several integers: the arithmetic intersection number of $\mathcal{Z}$-cycles $\textup{Int}^{\mathcal{Z}}(L)$, of $\mathcal{Y}$-cycles $\textup{Int}^{\mathcal{Y}}(L)$, and the derived local density of $L$ into some not self-dual lattices.
\par
Let $x_1,x_2,x_3$ be a basis of $L$. Define the arithmetic intersection number of $\mathcal{Z}$-cycles to be
\begin{equation}
    \textup{Int}^{\mathcal{Z}}(L)\coloneqq\chi(\M,\rO_{\mathcal{Z}(x_1)}\otimes^{\mathbb{L}}_{\rO_{\M}}\rO_{\mathcal{Z}(x_2)}\otimes^{\mathbb{L}}_{\rO_{\M}}\rO_{\mathcal{Z}(x_3)}).
    \label{Z-Int}
\end{equation}
where $\mathcal{O}_{\mathcal{Z}(x_i)}$ denotes the structure sheaf of the special divisor $\mathcal{Z}(x_i)$, $\otimes^{\mathbb{L}}_{\mathcal{O}_{\M}}$ denotes the derived tensor product of coherent sheaves on $\M$ , and $\chi$ denotes the Euler--Poincar{\'e} characteristic. The arithmetic intersection number of $\mathcal{Y}$-cycles $\textup{Int}^{\mathcal{Y}}(L)$ are defined by similar formula (\ref{Z-Int}) but replacing $\mathcal{Z}$ by $\mathcal{Y}$. Both the numbers $\textup{Int}^{\mathcal{Z}}(L)$ and $\textup{Int}^{\mathcal{Y}}(L)$ are independent of the choice of the basis of $L$ by the \textit{linear invariance} property (cf. Corollary \ref{linear-invariance-2}).
\par
Now we define the local density and derived local density. For another quadratic $\zp$-lattice (of arbitrary rank) $M$, define $\textup{Rep}_{M,L}$ to be the scheme of integral representations, an $\zp$-scheme such that for any $\zp$-algebra $R$, $\textup{Rep}_{M,L}(R) = \textup{QHom}(L\otimes_{\zp}
R, M\otimes_{\zp}R)$. Here $\textup{QHom}$ denotes the set of homomorphisms of quadratic modules. The local density of associated to $M$ and $L$ is defined to be
\begin{equation*}
    \textup{Den}(M,L)=\lim\limits_{d\rightarrow\infty}\frac{\#\textup{Rep}_{M,L}(\zp/p^{d})}{p^{d\cdot \textup{dim}(\textup{Rep}(M,L))_{\qp}}}.
\end{equation*}
\par
Let $H_0(p)$ and $H_0(p)^{\vee}$ be the following two quadratic lattices of rank $4$ over $\zp$:
\begin{equation}
    H_0(p)=\left\{\begin{pmatrix}
        a & b\\
        p c & d
    \end{pmatrix}:a,b,c,d\in\zp\right\},\,\,\,\,H_0(p)^{\vee}=\left\{\begin{pmatrix}
        a & p^{-1}b\\
        c & d
    \end{pmatrix}:a,b,c,d\in\zp\right\}.
    \label{thelattice}
\end{equation}
Let $H_{2k}^{+}=(H_{2}^{+})^{\obot k}$ be a self-dual lattice of rank $2k$, where the quadratic form on $H_{2}^{+}=\zp^{2}$ is given by $(x,y)\in \zp^{2}\mapsto xy$. There exist local density polynomials $\den(X,H_0(p),L),\den(X,H_0(p)^{\vee},L)\in\mathbb{Q}[X]$ such that
\begin{align*}
    \den(X,H_0(p),L)\big\vert_{X=p^{-k}}&=\den(H_0(p)\obot H_{2k}^{+},L),\\\den(X,H_0(p)^{\vee},L)\big\vert_{X=p^{-k}}&=\den(H_0(p)^{\vee}\obot H_{2k}^{+},L).
\end{align*}
Both polynomial vanishes at $X=1$ since $L\subset\B$ is anisotropic and cannot be embedded into $H_0(p)$ and $H_0(p)^\vee$.  Therefore we consider the (normalized) \textit{derived local density}
\begin{align*}
    \partial\den\left(H_0(p), L\right)&\coloneqq-2\cdot\frac{\textup{d}}{\textup{d}X}\bigg\vert_{X=1}\frac{\den(X,H_0(p),L)}{\den(H_0(p),H_2^{+}\obot H_1^{+}[p])},\\
        \partial\den\left(H_0(p)^{\vee}, L\right)&\coloneqq-2\cdot\frac{\textup{d}}{\textup{d}X}\bigg\vert_{X=1}\frac{\den(X,H_0(p)^{\vee},L)}{\den(H_0(p)^{\vee},H_2^{+}\obot H_1^{+}[p^{-1}])},
\end{align*}
\begin{remark}
    It is pointed out by Chao Li that the extra scaling scalar $2$ comes from the fact that the formal scheme $\mathcal{N}(x_0)$ is a double cover of the Rapoport--Zink space associated to the lattice $H_0(p)$.
\end{remark}
Let $\mathcal{O}_{\B}$ be the maximal order of the division quaternion algebra $\B$. Our main theorem is the following two precise identities between several quantities just defined.
\begin{theorem}[Theorem \ref{main-theorem}]
    Let $L\subset\B$ be a $\zp$-lattice of rank 3. Then
    \begin{equation*}
        \Int^{\mathcal{Z}}(L)=\partial\den\left(H_0(p),L\right),
    \end{equation*}
    and 
    \begin{equation*}
        \Int^{\mathcal{Y}}(L)=\partial\den\left(H_0(p)^{\vee},L\right)-1=\partial\den\left(H_0(p)^{\vee},L\right)-\frac{p^{7}}{2(p+1)^{2}}\cdot\den(\mathcal{O}_{\B}^{\vee},L).
    \end{equation*}
    \label{main-intro}
\end{theorem}
 We remark that in an ongoing work, we have a precise conjecture for the general GSpin RZ space with maximal parahoric level which specialize to Theorem \ref{main-intro}. We hope to establish more special cases to provide evidence for the general conjecture in the future. We also remark that the formulas in Theorem \ref{main-intro} is a manifestation of the duality phenomenon studied in   \cite{Cho}(cf. \cite[Conjecture 1.6]{Cho}) in the GSpin group setting. 

\subsection{Strategy of the proof}
\subsubsection{Difference formula on the analytic side}
Let $\mathbf{L}$ be a vertex lattice, i.e., $p\mathbf{L}^{\vee}\subset \mathbf{L}\subset\mathbf{L}^{\vee}$. Let $\mathbf{M}$ be another quadratic lattice. Let $\lx$ be a rank $1$ quadratic lattice generated by $x$ such that $n\coloneqq\nu_p(q(x))\geq0$. In $\S$\ref{diff-analytic}, we compute the difference of two local densities $\den(\mathbf{L},\mathbf{M}\obot\lx)$ and $\den(\mathbf{L},\mathbf{M}\obot\langle p^{-1}x\rangle)$. The main result is Theorem \ref{diff-loc-density}. Specialize it to our situation, we get
\begin{lemma}[Lemma \ref{der-diff}]
    Let $L^{\flat}\subset\mathbb{B}$ be an $\zp$-lattice of rank $2$. Let $x\in\mathbb{B}$ be an element such that $\nu_p(q(x))\geq0$ and $x\,\bot\, L^{\flat}$. Then
    \begin{align}
        \partial\den\left(H_0(p), L^{\flat}\obot\lx\right)-&\partial\den\left(H_0(p), L^{\flat}\obot\langle p^{-1}x\rangle\right)\notag\\
        =&\begin{cases}
            \partial\den\left(\lx[-1]\obot H_2^{+}[p], L^{\flat}\right), &\textup{if $n=0$;}\\
            &\\
            2\cdot\partial\den\left(\lx[-1]\obot H_2^{+}[p], L^{\flat}\right)&\\
            \,\,\,+\partial\den\left(\lx[-1]\obot H_2^{+}, L^{\flat}\right), &\textup{if $n=1$;}\\
            &\\
            2\cdot\partial\den\left(\lx[-1]\obot H_2^{+}[p], L^{\flat}\right)&\\
            +2\cdot\partial\den\left(\lx[-1]\obot H_2^{+}, L^{\flat}\right),&\textup{if $n\geq2$.}
        \end{cases}\label{explicit-diff-der}
    \end{align}
    \label{diff-ana-intro}
\end{lemma}
Here for a quadratic lattice $N$ and an element $a\in\qp$, the symbol $N[a]$ means another quadratic lattice with the same $\zp$-module as $N$ but with the quadratic form multiplied by $a$. The difference formula is related to the double coset decomposition of the lattice $H_0(p)$ under the left and right multiplication of $\Gamma_0(p)\coloneqq H_0(p)^{\times}$. Let $x\in H_0(p)$ be a \textit{primitive} element such that $\nu_p(\textup{det}(x))=n\geq0$, the primitivity of $x$ means that $x\in H_0(p)\backslash pH_0(p)$. Then the double coset $\Gamma_0(p)x\Gamma_0(p)$ has the following possibilities (Lemma \ref{double-coset}): 
\begin{align*}
    \Gamma_0(p)\begin{pmatrix}
        1 & 0\\
        0 & p^{n}
    \end{pmatrix}\Gamma_0(p),\,\,\Gamma_0(p)\begin{pmatrix}
        p^{n} & 0\\
        0 & 1
    \end{pmatrix}\Gamma_0(p),\,\,
    \Gamma_0(p)\begin{pmatrix}
        0 & p^{n-1}\\
        p & 0
    \end{pmatrix}\Gamma_0(p),\,\,
    \Gamma_0(p)\begin{pmatrix}
        0 & 1\\
        p^{n} & 0
    \end{pmatrix}\Gamma_0(p).
\end{align*}
When $n\geq2$, the four cosets are different. When $n=1$, the last two cosets are the same, therefore there are 3 different cosets. When $n=0$, the last two cosets do not exist, while the first two are the same, therefore there are only 1 coset. That corresponds the number of terms appearing on the right hand side of the difference formula (\ref{explicit-diff-der}).

\subsubsection{Difference formula on the geometric side}
Denote by $\M_{\F}$ the base change $\M\times_{W}\F$. It can be viewed as the Cartier divisor of the ideal sheaf generated by $p$. There are four irreducible components of $\M_{\F}$ labeled by $\MFF, \MFV, \MVF$ and $\MVV$ (see $\S$\ref{RZBLOWUP}). The superscripts here indicate the isogeny type (Frobenius or Verschiebung) of the first and second universal deformation of $x_0$ over the corresponding irreducible component. All four irreducible components are isomorphic to the formal completion of the scheme $\textup{Bl}_{(0,0)}\mathbb{A}_{\mathbb{F}}^{2}$ along its exceptional divisor. We have the following identity of Cartier divisors on $\M$ (Proposition \ref{geometry-M}):
\begin{equation*}
   \M_{\F}=2\cdot\exc_{\M}+\MFF+\MVV+\MFV+\MVF.
\end{equation*}
\par
We prove the following lemma: 
\begin{lemma}[Lemma \ref{geo-diff-M}]
    Let $L^{\flat}\subset\B$ be a $\zp$-lattice of rank $2$. Let $x\in\B$ be another element such that $\nu_p(q(x))\geq\max\{\textup{val}(L^{\flat}),2\}$ and $x\perp L^{\flat}$, then
    \begin{align*}
        \Int^{\mathcal{Z}}(L^{\flat}\obot\langle x\rangle)-\Int^{\mathcal{Z}}(L^{\flat}\obot\langle p^{-1}x\rangle)&=\chi(\M,{^{\mathbb{L}}\mathcal{Z}}(L^{\flat})\otimes^{\mathbb{L}}_{\rO_{\M}}\rO_{\M_{\F}})\\
        &=\chi(\M,{^{\mathbb{L}}\mathcal{Z}}(L^{\flat})\otimes^{\mathbb{L}}_{\rO_{\M}}\rO_{\MFF})+\chi(\M,{^{\mathbb{L}}\mathcal{Z}}(L^{\flat})\otimes^{\mathbb{L}}_{\rO_{\M}}\rO_{\MVV})\\
        &+\chi(\M,{^{\mathbb{L}}\mathcal{Z}}(L^{\flat})\otimes^{\mathbb{L}}_{\rO_{\M}}\rO_{\MFV})+\chi(\M,{^{\mathbb{L}}\mathcal{Z}}(L^{\flat})\otimes^{\mathbb{L}}_{\rO_{\M}}\rO_{\MVF}).
    \end{align*}
    \label{diff-geo-intro}
\end{lemma}

\subsubsection{Difference formula combined and the induction method}
For quadratic lattice of rank less or equal to $3$, the local density polynomial has been computed explicitly in Yang's work \cite{yanglocaldensity}. Combining with some calculations of intersection numbers in \cite{GK93} which are blow up invariant by $\S$\ref{invariance-int}, we prove the following: 
\begin{lemma}[Corollary \ref{diff-com}]
    Let $L^{\flat}\subset\B$ be a $\zp$-lattice of rank $2$. Let $x\in\B$ be a nonzero element such that $x\,\bot\,L^{\flat}$ and $\nu_p(q(x))\geq\max(L^{\flat})$. Then
    \begin{equation*}
        \Int^{\mathcal{Z}}\left(L^{\flat}\obot\lx\right)-\Int^{\mathcal{Z}}\left(L^{\flat}\obot\langle p^{-1}x\rangle\right)=\partial\den\left(H_0(p),L^{\flat}\obot\lx\right)-\partial\den\left(H_0(p),L^{\flat}\obot\langle p^{-1}x\rangle\right).
    \end{equation*}
    \label{diff-intro}
\end{lemma}
Actually, there is a correspondence between the terms on the right hand side of the geometric difference formula (Lemma \ref{diff-geo-intro}) and that of the analytic difference formula Lemma (\ref{diff-ana-intro}). Namely (Lemma \ref{term-identity}), 
\begin{align*}
    &\chi(\M,{^{\mathbb{L}}\mathcal{Z}}(L^{\flat})\otimes^{\mathbb{L}}_{\rO_{\M}}\rO_{\MFF})=\chi(\M,{^{\mathbb{L}}\mathcal{Z}}(L^{\flat})\otimes^{\mathbb{L}}_{\rO_{\M}}\rO_{\MVV})=\partial\den\left(\lx[-1]\obot H_2^{+}, L^{\flat}\right),\\
    &\chi(\M,{^{\mathbb{L}}\mathcal{Z}}(L^{\flat})\otimes^{\mathbb{L}}_{\rO_{\M}}\rO_{\MFV})=\chi(\M,{^{\mathbb{L}}\mathcal{Z}}(L^{\flat})\otimes^{\mathbb{L}}_{\rO_{\M}}\rO_{\MVF})=\partial\den\left(\lx[-1]\obot H_2^{+}[p], L^{\flat}\right).
\end{align*}
\par
Let $(a_1,a_2,a_3)$ be the Gross-Keating invariant $\textup{GK}(L)$ of the quadratic lattice $L$ where $a_1\leq a_2\leq a_3$. We prove Theorem \ref{main-intro} for $\mathcal{Z}$-cycles in the two base cases that the $\textup{GK}(L)=(0,0,1)$ and $(0,1,1)$ by computing both sides explicitly ($\S$\ref{base-ana} and $\S$\ref{base-geo}). Then the statement in Theorem \ref{main-intro} about $\mathcal{Z}$-cycles follows from Lemma \ref{diff-intro} by inducting on the integer $n(L)=a_1+a_2+a_3$.  

\subsubsection{Automorphism of $\M$}
There is a naturally defined automorphism $\iota_{\M}$ of the formal scheme $\M$ ($\S$\ref{aut-M}). For a subset $H\subset\B$, we have $\mathcal{Y}(H)=\left(\iota^{\M}\right)^{\ast}\left(\mathcal{Z}(x_0\cdot H)\right)$. Combining the geometric difference formula for $\mathcal{Z}$-cycles and some identities between local densities (Lemma \ref{analytic-calculations-1}), we obtain:
\begin{lemma}[Corollary \ref{diff-com}]
    Let $L^{\flat}\subset\B$ be a $\zp$-lattice of rank $2$. Let $x\in\B$ be a nonzero element such that $x\,\bot\,L^{\flat}$ and $\nu_p(q(x))\geq\max(L^{\flat})$. Then
    \begin{equation*}
        \Int^{\mathcal{Y}}\left(L^{\flat}\obot\lx\right)-\Int^{\mathcal{Y}}\left(L^{\flat}\obot\langle p^{-1}x\rangle\right)=\partial\den\left(H_0(p)^{\vee},L^{\flat}\obot\lx\right)-\partial\den\left(H_0(p)^{\vee},L^{\flat}\obot\langle p^{-1}x\rangle\right).
    \end{equation*}
    \label{diff-intro-y}
\end{lemma}
Then the statement in Theorem \ref{main-intro} about $\mathcal{Y}$-cycles follows from Lemma \ref{diff-intro-y} and similar induction method as the $\mathcal{Z}$-cycle case.

\subsection{Applications}
\subsubsection{A conjecture of Kudla--Rapoport}
In 2006, Kudla and Rapoport defined CM cycles on the modular curve $X_0(N)$ for an arbitrary positive integer $N$. Roughly speaking, a CM cycle parameterizes a pair of CM isogenies $(j,j^{\prime})$ of two elliptic curves $E,E^{\prime}$ which admits a cyclic isogeny $\pi:E\rightarrow E^{\prime}$ of degree $N$, where $j\in\textup{End}(E)$ and $j^{\prime}\in\textup{End}(E^{\prime})$ such that $j^{\prime}\circ\pi=\pi\circ j$. It is conjectured that the local arithmetic intersection numbers of two CM cycles on the Rapoport--Zink space associated to the modular curve $X_0(N)$ are related to the derivatives of local densities in the favor of (\ref{KR}). Our result can be applied to confirm this conjecture when $N$ is odd and squarefree.
\par
We only state the ``local'' theorem, from which the global one can be deduced easily. The key observation is the following: Suppose that the element $x_0\in\B$ we picked satisfies $q(x_0)=N$. Then the strict transform $\widetilde{\mathcal{Z}}(1)$ of the special cycle $\mathcal{Z}(1)$ on $\M$ is isomorphic to the blow up of the Rapoport--Zink space associated to $X_0(N)$ along its maximal ideal. The pullback of special cycles $\mathcal{Z}(x)$ to $\widetilde{\mathcal{Z}}(1)$ coincides with CM cycles defined by Kudla--Rapoport. This fact can be viewed as a \textit{geometric cancellation law} (cf. $\S$\ref{geo-cancel}).
\par
Let $M\subset\mathbb{B}^{0}$ (trace 0 elements in $\B$) be a rank 2 lattice with basis $x,y$. Define 
\begin{equation*}
    \textup{Int}^{\textup{CM}}(M)=\chi(\mathcal{N}_0(N),\rO_{\mathcal{Z}^{\textup{CM}}(x)}\otimes^{\mathbb{L}}_{\rO_{\M}}\rO_{\mathcal{Z}^{\textup{CM}}(y)}).
\end{equation*}
This number is independent of the choice of basis $x,y$. Let $S\subset\textup{M}_2(\mathbb{Z}_p)$ be the sublattice of trace 0 elements. Combining with \textit{analytic cancellation law} (cf. $\S$\ref{ana-cancel}) on the local densities, we obtain:
\begin{theorem}[Theorem \ref{int-CM}]
    Let $M\subset\B$ be a $\zp$-lattice of rank 2. Then
    \begin{equation*}
        \Int^{\textup{CM}}(M)=\partial\den(S,M).
    \end{equation*}
    \label{int-CM-intro}
\end{theorem}
We remark here that Shi \cite{ShiCM} has a different method to calculate $\Int^{\textup{CM}}(M)$.

\subsubsection{Arithmetic intersections of Hecke correspondences}
Let $\mathbb{V}$ be a rank 4 \textit{incoherent} quadratic space over $\mathbb{A}$ given by the following:
\begin{equation}
    \mathbb{V}_{v}=V_v=\textup{M}_2(\mathbb{Q}_v)\,\,\textup{if $v<\infty$, and $\mathbb{V}_{\infty}$ is positive definite.}
    \label{incoherent-intro}
\end{equation}
Given a Schwartz function $\Phi=\Phi_{\infty}\otimes\Phi_{f}\in\mathscr{S}(\mathbb{V}^{3})$ where $\Phi_{\infty}$ is given by the standard Gaussian function on $\mathbb{V}_{\infty}^{3}$. There is a classical incoherent Eisenstein series $E(g,s,\Phi)$ ($\S$\ref{eisen}) on the group $\textup{GSp}_6$. Denote by $E_T(g,s,\Phi)$ the $T$-th Fourier expansion of it. The central value of $E(z,s,\Phi)$ is $0$ by the incoherence. Let $E^{\prime}_T(g,0,\Phi)$ be the derivative of $E_T(g,s,\Phi)$ at $0$.
\par
Let $N$ be an odd squarefree positive integer. Let $\M_0(N)$ be the blow up of the Deligne-Mumford stack $\mathcal{X}_0(N)\times_{\mathbb{Z}}\mathcal{X}_0(N)$ along its supersingular points with residue field characteristic $p\vert N$. For a positive integer $m$, we define a Cartier divisor $\widehat{\mathsf{T}}(m)$ (cf. (\ref{integral-m-hecke})) on $\M_0(N)$ such that the $\mathbb{C}$-points of $\widehat{\mathsf{T}}(m)$ equals to the classically defined $m$-th Hecke correspondence on $X_0(N)^{2}$ (Lemma \ref{C-points}).
\begin{theorem}[Theorem \ref{a-i-h-1}]
    Let $m_1,m_2,m_3$ be three positive integers such that there is no positive definite binary quadratic form over $\mathbb{Z}$ which represents the three integers $m_1,m_2,m_3$. Then
    \begin{equation*}
        \left(\widehat{\mathsf{T}}(m_1)\cdot\widehat{\mathsf{T}}(m_2)\cdot\widehat{\mathsf{T}}(m_3)\right)=\sum\limits_{T}-2E_T^{\prime}(1,0,\Phi_{\infty}\otimes\mathbf{1}_{H_0(N)\otimes\widehat{\mathbb{Z}}}^{3}),
    \end{equation*}
    where the summation ranges over all the half-integral symmetric positive definite $3\times3$ matrices $T$ with diagonal elements $m_1,m_2,m_3$.
\end{theorem}
The condition on $m_1,m_2,m_3$ guarantees that there is no self-intersections between the three divisors $\widehat{\mathsf{T}}(m_i)$ on the generic fiber. In the works of Yuan--Zhang--Zhang \cite{YZZ}, they use ``regular'' test functions (cf. Definition 4.4.1 of \textit{loc. cit.}) to avoid the self-intersection. Using this method, they proved the (semi-global) arithmetic Siegel--Weil formula on the product of modular curves at a place $p$ where the modular curve has good reduction (cf. Theorem 5.4.3 of \textit{loc. cit.}). Our work gives a generalization of their formula at a prime $p$ where the level structure is given by $\Gamma_0(p)$.
\par
Let $p$ be an odd prime. Let $U=\Gamma_0(p)\times U^{p}$ be a compact subgroup of $\textup{GL}_2(\mathbb{A}_f)$. Let $K= U\times_{\mathbb{G}_m}U$. We define a regular integral model $\M_{K,(p)}$ over $\mathbb{Z}_{(p)}$ for the product of modular curves $Y_U\times Y_U$. Using Weil representation, we can vary the Hecke correspondence and define a divisor $\widehat{\mathsf{T}}(g,\phi)$ on $\M_{K,(p)}$ where $g\in\textup{GL}_2(\mathbb{A})$ and $\phi\in\mathscr{S}(\mathbb{V})$ following the works of Kudla \cite{Kudla2004} and Yuan--Zhang--Zhang.
\begin{theorem}[Theorem \ref{int-yzz}]
    Let $\phi_i\in\mathscr{S}(\mathbb{V})$ be three Schwartz functions satisfying 
    \begin{itemize}
    \item [(1)] $\phi_{i,\infty}$ is the standard Gaussian function on $\mathbb{V}_{\infty}$.
    \item[(2)] $\phi_{i,p}=1_{H_0(p)}$ or $1_{H_0(p)^{\vee}}$ and $\phi_i^{p}$ is invariant under the group $K^p$.
    \item[(3)] There exists a finite place $v$ prime to $p$ such that the Schwartz function $\phi_v=\phi_{1,v}\otimes\phi_{2,v}\otimes\phi_{3,v}\in\mathscr{S}(\mathbb{V}_v^{3})$ is regularly supported in the sense of \cite[Definition 4.4.1]{YZZ}.
\end{itemize}
     Let $\Phi=\phi_1\otimes\phi_2\otimes\phi_3\in\mathscr{S}(\mathbb{V}^{3})$. Then for all elements $g=(g_1,g_2,g_3)\in\left(\textup{GL}_2\times_{\mathbb{G}_m}\textup{GL}_2\times_{\mathbb{G}_m}\textup{GL}_2\right)(\mathbb{A})$ such that $g_{v}=\mathbf{1}_v$ and $g_{p}=\mathbf{1}_p$. We have
    \begin{equation}
        \left(\widehat{\mathsf{T}}(g_1,\phi_1)\cdot\widehat{\mathsf{T}}(g_2,\phi_2)\cdot\widehat{\mathsf{T}}(g_3,\phi_3)\right)_p=\sum\limits_{T:\textup{Diff}(T)=\{p\}}-2E_{T}^{\prime}(g,0,\Phi),\,\,\,\,\textup{if $\Phi_p=1_{H_0(p)^{3}}$ or $1_{H_0(p)^{\vee3}}$}.
    \end{equation}
    The set $\textup{Diff}(T)$ consists of all finite places $l$ where $T$ is not represented by $\mathbb{V}_l$.
    \label{int-hecke-main-intro}
\end{theorem}
\begin{remark}
    During the proof of the above theorem, we find that the case $\Phi_p=1_{H_0(p)^{3}}$ corresponds to the global intersections of $\mathcal{Z}$-cycles on $\M_{K,(p)}$, while the case $\Phi_p=1_{H_0(p)^{\vee3}}$ corresponds to the global intersections of $\mathcal{Y}$-cycles.
\end{remark}

\subsection{The structure of the paper}
In Part \ref{ana-part}, we introduce notations on quadratic lattices and local density ($\S$\ref{nota-qua-loc-den}). Then we establish a difference formula of local density and apply it to cases that we are interested ($\S$\ref{diff-analytic}).
\par
In Part \ref{geo-part}, we define the Rapoport--Zink space and special cycles on it in $\S$\ref{rz-special}. Then we study the intersection of special divisors and the exceptional divisor in $\S$\ref{spe-exc-part}. The main results in this section allows us to give a complete decomposition of the special divisors. We also prove the linear invariance of the derived special cycles. In $\S$\ref{diff-geo-part}, we establish the difference formula on the geometric side. In $\S$\ref{main-proof-part}, we combine the geometric and analytic difference formulas and also calculate the base cases to prove the main theorem (Theorem \ref{main-intro}).
\par
Finally in Part \ref{app-part}, we give two applications of our results: A conjecture of Rapoport ($\S$\ref{rapoport-conj}) and the arithmetic intersections of Hecke correspondences ($\S$\ref{ari-int-hecke-app}).

\subsection{Acknowledgments}
We would like to thank Chao Li, Yifeng Liu, Michael Rapoport and Wei Zhang for helpful conversations and comments.

\subsection{Notations}
Throughout this article, we fix an odd prime $p$. Let $F/\mathbb{Q}_{p}$ be a finite extension, with ring of integers $\mathcal{O}_{F}$, and uniformizer $\varpi$. Let $\nu_\varpi:F\rightarrow\mathbb{Z}\cup\{\infty\}$ be the valuation map on $F$ which maps $\varpi$ to 1. Let $\mathbb{F}$ be the algebraic closure of $\mathbb{F}_p$. Let $\fb$ be the completion of the maximal unramified extension of $F$. Let $\ofb$ be the integer ring of the field $\fb$.
\subsubsection{On crystalline sites}
For a scheme or formal scheme $S$ over $\ofb$, denote by $\textup{NCRIS}_{\ofb}(S/\textup{Spec}\,\ofb)$ the big fppf nilpotent crystalline site of $S$ over $\ofb$ (cf. \cite[Definition B.5.7.]{FGL08}), the definition is the same as the crystalline site defined in the works of Berthelot, Breen and Messing \cite[$\S$1.1.1.]{BBM82} except that we replace the pd-structure by $\mathcal{O}_F$-pd-structure \cite[Definition B.5.1.]{FGL08} (Notice that they are equivalent when $F=\mathbb{Q}_p$). Denote by $\mathcal{O}_{S}^{\textup{crys}}$ the structure sheaf in this site. For a point $z\in S(\mathbb{F})$, let $\widehat{\mathcal{O}}_{S,z}$ be the complete local ring of $S$ at the point $z$. Let $\textup{Nilp}_{\ofb}$ be the category of $\ofb$-schemes on which $\varpi$ is locally nilpotent. For an object $S$ in $\textup{Nilp}_{\ofb}$, we use $\overline{S}$ to denote the scheme $S\times_{\ofb}\mathbb{F}$.
\subsubsection{On Grothendieck $K$-groups}
For a locally noetherian formal scheme $X$ together with a formal subscheme $Y$, denote by $K_0^{Y}(X)$ the Grothendieck group of finite complexes of coherent locally free $\rO_{X}$-modules acyclic outside Y . We use $K_0(X)$ to denote $K_0^{X}(X)$. Let $K_0^{\prime}(Y)$ be the Grothendieck group of coherent sheaves of $\rO_{Y}$-modules on $Y$.
\par
Denote by $\textup{F}^{i}K_0^{Y}(X)$ the codimension $i$ descending filtration on $K_0^{Y}(X)$ and $\textup{Gr}^{i}K_0^{Y}(X)$ its $i$-th graded piece. For an object $A\in\textup{F}^{i}K_0^{Y}(X)$, denote by $[A]$ its image in $\textup{Gr}^{i}K_0^{Y}(X)$. There is also an ascending filtration $\textup{F}_iK^{\prime}_0(X)$ on $K^{\prime}_0(X)$
\begin{equation*}
    \textup{F}_iK^{\prime}_0(X)=\bigcup\limits_{\substack{Z\subset X\\\textup{dim}Z\leq i}}\textup{Im}\left(K_0^{\prime}(Z)\rightarrow K_0^{\prime}(X)\right).
\end{equation*}
\par
From now on we assume that the formal scheme $X$ is regular of pure dimension $d$. Then we have natural isomorphisms $K_0^{Y}(X)\stackrel{\sim}\rightarrow K^{\prime}_0(Y)$ and $\textup{F}^{d-i}K_0^{Y}(X)\stackrel{\sim}\rightarrow\textup{F}_iK_0^{\prime}(Y)$ for all integers $i\leq d$.
\par
We also have a cup product $\cdot$ on $K_0^{Y}(X)_{\Q}\coloneqq K_0^{Y}(X)\otimes_{\Z}\Q$ defined by tensor product of complexes. Under the identification $K_0^{Y}(X)\stackrel{\sim}\rightarrow K^{\prime}_0(Y)$, the cup product is derived tensor product. For two formal subschemes $Z$ and $Y$ and two positive integers $j_1,j_2$, it is expected that \cite[(B.3)]{zhang2021AFL}
\begin{equation}
    \textup{F}^{j_1}K_0^{Z}(X)_{\Q}\cdot\textup{F}^{j_2}K_0^{Y}(X)_{\Q}\subset\textup{F}^{j_1+j_2}K_0^{Y\cap Z}(X)_{\Q}.
    \label{codim-add}
\end{equation}

\part{Analytic side}
\label{ana-part}

\section{Quadratic lattices and local densities}
\label{nota-qua-loc-den}
\subsection{Notations on quadratic spaces and lattices}
A quadratic space $(U,q_{U})$ over $F$ is a finite dimensional vector space $U$ over $F$ equipped with a quadratic form $q_{U}:U\rightarrow F$, the quadratic form $q_{U}$ induces a symmetric bilinear form given by 
\begin{align}
    (\cdot,\cdot):U\times U&\rightarrow F,\notag\\
    (x,y)&\mapsto q_{U}(x+y)-q_{U}(x)-q_{U}(y).\label{to}
\end{align}
We say a quadratic space $U$ is non-degenerate if the bilinear form $(\cdot,\cdot)$ is non-degenerate.
\par
An isometry between two quadratic spaces $(U,q_{U})$ and $(U^{\prime},q_{U^{\prime}})$ is a linear isomorphism $\phi:U\rightarrow U^{\prime}$ preserving quadratic forms, i.e., $q_{U^{\prime}}(\phi(x))=q_{U}(x)$ for any $x\in U$. In that case, we say $U$ and $U^{\prime}$ are called isometric.
\par
Given by a quadratic space $(U,q_{U})$ over $F$. For a nonzero number $s\in F$, denote by $(U[s], q_{U}[s])$ the quadratic space whose underlying linear space $U[s]=U$, and the quadratic form $q_{U}[s]$ defined as $q_{U}[s](x)=s\cdot q_{U}(x)$, for all $x\in U$.
\par
A quadratic lattice $(L,q_{L})$ is a finite free $\mathcal{O}_{F}$-module equipped with a quadratic form $q_{L}:L\rightarrow F$. The quadratic form $q_{L}$ also induces a symmetric bilinear form $L\times L\stackrel{(\cdot,\cdot)}\longrightarrow F$ by similar formula (\ref{to}). The quadratic form $q_L$ can be extended to the linear space $L_{F}\coloneqq L\otimes_{\mathcal{O}_{F}}F$. We say a quadratic lattice $L$ is non-degenerate if the quadratic space $L_{F}$ is non-degenerate. In the following paragraphs, we will always assume the lattice is non-degenerate.
\par
For a nonzero number $s\in F$, denote by $(L[s], q_{L}[s])$ the quadratic lattice whose underlying lattice $L[s]=L$, and the quadratic form $q_{L}[s]$ defined as $q_{L}[s](x)=s\cdot q_{L}(x)$, for all $x\in L$.
\par
For a quadratic lattice $L$ and a sublattice $M\subset L$, define $M^{\bot}\coloneqq\{x\in L:(x,M)=0\}$. For two quadratic lattices $(L_1,q_1)$ and $(L_2,q_2)$, define $L_1\obot L_2$ to be the quadratic lattice whose underlying $\of$-lattice is $L_1\oplus L_2$, and the quadratic form $q=q_1\oplus q_2$.
\par
We say a quadratic lattice is integral if $q_{L}(x)\in\mathcal{O}_{F}$ for all $x\in L$. For an integral lattice $L$, define $L^{\vee}=\{x\in L\otimes_{\mathcal{O}_{F}}F:(x,L)\subset \mathcal{O}_{F}\}$.
Define its fundamental invariants to be the unique sequence of integers $(a_1, \cdots , a_n)$ such that $0 \leq a_1\leq\cdots\leq a_n$, and $L^{\vee}/L\simeq\oplus_{i=1}^{n}\mathcal{O}_F/a_i$ as $\mathcal{O}_F$-modules. Define
\begin{equation*}
    \min(L)=a_1,\,\,\,\,\max(L)=a_n.
\end{equation*}
\begin{definition}
    We say a quadratic lattice $L$ is a vertex lattice if it is integral and all the fundamental invariants $a_i$ satisfy that $0\leq a_i\leq1$, i.e., 
    \begin{equation*}
        \pi L^{\vee}\subset L\subset L^{\vee}.
    \end{equation*}
    When all the fundamental invariants are $0$, i.e., $L=L^{\vee}$, we say the quadratic lattice $L$ is self-dual.
\end{definition}
Let $a_1^{\prime}<\cdots<a_r^{\prime}$ be all the different numbers appearing in $a_1,\cdots,a_n$. There exists a Jordan decomposition of the quadratic lattice $L$ as follows,
\begin{equation*}
    L\simeq\obot_{i=1}^{r}L_i,
\end{equation*}
where $L_i[\pi^{-a_i^{\prime}}]$ is a self-dual lattice.
\begin{lemma}
    The Jordan decomposition is unique in the sense that if we have two decompositions $L\simeq\obot_{i=1}^{r}L_i$ and $L\simeq\obot_{i=1}^{r^{\prime}}L_i^{\prime}$, where $L_i[a_i^{\prime}], L_i^{\prime}[\pi^{-b_i^{\prime}}]$ are self-dual quadratic lattices for some integers $a_i^{\prime},b_i^{\prime}\geq0$ and $a_1^{\prime}<\cdots<a_r^{\prime}, b_1^{\prime}<\cdots<b_r^{\prime}$. Then we must have $r^{\prime}=r$ and $L_i^{\prime}\simeq L_i$ (hence $a_i^{\prime}=b_i^{\prime}$) for all $i$.
    \label{unique-jordan-decom}
\end{lemma}
\begin{proof}
    This is proved in the works of R. Schulze-Pillot \cite[Corollary 5.21]{schulzepillot2021lecture}.
\end{proof}
 
\begin{example}
    Let $(U,q_{U})$ be a quadratic space. For an element $x\in U$, we use $\langle x\rangle$ to denote the rank $1$ quadratic lattice generated by $x$. It is non-degenerate if $q(x)\neq0$. It is integral if $q(x)\in\mathcal{O}_{F}$. Its fundamental invariant is $\nu_\varpi(2q(x))$. When $p$ is odd, the lattice $\lx$ is self-dual if and only if $\nu_\varpi(q(x))=0$.
\end{example}

\subsection{Some invariants}
Let $(U,q_{U})$ be a quadratic space. Let's assume that $\textup{dim}_{F}U=n$ and the symmetric bilinear form $(\cdot,\cdot)$ is nondegenerate. Let $\{x_{i}\}_{i=1}^{n}$ be a basis of $U$, and $t_{ij}=\frac{1}{2}(x_{i},x_{j})$, we define the discriminant of the quadratic space $U$ to be:
\begin{equation*}
    \textup{disc}(U) = (-1)^{n(n-1)/2}\textup{det}\left((t_{ij})\right)\in F^{\times}/(F^{\times})^{2}.
\end{equation*}
If $\{x_{i}\}_{i=1}^{n}$ is an orthogonal basis of $U$ then $t_{ij}=0$ if $i\neq j$ and $t_{ii}\neq 0$ by the nondegeneracy of $(\cdot,\cdot)$. The Hasse invariant of the quadratic space $U$ is
\begin{equation*}
    \epsilon(U) = \prod\limits_{i<j}(t_{ii},t_{jj})_{F},
\end{equation*}
\par
For a quadratic lattice $L$, we use $\textup{disc}(L)$ and $\epsilon(L)$ to denote the corresponding invariants on the quadratic space $L_{F}=L\otimes_{\mathcal{O}_{F}}F$. Recall that when $p$ is odd, quadratic spaces $U$ over $F$ are classified by the following three invariants:
\begin{equation*}
    \textup{dim}_{F}\,U,\,\,\,\,\,\textup{disc}(U),\,\,\,\,\,\epsilon(U).
\end{equation*}
i.e., two quadratic spaces $U$ and $U^{\prime}$ are isometric if and only if the above three invariants for $U$ and $U^{\prime}$ are the same. 
\par
Let $\chi_F = \left(\frac{\cdot}{\varpi}\right):F^{\times}/(F^{\times})^{2}\rightarrow \{\pm1, 0\}$ be the quadratic residue symbol. For a quadratic space $U$, define $\chi(U)=\chi(\textup{disc}\,U)$. For a quadratic lattice $L$, define $\chi(L)=\chi(\textup{disc}\,L_{F})$.
\par
When $p$ is odd, the quadratic space $U$ admits a self-dual sub-lattice if and only if $\epsilon(U)=+1$ and $\chi_{F}(U)\neq0$, we will use $H_{k}^{\varepsilon}$ to denote the unique self-dual lattice of rank $k$ and 
\begin{equation*}
    \chi_{F}(H_{k}^{\varepsilon})\coloneqq\chi_{F}(\textup{disc}(H_{k}^{\varepsilon}))=\varepsilon.
\end{equation*}
When $p=2$, let $H_{2n}^{+}=(H_{2}^{+})^{\obot n}$ be a self-dual lattice of rank $2n$, where the quadratic form on $H_{2}^{+}=\mathcal{O}_{F}^{2}$ is given by $(x,y)\in \mathcal{O}_{F}^{2}\mapsto xy$.
\label{quadratic}

\subsection{Local densities of quadratic lattices}
\begin{definition}
Let $L, M$ be two quadratic $\mathcal{O}_{F}$-lattices. Let $\textup{Rep}_{M,L}$ be the scheme of integral representations, an $\mathcal{O}_{F}$-scheme such that for any $\mathcal{O}_{F}$-algebra $R$,
\begin{equation*}
    \textup{Rep}(M,L)(R)=\textup{QHom}(L\otimes_{\mathcal{O}_{F}}R, M\otimes_{\mathcal{O}_{F}}R),
\end{equation*}
where \textup{QHom} denotes the set of injective module homomorphisms which preserve the quadratic forms. The local density of integral representations is defined to be
\begin{equation*}
    \textup{Den}(M,L)=\lim\limits_{d\rightarrow\infty}\frac{\#\textup{Rep}(M,L)(\mathcal{O}_{F}/\pi^{d})}{q^{d\cdot \textup{dim}(\textup{Rep}(M,L))_{F}}}.
\end{equation*}
\end{definition}
\begin{remark}
If $L,M$ have rank $n,m$ respectively and the generic fiber $(\textup{Rep}_{M,L})_{F}\neq\varnothing$, then $n\leq m$ and 
\begin{equation*}
    \textup{dim}(\textup{Rep}(M,L))_{F}=\textup{dim} \,\textup{O}_{m}-\textup{dim}\, \textup{O}_{m-n}=\tbinom{m}{2}-\tbinom{m-n}{2}=mn-\frac{n(n+1)}{2}.
\end{equation*}
\end{remark}
It is well-known that there exists a local density polynomial $\den(X,M,L)\in\Q[X]$ such that 
\begin{equation*}
    \den(q^{-k},M,L)=\den(M\obot H_{2k}^{+},L)
\end{equation*}
for all positive integers $k$.
\begin{definition}
Let $L, M$ be two quadratic $\mathcal{O}_{F}$-lattices. Let $\textup{PRep}_{M,L}$ be the $\mathcal{O}_{F}$-scheme of primitive integral representations such that for any $\mathcal{O}_{F}$-algebra $R$,
\begin{equation*}
    \textup{PRep}(M,L)(R)=\{\phi\in\textup{Rep}(M,L)(R):\textup{$\phi$ is an isomorphism between $L_{R}$ and a direct summand of $M_{R}$}\}.
\end{equation*}
where $L_{R}$ (resp. $M_{R}$) is $L\otimes_{\mathcal{O}_{F}}R$ (resp. $M\otimes_{\mathcal{O}_{F}}R$). The primitive local density is defined to be
\begin{equation*}
    \textup{Pden}(M,L)=\lim\limits_{d\rightarrow\infty}\frac{\#\textup{PRep}(M,L)(\mathcal{O}_{F}/\pi^{d})}{q^{d\cdot \textup{dim}(\textup{Rep}_{M,L})_{F}}}.
\end{equation*}
\end{definition}
\begin{remark}
For any positive integer $d$, a homomorphism $\phi\in\textup{Rep}(M,L)(\mathcal{O}_{F}/\pi^{d})$ or $\textup{Rep}(M,L)(\mathcal{O}_{F})$ is primitive if and only if $\overline{\phi} \coloneqq \phi\,\,\textup{mod}\,\pi\in\textup{PRep}(M,L)(\mathcal{O}_{F}/\pi)$, which is equivalent to $\textup{dim}_{\mathbb{F}_{q}}(\phi(L)+\pi\cdot M)/\pi\cdot M = \textup{rank}_{\mathcal{O}_{F}}(L)$. 
\end{remark}
\begin{example}
    Let $k$ be a positive integer. Let $\varepsilon\in\{\pm\}$. Let $s\in\rO_{F}$ be a nonzero number. Let $\lx$ be a rank $1$ quadratic lattice such that $q(x)=s$. it has been calculated explicitly that (\cite[(3.3.2.1)]{LZ22b})
    \begin{equation}
    \textup{Pden}(H_{k}^{\varepsilon}, \lx) = \begin{cases}
    1-q^{1-k}, & \textup{when $k$ is odd and $\pi\,\vert\, s$;}\\
    1+\varepsilon\chi_{F}(s)q^{(1-k)/2}, &\textup{when $k$ is odd and $\pi\nmid s$;}\\
    (1-\varepsilon q^{-k/2})(1+\varepsilon q^{1-k/2}), & \textup{when $k$ is even and $\pi\,\vert\, s$;}\\
    1-\varepsilon q^{-k/2}, & \textup{when $k$ is even and $\pi\nmid s$.}
    \end{cases}
    \label{rank1}
\end{equation}
\end{example}

\section{Difference formula on the analytic side}
\label{diff-analytic}
\subsection{Primitive elements in vertex lattices}
Let $(L,q)$ be a vertex quadratic lattices such that $L^{\vee}\neq L$ or $\pi^{-1}L$. Let $(\cdot,\cdot)$ be the bilinear form on $L$. There exists two self-dual lattices $(H_1,q_1)$ and $(H_2,q_2)$ such that $\textup{rank}_{\of}H_1,\textup{rank}_{\of}H_2\geq1$.
\begin{equation}
    L\simeq H_1[\pi]\obot H_2.
    \label{decom-vertex}
\end{equation}
By the uniqueness of the Jordan decomposition, We will fix an isomorphism (\ref{decom-vertex}) in the following paragraphs.
\par
We say an element $x\in L$ is primitive in $L$ if $x\notin \pi L$. Let $\textup{O}(L)$ be the orthogonal group of the lattice $L$, i.e., $\textup{O}(L)=\{g\in\textup{Aut}_{\of}(L):(g(x_1),g(x_2))=(x_1,x_2)\,\,\textup{for all }x_1,x_2\in L\}$.
\begin{lemma}
    Let $x=x_1+x_2\in L$ be a primitive element where $x_1\in H_1[\pi]$ and $x_2\in H_2$.
    \begin{itemize}
        \item [(a)]If $x_2$ is a primitive element in $H_2$, there exists an element $g\in\textup{O}(L)$ such that $g(x)\in H_2$.
        \item[(b)] If $x_2$ is not a primitive element in $H_2$, there exists an element $g\in\textup{O}(L)$ such that $g(x)\in H_1[\pi]$.
    \end{itemize}    
    \label{primitive-in-vertex}
\end{lemma}
\begin{proof}
    We first prove (a). If $\textup{rank}_{\of}H_2=1$, the primitivity of $x_2$ implies that $\nu_\varpi(q(x))=0$. We have the following decomposition
    \begin{equation*}
        L=\lx\obot\lx^{\bot}.
    \end{equation*}
    By the uniqueness of the Jordan decomposition (Lemma \ref{unique-jordan-decom}), there exist two isometric maps $\phi_1:\lx\xrightarrow{\sim}H_2$ and $\phi_2:\lx^{\bot}\xrightarrow{\sim}H_1[\pi]$. Define $g\in\textup{O}(L)$ to be the composition $L=\lx^{\bot}\obot\lx\xrightarrow{\phi_1\obot\phi_2}H_1[\pi]\obot H_2\xrightarrow{(\ref{decom-vertex})}L$, then $g(x)\in H_2$.
    \par
    If $\textup{rank}_{\of}H_2\geq2$. The primitivity of $x_2$ implies that there exists an element $y\in H_2$ such that $\F_q\cdot\overline{x}+\F_q\cdot\overline{y}\subset L/\pi L$ is a non-degenerate quadratic subspace since $\overline{x_2}\neq0$ and $H_2$ is self-dual. Then the lattice $M\coloneqq\of\cdot x+\of\cdot y$ is self-dual. We have the following decomposition
    \begin{equation*}
        L=M\obot M^{\bot}.
    \end{equation*}
    Then $M^{\bot}$ is a vertex lattice, therefore there exists two self-dual lattices $M_1$ and $M_2$ such that $M^{\bot}\simeq M_1[\pi]\obot M_2$. We get $L\simeq M_1[\pi]\obot M\obot M_2$. Lemma \ref{unique-jordan-decom} implies that there exist two isometric maps $\phi_1: M_1\xrightarrow{\sim}H_1$ and $\phi_2:M\obot M_2\xrightarrow{\sim}H_2$. Define $g\in\textup{O}(L)$ to be the composition $L\simeq M_1[\pi]\obot M\obot M_2\xrightarrow{\phi_1\obot\phi_2}H_1[\pi]\obot H_2\xrightarrow{(\ref{decom-vertex})}L$, then $g(x)\in H_2$.
    \par
    Now we prove (b). We know that $x_1$ is primitive in $H_1$ since $x_2$ is not primitive in $H_2$. If $\textup{rank}_{\of}H_1=1$, the primitivity of $x_1$ implies that $\nu_\varpi(q_1[\pi](x))=1$. Notice that for all elements $l\in L$, we have $\nu_\varpi\left((x,l)\right)=\nu_{\varpi}((x_1,l)+(x_2,l))\geq\min\{\nu_{\varpi}((x_1,l)),\nu_{\varpi}((x_2,l))\}\geq1$. We have the following decomposition
    \begin{equation*}
        L=\lx\obot\lx^{\bot}.
    \end{equation*}
    By the uniqueness of the Jordan decomposition (Lemma \ref{unique-jordan-decom}), there exist two isometric maps $\phi_1:\lx\xrightarrow{\sim}H_1[\pi]$ and $\phi_2:\lx^{\bot}\xrightarrow{\sim}H_2$. Define $g\in\textup{O}(L)$ to be the composition $L=\lx\obot\lx^{\bot}\xrightarrow{\phi_1\obot\phi_2}H_1[\pi]\obot H_2\xrightarrow{(\ref{decom-vertex})}L$, then $g(x)\in H_1[\pi]$.
    \par
    If $\textup{rank}_{\of}H_2\geq2$. The primitivity of $x_1$ in $H_1$ implies that there exists an element $y\in H_1$ such that $\F_q\cdot\overline{x_1}+\F_q\cdot\overline{y}\subset H_1/\pi H_1$ is a non-degenerate quadratic subspace since $\overline{x_1}\neq0$ and $H_1$ is self-dual. Let $M\coloneqq\of\cdot x+\of\cdot y\subset L$. The inner product matrix of $M$ under the basis $\{x,y\}$ is
    \begin{equation*}
        \begin{pmatrix}
            \pi q_1(x_1)+q_2(x_2) & \frac{\pi}{2}(x_1,y)\\
            \frac{\pi}{2}(x_1,y) & \pi q_1(y)
        \end{pmatrix}=\pi\cdot\begin{pmatrix}
            q_1(x_1) & \frac{1}{2}(x_1,y)\\
            \frac{1}{2}(x_1,y) & q_1(y)
        \end{pmatrix}+\pi^{2}\begin{pmatrix}
            q_2(x_2/\pi) & 0\\
            0 & 0
        \end{pmatrix}.
    \end{equation*}
    Therefore $M^{\vee}/M\simeq\of/\pi\oplus\of/\pi$. Notice that for all elements $l\in L$ and $m=ax+by\in M$, we have $\nu_\varpi\left((m,l)\right)=\nu_{\varpi}(a(x,l)+b(y,l))\geq\min\{\nu_{\varpi}((x_1,l)),\nu_{\varpi}((x_2,l)),\nu_{\varpi}((y,l))\}\geq1$. Then we have a decomposition
    \begin{equation*}
        L=M\obot M^{\bot}.
    \end{equation*}
    Then $M^{\bot}$ is a vertex lattice, therefore there exists two self-dual lattices $M_1$ and $M_2$ such that $M^{\bot}\simeq M_1[\pi]\obot M_2$. We get $L\simeq M\obot M_1[\pi]\obot M_2$. Lemma \ref{unique-jordan-decom} implies that there exist two isometric maps $\phi_1: M\obot M_1[\pi]\xrightarrow{\sim}H_1[\pi]$ and $\phi_2:M_2\xrightarrow{\sim}H_2$. Define $g\in\textup{O}(L)$ to be the composition $L\simeq M\obot M_1[\pi]\obot M_2\xrightarrow{\phi_1\obot\phi_2}H_1[\pi]\obot H_2\xrightarrow{(\ref{decom-vertex})}L$, then $g(x)\in H_2$.
\end{proof}

\begin{lemma}
    Let $n\geq3$ be a positive integer and $\varepsilon\in\{\pm\}$.
    \begin{itemize}
        \item [(a)]For a primitive element $x\in H_{n}^{\varepsilon}$, we have
    \begin{equation*}
        \lx^{\bot}\simeq\lx[-1]\obot H_{n-2}^{\varepsilon}.
    \end{equation*}
        \item[(b)] For a primitive element $x\in H_{n}^{\varepsilon}[\pi]$, we have
    \begin{equation*}
        \lx^{\bot}\simeq\lx[-1]\obot H_{n-2}^{\varepsilon}[\pi].
    \end{equation*}
    \end{itemize}
\label{primitive-self-dual}
\end{lemma}
\begin{proof}
    The quadratic lattice $H_{n}^{\varepsilon}$ is isometric to the quadratic lattice $H_{2}^{+}\obot H_{n-2}^{\varepsilon}$. For (a), we can assume $x\in H_2^{+}$, then it's easy to see that the orthogonal complement of $x$ in $H_2^{+}$ is isometric to $\lx[-1]$. The statement (b) follows from (a).
\end{proof}

\subsection{Difference formula for vertex lattices}
The main goal of this section is to prove the following formula.
\begin{theorem}
    Let $n_1, n_2\geq2$ be two integers. Let $\varepsilon_1,\varepsilon_2\in\{\pm\}$. Let $L=H_{n_1}^{\varepsilon_1}[\pi]\obot H_{n_2}^{\varepsilon_2}$. Let $M$ be a quadratic lattice of rank $m$. Let $\lx$ be a rank $1$ quadratic lattice generated by $x$ such that $n\coloneqq\nu_\varpi(q(x))\geq0$. Then
    \begin{align*}
        &\den\left(L,M\obot\lx\right)-q^{m+2-n_1-n_2}\cdot\den\left(L,M\obot\langle \pi^{-1}x\rangle\right)\\
        =&\,q^{m+1-n_2}\cdot\den\left(H_{n_1-2}^{\varepsilon_1}[\pi]\obot H_{n_2}^{\varepsilon_2}\obot\lx[-1],M\right)\cdot\pden\left(H_{n_1}^{\varepsilon_1},\lx[\pi^{-1}]\right)\\
        +&\,\den\left(H_{n_1}^{\varepsilon_1}[\pi]\obot H_{n_2-2}^{\varepsilon_2}\obot\lx[-1],M\right)\cdot\pden\left(H_{n_2}^{\varepsilon_2},\lx\right).
    \end{align*}
    \label{diff-loc-density}
\end{theorem}
\subsubsection{Decomposition of some sets}
\begin{lemma}
    Let $d\geq n+2$ be an integer. Then we have a decomposition of the set $\textup{Rep}(L,\lx)(\of/\pi^{d})$ as follows,
    \begin{align*}
        \textup{Rep}(L,\lx)(\of/\pi^{d})&=\bigsqcup\limits_{i=0}^{[n/2]}\bigsqcup\limits_{t\in\of/(\pi^{i})}\prep\left(L,\langle \pi^{-i}x\rangle[1+\pi^{d}q(x)^{-1}t]\right)(\of/\pi^{d-i})\\
        &\simeq\bigsqcup\limits_{i=0}^{[n/2]}\prep\left(L,\langle \pi^{-i}x\rangle\right)(\of/\pi^{d-i})^{q^{i}}.
    \end{align*}
\end{lemma}
\begin{proof}
    Let $\phi\in\textup{Rep}(L,\lx)(\of/\pi^{d})$. Let $i$ be the largest integer such that $\phi(\overline{x})\in\pi^{i}L/\pi^d L$. Then we must have $0\leq i\leq [\frac{n}{2}]$ because $\nu_\varpi(q(x))=n$. Define $\widetilde{\phi}:\langle\pi^{-i}x\rangle/\pi^{d-i}\langle\pi^{-i}x\rangle\rightarrow L/\pi^{d-i}L$ to be $\widetilde{\phi}(\overline{\pi^{-i}x})=\pi^{-i}\phi(\overline{x})$. Since we only know that $q\left(\phi(\overline{x})\right)\equiv q(x)\,\textup{mod}\,\pi^{d}$, there exists an element $t\in\of/\pi^{i}$ such that
    \begin{equation*}
        q\left(\pi^{-i}\phi(\overline{x})\right)\equiv\pi^{-2i}q(x)+\pi^{d-2i}t\equiv \pi^{-2i}q(x)\left(1+\pi^{d}q(x)^{-1}t\right)\,\textup{mod}\,\,\pi^{d-i}.
    \end{equation*}
    Hence $\widetilde{\phi}\in\prep\left(L,\langle \pi^{-i}x\rangle[1+\pi^{d}q(x)^{-1}t]\right)(\of/\pi^{d-i})$.
    \par
    Since $d\neq n+2$, we know that the lattice $\pi^{-i}x\rangle[1+\pi^{d}q(x)^{-1}t]$ is isometric to $ \pi^{-i}x\rangle$, therefore
    \begin{equation*}
        \bigsqcup\limits_{t\in\of/(\pi^{i})}\prep\left(L,\langle \pi^{-i}x\rangle[1+\pi^{d}q(x)^{-1}t]\right)(\of/\pi^{d-i})\simeq\prep\left(L,\langle \pi^{-i}x\rangle\right)(\of/\pi^{d-i})^{q^{i}}.
    \end{equation*}
\end{proof}
Let $0\leq i\leq[\frac{n}{2}]$ be an integer. Let $d\geq n+2$ be an integer. For an element $\phi\in\prep\left(L,\langle \pi^{-i}x\rangle\right)$, we have
\begin{equation*}
    \phi(\overline{\pi^{-i}x})=\overline{x_1}+\overline{x_2},
\end{equation*}
where $\overline{x_1}\in H_1[\pi]/\pi^{d-i}H_1[\pi]$ and $\overline{x_2}\in H_2/\pi^{d-i}H_2$. Define
\begin{align*}
    &\prep_1(L,\langle \pi^{-i}x\rangle)(\of/\pi^{d-i})\coloneqq\{\phi\in\prep\left(L,\langle \pi^{-i}x\rangle\right)(\of/\pi^{d-i}):\overline{x_2}\in\pi H_2/\pi^{d-i}H_2\},\\
    &\prep_2(L,\langle \pi^{-i}x\rangle)(\of/\pi^{d-i})\coloneqq\{\phi\in\prep\left(L,\langle \pi^{-i}x\rangle\right)(\of/\pi^{d-i}):\overline{x_2}\notin\pi H_2/\pi^{d-i}H_2\}.
\end{align*}
Then we have a decomposition 
\begin{equation*}
    \prep\left(L,\langle \pi^{-i}x\rangle\right)(\of/\pi^{d-i})=\prep_1(L,\langle \pi^{-i}x\rangle)(\of/\pi^{d-i})\bigsqcup\prep_2(L,\langle \pi^{-i}x\rangle)(\of/\pi^{d-i}).
\end{equation*}
\begin{lemma}
    Let $0\leq i\leq[\frac{n}{2}]$ be an integer. Let $d\geq n+2$ be an integer. We have
    \begin{align*}
    &\#\prep_1(L,\langle \pi^{-i}x\rangle)(\of/\pi^{d-i})=q^{n_2(d-i-1)+1}\cdot\#\prep\left(H_{n_1}^{\varepsilon_1},\langle\pi^{-i}x\rangle[\pi^{-1}]\right)(\of/\pi^{d-i}),\\
    &\#\prep_2(L,\langle \pi^{-i}x\rangle)(\of/\pi^{d-i})=q^{n_1(d-i)}\cdot\#\prep(H_{n_2}^{\varepsilon_2},\langle\pi^{-i}x\rangle)(\of/\pi^{d-i}).
\end{align*}
\end{lemma}
\begin{proof}
    For simplicity, in the proof we denote $H_{n_1}^{\varepsilon_1}$ by $H_1$, denote $H_{n_2}^{\varepsilon_2}$ by $H_2$.
    \par
    Let $\phi\in\prep_1\left(L,\langle \pi^{-i}x\rangle\right)(\of/\pi^{d-i})$. Then we have
    \begin{equation*}
        \phi(\overline{x})=\overline{x_1}+\overline{\pi x_2^{\prime}}\in L/\pi^{d-i}L,
    \end{equation*}
    where $\overline{x_1}\in H_1[\pi]/\pi^{d-i}H_1[\pi]$ and $\overline{x_1}\notin H_1[\pi]/\pi^{d-i}H_1[\pi]$, and $\overline{x_2^{\prime}}\in H_2/\pi^{d-i-1}H_2$. Therefore we get a map
    \begin{equation*}
        \pr_2:\prep_1(L,\langle \pi^{-i}x\rangle)(\of/\pi^{d-i})\rightarrow H_2/\pi^{d-i-1}H_2,\,\,\,\,\phi\mapsto \overline{x_2^{\prime}}.
    \end{equation*}
    Conversely, given an element $\overline{x_2^{\prime}}\in H_2/\pi^{d-i-1}H_2$, the inverse image of $\overline{x_2^{\prime}}$ under the map $\pr_2$ is
    \begin{align*}
        \pr_2^{-1}(\overline{x_2^{\prime}})=\bigsqcup\limits_{x\in\of/\pi}\{\overline{x_1}\in H_1[\pi]/\pi^{d-i}H_1[\pi]:\,\,&q(\overline{x_1})\equiv\pi^{-2i-1}q(x)-\pi q(\overline{x_2^{\prime}})+\pi^{d-i-1}x\,\textup{mod}\,\,\pi^{d-i},\\
        &\textup{and}\,\,\overline{x_1}\notin \pi H_1[\pi]/\pi^{d-i}H_1[\pi]\}.
    \end{align*}
    Notice that $\pi^{-2i-1}q(x)-\pi q(\overline{x_2^{\prime}})+\pi^{d-i-1}x\equiv\pi^{-2i-1}q(x)\,\,\textup{mod}\,\,\pi$. By the formula (\ref{rank1}), we get
    \begin{equation*}
        \#\pr_2^{-1}(\overline{x_2^{\prime}})=q\cdot\#\prep\left(H_{1},\langle\pi^{-i}x\rangle[\pi^{-1}]\right)(\of/\pi^{d-i}).
    \end{equation*}
    Therefore
    \begin{align*}
        \#\prep_1(L,\langle \pi^{-i}x\rangle)(\of/\pi^{d-i})&=\#H_2/\pi^{d-i-1}H_2\cdot q\cdot\#\prep\left(H_{1},\langle\pi^{-i}x\rangle[\pi^{-1}]\right)(\of/\pi^{d-i})\\
        &=q^{n_2(d-i-1)+1}\cdot\#\prep\left(H_{n_1}^{\varepsilon_1},\langle\pi^{-i}x\rangle[\pi^{-1}]\right)(\of/\pi^{d-i}).
    \end{align*}
    \par
    Now we compute $\#\prep_2(L,\langle \pi^{-i}x\rangle)(\of/\pi^{d-i})$. The idea is similar to the previous computations. Let $\phi\in\prep_2\left(L,\langle \pi^{-i}x\rangle\right)(\of/\pi^{d-i})$. Then we have
    \begin{equation*}
        \phi(\overline{x})=\overline{x_1}+\overline{x_2}\in L/\pi^{d-i}L,
    \end{equation*}
    where $\overline{x_1}\in H_1[\pi]/\pi^{d-i}H_1[\pi]$, $\overline{x_2}\in H_2/\pi^{d-i}H_2$ and $\overline{x_2}\notin \pi H_2/\pi^{d-i}H_2$. Therefore we get a map
    \begin{equation*}
        \pr_1:\prep_2(L,\langle \pi^{-i}x\rangle)(\of/\pi^{d-i})\rightarrow H_1[\pi]/\pi^{d-i}H_1[\pi],\,\,\,\,\phi\mapsto \overline{x_1}.
    \end{equation*}
    Conversely, given an element $\overline{x_1}\in H_1[\pi]/\pi^{d-i}H_1[\pi]$, the inverse image of $\overline{x_1}$ under the map $\pr_1$ is
    \begin{align*}
        \pr_1^{-1}(\overline{x_1})=\{\overline{x_2}\in H_2/\pi^{d-i}H_2:\,\,&q(\overline{x_2})\equiv q(x)- q(\overline{x_1})\,\textup{mod}\,\,\pi^{d-i},\\
        &\textup{and}\,\,\overline{x_2}\notin \pi H_2/\pi^{d-i}H_2\}.
    \end{align*}
    Notice that $q(x)- q(\overline{x_1})\equiv q(x)\,\,\textup{mod}\,\,\pi$. By the formula (\ref{rank1}), we get
    \begin{equation*}
        \#\pr_1^{-1}(\overline{x_1})=\#\prep\left(H_{2},\langle\pi^{-i}x\rangle\right)(\of/\pi^{d-i}).
    \end{equation*}
    Therefore
    \begin{align*}
        \#\prep_2(L,\langle \pi^{-i}x\rangle)(\of/\pi^{d-i})&=\#H_1[\pi]/\pi^{d-i}H_1[\pi]\cdot \#\prep\left(H_{2},\langle\pi^{-i}x\rangle\right)(\of/\pi^{d-i})\\
        &=q^{n_1(d-i)}\cdot\#\prep\left(H_{n_2}^{\varepsilon_2},\langle\pi^{-i}x\rangle\right)(\of/\pi^{d-i}).
    \end{align*}
\end{proof}
\subsubsection{A restriction map}
Let $d$ be a positive integer. We define a map $\rep(L,M\obot\lx)(\of/\pi^{d})\rightarrow\rep(L,\lx)(\of/\pi^{d})$ as follows: for an element $\phi\in\rep(L,M\obot\lx)(\of/\pi^{d})$, define
\begin{equation*}
    \res(\phi)(\overline{x})=\phi(\overline{x}).
\end{equation*}
\begin{lemma}
    Let $0\leq i\leq[\frac{n}{2}]$ be an integer. When $d\geq n+2$ is a large enough integer,
    \begin{itemize}
        \item [(a)]For an element $\phi\in\prep_1(L,\langle \pi^{-i}x\rangle)(\of/\pi^{d-i})$, we have
        \begin{equation*}
            \#\res^{-1}(\phi)=\#\rep(H_{n_1-2}^{\varepsilon_1}[\pi]\obot H_{n_2}^{\varepsilon_2}\obot \langle\pi^{-i}x\rangle[-1],M)(\of/\pi^{d})\cdot q^{(i+1)m}.
        \end{equation*}
        \item [(b)]For an element $\phi\in\prep_2(L,\langle \pi^{-i}x\rangle)(\of/\pi^{d-i})$, we have
        \begin{equation*}
            \#\res^{-1}(\phi)=\#\rep(H_{n_1}^{\varepsilon_1}[\pi]\obot H_{n_2-2}^{\varepsilon_2}\obot \langle\pi^{-i}x\rangle[-1],M)(\of/\pi^{d})\cdot q^{im}.
        \end{equation*}
    \end{itemize}  
    \label{size-res-inverse}
\end{lemma}
\begin{proof}
    Let $\{e_i\}_{i=1}^{m}$ be a basis of the lattice $M$.
    \par
    We first (a).  Let $\phi\in\prep_1(L,\langle \pi^{-i}x\rangle)(\of/\pi^{d-i})$ be an element. Let $\overline{\theta}=\phi(\overline{\pi^{-i}x})\in L/\pi^{d-i}L$. Let $\theta\in L$ be a lift of $\overline{\theta}$ such that $q(\theta)=\pi^{-2i}q(x)$. Then
    \begin{align*}
        \res^{-1}(\phi)=\{(\overline{x_1},\cdots,\overline{x_m})\in (L/\pi^{d}L)^{m}:(\overline{x_k},\overline{x_t})\equiv (\overline{e_k},\overline{e_t})\,\,\textup{mod}\,\,\pi^{d},(\overline{x_k},\overline{\pi^{i}\theta})\equiv 0\,\,\textup{mod}\,\,\pi^{d}.\}.
    \end{align*}
    By Lemma \ref{primitive-in-vertex} and Lemma \ref{primitive-self-dual}, we know that $L_1\coloneqq\langle\theta\rangle^{\bot}\simeq H_{n_1-2}^{\varepsilon_1}[\pi]\obot H_{n_2}^{\varepsilon_2}\obot \langle\pi^{-i}x\rangle[-1]$. We have an exact sequence
    \begin{equation}
        0\longrightarrow L_1\obot\langle\theta\rangle\stackrel{i_1}\longrightarrow L\longrightarrow Q_1\longrightarrow 0.
        \label{exact-sequence-1}
    \end{equation}
    The quotient $Q_1$ is a finite group since $L_1\obot\langle\theta\rangle$ and $L$ have the same rank. Tensoring the sequence (\ref{exact-sequence-1}) by $\of/\pi^{d}$ for a sufficiently large integer $d$, we get
    \begin{equation*}
        0\longrightarrow K_1\longrightarrow L_1/\pi^{d}L_1\obot\langle\theta\rangle/\pi^{d}\langle\theta\rangle\stackrel{\overline{i_1}}\longrightarrow L/\pi^{d}L\longrightarrow Q_1\longrightarrow 0,
    \end{equation*}
    and $\# K_1=\# Q_1=q^{n-2i-1}$. Let $\overline{i_1}^{m}=\overline{i_1}\times\cdots\times\overline{i_1}$.
    \par
    Claim: When $d$ is large enough, the map $(\overline{i_1}^{m})^{-1}(\res^{-1}(\phi))\rightarrow\res^{-1}(\phi)$ is surjective.
    \par
    \textit{Proof of the claim}: Let $(\overline{x_1},\cdots,\overline{x_m})\in\res^{-1}(\phi)$. Let $x_1,\cdots,x_m\in L$ be lifts of the elements $\overline{x_1},\cdots,\overline{x_m}\in L/\pi^{d}L$. Then for all integers $1\leq i\leq m$,
    \begin{equation*}
        x_i=\frac{(x_i,\theta)}{2q(\theta)}\cdot\theta+x_i^{\prime},
    \end{equation*}
    where $(x_i^{\prime},\theta)=0$. When $d$ is large enough, the element $\frac{(x_i,\theta)}{2q(\theta)}\in\of$. Hence $x_i^{\prime}\in L$. Therefore $x_i^{\prime}\in\langle\theta\rangle^{\bot}=L_1$. Hence
    \begin{equation*}
        (\overline{x_1},\cdots,\overline{x_m})=\left(\overline{\frac{(x_i,\theta)}{2q(\theta)}\cdot\theta},\cdots, \overline{\frac{(x_m,\theta)}{2q(\theta)}\cdot\theta}\right)+\left(\overline{x_1^{\prime}},\cdots,\overline{x_m^{\prime}}\right)\in\textup{Im}(\overline{i_1}^{m}).
    \end{equation*}
    \par
    Now we fix an isomorphism of $\of$-modules $\langle\theta\rangle\simeq\of$ such that $\theta$ is mapped to $1$. Then when $d$ is large enough,
    \begin{align*}
        \overline{i_1}^{-1}(\res^{-1}(\phi))=\{(\overline{x_1^{\prime}},\cdots,\overline{x_m^{\prime}})\in (L_1/\pi^{d}L_1)^{m},(\overline{a_1},\cdots,\overline{a_m})\in(\of&/\pi^{d})^{m}:\overline{a_k}\in\pi^{d-n+i}/\pi^{d},\\
        &(\overline{x_k^{\prime}},\overline{x_t^{\prime}})\equiv (\overline{e_k},\overline{e_t})\,\,\textup{mod}\,\,\pi^{d}.\}
    \end{align*}
    Therefore $(\overline{i_1}^{m})^{-1}(\res^{-1}(\phi))\simeq\rep(L_1,M)(\of/\pi^{d})\times(\of/\pi^{n-i})^{m}$. Hence
    \begin{equation*}
        \#\res^{-1}(\phi)=\frac{\#(\overline{i_1}^{m})^{-1}(\res^{-1}(\phi))}{\# K_1^{m}}=\#\rep(L_1,M)(\of/\pi^{d})\cdot q^{(i+1)m}
    \end{equation*}
\par
Now we prove (b). The idea is similar to that of (a). Let $\phi\in\prep_2(L,\langle \pi^{-i}x\rangle)(\of/\pi^{d-i})$ be an element. Let $\overline{\theta}=\phi(\overline{\pi^{-i}x})\in L/\pi^{d-i}L$. Let $\theta\in L$ be a lift of $\overline{\theta}$ such that $q(\theta)=\pi^{-2i}q(x)$. Then
    \begin{align*}
        \res^{-1}(\phi)=\{(\overline{x_1},\cdots,\overline{x_m})\in (L/\pi^{d}L)^{m}:(\overline{x_k},\overline{x_t})\equiv (\overline{e_k},\overline{e_t})\,\,\textup{mod}\,\,\pi^{d},(\overline{x_k},\overline{\pi^{i}\theta})\equiv 0\,\,\textup{mod}\,\,\pi^{d}.\}.
    \end{align*}
    By Lemma \ref{primitive-in-vertex} and Lemma \ref{primitive-self-dual}, we know that $L_2\coloneqq\langle\theta\rangle^{\bot}\simeq H_{n_1}^{\varepsilon_1}[\pi]\obot H_{n_2-2}^{\varepsilon_2}\obot \langle\pi^{-i}x\rangle[-1]$. We have an exact sequence
    \begin{equation}
        0\longrightarrow L_1\obot\langle\theta\rangle\stackrel{i_2}\longrightarrow L\longrightarrow Q_2\longrightarrow 0.
        \label{exact-sequence-2}
    \end{equation}
    The quotient $Q_2$ is a finite group since $L_2\obot\langle\theta\rangle$ and $L$ have the same rank. Tensoring the sequence (\ref{exact-sequence-2}) by $\of/\pi^{d}$ for a sufficiently large integer $d$, we get
    \begin{equation*}
        0\longrightarrow K_2\longrightarrow L_2/\pi^{d}L_2\obot\langle\theta\rangle/\pi^{d}\langle\theta\rangle\stackrel{\overline{i_2}}\longrightarrow L/\pi^{d}L\longrightarrow Q_2\longrightarrow 0,
    \end{equation*}
    and $\# K_2=\# Q_2=q^{n-2i}$. Let $\overline{i_2}^{m}=\overline{i_2}\times\cdots\times\overline{i_2}$.
    \par
    Claim: When $d$ is large enough, the map $(\overline{i_2}^{m})^{-1}(\res^{-1}(\phi))\rightarrow\res^{-1}(\phi)$ is surjective.
    \par
    \textit{Proof of the claim}: Let $(\overline{x_1},\cdots,\overline{x_m})\in\res^{-1}(\phi)$. Let $x_1,\cdots,x_m\in L$ be lifts of the elements $\overline{x_1},\cdots,\overline{x_m}\in L/\pi^{d}L$. Then for all integers $1\leq i\leq m$,
    \begin{equation*}
        x_i=\frac{(x_i,\theta)}{2q(\theta)}\cdot\theta+x_i^{\prime},
    \end{equation*}
    where $(x_i^{\prime},\theta)=0$. When $d$ is large enough, the element $\frac{(x_i,\theta)}{2q(\theta)}\in\of$. Hence $x_i^{\prime}\in L$. Therefore $x_i^{\prime}\in\langle\theta\rangle^{\bot}=L_2$. Hence
    \begin{equation*}
        (\overline{x_1},\cdots,\overline{x_m})=\left(\overline{\frac{(x_i,\theta)}{2q(\theta)}\cdot\theta},\cdots, \overline{\frac{(x_m,\theta)}{2q(\theta)}\cdot\theta}\right)+\left(\overline{x_1^{\prime}},\cdots,\overline{x_m^{\prime}}\right)\in\textup{Im}(\overline{i_2}^{m}).
    \end{equation*}
    \par
    Now we fix an isomorphism of $\of$-modules $\langle\theta\rangle\simeq\of$ such that $\theta$ is mapped to $1$. Then when $d$ is large enough,
    \begin{align*}
        \overline{i_2}^{-1}(\res^{-1}(\phi))=\{(\overline{x_1^{\prime}},\cdots,\overline{x_m^{\prime}})\in (L_2/\pi^{d}L_2)^{m},(\overline{a_1},\cdots,\overline{a_m})\in(\of&/\pi^{d})^{m}:\overline{a_k}\in\pi^{d-n+i}/\pi^{d},\\
        &(\overline{x_k^{\prime}},\overline{x_t^{\prime}})\equiv (\overline{e_k},\overline{e_t})\,\,\textup{mod}\,\,\pi^{d}.\}
    \end{align*}
    Therefore $(\overline{i_2}^{m})^{-1}(\res^{-1}(\phi))\simeq\rep(L_2,M)(\of/\pi^{d})\times(\of/\pi^{n-i})^{m}$. Hence
    \begin{equation*}
        \#\res^{-1}(\phi)=\frac{\#(\overline{i_2}^{m})^{-1}(\res^{-1}(\phi))}{\# K_2^{m}}=\#\rep(L_2,M)(\of/\pi^{d})\cdot q^{im}
    \end{equation*}
\end{proof}
\noindent\textit{Proof of Theorem \ref{diff-loc-density}}. Let $d\geq n+2$ be a large enough integer, we have
\begin{align}
    \#\rep(L,M\obot\lx)=\sum\limits_{i=0}^{[n/2]}&\#\prep_1(L,\langle \pi^{-i}x\rangle)(\of/\pi^{d-i})\cdot\#\res^{-1}(\phi_{1,i})\notag\\
    &+\#\prep_2(L,\langle \pi^{-i}x\rangle)(\of/\pi^{d-i})\cdot\#\res^{-1}(\phi_{2,i})\label{proof-1}
\end{align}
where $\phi_{1,i}\in\prep_1(L,\langle \pi^{-i}x\rangle)$ and $\phi_{2,i}\in\prep_2(L,\langle \pi^{-i}x\rangle)$. Let $L_{1,i}=H_{n_1-2}^{\varepsilon_1}[\pi]\obot H_{n_2}^{\varepsilon_2}\obot \langle\pi^{-i}x\rangle[-1]$  and $L_{2,i}=H_{n_1}^{\varepsilon_1}[\pi]\obot H_{n_2-2}^{\varepsilon_2}\obot \langle\pi^{-i}x\rangle[-1]$. By (\ref{proof-1}), we have
\begin{align*}
    \den(L,M\obot\lx)=\sum\limits_{i=0}^{[n/2]}&q^{m+1-n_2}\cdot q^{i(m+2-n_1-n_2)}\cdot\den(L_{1,i},M)\cdot\pden(H_{n_1}^{\varepsilon_1},\langle\pi^{-i}x\rangle[\pi^{-1}])\\
    &+q^{i(m+2-n_1-n_2)}\cdot\den(L_{2,i},M)\cdot\pden(H_{n_2}^{\varepsilon_2},\langle\pi^{-i}x\rangle).
\end{align*}
Therefore
\begin{align*}
        &\den\left(L,M\obot\lx\right)-q^{m+2-n_1-n_2}\cdot\den\left(L,M\obot\langle \pi^{-1}x\rangle\right)\\
        =&\,q^{m+1-n_2}\cdot\den\left(H_{n_1-2}^{\varepsilon_1}[\pi]\obot H_{n_2}^{\varepsilon_2}\obot\lx[-1],M\right)\cdot\pden\left(H_{n_1}^{\varepsilon_1},\lx[\pi^{-1}]\right)\\
        +&\,\den\left(H_{n_1}^{\varepsilon_1}[\pi]\obot H_{n_2-2}^{\varepsilon_2}\obot\lx[-1],M\right)\cdot\pden\left(H_{n_2}^{\varepsilon_2},\lx\right).
    \end{align*}
\subsection{Analytic difference formula}
Let $s\in F$ be a nonzero number. $H_0(s)$ be the following lattice over $\of$ of rank $4$,
\begin{equation*}
    H_0(s)=\left\{\begin{pmatrix}
        a & b\\
        s c & d
    \end{pmatrix}:a,b,c,d\in\of\right\}.
\end{equation*}
We equip the lattice $H_0(s)$ with the quadratic form given by the determinant morphism $\det:H_0(s)\rightarrow F$. Then
\begin{equation*}
    H_0(s)\simeq H_2^{+}[s]\obot H_2^{+}.
\end{equation*}
\par
Let $H_{n_1}^{\varepsilon_1}=H_2^{+}$ and $H_{n_2}^{\varepsilon_2}=H_{2k+2}^{+}$. Applying Lemma \ref{diff-loc-density} and the formula (\ref{rank1}), we get
\begin{lemma}
    Let $\lx$ be a rank $1$ quadratic lattice generated by $x$ such that $n\coloneqq\nu_\varpi(q(x))\geq0$. Let $L^{\flat}$ be a rank $2$ quadratic lattice over $\of$. Then
    \begin{align*}
        \den\left(X, H_0(\pi), L^{\flat}\obot\lx\right)-&X^{2}\cdot\den\left(X,H_0(\pi),L^{\flat}\obot\langle\pi^{-1}x\rangle\right)\\
        =&\begin{cases}
            \left(1-q^{-1}X\right)\cdot\den\left(X,\lx[-1]\obot H_2^{+}[\pi],L^{\flat}\right), &\textup{if $n=0$;}\\
            &\\
            \left(1-q^{-1}X\right)\left(1+X\right)\cdot\den\left(X,\lx[-1]\obot H_2^{+}[\pi],L^{\flat}\right)&\\
            \quad\quad\quad\,\,\,+(q-1)X^{2}\cdot\den\left(X,\lx[-1]\obot H_2^{+},L^{\flat}\right), &\textup{if $n=1$;}\\
            &\\
            \left(1-q^{-1}X\right)\left(1+X\right)\cdot\den\left(X,\lx[-1]\obot H_2^{+}[\pi],L^{\flat}\right) &\\
            \quad\quad\quad+2(q-1)X^{2}\cdot\den\left(X,\lx[-1]\obot H_2^{+},L^{\flat}\right),&\textup{if $n\geq2$.}
        \end{cases}
    \end{align*}\label{diff-ana}
\end{lemma}
\begin{definition}
    Let $L$ and $L^{\flat}$ be quadratic lattices of rank $3$ and $2$ over $\of$ respectively. Let $\lx$ be a rank $1$ quadratic lattice generated by $x$. Define the (normalized) \textit{derived local densities}
    \begin{align*}
        \partial\den\left(H_0(\pi), L\right)&\coloneqq-2\cdot\frac{\textup{d}}{\textup{d}X}\bigg\vert_{X=1}\frac{\den(X,H_0(\pi),L)}{\den(H_0(\pi),H_2^{+}\obot H_1^{+}[\pi])},\\
        \partial\den\left(H_0(\pi)^{\vee}, L\right)&\coloneqq-2\cdot\frac{\textup{d}}{\textup{d}X}\bigg\vert_{X=1}\frac{\den(X,H_0(\pi)^{\vee},L)}{\den(H_0(\pi)^{\vee},H_2^{+}\obot H_1^{+}[\pi^{-1}])},\\
        \partial\den\left(\lx[-1]\obot H_2^{+},L^{\flat}\right)&\coloneqq-\frac{\textup{d}}{\textup{d}X}\bigg\vert_{X=1}\frac{\den(X,\lx[-1]\obot H_2^{+},L^{\flat})}{\den(H_2^{+},H_1^{+})},\\
        \partial\den\left(\lx[-1]\obot H_2^{+}[\pi], L^{\flat}\right)&\coloneqq-\frac{\textup{d}}{\textup{d}X}\bigg\vert_{X=1}\frac{\den(X,\lx[-1]\obot H_2^{+}[\pi],L^{\flat})}{\den(H_2^{+}[\pi],H_1^{+}[\pi])}.\\
    \end{align*}
\end{definition}
Notice that $\den(H_2^{+},H_1^{+})=1-q^{-1}$ and hence $\den(H_2^{+}[\pi],H_1^{+}[\pi]])=q-1$. The following lemma follows from the above definition, Lemma \ref{diff-ana} and Lemma \ref{base:analytic} below.
\begin{lemma}
    Let $\mathbb{V}$ be the unique division quaternion algebra over $F$. Let $L^{\flat}\subset\mathbb{V}$ be an $\of$-lattice of rank $2$. Let $x\in\mathbb{V}$ be an element such that $\nu_\varpi(q(x))\geq0$ and $x\,\bot\, L^{\flat}$. Then
    \begin{align*}
        \partial\den\left(H_0(\pi), L^{\flat}\obot\lx\right)-&\partial\den\left(H_0(\pi), L^{\flat}\obot\langle\pi^{-1}x\rangle\right)\\
        =&\begin{cases}
            \partial\den\left(\lx[-1]\obot H_2^{+}[\pi], L^{\flat}\right), &\textup{if $n=0$;}\\
            &\\
            2\cdot\partial\den\left(\lx[-1]\obot H_2^{+}[\pi], L^{\flat}\right)&\\
            \,\,\,+\partial\den\left(\lx[-1]\obot H_2^{+}, L^{\flat}\right), &\textup{if $n=1$;}\\
            &\\
            2\cdot\partial\den\left(\lx[-1]\obot H_2^{+}[\pi], L^{\flat}\right)&\\
            +2\cdot\partial\den\left(\lx[-1]\obot H_2^{+}, L^{\flat}\right),&\textup{if $n\geq2$.}
        \end{cases}
    \end{align*}
    \label{der-diff}
\end{lemma}

\subsection{Two identities}
\begin{lemma}
    Let $\mathbb{B}$ be the unique division quaternion algebra over $\qp$. Let $L^{\flat}\subset\B$ be an $\zp$-lattice of rank $2$. Let $x\in\mathbb{B}$ be an element such that $x\,\bot\, L^{\flat}$ and $\nu_p(q(x))\geq\max\{\max(L^{\flat}),2\}$. Then
    \begin{equation*}
        \partial\den\left(\lx\obot H_2^{+}[p], L^{\flat}\right)=p^{2}\cdot\partial\den\left(\lx[p^{-1}]\obot H_2^{+} ,L^{\flat}[p^{-1}]\right)-1.
    \end{equation*}
    \label{analytic-calculations-1}
\end{lemma}
\begin{proof}
    Since $L^{\flat}$ has rank $2$, we can apply the results of Yang \cite{yanglocaldensity}. Let $0\leq a\leq b$ be the fundamental invariants of the quadratic lattice $L^{\flat}$. By direct computations, we get
    \begin{equation*}
        \partial\den\left(\lx\obot H_2^{+}[p], L^{\flat}\right)=\frac{1}{(p-1)^{2}}\cdot\begin{cases}
            ap^{(a+b+6)/2}-ap^{(a+b+2)/2}-2p^{a+2}+p^{2}+2p-1, &\textup{if $a\equiv b$\,\,mod $2$};\\
            ap^{(a+b+5)/2}-ap^{(a+b+3)/2}-2p^{a+2}+p^{2}+2p-1, &\textup{if $a\not\equiv b$\,\,mod $2$}.
        \end{cases}
    \end{equation*}
    And
    \begin{equation*}
        \partial\den\left(\lx[p^{-1}]\obot H_2^{+} ,L^{\flat}[p^{-1}]\right)=\frac{1}{(p-1)^{2}}\cdot\begin{cases}
            ap^{(a+b+2)/2}-ap^{(a+b-2)/2}-2p^{a}+2, &\textup{if $a\equiv b$\,\,mod $2$};\\
            ap^{(a+b+1)/2}-ap^{(a+b-1)/2}-2p^{a}+2, &\textup{if $a\not\equiv b$\,\,mod $2$}.
        \end{cases}
    \end{equation*}
    The identity in the lemma can be verified by comparing the above two formulas.
\end{proof}
\begin{lemma}
    Let $\mathbb{B}$ be the unique division quaternion algebra over $\qp$. Let $L^{\flat}\subset\B$ be an $\zp$-lattice of rank $2$. Let $x\in\mathbb{B}$ be an element such that $x\,\bot\, L^{\flat}$ and $\nu_p(q(x))\geq\max\{\max(L^{\flat}),1\}$. Then
    \begin{align*}
        &\partial\den\left(H_0(p)^{\vee}, L^{\flat}\obot\langle x\rangle\right)-\partial\den\left(H_0(p)^{\vee}, L^{\flat}\obot\langle p^{-1}x\rangle\right)\\
        =\,&\partial\den\left(H_0(p), L^{\flat}[p]\obot\langle x\rangle[p]\right)-\partial\den\left(H_0(p), L^{\flat}[p]\obot\langle p^{-1}x\rangle[p]\right).
    \end{align*}
    \label{analytic-calculations-2}
\end{lemma}
\begin{proof}
     Let $-1\leq a\leq b$ be the fundamental invariants of the quadratic lattice $L^{\flat}$. In this case, we have
     \begin{align*}
         &\den(H_0(p)^{\vee}\obot H_{2k}^{+},L^{\flat}\obot\lx)-p^{-2k}\cdot\den(H_0(p)^{\vee}\obot H_{2k}^{+},L^{\flat}\obot\langle p^{-1}x\rangle)\\
         =&p^{-6}\left(\den\left(H_0(p)\obot H_{2k}^{+}[p],L^{\flat}[p]\obot\lx[p]\right)-p^{-2k}\cdot\den\left(H_0(p)\obot H_{2k}^{+}[p],L^{\flat}[p]\obot\langle p^{-1}x\rangle[p]\right)\right)
         \\
         =&p^{-6}\bigg{(}p\cdot\pden\left(H_{2k+2}^{+},\lx\right)\cdot\den\left(H_{2k}^{+}[p]\obot H_{2}^{+}\obot\lx[-p],L^{\flat}[p]\right)\\
         &\,\,\,\,\,\,\,\,\,\,\,\,\,\,\,\,\,\,\,\,\,\,\,\,\,\,\,\,\,\,\,\,\,\,\,\,\,\,\,\,\,\,\,\,\,\,\,\,\,\,\,\,\,\,\,\,\,\,\,\,\,\,\,\,\,\,\,\,\,\,\,\,\,\,\,\,+\pden\left(H_2^{+},\lx[p]\right)\cdot\den\left(H_{2k+2}^{+}[p]\obot\lx[-p],L^{\flat}[p]\right)\bigg{)}.
     \end{align*}
     Let $f(X),f_1(X),f_2(X)$ be two polynomials such that for all $k\geq0$,
     \begin{align*}
     f(p^{-k})&=\den\left(H_2^{+}[p]\obot\lx[-p]\obot H_{2k}^{+},L^{\flat}[p]\right),\\
         f_1(p^{-k})&=\den\left(H_{2k}^{+}[p]\obot H_{2}^{+}\obot\lx[-p],L^{\flat}[p]\right),\\
         f_2(p^{-k})&=\den\left(H_{2k+2}^{+}[p]\obot\lx[-p],L^{\flat}[p]\right).
     \end{align*}
     By the main theorem of \cite{yanglocaldensity}, there exists two polynomials $R_1(X),R_2(X)$ such that
     \begin{align*}
         f(X)&=1+R_1(X)+R_2(X),\\
         f_1(X)=1+p^{-1}X^{-1}R_1(X)+p^{-2}&X^{-2}R_2(X),\,\,f_2(X)=1+X^{-1}R_1(X)+X^{-2}R_2(X).
     \end{align*}
     Then we conclude that
     \begin{align*}
        &\partial\den\left(H_0(p)^{\vee}, L^{\flat}\obot\langle x\rangle\right)-\partial\den\left(H_0(p)^{\vee}, L^{\flat}\obot\langle p^{-1}x\rangle\right)\\
        =\,&2p^{-1}\left((p+1)\cdot\partial\den\left(\lx[-p]\obot H_2^{+}[p],L^{\flat}[p]\right)-R_1^{\prime}(1)\right).
    \end{align*}
    The number $R_1^{\prime}(1)$ is given by the following explicit formula
    \begin{equation*}
        R_1^{\prime}(1)=\frac{1}{p-1}\cdot\begin{cases}
            (1+p)(1-p^{(a+b)/2+2}), &\textup{if $a\equiv b$\,mod 2};\\
            1+p-2p^{(a+b+1)+2}, &\textup{if $a\not\equiv b$\,mod 2}.
        \end{cases}
    \end{equation*}
    Then the identity in the lemma can be verified by using Lemma \ref{der-diff} and computations of $\partial\den\left(\lx[-p]\obot H_2^{+}[p],L^{\flat}[p]\right)$ in the proof of Lemma \ref{analytic-calculations-1}.
\end{proof}

\subsection{Base cases: the analytic side}
\label{base-ana}
\begin{lemma}
    Suppose $F=\qp$. Let $\varepsilon\in\zp^{\times}$ be an element. Then
    \begin{align*}
        &\den\left(X,H_0(p), H_2^{-}\obot H_1^{+}[\varepsilon p]\right)=\left(1-p^{-1}X\right)\cdot\begin{cases}
            \left(1-X\right)\left(1-(p-1)X-X^{2}\right), &\textup{if $\varepsilon=-1$};\\
            \left(1+X\right)\left(1+(p-1)X-X^{2}\right), &\textup{if $\varepsilon=1$.}
        \end{cases}\\  
        &\den\left(X,H_0(p),H_1^{+}[\varepsilon]\obot H_2^{-}[p]\right)=\left(1-p^{-1}X\right)\left(1-2X^{2}+X^{4}\right);\\
        &\den\left(X,H_0(p)^{\vee},H_1^{+}[\varepsilon p^{-1}]\obot H_2^{-}\right)=p^{-7}(p-1)\cdot\left(p-(1+p)X+X^{2}\right);\\
        &\den\left(X,H_0(p)^{\vee},H_2^{-}[p^{-1}]\obot H_1^{+}[\varepsilon]\right)=0.
    \end{align*}
    Hence
     \begin{align*}
        &\partial\den\left(H_0(p), H_2^{-}\obot H_1^{+}[\varepsilon p]\right)=-1;\\   
        &\partial\den\left(H_0(p),H_1^{+}[\varepsilon]\obot H_2^{-}[p]\right)=0;\\
        &\partial\den\left(H_0(p)^{\vee}, H_2^{-}[p^{-1}]\obot H_1^{+}[\varepsilon]\right)=0;\\   
        &\partial\den\left(H_0(p),H_1^{+}[\varepsilon p^{-1}]\obot H_2^{-}\right)=1;
    \end{align*}
    \label{base:analytic}
\end{lemma}
\begin{proof}
By Lemma \ref{diff-ana}, we have that for positive integers $k$,
\begin{align*}
    \den\left(p^{-k},H_0(p), H_2^{\varepsilon}\obot H_1^{+}[\varepsilon p]\right)&=(1-p^{-1-k})(1+p^{-k})\cdot\den(H_{2k}^{+},H_2^{\varepsilon})\\&+p^{1-2k}(1-p^{-1})\cdot\den(H_{2k+2}^{+},H_2^{\varepsilon}).
\end{align*}
By the calculations of $\den(H_{m}^{\varepsilon},H_{n}^{\varepsilon})$ in \cite[Definition 3.4.1, Definition 3.5.1]{LZ22b}, we have
\begin{equation*}
    \den(H_{2k+2}^{+},H_2^{\varepsilon})=(1-p^{-k-1})(1+\varepsilon p^{-k}).
\end{equation*}
Then we get the first formula.
\par
For the second formula, we have the following decomposition by Lemma \ref{diff-ana},
\begin{equation*}
    \den(H_{0}(p)\obot H_{2k}^{+},H_1^{+}[\varepsilon]\obot H_2^{-}[p])=\den(H_{0}(p)\obot H_{2k}^{+},H_1^{+}[\varepsilon])\cdot\den(H_2^{+}[p]\obot H_{2k}^{-\varepsilon},H_2^{-}[p]).
\end{equation*}
Both local densities $\den(H_{0}(p)\obot H_{2k}^{+},H_1^{+}[\varepsilon])$ and $\den(H_2^{+}[p]\obot H_{2k}^{-\varepsilon},H_2^{-}[p])$ can be computed by the formulas in \cite{yanglocaldensity}.
\par
For the third formula, notice that
\begin{equation*}
    \den\left(X,H_0(p)^{\vee},H_1^{+}[\varepsilon p^{-1}]\obot H_2^{-}\right)=p^{-6}\den\left(H_0(p)\obot H_{2k}^{+}[p],H_1^{+}[\varepsilon]\obot H_2^{-}[p]\right).
\end{equation*}
Again we can use Lemma \ref{diff-ana} to reduce the local densities to the cases where formulas in \cite{yanglocaldensity} can be applied.
\par
For the last one, we notice that $H_2^{-}$ can be mapped isometrically to $H_0(p)$, which implies that $H_2^{-}[p^{-1}]$ can't be mapped isometrically to $H_0(p)^{\vee}\obot H_{2k}^{+}$ for all $k\geq0$, hence the local density is identically zero.
\end{proof}

\part{Geometric side}
\label{geo-part}

\section{Rapoport-Zink spaces and special cycles}
\label{rz-special}
For simplicity, let $W=\zpb$. Let $\B$ be the unique division quaternion algebra over $\qp$. Let $\X$ be the unique (up to isomorphism) formal group of dimension 1 and height 2 over $\F$ with a principal polarization $\lambda:\X\rightarrow\X^{\vee}$.
\subsection{Rapoport-Zink spaces with hyperspecial level structures}
\label{hyperspecial0}
Let $\mathcal{N}_0$ be the following functor on the category $\textup{Nilp}_{W}$: for any $S\in \textup{Nilp}_{W}$, the set $\mathcal{N}_0(S)$ is the isomorphism classes of pairs $(X,\rho)$, where $X$ is a $p$-divisible group over $S$, $\rho$ is a height 0 quasi-isogenies between $p$-divisible groups $\rho: \mathbb{X}\times_{\mathbb{F}}\overline{S}\rightarrow X\times_{S}\overline{S}$. It is well-known that the functor $\mathcal{N}_0$ is represented by the formal scheme $\textup{Spf}\,W[[t]]$ over $\textup{Spf}\,W$.\par
Let $(X^{\textup{univ}},\rho^{\univ})$ be the universal $p$-divisible group over the formal scheme $\N_0$. Let $\mathbb{D}(X^{\univ})$ be the Dieudonne crystal of the $p$-divisible group $X^{\univ}$. It is a locally free $\mathcal{O}_{\N_0}^{\crys}$-module crystal of rank $2$. Given a morphism $S\rightarrow\N_0$ where $S$ is an object in $\textup{Nilp}_W$, let $\mathbb{D}(X^{\univ})_{S}$ be the pullback of the crystal $\mathbb{D}(X^{\univ})$ to the site $\textup{NCRIS}_{W}(S/\textup{Spec}\,W)$.
\par
Let $\mathbb{D}(\X)$ be the Dieudonne module of the $p$-divisible group $\X$. There exists a basis $[e_1,e_2]$ of the rank 2 free $W$-module $\mathbb{D}(\X)$ such that the Hodge filtration on $\mathbb{D}(\X)_{\F}\coloneqq\mathbb{D}(\X)\otimes_{W}\F$ is given by
    \begin{equation*}
        0\rightarrow\textup{Fil}^{1}\mathbb{D}(\X)_{\F}=\F\cdot \overline{e_2}\rightarrow\mathbb{D}(\X)_{\F}.
    \end{equation*}
Adjusting the element $t$ by some invertible element in the local ring $\rO_{\N_0}$, the Hodge filtration on the rank 2 free $\rO_{\N_0}$-module $\mathbb{D}(X^{\univ})_{\N_0}$ is given by 
\begin{equation*}
        0\rightarrow\textup{Fil}^{1}\mathbb{D}(X^{\univ})_{\N_0}=\rO_{\N_0}\cdot (e_2+te_1)\rightarrow\mathbb{D}(X^{\univ})_{\N_0}.
    \end{equation*}

\subsection{CM cycles: the hyperspecial case}
Recall we use $(X^{\textup{univ}},\rho^{\univ})$ be the universal $p$-divisible group over the formal scheme $\N_0$. Let $\B^{0}$ be the subgroup of $\B^{0}$ consisting of trace 0 elements.
\begin{definition}
    For any subset $H\subset\B^{0}$, define the CM cycle $\mathcal{Z}_{\N_0}(H)\subset\mathcal{N}_0$ to be the closed formal subscheme cut out by the condition,
\begin{equation*}
    \rho^{\textup{univ}}\circ x \circ (\rho^{\textup{univ}})^{-1} \in \textup{Hom}(X^{\textup{univ}}, X^{\textup{univ}}).
\end{equation*}
for all $x\in H$.
\label{spe-hyper-CM}
\end{definition}
    
\subsection{Special cycles on the product: the hyperspecial case}
Let $\N=\N_0\times_{W}\N_0$. It is a formal scheme which parameterizes two pairs $\left((X,\rho),(X^{\prime},\rho^{\prime})\right)$. Let $\left((X^{\textup{univ}},\rho^{\univ}),(X^{\prime\textup{univ}},\rho^{\prime\univ})\right)$ be the universal pairs over the formal scheme $\N$.
\begin{definition}
For any subset $H\subset\B$, define the special cycle $\mathcal{Z}_{\N}(H)\subset\mathcal{N}$ to be the closed formal subscheme cut out by the condition,
\begin{equation*}
    \rho^{\prime\textup{univ}}\circ x \circ (\rho^{\textup{univ}})^{-1} \in \textup{Hom}(X^{\textup{univ}}, X^{\prime\textup{univ}}).
\end{equation*}
for all $x\in H$.
\label{spe-hyper}
\end{definition}
\begin{lemma}
Let $x\in\mathbb{B}$ be a nonzero element such that $q(x)\in\zp$. Then $\mathcal{Z}_{\N}(x)$ is a Cartier divisor on $\mathcal{N}$ (i.e., defined by one nonzero equation) and flat over $W$ (i.e., the equation is not divisible by $p$). Moreover, let $x,y\in\B$ be two linearly independent elements, then the two divisors $\mathcal{Z}_{\N}(x)$ and $\mathcal{Z}_{\N}(y)$ intersect properly, and the irreducible components of the intersection $\mathcal{Z}_{\N}(x)\cap\mathcal{Z}_{\N}(y)$ are of the form $\textup{Spf}\,W_s$ where $W_s$ is the ring of definition of a quasi-canonical lifting of level $s$.
\label{proper-hpe}
\end{lemma}
\begin{proof}
    This is proved by Gross-Keating (see \cite[(5.10)]{GK93}).
\end{proof}
For $H=\{x\}$, we denote by $\mathcal{Z}_{\N}(x)$ the cycle $\mathcal{Z}_{\N}(\{x\})$. By the moduli interpretation of the special cycle $\mathcal{Z}_{\N}(x)$, there is a closed immersion $\mathcal{Z}_{\N}(p^{-1}x)\rightarrow\mathcal{Z}_{\N}(x)$.
\begin{definition}
    Let $x\in\mathbb{B}$ be a nonzero element. Define the difference divisor associated to $x$ on $\N$ to be the following effective Cartier divisor on the formal scheme $\N$,
\begin{equation*}
    \mathcal{D}_{\N}(x)=\mathcal{Z}_{\N}(x)-\mathcal{Z}_{\N}(p^{-1}x).
\end{equation*}
\label{diff-div-N}
\end{definition}
\begin{remark}
Terstiege first introduced the difference divisors on the unitary Rapoport--Zink spaces with hyperspecial level \cite{Ter13a} and proved the regularity of them \cite{Ter13b}. He also give the construction of difference divisors on the Rapoport--Zink space associated to a rank $4$ self-dual quadratic lattice over $\mathbb{Z}_p$ other than $\textup{M}_2(\mathbb{Z}_p)$ and studied the intersection numbers of them in \cite{Ter11}, where he also proved the regularity of these difference divisors.
\par
The formal scheme $\mathcal{N}$ is the Rapoport--Zink space associated to the lattice $\textup{M}_2(\mathbb{Z}_p)$. The difference divisors $\mathcal{D}_\N(x)$ are regular formal schemes (cf. \cite[Theorem 6.2.2]{Zhu23}, see also \cite{Zhu23diff}), it is formally smooth over $W$ if and only if $\nu_p(q(x))=0$. It's easy to see that $\mathcal{Z}_\N(1)\simeq\N_0$ and for an element $y\in\B^{0}$, we have the following isomorphism:
    \begin{equation*}
        \mathcal{Z}_{\N_0}(y)\simeq\mathcal{Z}_\N(y)\cap\mathcal{Z}(1).
    \end{equation*}
    Let $\mathcal{D}_{\N_0}(y)=\mathcal{Z}_{\N_0}(y)-\mathcal{Z}_{\N_0}(p^{-1}y)$. It is a regular divisor on the formal scheme $\N_0$. It is isomorphic to the ring of definition of a (quasi-)canonical lifting of $\mathbb{X}$ and formally smooth over $W$ if and only if $\nu_p(q(y))=0$.
    \label{regularity-diff-smooth}
\end{remark}

\subsection{Rapoport-Zink spaces with cyclic level structures}
\label{cyclic-def-space}
We have an isomorphism
\begin{equation}
    \iota:\B\simeq\textup{End}^{\circ}(\X)\coloneqq\textup{End}(\X)\otimes_{\Z}\Q.
    \label{end-div}
\end{equation}
\par
Let $x\mapsto \overline{x}$ be the main involution of the quaternion algebra $\B$. Let $x\mapsto x^{\vee}$ be the Rosati involution on $\textup{End}^{\circ}(\X)$ induced by the principal polarization $\lambda$. The isomorphism (\ref{end-div}) can be chosen such that the the main involution on $\B$ transforms to the Rosati involution on $\textup{End}^{\circ}(\X)$.
\par
On the algebra $\B$, denote by $q$ the quadratic form $q(x)=x\cdot\overline{x}\in\qp$. On the algebra $\textup{End}^{\circ}(\X)$, denote by $q_\lambda$ the quadratic form $q_{\lambda}(x)=x\circ x^{\vee}\in\qp$. The isomorphism $\iota$ is an isometry between the quadratic spaces $(\B,q)$ and $(\textup{End}^{\circ}(\X),q_{\lambda})$. Moreover, the maximal order $\mathcal{O}_{\B}$ of $\B$ is mapped isometrically to $\textup{End}(\X)$.
\par
For all $x\in\B$, we consider the following contravariant set-valued functor $\mathcal{N}_0(x)$ defined over $\textup{Nilp}_{W}$: for every $S\in \textup{Nilp}_{W}$, the set $\mathcal{N}_{0}(x)(S)$ consists of the isomorphism classes of elements of the following form $(X\stackrel{\pi}\rightarrow X^{\prime},(\rho,\rho^{\prime}))$, where
\begin{itemize}
    \item [(a)]$((X,\rho),(X^{\prime},\rho^{\prime}))\in\mathcal{N}(S)$;
    \item[(b)] $\pi:X\rightarrow X^{\prime}$ is a cyclic isogeny (i.e., $\textup{ker}(x)$ is a cyclic group scheme over $S$ in the sense of \cite[$\S$6.1]{KM85}) lifting $\rho^{\prime}\circ x\circ\rho^{-1}$.
\end{itemize}
There is a natural morphism $\textup{st}_x$ from $\N_0(x)$ to $\N$ given as follows,
\begin{align}
    \textup{st}_x: \N_0(x)&\longrightarrow \N;\label{st}\\
    (X\stackrel{\pi}\rightarrow X^{\prime},(\rho,\rho^{\prime}))&\longmapsto \left((X,\rho),(X^{\prime},\rho^{\prime})\right).\notag
\end{align}
Recall that by definition $\N=\N_0\times\N_0$. Let $\textup{s}_{x}:\N_0(x)\rightarrow\N_0$ be the composition of the morphism $\st_x$ with the projection to the first factor $\N\rightarrow\N_0$, and let $\textup{t}_{x}:\N_0(x)\rightarrow\N_0$ be the composition of the morphism $\st_x$ with the projection to the second factor $\N\rightarrow\N_0$. Let $\st_x^{\#}:\rO_{\N}\rightarrow\mathcal{O}_{\N_0(x)}$, $\textup{s}_x^{\#}:\rO_{\N_0}\rightarrow\rO_{\N_0(x)}$ and $\textup{t}_x^{\#}:\rO_{\N_0}\rightarrow\rO_{\N_0(x)}$ be the corresponding local ring homomorphisms.
\begin{lemma}
    Let $x\in\B$ be a nonzero element such that $\nu_p(q(x))=n$ for some positive integer $n\geq0$. The morphism $\st_x$ is a closed immersion and identifies $\N_0(x)$ as an effective Cartier divisor on $\N$. We have the following equality of Cartier divisors on $\N$,
    \begin{equation*}
        \mathcal{D}_{\N}(x)=\N_0(x).
    \end{equation*}
    Moreover,
    \begin{itemize}
        \item [(a)] The functor $\mathcal{N}_0(x)$ is represented by a 2 dimensional regular local ring. 
        \item [(b)] There exists two elements $t,t^{\prime}\in\rO_{\N_0(x)}$ such that
        \begin{itemize}
            \item[(b1)] There exists an element $t_0\in\rO_{\N_0}$ such that $\rO_{\N_0}\simeq W[[t_0]]$ and $\textup{s}_x^{\#}(t_0)=t$, $\textup{t}_x^{\#}(t_0)=t^{\prime}$.
            \item[(b2)] There exists an invertible element $\nu\in\rO_{\N_0(x)}$ such that
            \begin{equation*}
                 \rO_{\mathcal{D}_\N(x)}=\rO_{\N_0(x)}\simeq\begin{cases}
                    W[[t,t^{\prime}]]/\left(\nu p+(t-t^{\prime p^{n}})(t^{p^{n}}-t^{\prime})\left(\prod\limits_{\substack{a+b=n\\a,b\geq1}}(t^{p^{a-1}}-t^{\prime p^{b-1}})\right)^{p-1}\right), &\textup{if $n\geq1$;}\\
                    W[[t,t^{\prime}]]/(t-t^{\prime}), &\textup{if $n=0$.}
                \end{cases}
            \end{equation*}
        \end{itemize}
    \end{itemize}   
    \label{cyclic-diff}
\end{lemma}
\begin{proof}
    The equality $\mathcal{D}_{\mathcal{N}}(x)=\mathcal{N}_0(x)$ is proved by Zhu \cite[Theorem 6.2.3]{Zhu23}. The statements (a) and (b) are proved by Katz and Mazur \cite{KM85}.
\end{proof}

\subsection{Blow up of the cyclic deformation space}
For an element $x\in\B$ such that $\nu_p\left(q(x)\right)\geq0$. Let $\pi_x:\widetilde{\N}_0(x)\rightarrow\N_0(x)$ be the blow up morphism of the formal scheme $\N_0(x)$ along its unique closed $\F$-point.
\begin{lemma}
    For an element $x\in\B$ such that $n\coloneqq\nu_p\left(q(x)\right)\geq0$. Then
    \begin{itemize}
        \item [(a)] The formal scheme $\widetilde{\N}_0(x)$ is a 2-dimensional regular formal scheme and $\exc_{\widetilde{\N}_0(x)}\simeq\mathbb{P}_{\F}^{1}$. We also have $[\mathcal{O}_{\exc_{\widetilde{\N}_0(x)}}\otimes^{\mathbb{L}}_{\mathcal{O}_{\widetilde{\N}_0(x)}}\mathcal{O}_{\exc_{\widetilde{\N}_0(x)}}]=\mathcal{O}_{\mathbb{P}_{\mathbb{F}}^{1}}(-1)$.
        \item[(b)] If $n\geq1$, the multiplicity $r(n)$ of the exceptional divisor $\exc_{\widetilde{\N}_0(x)}$ in the divisor $\textup{div}(p)=\widetilde{\N}_0(x)_{\F}$ is given by
    \begin{equation*}
        r(n)=p^{[n/2]}\cdot\begin{cases}
            1+p^{-1}, &\textup{if $n$ is even};\\
            2, &\textup{if $n$ is odd}.
        \end{cases}
    \end{equation*}
    \item[(c)] Let $\mathcal{C}=\textup{Spf}\,R\subset\N_0(x)$ be a regular horizontal divisor where $R$ is a regular local ring, then
        \begin{equation*}
            \pi_x^{\ast}\mathcal{C}=\widetilde{\mathcal{C}}+\exc_{\widetilde{\N}_0(x)},
        \end{equation*}
        where $\widetilde{\mathcal{C}}\subset\widetilde{\N}_0(x)$ is the strict transform of $\mathcal{C}$ under the blow up morphism $\pi_x$, and isomorphic to $\mathcal{C}$ under the morphism $\pi_x$. Moreover,
        \begin{equation*}
            [\mathcal{O}_{\pi_x^{\ast}\mathcal{C}}\otimes^{\mathbb{L}}_{\mathcal{O}_{\widetilde{\N}_0(x)}}\mathcal{O}_{\exc_{\widetilde{\N}_0(x)}}]=\mathcal{O}_{\mathbb{P}_{\mathbb{F}}^{1}}(0),\,\,\,\,[\mathcal{O}_{\widetilde{\mathcal{C}}}\otimes^{\mathbb{L}}_{\mathcal{O}_{\widetilde{\N}_0(x)}}\mathcal{O}_{\exc_{\widetilde{\N}_0(x)}}]=\mathcal{O}_{\mathbb{P}_{\mathbb{F}}^{1}}(1).
        \end{equation*}
    \item[(d)] Let $\mathcal{C}_1=\textup{Spf}\,R_1,\mathcal{C}_2=\textup{Spf}\,R_2\subset\N_0(x)$ be two regular horizontal divisors such that $R_1$ and $R_2$ are two regular local rings, then
        \begin{equation*}
            \chi(\widetilde{\N}_0(x),\mathcal{O}_{\pi_x^{\ast}\mathcal{C}_1}\otimes_{\rO_{\widetilde{\N}_0(x)}}^{\mathbb{L}}\mathcal{O}_{\pi_x^{\ast}\mathcal{C}_2})=\chi(\N_0(x),\mathcal{O}_{\mathcal{C}_1}\otimes_{\rO_{\N_0(x)}}^{\mathbb{L}}\mathcal{O}_{\mathcal{C}_2}),
        \end{equation*}
    \end{itemize}
    \label{blow-up-N0(x)}
\end{lemma}
\begin{proof}
    We only give the proof for $n\geq1$ (the case $n=0$ is similar and easier). We first prove (a) and (b). By Lemma \ref{cyclic-diff} (b2), the formal scheme $\widetilde{\N}_0(x)$ is covered by the following two open formal subschemes:
    \begin{equation}
        U_0=\textup{Spf}\,W[u][[t]]\big{/}\left(\nu p+t^{r(n)}\cdot f_1(u,t)\right),\,\,\,\,U_1=\textup{Spf}\,W[v][[t^{\prime}]]\big{/}\left(\nu p+t^{\prime r(n)}\cdot f_2(v,t^{\prime})\right),
        \label{explicit-cover}
    \end{equation}
    where $uv=1$ and $f_1(u,t)\in W[u][[t]],f_2(v,t^{\prime})\in W[v][[t^{\prime}]]$ are two non-unit elements such that $f_1(u,0)\neq0$ and $f_2(v,0)\neq0$. Therefore (a) and (b) are true by the explicit description of $\widetilde{\N}_0(x)$ in (\ref{explicit-cover}).
    \par
    Now we prove (c). By the regularity of the divisor $\mathcal{C}$, the equation $f_{\mathcal{C}}\in \mathcal{O}_{\N_0(x)}\simeq W[[t,t^{\prime}]]/(d_x)$ cutting out the divisor $\mathcal{C}$ must take the following form
    \begin{equation*}
        f_{\mathcal{C}}\equiv a\cdot t+a^{\prime}\cdot t^{\prime}\,\,\textup{mod}\,\,(p,t,t^{\prime})^{2},
    \end{equation*}
    where $a,a^{\prime}\in W$ and at least one of them is a unit. Then under the explicit equation (\ref{explicit-cover}) of $\widetilde{\N}_0(x)$, the equation $f_{\pi^{\ast}\mathcal{C}}$ of the pullback $\pi_x^{\ast}\mathcal{C}$ takes the following form:
    \begin{equation*}
        f_{\pi^{\ast}\mathcal{C}}\vert_{U_0}=t\cdot(a+a^{\prime}u)\,\,\textup{mod}\,\,t^{2},\,\,\,\, f_{\pi^{\ast}\mathcal{C}}\vert_{U_1}=t^{\prime}\cdot(av+a^{\prime})\,\,\textup{mod}\,\,t^{\prime2}.
    \end{equation*}
    Therefore we have $\pi^{\ast}\mathcal{C}=\exc_{\widetilde{\N}_0(x)}+\mathcal{C}_1$ where $\mathcal{C}_1$ is an irreducible divisor on $\widetilde{\N}_0(x)$. The set $\mathcal{C}_1\cap\exc_{\widetilde{\N}_0(x)}$ consists of only one point since at least one of $a,a^{\prime}$ is a unit. Therefore
    \begin{equation*}
        [\mathcal{O}_{\widetilde{\mathcal{C}}}\otimes^{\mathbb{L}}_{\mathcal{O}_{\widetilde{\N}_0(x)}}\mathcal{O}_{\exc_{\widetilde{\N}_0(x)}}]=\mathcal{O}_{\mathbb{P}_{\mathbb{F}}^{1}}(1).
    \end{equation*}
    Notice that the maximal ideal of $\mathcal{C}$ is already principal since $\mathcal{C}$ is a 1-dimensional regular local ring, therefore the strict transform of $\widetilde{\mathcal{C}}$ is isomorphic to $\mathcal{C}$ under the morphism $\pi_x$. Notice that we have $\widetilde{\mathcal{C}}\subset\mathcal{C}_1$, hence $\widetilde{\mathcal{C}}=\mathcal{C}_1$ by the irreducibility of both divisors. Therefore
    \begin{equation*}
        [\mathcal{O}_{\pi_x^{\ast}\mathcal{C}}\otimes^{\mathbb{L}}_{\mathcal{O}_{\widetilde{\N}_0(x)}}\mathcal{O}_{\exc_{\widetilde{\N}_0(x)}}]=[\mathcal{O}_{\widetilde{\mathcal{C}}}\otimes^{\mathbb{L}}_{\mathcal{O}_{\widetilde{\N}_0(x)}}\mathcal{O}_{\exc_{\widetilde{\N}_0(x)}}]+[\mathcal{O}_{\exc_{\widetilde{\N}_0(x)}}\otimes^{\mathbb{L}}_{\mathcal{O}_{\widetilde{\N}_0(x)}}\mathcal{O}_{\exc_{\widetilde{\N}_0(x)}}]=\mathcal{O}_{\mathbb{P}_{\F}^{1}}(0).
    \end{equation*}
    \par
    Now we prove (d). Using the intersection pairing on the divisors on a regular surface, we have
    \begin{align*}
        \chi(\widetilde{\N}_0(x),\mathcal{O}_{\pi_x^{\ast}\mathcal{C}_1}\otimes_{\rO_{\widetilde{\N}_0(x)}}^{\mathbb{L}}\mathcal{O}_{\pi_x^{\ast}\mathcal{C}_2})&=\left(\pi_x^{\ast}\mathcal{C}_1\cdot\pi_x^{\ast}\mathcal{C}_2\right)_{\widetilde{\N}_0(x)}=\left((\widetilde{\mathcal{C}}_1+\exc_{\widetilde{\N}_0(x)})\cdot\pi_x^{\ast}\mathcal{C}_2\right)_{\widetilde{\N}_0(x)}\\
        &=\left(\widetilde{\mathcal{C}}_1\cdot\pi_x^{\ast}\mathcal{C}_2\right)_{\widetilde{\N}_0(x)}\overset{\textup{projection formula}}{=}\left(\mathcal{C}_1\cdot\mathcal{C}_2\right)_{\N_0(x)}\\
        &=\chi(\N_0(x),\mathcal{O}_{\mathcal{C}_1}\otimes_{\rO_{\N_0(x)}}^{\mathbb{L}}\mathcal{O}_{\mathcal{C}_2}).
    \end{align*}
\end{proof}

\subsection{Rapoport-Zink space with level $\Gp$}
\label{description-x0}
Let $x_0\in\B$ be an element such that $\nu_p(q(x_0))=1$. The cyclic deformation space $\mathcal{N}_0(x_0)$ is a Rapoport-Zink space with level $\Gp$. Let 
\begin{equation*}
    \left((X^{\univ}\stackrel{x_0^{\univ}}\longrightarrow X^{\prime\univ}),(\rho^{\univ},\rho^{\prime\univ})\right)
\end{equation*}
be the universal deformation of the quasi-isogeny $x_0$ over the formal scheme $\N_0(x_0)$. By Lemma \ref{cyclic-diff} (b), there exists two elements $t,t^{\prime}\in\rO_{\N_0(x_0)}$ such that
    \begin{equation}
        \rO_{\N_0(x_0)}\simeq W[[t,t^{\prime}]]/\left(\nu p+(t^{p}-t^{\prime})(t-t^{\prime p})\right)
        \label{equation-x0}
    \end{equation}
\par
Let $\mathfrak{m}\subset\rO_{\N_0(x_0)}$ be the maximal ideal of the local ring $\rO_{\N_0(x_0)}$. Let $\N_0(x_0)_{\F}\coloneqq\N_0(x_0)\times_{W}\F$ be the reduction mod $p$ of the formal scheme $\N_0(x_0)$, it is a formal scheme over $\F$. The following lemma is clear from the above description of the local ring $\rO_{\N_0(x_0)}$.
\begin{lemma}
    The following facts hold:
    \begin{itemize}
        \item [(a)]The element $p\in\mathfrak{m}^{2}\backslash\mathfrak{m}^{3}$.
        \item [(b)]The formal scheme $\N_0(x_0)_{\F}$ has two irreducible components $\N_0(x_0)^{\textup{F}}$ and $\N_0(x_0)^{\textup{V}}$ such that:\\
        $\bullet$\,The two irreducible components $\N_0(x_0)^{\textup{F}}$ and $\N_0(x_0)^{\textup{V}}$ intersect properly at the unique closed point of $\N_0(x_0)_{\F}$.\\
        $\bullet$\,The formal scheme $\N_0(x_0)^{\textup{F}}\simeq\textup{Spf}\,\F[[t]]$. Over the formal scheme $\N_0(x_0)^{\textup{F}}$, the universal isogeny $X^{\univ}\stackrel{x_0^{\univ}}\longrightarrow X^{\prime\univ}$ is isomorphic to the Frobenius morphism. \\
        $\bullet$\,The formal scheme $\N_0(x_0)^{\textup{V}}\simeq\textup{Spf}\,\F[[t^{\prime}]]$. Over the formal scheme $\N_0(x_0)^{\textup{V}}$, the universal isogeny $X^{\univ}\stackrel{x_0^{\univ}}\longrightarrow X^{\prime\univ}$ is isomorphic to the Verschiebung morphism. 
    \end{itemize}   
    \label{semistable}
\end{lemma}
\begin{proof}
    By the equation (\ref{equation-x0}), the irreducible component $\N_0(x_0)^{\textup{F}}$ is given by the equation $t^{p}-t^{\prime}$, while the irreducible component $\N_0(x_0)^{\textup{V}}$ is given by the equation $t-t^{\prime p}$. All the statements follow easily.
\end{proof}
\begin{proposition}
    There exists a basis $e,f$ for the rank 2 free $\rO_{\N_0(x_0)}$-module $\mathbb{D}(X^{\univ})_{\N_0(x_0)}$ and a basis $e^{\prime},f^{\prime}$ for the rank 2 free $\rO_{\N_0(x_0)}$-module $\mathbb{D}(X^{\prime\univ})_{\N_0(x_0)}$ such that the Hodge filtrations are given by
    \begin{equation}
        0\rightarrow\textup{Fil}^{1}\mathbb{D}(X^{\univ})_{\N_0(x_0)}=\rO_{\N_0(x_0)}\cdot(f+x e)\rightarrow\mathbb{D}(X^{\univ})_{\N_0(x_0)},
    \end{equation}
    \begin{equation}
        0\rightarrow\textup{Fil}^{1}\mathbb{D}(X^{\prime\univ})_{\N_0(x_0)}=\rO_{\N_0(x_0)}\cdot(f^{\prime}+y e^{\prime})\rightarrow\mathbb{D}(X^{\prime\univ})_{\N_0(x_0)},
    \end{equation}
    and the morphism $\mathbb{D}(x_0)_{\N_0(x_0)}$ takes the following form
    \begin{equation*}
        \mathbb{D}(x_0)_{\N_0(x_0)}[e,f]=[e^{\prime},f^{\prime}]\begin{pmatrix}
            0 & -p\\
            1 & 0
        \end{pmatrix},
    \end{equation*}
    Moreover, the local ring $\rO_{\N_0(x_0)}$ is isomorphic to $ W[[x,y]]/(p+xy)$. 
    \label{local-gamma0p}
\end{proposition}
\begin{proof}
   Scaling the basis $[e_1,e_2]$ of the rank 2 free $W$-module $\mathbb{D}(\X)$ in the end of $\S$\ref{hyperspecial0} by invertible elements in $W$, we can assume that the morphism $\mathbb{D}(x_0)$ takes the form $\mathbb{D}(x_0)[e_1,e_2]=[e_1,e_2]\begin{pmatrix}
            0 & -p\\
            1 & 0
        \end{pmatrix}$.
    \par
    As $\rO_{\N_0(x_0)}$-modules, we have $\mathbb{D}(X^{\univ})_{\N_0(x_0)}\simeq\mathbb{D}(\X)\otimes_{W}\rO_{\N_0(x_0)}$ and $\mathbb{D}(X^{\prime\univ})_{\N_0(x_0)}\simeq\mathbb{D}(\X)\otimes_{W}\rO_{\N_0(x_0)}$. By the closed immersion $\textup{Spec}\,\F\rightarrow\N_0(x_0)$ and the property of crystals, the morphism $\mathbb{D}(x_0)_{\N_0(x_0)}$ takes the following form after scaling the element $e_1$ and $e_2$ by invertible elements in the local ring $\rO_{\N_0(x_0)}$:
    \begin{equation*}
        \mathbb{D}(x_0)_{\N_0(x_0)}[e_1,e_2]=[e_1^{\prime},e_2^{\prime}]\begin{pmatrix}
            h & -p+mh\\
            1 & m
        \end{pmatrix}
    \end{equation*}
    for some elements $h,m\in\mathfrak{m}$. Let $e=e_1$, $f=e_2-me_1$, $e^{\prime}=e_1^{\prime}$, $f^{\prime}=e_2^{\prime}+he_1^{\prime}$. We have
    \begin{equation}
        \mathbb{D}(x_0)_{\N_0(x_0)}[e,f]=[e^{\prime},f^{\prime}]\begin{pmatrix}
            0 & -p\\
            1 & 0
        \end{pmatrix}.
        \label{proof-standard-p}
    \end{equation}
    \par
    Recall that $\rO_{\N_0}\simeq W[[t]]$. Therefore the local ring $\rO_{\N}$ is isomorphic to $W[[t,t^{\prime}]]$. We can choose two uniformizers $t$ and $t^{\prime}$ such that the Hodge filtrations on $\mathbb{D}(X^{\univ})_{\N}$ and $\mathbb{D}(X^{\prime\univ})_{\N}$ are given by
    \begin{equation*}
        0\rightarrow\textup{Fil}^{1}\mathbb{D}(X^{\univ})_{\N}=\rO_{\N}\cdot (e_2+te_1)\rightarrow\mathbb{D}(X^{\univ})_{\N}.
    \end{equation*}
    \begin{equation*}
        0\rightarrow\textup{Fil}^{1}\mathbb{D}(X^{\prime\univ})_{\N}=\rO_{\N}\cdot (e_2^{\prime}+t^{\prime}e_1^{\prime})\rightarrow\mathbb{D}(X^{\univ})_{\N}.
    \end{equation*}
    Then the Hodge filtrations on $\mathbb{D}(X^{\univ})_{\N_0(x_0)}$ and $\mathbb{D}(X^{\prime\univ})_{\N_0(x_0)}$ are given by
    \begin{equation*}
        0\rightarrow\rO_{\N_0(x_0)}\cdot (e_2+te_1)=\rO_{\N_0(x_0)}\cdot \left(f+(t+m)e\right)\rightarrow\mathbb{D}(X^{\univ})_{\N_0(x_0)}.
    \end{equation*}
    \begin{equation*}
        0\rightarrow\rO_{\N_0(x_0)}\cdot (e_2^{\prime}+t^{\prime}e_1^{\prime})=\rO_{\N_0(x_0)}\cdot \left(f^{\prime}+(t^{\prime}-h)e^{\prime}\right)\rightarrow\mathbb{D}(X^{\univ})_{\N_0(x_0)}.
    \end{equation*}
    \par
    Let $x=t+m\in\mathfrak{m}, y=t^{\prime}-h\in\mathfrak{m}$. Since $x_0$ lifts to an isogeny over the formal scheme $\N_0(x_0)$, we have $\mathbb{D}(x_0)_{\N_0(x_0)}\left(f+xe\right)\subset\rO_{\N_0(x_0)}\cdot \left(f^{\prime}+ye^{\prime}\right)$. By (\ref{proof-standard-p}), we have $\mathbb{D}(x_0)_{\N_0(x_0)}\left(f+xe\right)=xf^{\prime}-pe^{\prime}\in \rO_{\N_0(x_0)}\cdot \left(f^{\prime}+ye^{\prime}\right)$. Therefore $p+xy=0$.
    \par
    Next we want to show that $\mathfrak{m}=(x,y)$. By the equation $p+xy=0$, the two irreducible components of the formal scheme $\N_0(x_0)_{\F}$ are given by $x=0$ and $y=0$. By Lemma \ref{semistable}, they intersect properly at the unique closed point of $\N_0(x_0)$, hence $\mathfrak{m}=(x,y)$. Therefore we conclude that $\rO_{\N_0(x_0)}\simeq W[[x,y]]/(p+xy)$.
\end{proof}
\begin{remark}
    We shall make the convention that under the isomorphism $\rO_{\N_0(x_0)}\simeq W[[x,y]]/(p+xy)$, the equation $y=0$ is the equation of the irreducible component $\N_0(x_0)^{\textup{F}}$, while the equation $x=0$ is the equation of the irreducible component $\N_0(x_0)^{\textup{V}}$.
    \label{convention-FV}
\end{remark}
\begin{lemma}
    Let $t,t^{\prime}\in\rO_{\N_0(x_0)}$ be two elements guaranteed by Lemma \ref{cyclic-diff} (b) for the element $x_0$. Then there exists invertible elements $\nu_1,\nu_2,\omega_2,\omega_2\in W[[x,y]]$ such that
    \begin{equation*}
        \textup{s}_x^{\#}(t)=\nu_1x+\nu_2y^{p}+p\cdot h,
    \end{equation*}
    \begin{equation*}
        \textup{t}_x^{\#}(t^{\prime})=\omega_1y+\omega_2x^{p}+p\cdot h^{\prime},
    \end{equation*}
    where $h$ and $h^{\prime}$ are two elements in the local ring $\rO_{\N_0(x_0)}$.
    \label{coordinate-change}
\end{lemma}
\begin{proof}
    Under the coordinate in Proposition \ref{local-gamma0p}, the quotient ring $\mathcal{O}_{\N_0(x_0)}/(p)\simeq\mathbb{F}[[x,y]]/(x\cdot y)$. While under the coordinate $t,t^{\prime}$, we have $\mathcal{O}_{\N_0(x_0)}/(p)\simeq\mathbb{F}[[t,t^{\prime}]]/\left((t^{p}-t^{\prime})(t-t^{\prime p})\right)$. Therefore there exists invertible elements $\nu,\omega\in\mathcal{O}_{\N_0(x_0)}^{\times}$ such that
    \begin{equation*}
        x=\nu\cdot(t-t^{\prime p}),\,\, y=\omega\cdot(t^{p}-t^{\prime})\,\,\,\textup{in}\,\,\mathcal{O}_{\N_0(x_0)}/(p).
    \end{equation*}
    The claim in the lemma follows by solving the above equations for $t,t^{\prime}$.
\end{proof}

\subsection{The product space $\N_0(x_0)\times\N_0(x_0)$}
\label{product-space}
The formal scheme $\N(x_0)\coloneqq\N_0(x_0)\times_W\N_0(x_0)$ is represented by the local ring $W[[x_1,y_1,x_2,y_2]]/(p+x_1y_1,p+x_2y_2)$ by Lemma \ref{local-gamma0p}. The sequence $p+x_1y_1,p+x_2y_2=x_2y_2-x_1y_1$ is not a regular sequence since the second element is not regular. Hence the local ring is not a regular local ring. Therefore the formal scheme $\N(x_0)^{2}$ is not regular.
\par
Notice that $\N(x_0)\times_{W}{\F}\simeq\N_0(x_0)_{\F}\times_{\F}\N_0(x_0)_{\F}$. The formal scheme $\N(x_0)_{\F}$ has the following 4 irreducible components by Lemma \ref{semistable}.
\begin{align}
    \N(x_0)^{\textup{FF}}=\N_0(x_0)^{\textup{F}}\times\N_0(x_0)^{\textup{F}},\,\,\,\,&\N(x_0)^{\textup{FV}}=\N_0(x_0)^{\textup{F}}\times\N_0(x_0)^{\textup{V}},\,\,\\
    \N(x_0)^{\textup{VV}}=\N_0(x_0)^{\textup{V}}\times\N_0(x_0)^{\textup{V}},\,\,\,\,&\N(x_0)^{\textup{VF}}=\N_0(x_0)^{\textup{V}}\times\N_0(x_0)^{\textup{F}}.\notag
\end{align}
Let $\widehat{\mathbb{A}}_{\F}^{2}$ be the completion of the 2-dimensional affine plane at the point $(0,0)$. By Lemma \ref{semistable} (2), the formal schemes $\N(x_0)^{\textup{FF}}, \N(x_0)^{\textup{VV}}, \N(x_0)^{\textup{FV}}, \N(x_0)^{\textup{VF}}$ are all isomorphic to $\widehat{\mathbb{A}}_{\F}^{2}$.
\par
We define two morphisms $s_{+},s_{-}:\N(x_0)\rightarrow\N$ by the moduli interpretations: let $S$ be a scheme in the category $\textup{Nilp}_{W}$,
\begin{itemize}
    \item [$s_{+}$:] $\left(\left(X_1\stackrel{\pi_1}\rightarrow X_1^{\prime},\left(\rho_1,\rho_1^{\prime}\right)\right),\left(X_2\stackrel{\pi_2}\rightarrow X_2^{\prime},\left(\rho_2,\rho_2^{\prime}\right)\right)\right)\in\N(x_0)(S)\mapsto\left((X_1,\rho_1),(X_2,\rho_2)\right)\in\N(S)$.
    \item [$s_{-}$:] $\left(\left(X_1\stackrel{\pi_1}\rightarrow X_1^{\prime},\left(\rho_1,\rho_1^{\prime}\right)\right),\left(X_2\stackrel{\pi_2}\rightarrow X_2^{\prime},\left(\rho_2,\rho_2^{\prime}\right)\right)\right)\in\N(x_0)(S)\mapsto\left((X_1^{\prime},\rho_1^{\prime}),(X_2^{\prime},\rho_2^{\prime})\right)\in\N(S)$.
\end{itemize}

\subsection{Local Hecke correspondences}
\label{local-correspondences}
Let $x_0\in\B$ be an element such that $\nu_{p}(q(x_0))=1$. Let $(\cdot)^{\prime}:\B\rightarrow\B,\,\,b\mapsto b^{\prime}\coloneqq x_0\cdot b\cdot x_0^{-1}$ be the conjugation-by-$x_0$ automorphism on $\B$. Let $x$ be an element in $\B$ such that $n=\nu_p(q(x))\geq1$. For an object $(X\stackrel{\pi}\rightarrow X^{\prime},(\rho,\rho^{\prime}))\in\N_0(x)(S)$, there is a standard decomposition of the cyclic isogeny $\pi$ into $n$ degree $p$ isogenies:
\begin{equation}
    \pi: X=X_0\stackrel{\pi_1}\rightarrow X_1\stackrel{\pi_2}\rightarrow\cdots\stackrel{\pi_{n-1}}\rightarrow X_{n-1}\stackrel{\pi_n}\rightarrow X_n=X^{\prime}.
    \label{stdfac}
\end{equation}
We refer to \cite[$\S$6.7]{KM85} for the details on the notions of standard decomposition.
\par
Let $S$ be a $W$-scheme such that $p$ is locally nilpotent on $S$, for any object $(X\stackrel{\pi}\rightarrow X^{\prime},(\rho,\rho^{\prime}))\in\N_0(x)(S)$ with the standard decomposition of $\pi$ as (\ref{stdfac}). We construct two height $0$ quasi-isogenies $\rho_1:\X\times_{\F}\overline{S}\rightarrow X_1\times_{S}\overline{S}$ and $\rho_{n-1}:\X\times_{\F}\overline{S}\rightarrow X_{n-1}\times_{S}\overline{S}$ in the following way,
\begin{equation*}
\rho_1=(\pi_1\times_{S}\overline{S})\circ\rho\circ(x_0\times_\F\overline{S})^{-1},\,\,\,\rho_{n-1}=(\pi_n\times_{S}\overline{S})^{-1}\circ\rho\circ(x_0\times_\F\overline{S}).
\end{equation*}
\begin{equation*}
\rho_1^{-}=(\pi_1^{\vee}\times_{S}\overline{S})^{-1}\circ\rho\circ(x_0\times_\F\overline{S}),\,\,\,\rho_{n-1}^{+}=(\pi_n^{\vee}\times_{S}\overline{S})\circ\rho\circ(x_0\times_\F\overline{S})^{-1}.
\end{equation*}
Then we have
\begin{equation*}
    \left(X\stackrel{\pi_1}\rightarrow X_1,(\rho,\rho_1)\right), \left(X_{n-1}\stackrel{\pi_n}\rightarrow X^{\prime},(\rho_{n-1},\rho^{\prime})\right), \left(X_1\stackrel{\pi_1^{\vee}}\rightarrow X,(\rho_1^{-},\rho)\right), \left(X^{\prime}\stackrel{\pi_n^{\vee}}\rightarrow X_{n-1},(\rho^{\prime},\rho_{n-1}^{+})\right)\in\N_0(x_0)(S).
\end{equation*}
Define four morphisms $\st_x^{\textup{I}+}, \st_x^{\textup{I}-}, \st_x^{\textup{II}+}, \st_x^{\textup{II}-}$ as follows,
\begin{itemize}
     \item[$\st_x^{\textup{I}+}$:] $(X\stackrel{\pi}\rightarrow X^{\prime},(\rho,\rho^{\prime}))\in\N_0(x)(S)\mapsto\left(\left(X\stackrel{\pi_1}\rightarrow X_1,(\rho,\rho_1)\right),\left(X_{n-1}\stackrel{\pi_n}\rightarrow X^{\prime},(\rho_{n-1},\rho^{\prime})\right)\right)\in\N(x_0)(S)$,
      \item[$\st_x^{\textup{I}-}$:] $(X\stackrel{\pi}\rightarrow X^{\prime},(\rho,\rho^{\prime}))\in\N_0(x)(S)\mapsto\left(\left(X_{n-1}\stackrel{\pi_n}\rightarrow X^{\prime},(\rho_{n-1},\rho^{\prime})\right),\left(X\stackrel{\pi_1}\rightarrow X_1,(\rho,\rho_1)\right)\right)\in\N(x_0)(S)$,
    \item[$\st_x^{\textup{II}+}$:] $(X\stackrel{\pi}\rightarrow X^{\prime},(\rho,\rho^{\prime}))\in\N_0(x)(S)\mapsto\left(\left(X\stackrel{\pi_1}\rightarrow X_1,(\rho,\rho_1)\right),\left(X^{\prime}\stackrel{\pi_n^{\vee}}\rightarrow X_{n-1},(\rho^{\prime},\rho_{n-1}^{+})\right)\right)\in\N(x_0)(S)$,
    \item[$\st_x^{\textup{II}-}$:] $(X\stackrel{\pi}\rightarrow X^{\prime},(\rho,\rho^{\prime}))\in\N_0(x)(S)\mapsto\left(\left(X_1\stackrel{\pi_1^{\vee}}\rightarrow X,(\rho_1^{-},\rho)\right),\left(X_{n-1}\stackrel{\pi_n}\rightarrow X^{\prime},(\rho_{n-1},\rho^{\prime})\right)\right)\in\N(x_0)(S)$.
\end{itemize}
\begin{lemma}
    Let $x$ be an element in $\B$ such that $n=\nu_p(q(x))\geq1$. Then the four morphisms $\st_x^{\textup{I}+}, \st_x^{\textup{I}-}, \st_x^{\textup{II}+},\st_x^{\textup{II}-}:\N_0(x)\rightarrow\N(x_0)$ are closed immersions.
\end{lemma}
\begin{proof}
    We only give the proof for the morphism $\st_x^{\textup{I}+}$. Notice that 
    \begin{equation*}
        (s_{x_0}\times t_{x_0})\circ\st_x^{\textup{I}+}=\st_x:\N_0(x)\rightarrow\N.
    \end{equation*}
    The morphism $\st_x$ is a closed immersion by Lemma \ref{cyclic-diff}, hence $\st_x^{\textup{I}+}$ is a closed immersion.
\end{proof}
\begin{lemma}
    Let $x$ be an element in $\B$ such that $n=\nu_p(q(x))\geq1$. Then the following diagram is Cartesian,
    \begin{equation*}
        \begin{tikzcd}
    {\textup{Spec}\,\F}
    \arrow[d] \arrow[r]
    & {\N(x_0)^{\textup{VF}}}
    \arrow[d]
    \\
    {\N_0(x)}
    \arrow[r,"{\st_x^{\textup{I}+}}"]
    &{\N(x_0)}.
\end{tikzcd}
\end{equation*}
\label{VF}
\end{lemma}
\begin{proof}
    It is proved in \cite[Theorem 13.3.5]{KM85} that a cyclic isogeny $\pi$ of degree $p^{n}$ between two $p$-divisible groups over an $\F$-scheme $S$ whose first decomposition $\pi_1$ is isomorphic to Verschiebung morphism $V$ is isomorphic to $V^{n}$. Then by Lemma \ref{cyclic-diff} (b2), the fiber product $\mathcal{N}_0(x)\times_{\st_x^{\textup{I}+},\N(x_0)}\N(x_0)^{\textup{VF}}$ is given by the following two equations
    \begin{equation*}
        t^{\prime p^{n}}-t\,\,\,\,\textup{and}\,\,\,\,t^{\prime}-t^{p^{n}}.
    \end{equation*}
    These two elements generate the maximal ideal of the unique closed $\F$-point of $\N_0(x)$, hence $\mathcal{N}_0(x)\times_{\st_x^{\textup{I}+},\N(x_0)}\N(x_0)^{\textup{VF}}\simeq\textup{Spec}\,\F$. 
\end{proof}
\begin{remark}
    We can summarize the above lemma as $\N_0(x)\times_{\st_x^{\textup{I}+},\N(x_0)}\N(x_0)^{\textup{VF}}\simeq\textup{Spec}\,\F$. Similar arguments also imply that $\N_0(x)\times_{\st_x^{\textup{I}-},\N(x_0)}\N(x_0)^{\textup{FV}}\simeq\textup{Spec}\,\F$, $\N_0(x)\times_{\st_x^{\textup{II}+},\N(x_0)}\N(x_0)^{\textup{VV}}\simeq\textup{Spec}\,\F$ and $\N_0(x)\times_{\st_x^{\textup{II}-},\N(x_0)}\N(x_0)^{\textup{FF}}\simeq\textup{Spec}\,\F$.
    \label{Stand-FV}
\end{remark}

\subsection{Special cycles on the product: the $\Gp\times_{\mathbb{G}_m}\Gp$ case} The formal scheme $\N(x_0)$ parameterizes a pair of deformations of the quasi-isogeny $x_0$. Let 
\begin{equation}
   \left(X_1^{\textup{univ}}\stackrel{x_{0,1}^{\univ}}\longrightarrow X_{1}^{\prime\textup{\univ}},\left(\rho_1^{\univ},\rho_1^{\prime\univ}\right)\right),\,\,\,\,\left(X_2^{\textup{univ}}\stackrel{x_{0,2}^{\univ}}\longrightarrow X_{2}^{\prime\textup{\univ}},\left(\rho_2^{\univ},\rho_2^{\prime\univ}\right)\right)
   \label{ori-univ}
\end{equation}
be the universal pairs over the formal scheme $\N(x_0)$.
\par
We have the following commutative diagram
\begin{equation*}
        \begin{tikzcd}
    {X_1^{\textup{univ}}}
    \arrow[d,"{x_{0,1}^{\univ}}"swap] \arrow[r, dashed, "{x}"]
    & {X_2^{\textup{univ}}}
    \arrow[d,"{x_{0,2}^{\univ}}"]
    \\
    {X_1^{\prime\textup{univ}}}
    \arrow[r, dashed, "{x^{\prime}}"]
    &{X_2^{\prime\textup{univ}}},
\end{tikzcd}
\end{equation*}
here the dotted arrows below $x$ and $x^{\prime}$ mean that they are quasi-isogenies.
\begin{definition}
For any subset $H\subset\B$. Define the special cycle $\mathcal{Z}_{\N(x_0)}(H)\subset\N(x_0)$ to be the closed formal subscheme cut out by the conditions,
\begin{equation*}
    \rho_2^{\textup{univ}}\circ x \circ (\rho_1^{\textup{univ}})^{-1}\in \textup{Hom}(X_1^{\textup{univ}}, X_2^{\textup{univ}});
\end{equation*}
\begin{equation*}
    \rho_2^{\prime\textup{univ}}\circ x^{\prime} \circ (\rho_1^{\prime\textup{univ}})^{-1} \in \textup{Hom}(X_1^{\prime\textup{univ}}, X_2^{\prime\textup{univ}}).
\end{equation*}
for all $x\in H$.
\par
Define the special cycle $\mathcal{Y}_{\N(x_0)}(H)\subset\N(x_0)$ to be the closed formal subscheme cut out by the conditions,
\begin{equation*}
    \rho_2^{\prime\textup{univ}}\circ x_0\cdot x \circ (\rho_1^{\textup{univ}})^{-1}\in \textup{Hom}(X_1^{\textup{univ}}, X_2^{\textup{univ}});
\end{equation*}
\begin{equation*}
    \rho_2^{\textup{univ}}\circ \overline{x_0}\cdot x^{\prime} \circ (\rho_1^{\prime\textup{univ}})^{-1} \in \textup{Hom}(X_1^{\prime\textup{univ}}, X_2^{\prime\textup{univ}}).
\end{equation*}
for all $x\in H$.
\label{special-cycle-originalmodel}
\end{definition}
\begin{lemma}
    Let $x\in\B$ be an element such that $\nu_p\left(q(x)\right)\geq0$. The morphisms $\st_{x_0\cdot x}^{\textup{I}+}$, $\st_{x_0\cdot\overline{x}}^{\textup{I}-}$, $\st_x^{\textup{II}+}$ and $\st_{x^{\prime}}^{\textup{II}-}$ (the later two morphisms are only defined for elements $x$ such that $\nu_p\left(q(x)\right)\geq1$) sends the corresponding source formal scheme to the special divisor $\mathcal{Z}_{\N(x_0)}(x)$.
    \label{ori-diff}
\end{lemma}
\begin{proof}
    Let's first consider the morphism $\st_{x_0\cdot x}^{\textup{I}+}:\N_0(x_0\cdot x)\rightarrow\N(x_0)$. We still use $\N_0(x_0\cdot x)$ to denote the image in $\N(x_0)$ of the closed immersion $\st_{x_0\cdot x}^{\textup{I}+}$. Denote by 
     \begin{equation*}
    \left(X_{1}^{\textup{I}+}\stackrel{(x_0)_1^{\textup{I}+}}\longrightarrow X_{1}^{\prime\textup{I}+},\left(\rho_{1}^{\textup{I}+},\rho_{1}^{\prime\textup{I}+}\right)\right),\,\,\,\,\left(X_{2}^{\textup{I}+}\stackrel{(x_0)_2^{\textup{I}+}}\longrightarrow X_{2}^{\prime\textup{I}+},\left(\rho_{2}^{\textup{I}+},\rho_{2}^{\prime\textup{I}+}\right)\right)
    \end{equation*}
    be the base change of the universal pair (\ref{ori-univ}) to the formal scheme $\N_0(x_0\cdot x)$ through the morphism $\textup{st}_{x_0\cdot x}^{\textup{I}+}$.
    \par
    The quasi-isogeny $x_0\cdot x$ lifts to a cyclic isogeny $(x_0\cdot x)^{\textup{I}+}:X_{1}^{\textup{I}+}\rightarrow X_{2}^{\prime\textup{I}+}$. By the definition of the morphism $\st_{x_0\cdot x}^{\textup{I}+}$ in \ref{local-correspondences}, the isogeny $\left(X_{2}^{\textup{I}+}\stackrel{(x_0)_{2}^{\textup{I}+}}\longrightarrow X_{2}^{\prime\textup{I}+}\right)$ is the last term in the factorization of the cyclic isogeny $(x_0\cdot x)^{\textup{I}+}$. Therefore the quasi-isogeny $\left((x_0)_2^{\textup{I}+}\right)^{-1}\circ(x_0\cdot x)^{\textup{I}+}:X_1^{\textup{I}+}\rightarrow X_{2}^{\textup{I}+}$ is an isogeny. Notice that 
    \begin{equation*}  \left(X_1^{\textup{I}+}\xrightarrow{\left((x_0)_2^{\textup{I}+}\right)^{-1}\circ(x_0\cdot x)^{\textup{I}+}}X_{2}^{\textup{I}+}, (\rho_{1}^{\textup{I}+},\rho_2^{\textup{I}+})\right)
    \end{equation*}
    is a lift of the quasi-isogeny $x$. Therefore the morphism $\st_{x_0\cdot x}^{\textup{I}+}$ maps $\N_0(x_0\cdot x)$ to $\mathcal{Z}_{\N(x_0)}(x)$.
    \par
    The ideas and proofs for other morphisms $\st_{x_0\cdot\overline{x}}^{\textup{I}-}$, $\st_x^{\textup{II}+}$ and $\st_{x^{\prime}}^{\textup{II}-}$ are similar.
\end{proof}
\begin{lemma}
    Let $x\in\B$ be an element.
    \begin{itemize}
        \item [(a)] If $\nu_p(q(x))=0$, we have the following equality of closed formal subschemes of $\N(x_0)$:
        \begin{equation*}
            \st_{x_0\cdot x}^{\textup{I}+}\left(\N_0(x_0\cdot x)\right)=\st_{x_0\cdot\overline{x}}^{\textup{I}-}\left(\N_0(x_0\cdot\overline{x})\right).
        \end{equation*}
        \item[(b)] If $\nu_p(q(x))=1$, we have the following equality of closed formal subschemes of $\N(x_0)$:
        \begin{equation*}
            \st_{x}^{\textup{II}+}\left(\N_0(x)\right)=\st_{x^{\prime}}^{\textup{II}-}\left(\N_0(x^{\prime})\right).
        \end{equation*}
    \end{itemize}
    \label{special-0-1}
\end{lemma}
\begin{proof}
    Let $S$ be a $W$-scheme such that $p$ is locally nilpotent on $S$. We first prove (a). Let 
    \begin{equation}
        \left(X_1\stackrel{\pi_1}\rightarrow X_1^{\prime}, (\rho_1,\rho_1^{\prime})\right), \left(X_2\stackrel{\pi_2}\rightarrow X_2^{\prime}, (\rho_2,\rho_2^{\prime})\right)
        \label{obj-in-i+}
    \end{equation}
    be an object in the set $\st_{x_0\cdot x}^{\textup{I}+}\left(\N_0(x_0\cdot x)\right)(S)$. By the definition of the morphism $\st_{x_0\cdot x}^{\textup{I}+}$ there exist two isomorphisms $\phi: X_1\rightarrow X_2$ and $\phi^{\prime}:X_1^{\prime}\rightarrow X_2^{\prime}$ such that $\pi_2\circ\phi=\phi^{\prime}\circ\pi_1$. Then it's easy to see that the object 
     \begin{equation*}
     \left(X_2\xrightarrow{\pi_1\circ\phi^{\vee}} X_1^{\prime}, (\rho_2,\rho_1^{\prime})\right)\in\N_0(x_0\cdot\overline{x})(S)
    \end{equation*}
    is mapped to the object (\ref{obj-in-i+}) under the morphism $\st_{x_0\cdot\overline{x}}^{\textup{I}-}$. Hence $\st_{x_0\cdot x}^{\textup{I}+}\left(\N_0(x_0\cdot x)\right)\subset\st_{x_0\cdot\overline{x}}^{\textup{I}-}\left(\N_0(x_0\cdot\overline{x})\right)$ as closed formal subschemes. The converse direction $\st_{x_0\cdot x}^{\textup{I}+}\left(\N_0(x_0\cdot x)\right)\supset\st_{x_0\cdot\overline{x}}^{\textup{I}-}\left(\N_0(x_0\cdot\overline{x})\right)$ can be proved similarly. Hence (a) is true. The idea of the proof of (b) is similar so we omit it.
\end{proof}

\subsection{Rapoport-Zink space with level $\Gp\times_{\mathbb{G}_{m}}\Gp$}
\label{RZBLOWUP}
Let $x_0\in\B$ be an element such that $\nu_p(q(x_0))=1$. Let $\pi:\M\rightarrow\N(x_0)$ be the blow-up morphism of the formal scheme $\N(x_0)$ along its unique closed point.
\par
Let $\exc_{\M}$ be the exceptional divisor on the formal scheme $\M$. Let $\M_{\F}\coloneqq\M\times_{W}\F$ be the reduction mod $p$ of the formal scheme $\M$.
\par
Notice that the unique closed point pulls back to the unique closed point of the closed formal subschemes $\N(x_0)^{\textup{FF}}, \N(x_0)^{\textup{VV}}, \N(x_0)^{\textup{FV}}$ and $\N(x_0)^{\textup{VF}}$ of $ \N(x_0)$. Let $\MFF,\MVV,\MFV$ and $\MVF$ be the strict transforms of the formal schemes $\N(x_0)^{\textup{FF}}, \N(x_0)^{\textup{VV}}, \N(x_0)^{\textup{FV}}$ and $\N(x_0)^{\textup{VF}}$ respectively under the blow up morphism $\pi: \M\rightarrow\N(x_0)$. The four closed formal subschemes $\MFF,\MVV,\MFV$ and $\MVF$ of $\M$ are all isomorphic to the formal scheme $\textup{Bl}_{(0,0)}\widehat{\mathbb{A}}_{\F}^{2}$.
\par
Let $x\in\B$ be an element such that $\nu_p(q(x))\geq1$. The unique closed point of $\N(x_0)$ also pulls back to the unique closed point of the formal subschemes $\N_0(x)$ through the four closed immersion $\st_{x}^{\textup{I}+},\st_{x}^{\textup{I}-},\st_{x}^{\textup{II}+},\st_x^{\textup{II}-}:\N_0(x)\rightarrow\N(x_0)$. Denote by $\N_0^{\textup{I}+}(x), \N_0^{\textup{I}-}(x),\N_0^{\textup{II}+}(x)$ and $\N_0^{\textup{II}-}(x)$ the strict transform of the formal scheme $\N_0(x)$ under the three morphisms $\st_{x}^{\textup{I}+},\st_{x}^{\textup{I}-}, \st_{x}^{\textup{II}+}$ and $\st_{x}^{\textup{II}-}$ respectively. They are closed formal schemes of the formal scheme $\M$, and are all isomorphic to the formal scheme $\widetilde{\N}_0(x)$.

\subsection{An open cover of the Rapoport-Zink space $\M$}
\label{opencover}
We give a detailed description of the formal scheme $\M$ by giving an explicit open cover of it. By \cite[Theorem 13.4.6]{KM85}, the local ring $\rO_{\N(x_0)}$ is isomorphic to $W[[x_1,y_1,x_2,y_2]]/(p+x_1y_1,p+x_2y_2)$. The maximal ideal is generated by the images of $x_1,y_1,x_2$ and $y_2$. Then the blow up formal scheme $\M$ is covered by the following 4 open formal subschemes:\begin{itemize}
    \item [$\M_1^{+}$:] Let $y_1=u_{11}x_1, x_2=v_{21}x_1$ and $y_2=u_{21}x_1$. Then $u_{11}=u_{21}v_{21}$. Define 
    \begin{equation*}
        \M_1^{+}=\textup{Spf}\,W[u_{21},v_{21}][[x_1]]/(p+u_{21}v_{21}x_1^{2}).
    \end{equation*}
    \item [$\M_2^{+}$:] Let $x_1=v_{12}x_2, y_1=u_{12}x_2$ and $y_2=u_{22}x_2$. Then $u_{22}=v_{12}u_{12}$. Define 
    \begin{equation*}
        \M_2^{+}=\textup{Spf}\,W[v_{12},u_{12}][[x_2]]/(p+v_{12}u_{12}x_2^{2}).
    \end{equation*}
    \item [$\M_1^{-}$:] Let $x_1=w_{11}y_1, x_2=w_{21}y_1$ and $y_2=t_{21}y_1$. Then $w_{11}=w_{21}t_{21}$. Define 
    \begin{equation*}
        \M_1^{-}=\textup{Spf}\,W[w_{21},t_{21}][[y_1]]/(p+w_{21}t_{21}y_1^{2}).
    \end{equation*}
    \item [$\M_2^{-}$:] Let $x_1=w_{12}y_2, y_1=t_{12}y_2$ and $x_2=w_{22}y_2$. Then $w_{22}=w_{12}t_{12}$. Define 
    \begin{equation*}
        \M_2^{-}=\textup{Spf}\,W[w_{12},t_{12}][[y_2]]/(p+w_{12}t_{12}y_2^{2}).
    \end{equation*}
\end{itemize}
Notice that we have the following relations (on the intersections of the corresponding opens):
\begin{equation}
    u_{21}=u_{12},\,\,v_{21}v_{12}=1,\,\,w_{21}u_{21}=1,\,\,t_{21}v_{21}=1,\,\,u_{21}w_{12}=1,\,\,t_{12}=v_{21}.
    \label{relation-1}
\end{equation}
We denote this cover by $\mathcal{C}$. Let $\M^{+}=\M_1^{+}\cup\M_2^{+}$ and $\M^{-}=\M_1^{-}\cup\M_2^{-}$.
\begin{proposition}
The following facts hold:
\begin{itemize}
    \item [(i)]The 3-dimensional formal scheme $\M$ is regular. 
    \item[(ii)] The exceptional divisor $\exc_{\M}$ is isomorphic to $\bP_{\F}^{1}\times\bP_{\F}^{1}$.
    \item[(iii)] The closed formal subschemes $\MFF,\MVV,\MFV$ and $\MVF$ are regular Cartier divisors on $\M$, and we have the following equality of Cartier divisors:
    \begin{equation*}
        \M_{\F}=2\cdot\exc_{\M}+\MFF+\MVV+\MFV+\MVF.
    \end{equation*}
    \item[(iv)] Let $x\in\B$ be an element such that $\nu_p(q(x))\geq1$, the closed formal subschemes $\N_0^{\textup{I}+}(x), \N_0^{\textup{I}-}(x)$, $ \N_0^{\textup{II}+}(x)$ and $\N_0^{\textup{II}-}(x)$ are regular Cartier divisors on $\M$.
\end{itemize}
\label{geometry-M}
\end{proposition}
\begin{proof}
    The four open formal subschemes $\M_1^{+},\M_2^{+},\M_1^{-}$ and $\M_2^{-}$ of the formal scheme $\M$ cover $\M$ and are all regular formal schemes of dimension $3$, hence (i) is true. Notice that the formal scheme $\textup{Bl}_{(0,0)}\widehat{\mathbb{A}}_{\F}^{2}$ and $\widetilde{\N}_0(x)$ are $2$-dimensional regular formal schemes, hence the regular closed formal subschemes $\MFF,\MVV,\MFV,\MVF, \N_0^{\textup{I}+}(x), \N_0^{\textup{I}-}(x), \N_0^{\textup{II}+}(x), \N_0^{\textup{II}-}(x)$ of $\M$ must be regular Cartier divisors on $\M$, i.e., (iv) and the first part of (iii) are true.
    \par
    The exceptional divisor $\exc_{\M}$ is covered by the following 4 open subschemes 
    \begin{equation*} \exc_{\M,1}^{+}\coloneqq\exc_{\M}\cap\M_1^{+}\simeq\textup{Spec}\,\F[u_{21},v_{21}],\,\,\exc_{\M,2}^{+}\coloneqq\exc_{\M}\cap\M_2^{+}\simeq\textup{Spec}\,\F[v_{12},u_{12}],
    \end{equation*}
    \begin{equation*} \exc_{\M,1}^{-}\coloneqq\exc_{\M}\cap\M_1^{+}\simeq\textup{Spec}\,\F[w_{21},t_{21}],\,\,\exc_{\M,2}^{-}\coloneqq\exc_{\M}\cap\M_2^{-}\simeq\textup{Spec}\,\F[w_{12},t_{12}].
    \end{equation*}
    By the relations between the coordinates in (\ref{relation-1}), the above $4$ schemes glue together and the resulting scheme $\bP_{\F}^{1}\times\bP_{\F}^{1}$. Hence (ii) is true.
    \par
    Now we prove the second part of (iii). By $\S$\ref{description-x0}, we can assume that the irreducible component $\N_0(x_0)^{\textup{F}}$ is given by the equation $x=0$, while the irreducible component $\N_0(x_0)^{\textup{F}}$ is given by the equation $y=0$. Hence the equations of the four closed formal subschemes $\N(x_0)^{\textup{FF}}, \N(x_0)^{\textup{VV}}, \N(x_0)^{\textup{FV}}$ and $\N(x_0)^{\textup{VF}}$ are given by
    \begin{equation*}
        \N(x_0)^{\textup{FF}}: (y_1,y_2),\,\,\,\,
        \N(x_0)^{\textup{VV}}: (x_1,x_2),\,\,\,\,\N(x_0)^{\textup{FV}}: (y_1,x_2),\,\,\,\,\N(x_0)^{\textup{VF}}: (x_1,y_2).
    \end{equation*}
    Hence it's easy to see that 
    \begin{itemize}
        \item[$\MFF$]: $\MFF\subset\M_1^{+}\cup\M_2^{+}$ is cut out by the equation $u_{21}=0$ on $\M_1^{+}$ and $u_{12}=0$ on $\M_2^{+}$;
        \item[$\MVV$]: $\MVV\subset\M_1^{-}\cup\M_2^{-}$ is cut out by the equation $w_{21}=0$ on $\M_1^{-}$ and $w_{12}=0$ on $\M_2^{-}$;
        \item[$\MFV$]: $\MFV\subset\M_1^{+}\cup\M_2^{-}$ is cut out by the equation $v_{21}=0$ on $\M_1^{+}$ and $t_{12}=0$ on $\M_2^{-}$;
        \item[$\MVF$]: $\MVF\subset\M_2^{+}\cup\M_1^{-}$ is cut out by the equation $v_{12}=0$ on $\M_2^{+}$ and $t_{21}=0$ on $\M_1^{-}$.  
    \end{itemize}
    Therefore the second part of (iii) is true by the explicit open covering of the formal scheme $\M$ by $\M_1^{+},\M_2^{+},\M_1^{-}$ and $\M_2^{-}$.
\end{proof}
\begin{remark}
    For simplicity, in the following paragraphs we denote $u_{21}$ by $u$, denote $v_{21}$ by $v$. Then the open covering of the exceptional divisor $\exc_{\M}$ can be written in the following way by the relations in (\ref{relation-1})
\begin{equation*}
    \exc_{\M,1}^{+}\simeq\textup{Spec}\,\F[v,u],\,\,\exc_{\M,2}^{+}\simeq\textup{Spec}\,\F[v^{\prime},u],
\end{equation*}
\begin{equation*}
    \exc_{\M,1}^{-}\simeq\textup{Spec}\,\F[v^{\prime},u^{\prime}],\,\,\exc_{\M,2}^{-}\simeq\textup{Spec}\,\F[v,u^{\prime}].
\end{equation*}
Notice that the coordinate $v^{\prime}$ (resp. $u^{\prime}$) satisfy $v\cdot v^{\prime}=1$ (resp. $u\cdot u^{\prime}=1$) when $v$ (resp. $u$) is nonzero.
\par
Using the coordinates $u,v,u^{\prime}$ and $v^{\prime}$. We fix an isomorphism $\iota:\exc_{\M}\simeq\bP_{\F}^{1}\times\bP_{\F}^{1}$ such that the first $\bP_{\F}^{1}$ is obtained by gluing $\textup{Spec}\,\F[v]$ and $\textup{Spec}\,\F[v^{\prime}]$ along $\textup{Spec}\,\F[v, v^{\prime}]\simeq\textup{Spec}\,\F[v, v^{-1}]$, while the second $\bP_{\F}^{1}$ is obtained by gluing $\textup{Spec}\,\F[u]$ and $\textup{Spec}\,\F[u^{\prime}]$ along $\textup{Spec}\,\F[u, u^{\prime}]\simeq\textup{Spec}\,\F[u,u^{-1}]$.
\par
\label{coordinate-exc}
\end{remark}

\subsection{An automorphism of $\M$}
\label{aut-M}
Let $\iota_{0}:\N_0(x_0)\rightarrow\N_0(x_0)$ be the following morphism: Let $S\in\textup{Nilp}_{W}$. Let $c_0=\overline{x}_0\cdot x_0^{-1}\in\B$. For an object $\left(X\xrightarrow{\pi}X^{\prime},(\rho,\rho^{\prime})\right)\in\N_0(x_0)(S)$,
\begin{align*}
    \iota_{0}:\left(X\xrightarrow{\pi}X^{\prime},(\rho,\rho^{\prime})\right)\mapsto\left(X^{\prime}\xrightarrow{\pi^{\vee}}X,(\rho^{\prime},\rho\circ c_0)\right).
\end{align*}
The morphism $\iota_{0}$ is an automorphism of $\N_0(x_0)$ over $W$ (but not an involution unless $\overline{x}_0=x_0$). It induces an automorphism $\textup{id}\times\iota_0:\N(x_0)\rightarrow\N(x_0)$ of the formal scheme $\N(x_0)$. Under this isomorphism, the universal pair (\ref{ori-univ}) is mapped to 
\begin{equation*}
   \left(X_1^{\textup{univ}}\stackrel{x_{0,1}^{\univ}}\longrightarrow X_{1}^{\prime\textup{\univ}},\left(\rho_1^{\univ},\rho_1^{\prime\univ}\right)\right),\,\,\,\,\left(X_2^{\prime\textup{univ}}\xrightarrow{\left(x_{0,2}^{\univ}\right)^{\vee}} X_{2}^{\textup{\univ}},\left(\rho_2^{\prime\univ},\rho_2^{\univ}\circ c_0\right)\right).
\end{equation*}
By the moduli interpretations of special cycles on $\N(x_0)$, we have that for all subsets $H\subset\B$,
\begin{equation*}
    \mathcal{Y}_{\N(x_0)}(H)=\left(\textup{id}\times\iota_0\right)^{\ast}\left(\mathcal{Z}_{\N(x_0)}(x_0\cdot H)\right)\coloneqq\mathcal{Z}_{\N(x_0)}(x_0\cdot H)\times_{\N(x_0),\textup{id}\times\iota_0}\N(x_0).
\end{equation*}
\par
Let $x\in\B$ be an element such that $\nu_p(q(x))\geq1$. By the definition of the morphisms $\st_{x}^{\textup{I}+}$,$\st_{x}^{\textup{I}-}$, $\st_{x}^{\textup{II}+}$ and $\st_x^{\textup{II}-}:\N_0(x)\rightarrow\N(x_0)$ in $\S$\ref{local-correspondences}, we have the following identities of closed formal subschemes of $\N(x_0)$:
\begin{align}
    \st_{c_0x}^{\textup{II}+}\left(\N_0(c_0x)\right)=\left(\textup{id}\times\iota_0\right)^{\ast}\left(\st_x^{\textup{I}+}\left(\N_0(x)\right)\right),\,\,\st_{\overline{x}}^{\textup{II}-}\left(\N_0(\overline{x})\right)=\left(\textup{id}\times\iota_0\right)^{\ast}\left(\st_x^{\textup{I}-}\left(\N_0(x)\right)\right),\label{local-id-aut-1}\\
    \st_{x}^{\textup{I}+}\left(\N_0(x)\right)=\left(\textup{id}\times\iota_0\right)^{\ast}\left(\st_x^{\textup{II}+}\left(\N_0(x)\right)\right),\,\,\st_{\overline{c_0x}}^{\textup{I}-}\left(\N_0(\overline{c_0x})\right)=\left(\textup{id}\times\iota_0\right)^{\ast}\left(\st_x^{\textup{II}-}\left(\N_0(x)\right)\right).
    \label{local-id-aut}
\end{align}
\par
Notice that by the universal property of blow up morphism $\pi:\M\rightarrow\N(x_0)$, the automorphism $\textup{id}\times\iota_0$ induces an automorphism $\iota^{\M}$ of the formal scheme $\M$. By (\ref{local-id-aut-1}) and (\ref{local-id-aut}), we have the following identities of closed formal subschemes of $\M$:
\begin{align}
    \N_0^{\textup{II}+}(c_0x)=\left(\iota_\M\right)^{\ast}\left(\N_0^{\textup{I}+}(x)\right),\,\,\N_0^{\textup{II}-}\left(\overline{x}\right)=\left(\iota_\M\right)^{\ast}\left(\N^{\textup{I}-}_0(x)\right),\label{local-id-aut-M-1}\\
    \N_0^{\textup{I}+}(x)=\left(\iota_\M\right)^{\ast}\left(\N_0^{\textup{II}+}(x)\right),\,\,\N_0^{\textup{I}-}\left(\overline{c_0x}\right)=\left(\iota_\M\right)^{\ast}\left(\N^{\textup{II}-}_0(x)\right).
    \label{local-id-aut-M}
\end{align}

\subsection{Special cycles on the Rapoport--Zink space $\M$}
\begin{definition}
    For any subset $H\subset\B$, define the special cycles $\mathcal{Z}_{\M}(H)\coloneqq\mathcal{Z}_{\N(x_0)}(H)\times_{\N(x_0)}\M$ and $\mathcal{Y}_{\M}(H)\coloneqq\mathcal{Y}_{\N(x_0)}(H)\times_{\N(x_0)}\M$, i.e.,  we have the following Cartesian diagrams
    \begin{equation*}
        \begin{tikzcd}
    {\mathcal{Z}_{\M}(H)}
    \arrow[d] \arrow[r]
    & {\M}
    \arrow[d, "\pi"]
    \\
    {\mathcal{Z}_{\N(x_0)}(H)}
    \arrow[r]
    &{\N(x_0)},
\end{tikzcd}
\begin{tikzcd}
    {\mathcal{Y}_{\M}(H)}
    \arrow[d] \arrow[r]
    & {\M}
    \arrow[d, "\pi"]
    \\
    {\mathcal{Y}_{\N(x_0)}(H)}
    \arrow[r]
    &{\N(x_0)},
\end{tikzcd}
    \end{equation*}
Since $\mathcal{Y}_{\N(x_0)}(H)=\left(\textup{id}\times\iota_0\right)^{\ast}\left(\mathcal{Z}_{\N(x_0)}(x_0\cdot H)\right)$ and $\mathcal{Z}_{\N(x_0)}(\overline{x_0}\cdot H)=\left(\textup{id}\times\iota_0\right)^{\ast}\left(\mathcal{Y}_{\N(x_0)}(H)\right)$, we have $\mathcal{Y}_{\M}(H)=\left(\iota^{\M}\right)^{\ast}\left(\mathcal{Z}_{\M}(x_0\cdot H)\right)$ and $\mathcal{Z}_{\M}(\overline{x_0}\cdot H)=\left(\iota^{\M}\right)^{\ast}\left(\mathcal{Y}_{\M}(H)\right)$.
\label{special-cycle-on-M}
\end{definition}

By the moduli interpretation of the formal scheme $\N(x_0)$, the morphism $\pi$ gives rise to the following two pairs of deformations of the quasi-isogeny $x_0$ over the formal scheme $\M$:
\begin{equation}
    \left(X_{1,\M}\stackrel{(x_0)_{1,\M}}\longrightarrow X_{1,\M}^{\prime},\left(\rho_{1,\M},\rho_{1,\M}^{\prime}\right)\right),\,\,\,\,\left(X_{2,\M}\stackrel{(x_0)_{2,\M}}\longrightarrow X_{2,\M}^{\prime},\left(\rho_{2,\M},\rho_{2,\M}^{\prime}\right)\right).
    \label{univ-obj-M}
\end{equation}
By the moduli interpretation of the special cycle $\mathcal{Z}_{\N(x_0)}(H)$ in Definition \ref{special-cycle-originalmodel}, the special cycle $\mathcal{Z}_{\M}(H)$ is cut out by the following conditions,
\begin{equation*}
    \rho_{2,\M}\circ x \circ (\rho_{1,\M})^{-1}\in \textup{Hom}(X_{1,\M}, X_{2,\M});
\end{equation*}
\begin{equation*}
    \rho_{2,\M}^{\prime}\circ x^{\prime} \circ (\rho_{1,\M}^{\prime})^{-1}\in \textup{Hom}(X_{1,\M}^{\prime}, X_{2,\M}^{\prime}).
\end{equation*}
for all $x\in H$.
\par
For simplicity, we denote the special cycle $\mathcal{Z}_{\M}(H)$ by $\mathcal{Z}(H)$ in all the following paragraphs. If $H=\{x\}$ where $x\in\B$, denote the special cycle $\mathcal{Z}(\{x\})$ by $\mathcal{Z}(x)$.
\par
Recall that we define $\M^{+}=\M_1^{+}\cup\M_2^{+}$ and $\M^{-}=\M_1^{-}\cup\M_2^{-}$. For a symbol $?\in\{+,-\}$, let the pair
\begin{equation*}
    \left((X_{1,\M^{?}}\stackrel{(x_0)_{1,\M^{?}}}\longrightarrow X_{1,\M^{?}}^{\prime}),(\rho_{1,\M^{?}},\rho_{1,\M^{?}}^{\prime})\right),\,\,\,\,\left((X_{2,\M^{?}}\stackrel{(x_0)_{2,\M^{?}}}\longrightarrow X_{2,\M^{?}}^{\prime}),(\rho_{2,\M^{?}},\rho_{2,\M^{?}}^{\prime})\right)
\end{equation*}
be the base change of the universal pair (\ref{univ-obj-M}) over $\M$ to $\M^{?}$. Define $\mathcal{Z}^{?}(x)\coloneqq\mathcal{Z}(x)\cap\M^{?}$. Let $p_{?}:\M^{?}\rightarrow\N$ be the following composition
\begin{equation}
    p_{?}:\M^{?}\rightarrow\M\stackrel{\pi}\rightarrow\N(x_0)\stackrel{s_{?}}\rightarrow\N.
    \label{pMtoN}
\end{equation} 
\begin{lemma}\label{lem: divisor}
    Let $x\in\mathbb{B}$ be a nonzero element.The closed formal subscheme $\mathcal{Z}(x)$ of the formal scheme $\M$ is cut out by the conditions
    \begin{itemize}
        \item [(+)] On $\M^{+}$: $\rho_{2,\M^{+}}\circ x \circ (\rho_{1,\M^{+}})^{-1}\in \textup{Hom}(X_{1,\M^{+}}, X_{2,\M^{+}})$, i.e., 
        \begin{equation*}
            \mathcal{Z}^{+}(x)=\mathcal{Z}_{\N}(x)\times_{\N,p_{+}}\M^{+}.
        \end{equation*}
        \item[($-$)]  On $\M^{-}$: $\rho_{2,\M^{-}}\circ x^{\prime} \circ (\rho_{1,\M^{-}})^{-1}\in \textup{Hom}(X_{1,\M^{-}}, X_{2,\M^{-}})$, i.e., 
        \begin{equation*}
            \mathcal{Z}^{-}(x)=\mathcal{Z}_{\N}(x^{\prime})\times_{\N,p_{-}}\M^{-}.
        \end{equation*}
    \end{itemize}
    Therefore the special cycle $\mathcal{Z}(x)$ is a Cartier divisor on $\M$.
    \label{divisor}
\end{lemma}
\begin{proof}
    We first prove the $(+)$ case. Let $z\in\mathcal{Z}^{+}(x)(\F)$ be a point. Let $R$ be the completed local ring of the formal scheme $\M$ at $z$. Without loss of generality, we can assume that $z\in\M_1^{+}(\F)$. Let $I\subset R$ be the ideal cutting out the divisor $\mathcal{Z}(x)$, we will show that $I$ is a principal ideal.
    \par
    By the moduli interpretation of $\mathcal{Z}(x)$, the ideal $I$ is generated by two elements $f_x$ and $f_{x^{\prime}}$, where $f_x$ (resp. $f_{x^{\prime}}$) is the equation of the universal closed formal subscheme of $\M$ over which the quasi-isogeny $x$ (resp. $x^{\prime}$) lifts to an isogeny. The surjective homomorphism $R/I^{2}\rightarrow R/I$ is equipped with a nilpotent pd structure. We have the following commutative diagram,
    \begin{equation*}
\begin{tikzcd}
    {\mathbb{D}(X_1)}
    \arrow[d, "\mathbb{D}(x_0)"'] \arrow[r,"\mathbb{D}(x)"]
    & {\mathbb{D}(X_2)}
    \arrow[d, "\mathbb{D}(x_0)"]
    \\
    {\mathbb{D}(X_1^{\prime})}
    \arrow[r,"\mathbb{D}(x^{\prime})"]
    &{\mathbb{D}(X_2^{\prime})}.
\end{tikzcd}
    \end{equation*}
    Let's use the specific basis of the four Dieudonne modules stated in Proposition \ref{local-gamma0p}. Then
    \begin{equation*}
        \mathbb{D}(x)[e_1,f_1]=[e_2,f_2]\begin{pmatrix}
            a & -pc^{\prime}\\
            c & d
        \end{pmatrix},\,\,\,\,\mathbb{D}(x^{\prime})[e_1^{\prime},f_1^{\prime}]=[e_2^{\prime},f_2^{\prime}]\begin{pmatrix}
            a^{\prime} & -pc\\
            c^{\prime} & a
        \end{pmatrix},
    \end{equation*}
    where $a,a^{\prime},c,c^{\prime},d\in R/I^{2}$ and $p(a^{\prime}-d)=0\in R/I^{2}$. Then
    \begin{align*}
        \mathbb{D}(x)(f_1+x_1e_1)=(d+cx_1)f_2+(ax_1-pc^{\prime})e_2,\\
        \mathbb{D}(x^{\prime})(f_1^{\prime}+y_1e_1^{\prime})=(a+c^{\prime}y_1)f_2^{\prime}+(a^{\prime}y_1-pc)e_2^{\prime}.
    \end{align*}
    
    Since $z\in\M_1^{+}$, we use the coordinate for $\M_1^{+}$ in $\S$\ref{opencover}. Recall that we have $y_1=ux_1$ and $x_2=v x_1$. The filtration on the Dieudonne crystals are given explicitly in Proposition \ref{local-gamma0p}, hence
    \begin{align*}
        &f_x=cvx_1^{2}+dvx_1-ax_1-c^{\prime}uvx_1^{2}\in R/I^{2};\\
        &f_{x^{\prime}}=c^{\prime}u^{2}vx_1^{2}+aux_1-a^{\prime}uvx_1-cuvx_1^{2}=-uf_x+uvx_1(d-a^{\prime})\in R/I^{2}.
    \end{align*}
    Notice that $x_1\vert f_x,f_{x^{\prime}}$, let $g=f_x/x_1$ and $g^{\prime}=f_{x^{\prime}}/x_1$ (In Proposition \ref{difference-decomposition}, we calculated the equations of $f_x$ and $f_{x^{\prime}}$ which imply that both $g$ and $g^{\prime}$ are not units in $R$). Since $p=-uvx_1^{2}$ and $p(a^{\prime}-d)=0\in R/I^{2}$, we have
    \begin{equation*}
        uvx_1(d-a^{\prime})\in(gf_x,g^{\prime}f_{x^{\prime}},gf_{x^{\prime}}).
    \end{equation*}
    Then $f_{x^{\prime}}-uf_x\in(gf_x,g^{\prime}f_{x^{\prime}},gf_{x^{\prime}})+I^{2}$. Therefore $f_{x^{\prime}}=t\cdot f_{x}$ for some element $t\in R$. Hence $I=(f_x,f_{x^{\prime}})=(f_x)$ is a principal ideal, and $\mathcal{Z}^{+}(x)=\mathcal{Z}_{\N}(x)\times_{\N,p_{+}}\M^{+}$. The proof of the $(-)$ case is similar so we omit the details here.
\end{proof}
By the identity $\mathcal{Y}(x)=\left(\iota^{\M}\right)^{\ast}(\mathcal{Z}(x_0\cdot x))$, we obtain the following
\begin{corollary}\label{cor: divisor}
     Let $x\in\mathbb{B}$ be a nonzero element.The closed formal subscheme $\mathcal{Y}(x)$ is a Cartier divisor on $\M$.
\end{corollary}
In the following paragraphs, for a nonzero element $x\in\mathbb{B}$, we refer to $\mathcal{Z}(x)$ as a special divisor on the formal scheme $\M$. By the moduli interpretation of the special divisor $\mathcal{Z}(x)$ in Proposition \ref{divisor}, there is a closed immersion $\mathcal{Z}(p^{-1}x)\rightarrow\mathcal{Z}(x)$.
\begin{definition}
    Let $x\in\mathbb{B}$ be a nonzero element. Define the difference divisor associated to $x$ to be the following effective Cartier divisor on the formal scheme $\M$,
\begin{equation*}
    \mathcal{D}(x)=\mathcal{Z}(x)-\mathcal{Z}(p^{-1}x).
\end{equation*}
For a symbol $?\in\{+,-\}$, define $\mathcal{D}^{?}(x)=\mathcal{D}(x)\cap\M^{?}$.
\label{diff-div}
\end{definition}
\begin{remark}
    By Proposition \ref{divisor}, we have
    \begin{equation*}
        \mathcal{D}^{+}(x)=\mathcal{N}_0(x)\times_{\N,p_{+}}\M^{+}\,\,\textup{and}\,\,\mathcal{D}^{+}(x)=\mathcal{N}_0(x^{\prime})\times_{\N,p_{-}}\M^{-}.
    \end{equation*}
    By the isomorphism $\mathcal{D}_{\N}(x)\simeq\N_0(x)$ in Lemma \ref{cyclic-diff}, we can give a moduli interpretation for the difference divisor $\mathcal{D}(x)$ as follows:
    \begin{align*}
        &\mathcal{D}^{+}(x): \textup{the quasi-isogeny} \rho_{2,\M^{+}}\circ x \circ (\rho_{1,\M^{+}})^{-1} \textup{lifts to a cyclic isogeny},\\
        &\mathcal{D}^{-}(x): \textup{the quasi-isogeny} \rho_{2,\M^{-}}\circ x^{\prime} \circ (\rho_{1,\M^{-}})^{-1} \textup{lifts to a cyclic isogeny}.
    \end{align*}
    \label{moduli-diff-M}
\end{remark}

\section{Special divisor and the exceptional divisor}
\label{spe-exc-part}
\subsection{Line bundles on the exceptional divisor $\exc_{\M}$}
Before stating the main results of this part, we record the following lemma.
\begin{lemma}
    Let $X$ be a regular formal scheme of dimension $d$. Let $C,D\subset X$ be two effective Cartier divisors on $X$, then we have the following equality in $\textup{Gr}^{1}K_0(X)$,
    \begin{equation*}
        [\rO_{C+D}]=[\rO_{C}]+[\rO_{D}].
    \end{equation*}
    \label{linearization}
\end{lemma}
\begin{proof}
    Let $\mathcal{I}$ and $\mathcal{J}$ be the ideal sheaf of the effective Cartier divisors $C$ and $D$ respectively. Then 
    \begin{equation*}
        [\rO_C]=[\rO_X]-[\mathcal{I}],\,\,\,\,[\rO_D]=[\rO_X]-[\mathcal{J}].
    \end{equation*}
    Since $[\rO_C],[\rO_D]\in\textup{F}^{1}K_0(X)$, we have $[\rO_C]\cdot[\rO_D]\in\textup{F}^{2}K_0(X)$ by (\ref{codim-add}). Therefore
    \begin{equation*}
        0=[\rO_C]\cdot[\rO_D]=\left([\rO_X]-[\mathcal{I}]\right)\cdot\left([\rO_X]-[\mathcal{J}]\right)\in\textup{Gr}^{1}K_0(X).
    \end{equation*}
    Then we have the following equality in $\textup{Gr}^{1}K_0(X)$,
    \begin{equation*}
        [\rO_X]-[\mathcal{I}]-[\mathcal{J}]+[\mathcal{I}]\cdot[\mathcal{J}]=[\rO_X]-[\mathcal{I}]-[\mathcal{J}]+[\mathcal{IJ}]=0.
    \end{equation*}
    Hence the following equality in $\textup{Gr}^{1}K_0(X)$,
    \begin{equation*}
        [\rO_{C+D}]=[\rO_X]-[\mathcal{IJ}]=[\rO_X]-[\mathcal{I}]+[\rO_X]-[\mathcal{J}]=[\rO_{C}]+[\rO_{D}].
    \end{equation*}
\end{proof}
Under the isomorphism $\iota:\exc_{\M}\simeq\bP_{\F}^{1}\times\bP_{\F}^{1}$ we fixed in Remark \ref{coordinate-change}, let $\textup{pr}_1:\exc_{\M}\rightarrow\bP_{\F}^{1}$ be the fist projection, while $\textup{pr}_2:\exc_{\M}\rightarrow\bP_{\F}^{1}$ be the second projection. Let $\mathcal{L}$ be a line bundle on $\exc_{\M}$, it is isomorphic to $\textup{pr}_1^{\ast}\mathcal{O}(m)\otimes\textup{pr}_2^{\ast}\mathcal{O}(n)$ for some integers $m$ and $n$. Define
\begin{equation*}
    \mathcal{O}(m,n)\coloneqq\textup{pr}_1^{\ast}\mathcal{O}(m)\otimes\textup{pr}_2^{\ast}\mathcal{O}(n)\in\textup{Pic}(\exc_{\M}).
\end{equation*}
\par
For two line bundles $\mathcal{L},\mathcal{G}\in\textup{Pic}(\exc_{\M})$, denote by $\mathcal{L}\cdot\mathcal{G}$ the intersection product $\textup{Pic}(\exc_{\M})\times\textup{Pic}(\exc_{\M})\rightarrow\mathbb{Z}$ on the group $\textup{Pic}(\exc_{\M})$.
\begin{lemma}
    Let $m_1,m_2,n_1,n_2$ be integers, then the intersection number of the line bundles $\mathcal{O}(m_1,n_1)$ and $\mathcal{O}(m_2,n_2)$ is
    \begin{equation*}
        \mathcal{O}(m_1,n_1)\cdot\mathcal{O}(m_2,n_2)=m_1n_2+m_2n_1.
    \end{equation*}
    \label{int-on-exc}
\end{lemma}
\begin{proof}
    This follows from the fact that $\mathcal{O}(1,0)\cdot\mathcal{O}(0,1)=1$ and the additivity of intersection numbers between line bundles.
\end{proof}
\subsection{Self-intersection of the exceptional divisor}
For $i=0,1,2$, we have $\textup{F}^{3-i}K_{0}^{\exc_{\M}}(\M)\simeq\textup{F}_{i}K_0(\exc_{\M})$. Notice that
\begin{align*}
    \textup{Gr}^{2}K_{0}^{\exc_{\M}}(\M)&\coloneqq\textup{F}^{2}K_{0}^{\exc_{\M}}(\M)\big/\textup{F}^{3}K_{0}^{\exc_{\M}}(\M)\\&\simeq\textup{F}_{1}K_0(\exc_{\M})\big/\textup{F}_{0}K_0(\exc_{\M})\simeq\textup{CH}^{1}(\exc_{\M})\simeq\textup{Pic}(\exc_{\M}).
\end{align*}
\begin{lemma}
    We have the following equality in the group $\textup{Gr}^{2}K_{0}^{\exc_{\M}}(\M)\simeq\textup{Pic}(\exc_{\M})$,
    \begin{equation*}
        [\mathcal{O}_{\exc_{\M}}\otimes^{\mathbb{L}}_{\rO_{\M}}\mathcal{O}_{\exc_{\M}}]=\rO(-1,-1).
    \end{equation*}
    \label{self-int-exc}
\end{lemma}
\begin{proof}
 We know that $[\mathcal{O}_{\exc_{\M}}\otimes^{\mathbb{L}}_{\rO_{\M}}\mathcal{O}_{\exc_{\M}}]$ is given by restricting the line bundle corresponding to $\exc_{\M}$ on $\M$ to $\exc_{\M}$. Under the open cover of $\M$ in $\S$\ref{opencover}, the divisor $\exc_{\M}$ is given by the following equations and transformation rules:
 \begin{equation*}
     \begin{tikzcd}
     \M_1^{+} & \M_1^{-} & \M_2^{+} & \M_2^{-}
    \\
x_1 \arrow[r,shift left,"\times uv"] \arrow[Subseteq]{u}{}
    &
y_1 \arrow[l,shift left,"\times (uv)^{-1}"] \arrow[r,shift left, "\times u^{-1}"] \arrow[Subseteq]{u}{}
&
x_2 \arrow[l, shift left, "\times u"] \arrow[r,shift left,"\times uv^{-1}"] \arrow[Subseteq]{u}{}
    &
y_2 \arrow[l,shift left,"\times vu^{-1}"] \arrow[Subseteq]{u}{}
\end{tikzcd}
 \end{equation*}
    The same transformation rule also applies to the corresponding open cover of $\exc_{\M}\simeq\mathbb{P}_{\mathbb{F}}^{1}\times\mathbb{P}_{\mathbb{F}}^{1}$ in Remark \ref{coordinate-exc}. Therefore the corresponding line bundle is $\mathcal{O}(-1,-1)$.
\end{proof}
\subsection{Intersections of special divisors and the exceptional divisor}
Let $x\in\mathbb{B}$ be a nonzero element such that $q(x)\in\zp$. Define $\sz(x)\subset\M$ to be the strict transformation of the cycle $\mathcal{Z}_{\N(x_0)}(\{x\})\subset\N(x_0)$ under the blow up morphism $\pi:\M\rightarrow\N(x_0)$.
\begin{proposition}
     Let $x\in\mathbb{B}$ be a nonzero element such that $q(x)\in\zp$ such that $\nu_p(q(x))=n$ for some integer $n\geq0$. Then there exists a Cartier divisor $\widetilde{\mathcal{D}}(x)$ on the formal scheme $\M$ which intersects with the exceptional divisor $\exc_{\M}$ properly and
     \begin{itemize}
         \item [(i)] When $n=0$, the following equality of Cartier divisors on $\M$ holds,
         \begin{equation*}
             \mathcal{D}(x)=\exc_{\M}+\widetilde{\mathcal{D}}(x).
         \end{equation*}
         Moreover, the Cartier divisor $\exc_{\M}\cap\widetilde{\mathcal{D}}(x)$ of $\exc_{\M}$ is
         \begin{equation*}
             \exc_{\M}\cap\widetilde{\mathcal{D}}(x)=\left(v=\textup{a nonzero number in $\F$}\right).
         \end{equation*}
         \item [(ii)] When $n\geq1$, the following equality of Cartier divisors on $\M$ holds,
         \begin{equation*}
             \mathcal{D}(x)=2\cdot\exc_{\M}+\widetilde{\mathcal{D}}(x),
         \end{equation*}
         Moreover, the Cartier divisor $\exc_{\M}\cap\widetilde{\mathcal{D}}(x)$ of $\exc_{\M}$ is
         \begin{equation*}
             \exc_{\M}\cap\widetilde{\mathcal{D}}(x)=\begin{cases}
                 \left(v=0\right)+\left(u=\textup{a nonzero number in $\F$}\right)+\left(v^{\prime}=0\right), &\textup{when $n=1$;}\\
            \left(v=0\right)+\left(u=0\right)+\left(v^{\prime}=0\right)+\left(u^{\prime}=0\right), &\textup{when $n\geq2$}.
             \end{cases}
         \end{equation*}
     \end{itemize}
     \label{difference-decomposition}
\end{proposition}
\begin{proof}
    By Lemma \ref{divisor}, we have
    \begin{equation*}
        \mathcal{Z}^{+}(x)=\mathcal{Z}_{\N}(x)\times_{\N,p_{+}}\M^{+}\,\,\textup{and}\,\,\mathcal{Z}^{-}(x)=\mathcal{Z}_{\N}(x^{\prime})\times_{\N,p_{-}}\M^{-}.
    \end{equation*}
    By Lemma \ref{cyclic-diff}, we have
    \begin{equation*}
        \mathcal{D}^{+}(x)=\mathcal{N}_0(x)\times_{\N,p_{+}}\M^{+}\,\,\textup{and}\,\,\mathcal{D}^{+}(x)=\mathcal{N}_0(x^{\prime})\times_{\N,p_{-}}\M^{-}.
    \end{equation*}
    For $i=1,2$, let $p_{?i}:\M_i^{?}\rightarrow\N$ be the composition $\M_i^{?}\rightarrow\M^{?}\stackrel{p_{?}}\rightarrow\N$. Let $p_{?i}^{\#}:\rO_{\N}\rightarrow\rO_{\M_i^{?}}$ be the corresponding ring homomorphism. We will prove (i) and (ii) by studying the equation of the divisor $\mathcal{Z}(x)$ on the open formal subscheme $\M_{i}^{?}$ for $i=1,2$ and $?\in\{+,-\}$. For $i=1,2$, let $t_i^{+},t_{i}^{-}\in\rO_{\N_0(x_0)}$ be two elements satisfying the assumption in Lemma \ref{coordinate-change}.
    \par
    We first prove (i). If $?=+$, the equation of the special cycle $\mathcal{Z}_{\N}(x)$ is $t_1^{+}-\nu_x t_2^{+}+p\cdot f_x$ for an element $f_{x}\in W[[t_1^{+},t_2^{+}]]$ and an invertible element $\nu_x\in W[[t_1^{+},t_2^{+}]]$. By Lemma \ref{coordinate-change}, we have
    \begin{equation*}
        s_{+}^{\#}(t_1^{+}-\nu_x t_2^{+})=\nu_{11}x_1+\nu_{12}y_1^{p}-\nu_x\nu_{21}x_2-\nu_x\nu_{22}y_2^{p}+p\cdot h,
    \end{equation*}
    where $h\in W[[t_1^{+},t_2^{+}]]$ is an element and $\nu_{ij}$ are invertible elements in the ring $W[[x_{i},y_{i}]]/(p+x_iy_i)$. Therefore the equation of the divisor $\mathcal{Z}(x)$ on $\M_i^{+}$ is
    \begin{align}
        &i=1:\,\,p_{+1}^{\#}(t_1^{+}-\nu_x t_2^{+})=x_1\cdot(\nu_{11}-\nu_x\nu_{21}v+\nu_{12}u^{p}v^{p}x_1^{p-1}-\nu_x\nu_{22}u^{p}x_{1}^{p-1}-uvx_1\cdot h).\label{0+1}\\
        &i=2:\,\,p_{+2}^{\#}(t_1^{+}-\nu_x t_2^{+})=x_2\cdot(\nu_{11}v^{\prime}-\nu_x\nu_{21}+\nu_{12}u^{p}x_2^{p-1}-\nu_x\nu_{22}v^{\prime p}u^{p}x_{2}^{p-1}-uv^{\prime}x_2\cdot h).
    \end{align}
    \par
    If $?=-$, the equation of the special cycle $\mathcal{Z}_{\N}(x^{\prime})$ is $t_1^{-}-\omega_{x^{\prime}} t_2^{-}+p\cdot g_{x^{\prime}}$ for an element $g_{x^{\prime}}\in W[[t_1^{-},t_2^{-}]]$ and an invertible element $\omega_{x^{\prime}}\in W[[t_1^{-},t_2^{-}]]$. By Lemma \ref{coordinate-change}, we have
    \begin{equation*}
        s_{-}^{\#}(t_1^{-}-\omega_{x^{\prime}} t_2^{-})=\omega_{11}y_1+\omega_{12}x_1^{p}-\omega_{x^{\prime}}\omega_{21}y_2-\omega_{x^{\prime}}\omega_{22}x_2^{p}+p\cdot g,
    \end{equation*}
    where $g\in W[[t_1^{-},t_2^{-}]]$ is an element and $\omega_{ij}$ are invertible elements in the ring $W[[x_{i},y_{i}]]/(p+x_iy_i)$. Therefore the equation of the divisor $\mathcal{Z}(x)$ on $\M_i^{-}$ is
    \begin{align}
        &i=1:\,\,p_{-1}^{\#}(t_1^{-}-\omega_{x^{\prime}} t_2^{-})=y_1\cdot(\omega_{11}-\omega_{x^{\prime}}\omega_{21}v^{\prime}+\omega_{12}u^{\prime p}v^{\prime p}y_1^{p-1}-\omega_{x^{\prime}}\omega_{22}u^{\prime p}y_{1}^{p-1}-u^{\prime}v^{\prime}y_1\cdot g).\\
        &i=2:\,\,p_{-2}^{\#}(t_1^{-}-\omega_{x^{\prime}} t_2^{-})=y_2\cdot(\omega_{11}v-\omega_{x^{\prime}}\omega_{21}+\omega_{12}u^{\prime p}y_2^{p-1}-\omega_{x^{\prime}}\omega_{22}v^{p}u^{\prime p}y_{2}^{p-1}-u^{\prime}vy_2\cdot g)\label{0-2}.
    \end{align}
    \par
    Define $\widetilde{\mathcal{D}}(x)\coloneqq\mathcal{Z}(x)-\exc_{\M}$. By the equations (\ref{0+1})-(\ref{0-2}) of the special divisor $\mathcal{Z}(x)$, we deduce that $\widetilde{\mathcal{D}}(x)$ is an effective Cartier divisor which intersects with the exceptional divisor $\exc_{\M}$ properly. The intersection $\widetilde{\mathcal{D}}(x)\cap\exc_{\M}\subset\exc_{\M}$ is given by the equation $v=\overline{\nu_x^{-1}\nu_{21}^{-1}\nu_{11}}=\overline{\omega_{x^{\prime}}\omega_{21}\omega_{11}^{-1}}$, where the symbol $\overline{(\cdot)}$ means the image of the corresponding element under the map $\mathcal{O}_{\M}\rightarrow\rO_{\exc_{\M}}$. Notice that $\overline{\nu_x^{-1}\nu_{21}^{-1}\nu_{11}}=\overline{\omega_{x^{\prime}}\omega_{21}\omega_{11}^{-1}}$ is a nonzero number in $\F$. Therefore (i) is true.
    \par
    Now we prove (ii). If $?=+$. By Lemma \ref{cyclic-diff}, the equation of the special cycle $\mathcal{D}_{\N}(x)$ is
    \begin{equation*}
        z_x=\nu p+\left(t_{1}^{+}-\left(\nu_xt_2^{+}\right)^{p^{n}}\right)\left(t_1^{+p^{n}}-\nu_xt_2^{+}\right)\cdot\prod\limits_{\substack{a+b=n\\a,b\geq1}}\left(t_1^{+p^{a-1}}-\left(\nu_xt_2^{+}\right)^{p^{b-1}}\right)^{p-1}
    \end{equation*}
    where $\nu_x$ and $\nu$ are two invertible elements in the ring $W[[t_1^{+},t_2^{+}]]$. By Lemma \ref{coordinate-change}, we have
    \begin{align}
        s_{+}^{\#}(z_x)=\nu p+&\left(\nu_{11}x_1+\nu_{12}y_1^{p}-\left(\nu_x\nu_{21}x_2+\nu_x\nu_{22}y_{2}^{p}\right)^{p^{n}}\right)\left(\left(\nu_{11}x_1+\nu_{12}y_1^{p}\right)^{p^{n}}-\nu_x\left(\nu_{21}x_2+\nu_{22}y_2^{p}\right)\right)\notag\\
        \cdot&\prod\limits_{\substack{a+b=n\\a,b\geq1}}\left(\left(\nu_{11}x_1+\nu_{12}y_1^{p}\right)^{p^{a-1}}-\left(\nu_x\nu_{21}x_2+\nu_x\nu_{22}y_2^{p}\right)^{p^{b-1}}\right)^{p-1},
        \label{equation-x}
    \end{align}
    Therefore the equation of the divisor $\mathcal{D}_{\N}(x)$ on $\M_i^{+}$ is
    \begin{align}
        &i=1: p_{+1}^{\#}(z_x)=-\nu uvx_1^{2}+x_1^{m_n}\cdot\begin{cases}
            \Tilde{\nu}_1v+x_1^{p-1}h_1, & $\textup{when $n=1$;}$\\
            h_1, & $\textup{when $n=2$.}$
        \end{cases}\label{eq_x+1}\\
        &i=2: p_{+2}^{\#}(z_x)=-\nu uv^{\prime}x_2^{2}+x_2^{m_n}\cdot\begin{cases}
            \Tilde{\nu}_2v^{\prime}+x_2^{p-1}h_2, & $\textup{when $n=1$;}$\\
            h_2, & $\textup{when $n=2$.}$
        \end{cases}
    \end{align}
    Here $\nu,\Tilde{\nu}_i$ are invertible elements in the ring $\rO_{\M_{i}^{+}}$, $h_i$ is an element in the ring $\rO_{\M_{i}^{+}}$. By the equation (\ref{equation-x}), the integer $m_n$ satisfies the following conditions
    \begin{equation*}
        m_n\begin{cases}
            =2, &\textup{when $n=1$;}\\
            \geq3, &\textup{when $n\geq2$.}
        \end{cases}
    \end{equation*}
    Similar method also gives the following equations of the special divisor $\mathcal{D}(x)$ on the open formal subscheme $\M^{-}=\M_1^{-}\cup\M_2^{-}$. Let $z_{x^{\prime}}\in W[[t_1^{-},t_{2}^{-}]]$ be the equation of the special divisor $\mathcal{Z}_{\N}(x^{\prime})$.
    \begin{align}
        &i=1: p_{-1}^{\#}(z_{x^{\prime}})=-\omega u^{\prime}v^{\prime}y_1^{2}+y_1^{m_n}\cdot\begin{cases}
            \Tilde{\omega}_1v^{\prime}+y_1^{p-1}g_1, & $\textup{when $n=1$;}$\\
            g_1, & $\textup{when $n=2$.}$
        \end{cases}\\
        &i=2: p_{-2}^{\#}(z_{x^{\prime}})=-\omega u^{\prime}v y_2^{2}+y_2^{m_n}\cdot\begin{cases}
            \Tilde{\omega}_2v+y_2^{p-1}g_2, & $\textup{when $n=1$;}$\\
            g_2, & $\textup{when $n=2$.}$\label{eq-x-2}
        \end{cases}
    \end{align}
    Here $\omega,\Tilde{\omega}_i$ are invertible elements in the ring $\rO_{\M_{i}^{-}}$, $g_i$ is an element in the ring $\rO_{\M_{i}^{-}}$.
    \par
    Define $\widetilde{\mathcal{D}}(x)\coloneqq\mathcal{D}(x)-2\cdot\exc_{\M}$. By the equations (\ref{eq_x+1})-(\ref{eq-x-2}) of the difference divisor $\mathcal{D}(x)$, we deduce that $\widetilde{\mathcal{D}}(x)\coloneqq\mathcal{D}(x)-2\cdot\exc_{\M}$ is an effective Cartier divisor which intersects with the exceptional divisor properly. The intersection $\widetilde{\mathcal{D}}(x)\cap\exc_{\M}$ is an effective Cartier divisor on the exceptional divisor $\exc_{\M}$. More specifically,
    \begin{equation*}
        \widetilde{\mathcal{D}}(x)\cap\exc_{\M}=\begin{cases}
            \left(v=0\right)+\left(u=\textup{a nonzero number in $\F$}\right)+\left(v^{\prime}=0\right), &\textup{when $n=1$;}\\
            \left(v=0\right)+\left(u=0\right)+\left(v^{\prime}=0\right)+\left(u^{\prime}=0\right), &\textup{when $n\geq2$}.
        \end{cases}
    \end{equation*}
    Therefore (ii) is true.
\end{proof}
\begin{corollary}
    Let $x\in\mathbb{B}$ be a nonzero element such that $q(x)\in\zp$ such that $\nu_p(q(x))\geq0$. Denote by $\widetilde{\mathcal{Z}}(x)$ the following Cartier divisor on $\M$:
    \begin{equation*}
        \widetilde{\mathcal{Z}}(x)=\sum\limits_{i=0}^{[n/2]}\widetilde{\mathcal{D}}(p^{-i}x).
    \end{equation*}
    \begin{itemize}
        \item [(a)] The following identities of Cartier divisors on $\M$ hold,
        \begin{equation*}
            \mathcal{Z}(x)=(n+1)\cdot\exc_{\M}+\sz(x).
        \end{equation*}
        \item[(b)] The following identities in the group $\textup{Gr}^{2}K_{0}^{\exc_{\M}}(\M)\simeq\textup{Pic}(\exc_{\M})$ holds,
        \begin{equation}
             [\mathcal{O}_{\exc_{\M}}\otimes^{\mathbb{L}}_{\rO_{\M}}\mathcal{O}_{\widetilde{\mathcal{D}}(x)}]=[\mathcal{O}_{\exc_{\M}}\otimes_{\rO_{\M}}\mathcal{O}_{\widetilde{\mathcal{D}}(x)}]=[\mathcal{O}_{\exc_{\M}\cap\widetilde{\mathcal{D}}(x)}]=\begin{cases}
                 \rO(1,0), &\textup{when $n=0$;}\\
                 \rO(2,1), &\textup{when $n=1$;}\\
                 \rO(2,2), &\textup{when $n\geq2$.}
             \end{cases}
             \label{diff-exc}
         \end{equation}
        \begin{equation}
            [\mathcal{O}_{\exc_{\M}}\otimes^{\mathbb{L}}_{\rO_{\M}}\mathcal{O}_{\mathcal{Z}(x)}]=\rO(0,-1).
            \label{spe-exc}
        \end{equation}
    \end{itemize}
    \label{spe-decom}
\end{corollary}
\begin{proof}
    By the definition of difference divisors in Definition \ref{diff-div}, we have the following equality as effective Cartier divisors on $\M$,
    \begin{equation*}
        \mathcal{Z}(x)=\sum\limits_{i=0}^{[n/2]}\mathcal{D}(p^{-i}x).
    \end{equation*}
    By Proposition \ref{difference-decomposition}, we have
    \begin{equation*}
        \mathcal{Z}(x)=(n+1)\cdot\exc_{\M}+\sum\limits_{i=0}^{[n/2]}\widetilde{\mathcal{D}}(p^{-i}x),
    \end{equation*}
    and the effective Cartier divisor $\sum\limits_{i=0}^{[n/2]}\widetilde{\mathcal{D}}(p^{-i}x)$ intersects with the exceptional divisor $\exc_{\M}$ properly. Therefore we conclude that
    \begin{equation*}
        \widetilde{\mathcal{Z}}(x)=\sum\limits_{i=0}^{[n/2]}\widetilde{\mathcal{D}}(p^{-i}x).
    \end{equation*}
    Hence $\mathcal{Z}(x)=(n+1)\cdot\exc_{\M}+\sz(x)$.
    \par
    The formula (\ref{diff-exc}) in (b) follows from Proposition \ref{difference-decomposition}. For the formula (\ref{spe-exc}) in (b): By Lemma \ref{linearization}, we know that $\rO_{\mathcal{Z}(x)}=(n+1)\rO_{\exc_{\M}}+\sum\limits_{i=0}^{[n/2]}\rO_{\widetilde{\mathcal{D}}(p^{-i}x)}$ in $\textup{Gr}^{1}K_0(X)$. Therefore by Lemma \ref{self-int-exc} and (\ref{diff-exc}), we have the following equality in $\textup{Gr}^{2}K_0^{\exc_{\M}}(X)\simeq\textup{Pic}(\exc_{\M})$,
    \begin{align*}
        [\rO_{\exc_{\M}}\otimes^{\mathbb{L}}_{\rO_{\M}}\rO_{\mathcal{Z}(x)}]&=(n+1)[\rO_{\exc_{\M}}\otimes^{\mathbb{L}}_{\rO_{\M}}\rO_{\exc_{\M}}]+\sum\limits_{i=0}^{[n/2]}[\rO_{\exc_{\M}}\otimes^{\mathbb{L}}_{\rO_{\M}}\rO_{\widetilde{\mathcal{D}}(p^{-i}x)}]\\
        &=(n+1)\rO(-1,-1)+\rO(n+1,n)=\rO(0,-1).
    \end{align*}
\end{proof}

\subsection{Local Hecke correspondences and the exceptional divisor}
Recall that for an element $x\in\B$ such that $\nu_p\left(q(x)\right)\geq1$, we defined closed formal subschemes $\N^{\textup{I}+}_0(x), \N^{\textup{I}-}_0(x), \N^{\textup{II}+}_0(x), \N^{\textup{II}-}_0(x)$ of $\M$.
\begin{lemma}
    Let $x\in\B$ be an element such that $\nu_p\left(q(x)\right)\geq0$. The formal subschemes $\N^{\textup{I}+}_0(x_0\cdot x)$, $\N^{\textup{I}-}_0(x_0\cdot\overline{x})$, $\N^{\textup{II}+}_0(x)$ and $\N^{\textup{II}-}_0(x^{\prime})$ (the later two spaces are only defined for elements $x$ such that $\nu_p\left(q(x)\right)\geq1$) are all closed formal subschemes of the special divisor $\mathcal{Z}(x)$.
    \label{lie-in-spe-div}
\end{lemma}
\begin{proof}
    By Lemma \ref{ori-diff}, the morphisms $\st_{x_0\cdot x}^{\textup{I}+}$, $\st_{x_0\cdot \overline{x}}^{\textup{I}-}$, $\st_{x}^{\textup{II}+}$ and $\st_{x^{\prime}}^{\textup{II}-}$ sends the corresponding source formal schemes to $\mathcal{Z}_{\N(x_0)}(x)$. Therefore the strict transforms of the corresponding formal schemes $\N^{\textup{I}+}_0(x_0\cdot x)$, $\N^{\textup{I}-}_0(x_0\cdot\overline{x})$, $\N^{\textup{II}+}_0(x)$ and $\N^{\textup{II}-}_0(x^{\prime})$ are all mapped into the direct base change $\mathcal{Z}_{\N(x_0)}(x)\times_{\N(x_0)}\M=\mathcal{Z}(x)$.
\end{proof}
\begin{remark}
    By Lemma \ref{special-0-1}, we know that
    \begin{itemize}
        \item The regular divisor $\N^{\textup{I}+}_0(x_0\cdot x)=\N^{\textup{I}-}_0(x_0\cdot\overline{x})$ if $\nu_p(q(x))=0$.
        \item The regular divisor $\N^{\textup{II}+}_0(x)=\N^{\textup{II}-}_0(x^{\prime})$ if $\nu_p(q(x))=1$.
    \end{itemize}
    \label{specialcase0-1}
\end{remark}
Notice that over the closed formal subschemes $\N^{\textup{I}+}_0(x_0\cdot x)$, $\N^{\textup{I}-}_0(x_0\cdot\overline{x})$, $\N^{\textup{II}+}_0(x)$ and $\N^{\textup{II}-}_0(x^{\prime})$, the quasi-isogenies $x_0\cdot x$ and $x_0\cdot\overline{x}$ lift to isogenies by Lemma \ref{lie-in-spe-div}. We still use $x_0\cdot x$ and $x_0\cdot\overline{x}$ to denote these two isogenies.
\begin{lemma}
    Let $x\in\B$ be an element such that $\nu_p\left(q(x)\right)\geq1$.
    \begin{itemize}
        \item [(I):]  
        \begin{itemize}
            \item [+]Over the formal scheme $\N_0^{\textup{I}+}(x_0\cdot x)$, 
        \begin{align*}
            &x_0\cdot x=\textup{a cyclic isogeny},&x=\textup{a cyclic isogeny},\\
            &x_0\cdot\overline{x}=p\times\textup{a cyclic isogeny}, &x^{\prime}=\textup{a cyclic isogeny}.
        \end{align*}
        \item[$-$] Over the formal scheme $\N_0^{\textup{I}-}(x_0\cdot\overline{x})$, 
        \begin{align*}
            &x_0\cdot x=p\times\textup{a cyclic isogeny},&x=\textup{a cyclic isogeny},\\
            &x_0\cdot\overline{x}=\textup{a cyclic isogeny}, &x^{\prime}=\textup{a cyclic isogeny}.
        \end{align*}
        \end{itemize}
        \item [(II):]
        \begin{itemize}
            \item [+] Over the formal scheme $\N_0^{\textup{II}+}(x)$,
            \begin{align*}
            &x_0\cdot x=p\times\textup{a cyclic isogeny},&x_0\cdot\overline{x}=p\times\textup{a cyclic isogeny}.
            \end{align*}
            When $\nu_p\left(q(x)\right)\geq2$, we also have
            \begin{align*}
            &x=\textup{a cyclic isogeny}, &x^{\prime}=p\times\textup{a cyclic isogeny}.
        \end{align*}
        \item[$-$] Over the formal scheme $\N_0^{\textup{II}-}(x^{\prime})$, 
        \begin{align*}
            &x_0\cdot x=p\times\textup{a cyclic isogeny},&x_0\cdot\overline{x}=p\times\textup{a cyclic isogeny}.
            \end{align*}
            When $\nu_p\left(q(x)\right)\geq2$, we also have
            \begin{align*}
            &x=p\times\textup{a cyclic isogeny}, &x^{\prime}=\textup{a cyclic isogeny}.
            \end{align*}
        \end{itemize}
    \end{itemize}
    \label{moduli-over-hec}
\end{lemma}
\begin{proof}
   Let $n=\nu_p(q(x))$. Let $S$ be a $W$-scheme such that $p$ is locally nilpotent. We first consider the case (I)+. Let
    \begin{equation}
        \left(X_{1}\stackrel{(x_0)_{1}}\longrightarrow X_{1}^{\prime},\left(\rho_{1},\rho_{1}^{\prime}\right)\right),\,\,\,\,\left(X_{2}\stackrel{(x_0)_{2}}\longrightarrow X_{2}^{\prime},\left(\rho_{2},\rho_{2}^{\prime}\right)\right)
    \end{equation}
    be an object in the set $\N_0^{\textup{I}+}(x_0\cdot x)(S)$. Since the formal scheme $\N_0^{\textup{I}+}(x_0\cdot x)$ is contained $\N_0(x_0\cdot x)\times_{\st_{x_0\cdot x}^{\textup{I}+},\N(x_0)}\M$, the quasi-isogeny $x_0\cdot x$ lifts to a cyclic isogeny $\pi:X_1\rightarrow X_2^{\prime}$.
    \par
    The cyclic isogeny $\pi$ factorizes into a composition of $n$ degree $p$ isogenies $\pi=\pi_n\circ\cdots\circ\pi_1$. The isogeny $\pi_1$ is isomorphic to $(x_0)_{1}$, while $\pi_n$ is isomorphic to $(x_0)_{2}$. Hence $x$ lifts to $\pi_{n-1}\circ\cdots\circ\pi_1$. Then $\overline{x}$ lifts to $\pi_1^{\vee}\circ\cdot\circ\pi_{n-1}^{\vee}$. Therefore
    \begin{align*}
        x_0\cdot\overline{x}&=(x_0)_1\circ\pi_1^{\vee}\circ\cdots\circ\pi_{n-1}^{\vee}=\pi_1\circ\pi_1^{\vee}\circ\cdots\circ\pi_{n-1}^{\vee}=q(x_0)\times\pi_2^{\vee}\circ\cdots\circ\pi_{n-1}^{\vee}=p\times\textup{a cyclic isogeny}.\\
        x&=x_0^{-1}\circ x_0\cdot x=\left((x_0)_2\right)^{-1}\circ\pi_n\circ\cdots\circ\pi_{1}=\pi_{n-1}\circ\cdots\circ\pi_{1}=\textup{a cyclic isogeny}.\\
        x^{\prime}&=(x_0)_2\circ x\circ\left((x_0)_1\right)^{-1}=\pi_n\circ\cdots\circ\pi_2=\textup{a cyclic isogeny}.
    \end{align*}
    Therefore we have shown that $x_0\cdot x$ is a cyclic isogeny, while $x_0\cdot\overline{x}$, $x$ and $x^{\prime}$ are of the form $p\times\textup{a cyclic isogeny}$ over the formal scheme $\N_0^{\textup{I}+}(x_0\cdot x)$. The proof for the other cases are similar so we omit it.
\end{proof}
Let $x\in\B$ be an element such that $\nu_p\left(q(x)\right)\geq0$. We know that the formal schemes $\N^{\textup{I}+}_0(x_0\cdot x)$, $\N^{\textup{I}-}_0(x_0\cdot\overline{x})$, $\N^{\textup{II}+}_0(x)$ and $\N^{\textup{II}-}_0(x^{\prime})$ (the later two spaces are only defined for elements $x$ such that $\nu_p\left(q(x)\right)\geq1$) are isomorphic to the blow up along the unique closed point of the corresponding cyclic deformations spaces $\N_0(x_0\cdot x)$, $\N_0(x_0\cdot\overline{x})$, $\N_0(x)$ and $\N_0(x^{\prime})$. Denote by $\exc_{x_0\cdot x}^{\textup{I}+}$, $\exc_{x_0\cdot\overline{x}}^{\textup{I}-}$, $\exc_{x}^{\textup{II}+}$ and $\exc_{x^{\prime}}^{\textup{II}-}$ the corresponding exceptional divisors. They are all isomorphic to $\bP_{\F}^{1}$. By the definition of these formal schemes, 
\begin{equation*}
    \exc_{x_0\cdot x}^{\textup{I}+}=\exc_{\M}\cap\N_0^{\textup{I}+}(x_0\cdot x),\,\,\,\,\exc_{x_0\cdot\overline{x}}^{\textup{I}-}=\exc_{\M}\cap\N_0^{\textup{I}-}(x_0\cdot \overline{x}),
\end{equation*}
\begin{equation*}
    \exc_{x}^{\textup{II}+}=\exc_{\M}\cap\N_0^{\textup{II}+}(x),\,\,\,\,\exc_{x^{\prime}}^{\textup{II}-}=\exc_{\M}\cap\N_0^{\textup{II}-}(x^{\prime}).
\end{equation*}
Hence $\exc_{x_0\cdot x}^{\textup{I}+}$, $\exc_{x_0\cdot\overline{x}}^{\textup{I}-}$, $\exc_{x}^{\textup{II}+}$ and $\exc_{x^{\prime}}^{\textup{II}-}$ are Cartier divisors on $\exc_{\M}$.
\begin{proposition}
    Let $x\in\B$ be an element such that $n\coloneqq\nu_p\left(q(x)\right)\geq0$.
    \begin{itemize}
        \item [(I)] The Cartier divisors $\exc_{x_0\cdot x}^{\textup{I}+}$ and $\exc_{x_0\cdot \overline{x}}^{\textup{I}-}$ of $\exc_{\M}$ are
        \begin{equation*}
            \exc_{x_0\cdot x}^{\textup{I}+}=\begin{cases}
                (v=\textup{a nonzero number in $\F$}), &\textup{when $n=0$;}\\
                (v=0), &\textup{when $n\geq1$.}
            \end{cases}
        \end{equation*}
        \begin{equation*}
            \exc_{x_0\cdot \overline{x}}^{\textup{I}-}=\begin{cases}
                (v^{\prime}=\textup{a nonzero number in $\F$}), &\textup{when $n=0$;}\\
                (v^{\prime}=0), &\textup{when $n\geq1$.}
            \end{cases}
        \end{equation*}
        \item [(II)] When $n\geq1$, the Cartier divisors $\exc_{x}^{\textup{II}+}$ and $\exc_{x^{\prime}}^{\textup{II}-}$ of $\exc_{\M}$ are
        \begin{equation*}
            \exc_{x}^{\textup{II}+}=\begin{cases}
                (u=\textup{a nonzero number in $\F$}), &\textup{when $n=1$;}\\
                (u=0), &\textup{when $n\geq2$.}
            \end{cases}
        \end{equation*}
        \begin{equation*}
            \exc_{x^{\prime}}^{\textup{II}-}=\begin{cases}
                (u^{\prime}=\textup{a nonzero number in $\F$}), &\textup{when $n=1$;}\\
                (u^{\prime}=0), &\textup{when $n\geq2$.}
            \end{cases}
        \end{equation*}
    \end{itemize}
    \label{equation-loc-exc}
\end{proposition}
\begin{proof}
    We first consider the case (I+). When $n=0$, by Lemma \ref{lie-in-spe-div}, we have $\N_0^{\textup{I}+}(x_0\cdot x)\subset\mathcal{Z}(x)$. By Proposition \ref{difference-decomposition}, we know that $\mathcal{Z}(x)=\exc_{\M}+\widetilde{\mathcal{D}}(x)$. Then $\N_0^{\textup{I}+}(x_0\cdot x)\subset\widetilde{\mathcal{D}}(x)$ because $\N_0^{\textup{I}+}(x_0\cdot x)$ intersects with $\exc_{\M}$ properly. Hence
    \begin{equation*}
        \bP_{\F}^{1}\simeq\exc_{x_0\cdot x}^{\textup{I}+}=\exc_{\M}\cap\N_0^{\textup{I}+}(x_0\cdot x)\subset\exc_{\M}\cap\widetilde{\mathcal{D}}(x).
    \end{equation*}
    By Proposition \ref{difference-decomposition}, the intersection $\exc_{\M}\cap\widetilde{\mathcal{D}}(x)$ is given by the equation $v=$ a nonzero number in $\F$, which also cuts out a projective line $\bP_{\F}^{1}$ in $\exc_{\M}$. Hence the Cartier divisor $\exc_{x_0\cdot x}^{\textup{I}+}$ is given by the equation ($v=$ a nonzero number in $\F$) in $\exc_{\M}$. Notice that $\N^{\textup{I}+}_0(x_0\cdot x)=\N^{\textup{I}-}_0(x_0\cdot\overline{x})$ by Remark \ref{specialcase0-1}. Hence $\exc_{x_0\cdot \overline{x}}^{\textup{I}-}=\exc_{x_0\cdot x}^{\textup{I}+}$. Therefore the Cartier divisor $\exc_{x_0\cdot \overline{x}}^{\textup{I}-}$ is given by the equation ($v^{\prime}=$ a nonzero number in $\F$) since $v^{\prime}\cdot v=1$.
    \par
    Now we consider the case $n\geq1$. Let's assume that $\N_0^{\textup{I}+}(x_0\cdot x)\cap(\M_2^{+}\cup\M_1^{-})\neq\varnothing$. Let $\M_{v^{\prime}}$ be the closed formal subscheme of $\M$ cut out by the equation $v^{\prime}=0$. Notice that $v^{\prime}=0$ also implies that $x_1=y_2=0$. Therefore
    \begin{equation*}
        \N_0^{\textup{I}+}(x_0\cdot x)\cap\M_{v^{\prime}}\subset\left(\st_{x_0\cdot x}^{\textup{I}+}\left(\N_0(x_0\cdot x)\right)\cap\left(x_1=y_2=0\right)\right)\times_{\N(x_0)}\M.
    \end{equation*}
    By our convention in Remark \ref{convention-FV}, $x_1=y_2=0$ cuts out the closed formal subscheme $\N(x_0)^{\textup{VF}}$ of $\N(x_0)$. Therefore
    \begin{equation*}
        \st_{x_0\cdot x}^{\textup{I}+}\left(\N_0(x_0\cdot x)\right)\cap\left(x_1=y_2=0\right)=\N_0(x_0\cdot x)\times_{\st_{x_0\cdot x}^{\textup{I}+}, \N(x_0)}\N(x_0)^{\textup{VF}}\simeq\textup{Spec}\,\F
    \end{equation*}
    by Lemma \ref{VF} and Remark \ref{Stand-FV}. Hence $\N_0^{\textup{I}+}(x_0\cdot x)\cap\M_{v^{\prime}}\subset\textup{Spec}\,\F\times_{\N(x_0)}\M=\exc_{\M}$.
    \par
    We assume that $\N_0^{\textup{I}+}(x_0\cdot x)\cap\M_2^{+}\neq\varnothing$ (the argument for the case $\N_0^{\textup{I}+}(x_0\cdot x)\cap\M_1^{-}\neq\varnothing$ is similar). Let $f\in\rO_{\M_2^{+}}$ be the equation of the regular divisor $\N_0^{\textup{I}+}(x_0\cdot x)$ in $\M_2^{+}$. The inclusion $\N_0^{\textup{I}+}(x_0\cdot x)\cap\M_{v^{\prime}}\subset\exc_{\M}$ implies that
    \begin{equation*}
        (x_2)\subset(f,v^{\prime})\subset\rO_{\M_2^{+}}\simeq W[v^{\prime}, u][[x_2]]/(p+uv^{\prime}x_2^{2}).
    \end{equation*}
    Therefore there exist $a,b\in\rO_{\M_2^{+}}$ such that $x_2=av^{\prime}+bf$.\par
    Claim: the element $a$ is invertible.
    \par
    \textit{Proof of the claim}: Let $\overline{a}\coloneqq a\,\textup{mod}\,(x_2)$ be an element in $\F[v^{\prime},u]$. Then $\overline{a}\neq0$ because otherwise $x_2\vert f$, which is impossible. Then the intersection $\exc_{x_0\cdot x}^{\textup{I}+}\cap\M_{2}^{+}$ is given by $(\overline{a}\cdot v^{\prime}=0)=(\overline{a}=0)+(v^{\prime}=0)$ by the equation $x_2=av^{\prime}+bf$. However, the intersection $\exc_{x_0\cdot x}^{\textup{I}+}\cap\M_{2}^{+}$ is an open subvariety of $\bP_{\F}^{1}$. Therefore we must have $(\overline{a}=0)=\varnothing$ and $\exc_{x_0\cdot x}^{\textup{I}+}=(v^{\prime}=0)$. Hence $\overline{a}$ is invertible, which implies that $a$ is invertible.
    \par
    The invertibility of $a$ implies that the element $bf=x_2-av^{\prime}$ is a regular element because the quotient ring $\rO_{\M_2^{+}}/(x_2-av^{\prime})$ is regular. Therefore the element $b$ must be invertible because $f$ is not invertible. Therefore we conclude that $f=\Tilde{a}v^{\prime}+\Tilde{b}x_2$ for invertible elements $\Tilde{a},\Tilde{b}\in\rO_{\M_2^{+}}$. Then
    \begin{equation*}
        \N_0^{\textup{I}+}(x_0\cdot x)\cap\M_2^{+}\simeq\textup{Spf}\,W[u][[x_2]]/(p-\Tilde{a}^{-1}\Tilde{b}ux_2^{3}).
    \end{equation*}
    By the above equation, the multiplicity of the exceptional divisor $\exc_{x_0\cdot x}^{\textup{I}+}$ in $\N_0^{\textup{I}+}(x_0\cdot x)_{\F}=\textup{div}(p)$ of $\N_0^{\textup{I}+}(x_0\cdot x)$ is 3 instead of $r(n)$ as is shown in Lemma \ref{blow-up-N0(x)}, which is a contradiction. Therefore the assumption $\N_0^{\textup{I}+}(x_0\cdot x)\cap(\M_2^{+}\cup\M_1^{-})\neq\varnothing$ is wrong. Since $\exc_{x_0\cdot x}^{\textup{I}+}\subset\widetilde{\mathcal{Z}}(x)\cap\exc_{\M}$ and the latter is a summation of divisors $(v=0)$, $(v^{\prime}=0)$, $(u=0)$, $(u^{\prime}=0)$ and ($u=$ a nonzero number in $\F$) on $\exc_{\M}$. Therefore the only possibility is
    \begin{equation*}
        \exc_{x_0\cdot x}^{\textup{I}+}=(v=0).
    \end{equation*}
    The proof for the other cases are similar, so we omit it.
\end{proof}

\subsection{Decomposition of the difference divisor}
\begin{lemma}
    Let $x\in\B$ be an element such that $\nu_p\left(q(x)\right)\geq2$. Then the regular divisors $\N^{\textup{I}+}_0(x_0\cdot x), \N^{\textup{I}-}_0(x_0\cdot\overline{x}), \N^{\textup{II}+}_0(x), \N^{\textup{II}-}_0(x^{\prime})$ are all contained in the divisor $\widetilde{\mathcal{D}}(x)$.
    \label{hec-in-diff}
\end{lemma}
\begin{proof}
    By the moduli interpretation of the divisors $\mathcal{D}^{+}(x)$ and $\mathcal{D}^{-}(x)$ in Remark \ref{moduli-diff-M} and Lemma \ref{moduli-over-hec}, we have
    \begin{align*}
        &\N^{\textup{I}+}_0(x_0\cdot x)=\N^{\textup{I}+}_0(x_0\cdot x)\cap\M^{+}\bigcup\N^{\textup{I}+}_0(x_0\cdot x)\cap\M^{-}\subset\mathcal{D}^{+}(x)\bigcup\mathcal{D}^{-}(x)=\mathcal{D}(x),\\
        &\N^{\textup{I}-}_0(x_0\cdot \overline{x})=\N^{\textup{I}-}_0(x_0\cdot \overline{x})\cap\M^{+}\bigcup\N^{\textup{I}-}_0(x_0\cdot \overline{x})\cap\M^{-}\subset\mathcal{D}^{+}(x)\bigcup\mathcal{D}^{-}(x)=\mathcal{D}(x).
    \end{align*}
    Therefore the two divisors $\N^{\textup{I}+}_0(x_0\cdot x),\N^{\textup{I}-}_0(x_0\cdot \overline{x})\subset\widetilde{\mathcal{D}}(x)$ since they intersect the exceptional divisor $\exc_{\M}$ properly.
    \par
    By Proposition \ref{equation-loc-exc}, we know that
    \begin{equation*}
        \exc_x^{\textup{II}+}=(u=0),\,\,\,\,\exc_{x^{\prime}}^{\textup{II}-}=(u^{\prime}=0).
    \end{equation*}
    Hence $\N_0^{\textup{II}+}(x)\subset\M^{+}$ and $\N_0^{\textup{II}-}(x^{\prime})\subset\M^{-}$. Therefore by Lemma \ref{moduli-over-hec}
    \begin{align*}
        &\N^{\textup{II}+}_0(x)=\N^{\textup{II}+}_0(x)\cap\M^{+}\subset\mathcal{D}^{+}(x)\subset\mathcal{D}(x),\\
        &\N^{\textup{II}-}_0(x^{\prime})=\N^{\textup{II}-}_0(x^{\prime})\cap\M^{-}\subset\mathcal{D}^{-}(x)\subset\mathcal{D}(x).
    \end{align*}
    Therefore the two divisors $\N^{\textup{II}+}_0(x),\N^{\textup{II}-}_0(x^{\prime})\subset\widetilde{\mathcal{D}}(x)$ since they intersect the exceptional divisor $\exc_{\M}$ properly.
\end{proof}
\begin{lemma}
Let $x\in\B$ be an element such that $\nu_p\left(q(x)\right)\geq0$. We have the following decomposition of the effective Cartier divisor $\widetilde{\mathcal{D}}(x)$,
    \begin{equation*}
        \widetilde{\mathcal{D}}(x)=\begin{cases}
           \N^{\textup{I}+}_0(x_0\cdot x), &\textup{if $\nu_p\left(q(x)\right)=0$;}\\
          \N^{\textup{I}+}_0(x_0\cdot x)+\N^{\textup{I}-}_0(x_0\cdot \overline{x})+\N^{\textup{II}+}_0(x), &\textup{if $\nu_p\left(q(x)\right)=1$;}\\
            \N^{\textup{I}+}_0(x_0\cdot x)+\N^{\textup{I}-}_0(x_0\cdot \overline{x})+\N^{\textup{II}+}_0(x)+\N_0^{\textup{II}-}(x^{\prime}), &\textup{if $\nu_p\left(q(x)\right)\geq2$.}
        \end{cases}
    \end{equation*}
    \label{dec-diff-div}
\end{lemma}
\begin{proof}
    Let $n=\nu_p\left(q(x)\right)$. Denote by $H(x)$ the effective Cartier divisor on the right hand side, i.e.,
    \begin{equation*}
        H(x)=\begin{cases}
           \N^{\textup{I}+}_0(x_0\cdot x), &\textup{if $\nu_p\left(q(x)\right)=0$;}\\
          \N^{\textup{I}+}_0(x_0\cdot x)+\N^{\textup{I}-}_0(x_0\cdot \overline{x})+\N^{\textup{II}+}_0(x), &\textup{if $\nu_p\left(q(x)\right)=1$;}\\
            \N^{\textup{I}+}_0(x_0\cdot x)+\N^{\textup{I}-}_0(x_0\cdot \overline{x})+\N^{\textup{II}+}_0(x)+\N_0^{\textup{II}-}(x^{\prime}), &\textup{if $\nu_p\left(q(x)\right)\geq2$.}
        \end{cases}
    \end{equation*}
    Let $R(x)=\widetilde{\mathcal{D}}(x)-H(x)$. By Proposition \ref{difference-decomposition} and Proposition \ref{equation-loc-exc}, we know that 
    \begin{equation*}
        \widetilde{\mathcal{D}}(x)\cap\exc_{\M}=H(x)\cap\exc_{\M}
    \end{equation*}
    as Cartier divisors on $\exc_{\M}$.
    \par
    If $R(x)\neq\varnothing$, it must be an effective Cartier divisor by Lemma \ref{hec-in-diff}, and intersects with $\exc_{\M}$ properly. The identity $\widetilde{\mathcal{D}}(x)\cap\exc_{\M}=H(x)\cap\exc_{\M}$ as Cartier divisors on $\exc_{\M}$ implies that $R(x)\cap\exc_{\M}=\varnothing$. However, the exceptional divisor is the reduced locus of the formal scheme $\M$. We conclude that $R(x)=\varnothing$. Therefore $\widetilde{\mathcal{D}}(x)=H(x)$.
\end{proof}
\begin{remark}
    We can also consider the difference divisor associated to $\mathcal{Y}$-cycles. Let $x\in\B$ be an element such that $\nu_p(q(x))\geq-1$. Define $\mathcal{D}^{\mathcal{Y}}(x)=\mathcal{Y}(x)-\mathcal{Y}(p^{-1}x)$. By the identity $\mathcal{Y}(x)=(\iota_\M)^{\ast}\mathcal{Z}(x_0\cdot x)$, we have $\mathcal{D}^{\mathcal{Y}}(x)=(\iota_\M)^{\ast}\mathcal{D}(x_0x)$. Define $\widetilde{\mathcal{D}}^{\mathcal{Y}}(x)=(\iota_\M)^{\ast}\widetilde{\mathcal{D}}(x_0\cdot x)$. Then by Lemma \ref{dec-diff-div} and (\ref{local-id-aut-M}), we have
    \begin{equation}
        \widetilde{\mathcal{D}}^{\mathcal{Y}}(x)=\begin{cases}
           \N^{\textup{II}+}_0(px), &\textup{if $\nu_p\left(q(x)\right)=-1$;}\\
          \N^{\textup{I}+}_0(x_0\cdot x)+\N^{\textup{II}+}_0(px)+\N^{\textup{II}-}_0(px^{\prime}), &\textup{if $\nu_p\left(q(x)\right)=0$;}\\
            \N^{\textup{I}+}_0(x_0\cdot x)+\N^{\textup{I}-}_0(x_0\cdot \overline{x})+\N^{\textup{II}+}_0(px)+\N_0^{\textup{II}-}(px^{\prime}), &\textup{if $\nu_p\left(q(x)\right)\geq1$.}
        \end{cases}
        \label{diff-Y}
    \end{equation}
\end{remark}

\subsection{Derived special cycles ${^{\mathbb{L}}\mathcal{Z}(L)}$}
\begin{lemma}
    Let $x,y\in\B$ be two linearly independent elements. Then
    \begin{itemize}
        \item [(a)] The two effective Cartier divisors $\sz(x)$ and $\mathcal{Z}(y)$ intersect properly.
        \item[(b)] The irreducible components of the intersection $\sz(x)\cap\mathcal{Z}(y)$ are of the form
        \begin{equation*}
            \textup{Spf}\,W_s\,\,\,\textup{or}\,\,\,\,\bP_{\F}^{1}\subset\exc_{\M}.
        \end{equation*}
        Here $W_s$ is the ring of definition of a quasi-canonical lifting of level $s$. 
    \end{itemize}   
    \label{proper-with-strict}
\end{lemma}
\begin{proof}
    By Lemma \ref{dec-diff-div}, an irreducible component of the divisor $\sz(x)$ are of the form $\N^{\textup{I}+}_0(x_0\cdot \Tilde{x}), \N^{\textup{I}-}_0(x_0\cdot\overline{\Tilde{x}}), \N^{\textup{II}+}_0(\Tilde{x}), \N^{\textup{II}-}_0(\Tilde{x}^{\prime})$ where $\Tilde{x}=p^{-i}x$ for some positive integer $i$. It's sufficient to prove (a) and (b) for the intersections of the divisors $\N^{\textup{I}+}_0(x_0\cdot \Tilde{x}), \N^{\textup{I}-}_0(x_0\cdot\overline{\Tilde{x}}), \N^{\textup{II}+}_0(\Tilde{x}), \N^{\textup{II}-}_0(\Tilde{x}^{\prime})$ and the divisor $\mathcal{Z}(y)$. Without loss of generality, we can assume $\Tilde{x}=x$.
    \par
    We first consider the intersection $\N^{\textup{I}+}_0(x_0\cdot x)\cap\mathcal{Z}(y)$. Denote by $\pi^{\textup{I}+}:\N^{\textup{I}+}_0(x_0\cdot x)\rightarrow\N_0(x_0\cdot x)$ the blow up morphism, then
    \begin{align*}
        \N^{\textup{I}+}_0(x_0\cdot x)\cap\mathcal{Z}(y)&\subset\left(\N_0(x_0\cdot x)\cap\mathcal{Z}_{\N}(x_0\cdot y)\right)\times_{\N_0(x_0\cdot x),\pi^{\textup{I}+}}\N^{\textup{I}+}_0(x_0\cdot x)\\
        &\subset\left(\mathcal{Z}_{\N}(x_0\cdot x)\cap\mathcal{Z}_{\N}(x_0\cdot y)\right)\times_{\N_0(x_0\cdot x),\pi^{\textup{I}+}}\N^{\textup{I}+}_0(x_0\cdot x).
    \end{align*}
    By Lemma \ref{proper-hpe}, we know that $\textup{dim}\,\mathcal{Z}_{\N}(x_0\cdot x)\cap\mathcal{Z}_{\N}(x_0\cdot y)=1$, hence $\textup{dim}\,\N^{\textup{I}+}_0(x_0\cdot x)\cap\mathcal{Z}(y)\leq\textup{dim}\,\left(\mathcal{Z}_{\N}(x_0\cdot x)\cap\mathcal{Z}_{\N}(x_0\cdot y)\right)\times_{\N_0(x_0\cdot x)}\N^{\textup{I}+}_0(x_0\cdot x)=1$. Therefore $\N^{\textup{I}+}_0(x_0\cdot x)$ and $\mathcal{Z}(y)$ intersect properly. Part (b) follows from Lemma \ref{proper-hpe} and Lemma \ref{blow-up-N0(x)}. The proof for the other intersections are similar so we omit it.    
\end{proof}
For a lattice $M\subset\B$, define $\min(M)=\min\limits_{x\in M}\left\{\nu_p(q(x))\right\}$. We say an element $x\in M$ is a minimal element of $M$ if $\nu_p(q(x))=\min(M)$. It's easy to see that if $x$ is a minimal element of $M$, we have $p^{-1}x\notin M$.
\begin{proposition}
    Let $M\subset\B$ be a $\zp$-lattice of rank $2$. Let $x\in M$ be a minimal element of $M$. Let $y\in M$ be another element such that $\{x,y\}$ is a $\zp$-basis of $M$. The element $[\rO_{\widetilde{\mathcal{Z}}(x)}\otimes^{\mathbb{L}}_{\rO_{\M}}\rO_{\mathcal{Z}(y)}]\in\textup{Gr}^{2}K_0^{\mathcal{Z}(M)}(\M)$ is independent of the choices of the minimal element $x$ and the element $y$.
\end{proposition}
\begin{proof}
    Let $n=\min(M)$. If $n<0$, then the element $\rO_{\widetilde{\mathcal{Z}}(x)}\otimes^{\mathbb{L}}_{\rO_{\M}}\rO_{\mathcal{Z}(y)}=0$ in the group $K_0(\M)$ for all choices of minimal elements $x\in M$ and $y\in M$ such that $\{x,y\}$ is a basis of $M$. Therefore we only need to consider the case $n\geq0$.
    \par
    Let $x^{\prime}\in M$ be another minimal element of $M$, i.e., $\nu_p(q(x^{\prime}))=n$. Let $y^{\prime}\in M$ be another element such that $\{x^{\prime},y^{\prime}\}$ is a basis of $M$. Then there exist $a,b,c,d\in\zp$ such that
    \begin{equation*}
        x^{\prime}=ax+by,\,\,\,\,y^{\prime}=cx+dy,\,\,\,\,ad-bc\in\zp^{\times}.
    \end{equation*}
    We first consider some special cases. 
    \begin{itemize}
        \item [Case 1]: $x^{\prime}=x$, $y^{\prime}=cx+y$. Then by Proposition \ref{proper-with-strict}, we have
        \begin{equation*}
            \rO_{\widetilde{\mathcal{Z}}(x^{\prime})}\otimes^{\mathbb{L}}_{\rO_{\M}}\rO_{\mathcal{Z}(y^{\prime})}=\rO_{\widetilde{\mathcal{Z}}(x)}\otimes_{\rO_{\M}}\rO_{\mathcal{Z}(y^{\prime})}=\rO_{\widetilde{\mathcal{Z}}(x)\cap\mathcal{Z}(y^{\prime})}.
        \end{equation*}
        By the moduli interpretations of the special cycles, we have $\widetilde{\mathcal{Z}}(x)\cap\mathcal{Z}(y^{\prime})=\widetilde{\mathcal{Z}}(x)\cap\mathcal{Z}(y)$. Therefore
        \begin{equation*}
            \rO_{\widetilde{\mathcal{Z}}(x^{\prime})}\otimes^{\mathbb{L}}_{\rO_{\M}}\rO_{\mathcal{Z}(y^{\prime})}=\rO_{\widetilde{\mathcal{Z}}(x)}\otimes^{\mathbb{L}}_{\rO_{\M}}\rO_{\mathcal{Z}(y)}.
        \end{equation*}
        \item [Case 2]: $x^{\prime}=x+ay$, $y^{\prime}=y$. Let $m=\nu_p(q(y))$. We have:
        \begin{align*}
            &[\rO_{\widetilde{\mathcal{Z}}(x^{\prime})}\otimes^{\mathbb{L}}_{\rO_{\M}}\rO_{\mathcal{Z}(y^{\prime})}]=[\rO_{\widetilde{\mathcal{Z}}(x^{\prime})}\otimes^{\mathbb{L}}_{\rO_{\M}}\rO_{\mathcal{Z}(y)}]=[\rO_{\widetilde{\mathcal{Z}}(x^{\prime})}\otimes^{\mathbb{L}}_{\rO_{\M}}\rO_{\widetilde{\mathcal{Z}}(y)}]+(m+1)\cdot[\rO_{\widetilde{\mathcal{Z}}(x^{\prime})}\otimes^{\mathbb{L}}_{\rO_{\M}}\mathcal{O}_{\exc_\M}]\\
            =&[\rO_{\mathcal{Z}(x^{\prime})}\otimes^{\mathbb{L}}_{\rO_{\M}}\rO_{\widetilde{\mathcal{Z}}(y)}]-(n+1)\cdot[\mathcal{O}_{\exc_\M}\otimes^{\mathbb{L}}_{\rO_{\M}}\rO_{\widetilde{\mathcal{Z}}(y)}]+(m+1)\cdot[\rO(n+1,n)]\\
            =&[\rO_{\mathcal{Z}(x^{\prime})\cap\widetilde{\mathcal{Z}}(y)}]-(n+1)\cdot[\rO(m+1,m)]+(m+1)\cdot[\rO(n+1,n)]=[\rO_{\mathcal{Z}(x)\cap\widetilde{\mathcal{Z}}(y)}]+[\rO(0,n-m)].
        \end{align*}
        On the other hand, similar computations also apply to $[\rO_{\widetilde{\mathcal{Z}}(x)}\otimes^{\mathbb{L}}_{\rO_{\M}}\rO_{\mathcal{Z}(y)}]$ just by replacing $x^{\prime}$ by $x$ in the above computations. We conclude that
        \begin{equation*}
            [\rO_{\widetilde{\mathcal{Z}}(x^{\prime})}\otimes^{\mathbb{L}}_{\rO_{\M}}\rO_{\mathcal{Z}(y^{\prime})}]=[\rO_{\widetilde{\mathcal{Z}}(x)}\otimes^{\mathbb{L}}_{\rO_{\M}}\rO_{\mathcal{Z}(y)}]=[\rO_{\mathcal{Z}(x)\cap\widetilde{\mathcal{Z}}(y)}]+[\rO(0,n-m)]\,\,\,\textup{in}\,\,\textup{Gr}^{2}K_0^{\mathcal{Z}(M)}(\M).
        \end{equation*}
    \end{itemize}
    \par
    Now we come back to prove the proposition. There are two situations.
    \begin{itemize}
        \item If $a\in\zp^{\times}$. Scaling $x^{\prime}$ by $a^{-1}\in\zp^{\times}$, we can assume $a=1$. Then $y^{\prime}=cx^{\prime}+(d-bc)y$. Scaling $y^{\prime}$ by $(d-bc)^{-1}\in\zp^{\times}$, we can assume $d-bc=1$. Then $x^{\prime}=x+by, y^{\prime}=cx^{\prime}+y$. Therefore
        \begin{equation*}
            [\rO_{\widetilde{\mathcal{Z}}(x^{\prime})}\otimes^{\mathbb{L}}_{\rO_{\M}}\rO_{\mathcal{Z}(y^{\prime})}]\stackrel{\textup{Case 1}}=[\rO_{\widetilde{\mathcal{Z}}(x^{\prime})}\otimes^{\mathbb{L}}_{\rO_{\M}}\rO_{\mathcal{Z}(y)}]\stackrel{\textup{Case 2}}=[\rO_{\widetilde{\mathcal{Z}}(x)}\otimes^{\mathbb{L}}_{\rO_{\M}}\rO_{\mathcal{Z}(y)}].
        \end{equation*}
        \item If $\nu_p(a)\geq1$. Then $b,c\in\zp^{\times}$. Scaling $x^{\prime}$ by $b^{-1}$, we can assume $b=1$. Then $y^{\prime}=dx^{\prime}+(c-ad)x$. Scaling $y^{\prime}$ by $(c-ad)^{-1}$, we can assume $c-ad=1$. Then $x^{\prime}=ax+y, y^{\prime}=x+dx^{\prime}$. Therefore
        \begin{align*}
            &[\rO_{\widetilde{\mathcal{Z}}(x^{\prime})}\otimes^{\mathbb{L}}_{\rO_{\M}}\rO_{\mathcal{Z}(y^{\prime})}]\stackrel{\textup{Case 1}}=[\rO_{\widetilde{\mathcal{Z}}(x^{\prime})}\otimes^{\mathbb{L}}_{\rO_{\M}}\rO_{\mathcal{Z}(x)}]=\left([\rO_{\mathcal{Z}(x^{\prime})}]-(n+1)[\rO_{\exc_{\M}}]\right)\cdot[\rO_{\mathcal{Z}(x)}]\\
            &\stackrel{(\ref{spe-exc})}=[\rO_{\mathcal{Z}(x^{\prime})}]\cdot[\rO_{\mathcal{Z}(x)}]-(n+1)\rO(0,-1)=\left([\rO_{\widetilde{\mathcal{Z}}(x)}]+(n+1)[\rO_{\exc_{\M}}]\right)\cdot[\rO_{\mathcal{Z}(x^{\prime})}]-\rO(0,-n-1)\\
            &\stackrel{(\ref{spe-exc})}=[\rO_{\widetilde{\mathcal{Z}}(x)}]\cdot[\rO_{\mathcal{Z}(x^{\prime})}]=[\rO_{\widetilde{\mathcal{Z}}(x)}\otimes^{\mathbb{L}}_{\rO_{\M}}\rO_{\mathcal{Z}(x^{\prime})}]\stackrel{\textup{Case 1}}=[\rO_{\widetilde{\mathcal{Z}}(x)}\otimes^{\mathbb{L}}_{\rO_{\M}}\rO_{\mathcal{Z}(y)}].
        \end{align*}
    \end{itemize}
\end{proof}
\begin{corollary}
    Let $M\subset\B$ be a $\zp$-lattice of rank $2$. Let $\{x,y\}$ be a basis of $M$, then the elements $[\rO_{\mathcal{Z}(x)}\otimes^{\mathbb{L}}_{\rO_{\M}}\rO_{\mathcal{Z}(y)}]\in\textup{Gr}^{2}K_0^{\mathcal{Z}(M)}(\M)$ and $[\rO_{\mathcal{Y}(x)}\otimes^{\mathbb{L}}_{\rO_{\M}}\rO_{\mathcal{Y}(y)}]\in\textup{Gr}^{2}K_0^{\mathcal{Y}(M)}(\M)$ only depend on the $\zp$-lattice $M$.
    \label{linear-invariance-2}
\end{corollary}
\begin{proof}
    Without loss of generality, we assume that $0\leq\nu_p(q(x))\leq\nu_p(q(y))$. For an arbitrary element $x^{\prime}\in M$, we have $x^{\prime}=ax+by$ for some elements $a,b\in\zp$. Then $q(x^{\prime})=a^{2}q(x)+ab(x,y)+b^{2}q(y)$. Lemma \ref{inner-product-B} implies that $\nu_p((x,y))\geq\nu_p(q(x))$, therefore $\nu_p(q(x^{\prime}))\geq\nu_p(q(x))$. Hence $x$ is a minimal element of $M$. Therefore the element $[\rO_{\widetilde{\mathcal{Z}}(x)}\otimes^{\mathbb{L}}_{\rO_{\M}}\rO_{\mathcal{Z}(y)}]$ only depends on $M$. Let $n=\nu_p(q(x))=\min(M)$. By (\ref{spe-exc}) and Lemma \ref{linearization}, we have
    \begin{align*}
        [\rO_{\mathcal{Z}(x)}\otimes^{\mathbb{L}}_{\rO_{\M}}\rO_{\mathcal{Z}(y)}]&=\left([\rO_{\widetilde{\mathcal{Z}}(x)}]+(n+1)[\rO_{\exc_{\M}}]\right)\cdot[\rO_{\mathcal{Z}(y)}]\\&=[\rO_{\widetilde{\mathcal{Z}}(x)}\otimes^{\mathbb{L}}_{\rO_{\M}}\rO_{\mathcal{Z}(y)}]+\rO(0,-n-1).
    \end{align*}
    Therefore the element $[\rO_{\mathcal{Z}(x)}\otimes^{\mathbb{L}}_{\rO_{\M}}\rO_{\mathcal{Z}(y)}]\in\textup{Gr}^{2}K_0^{\mathcal{Z}(M)}(\M)$ only depends on the $\zp$-lattice $M$.
    \par
    Notice that we have $\mathcal{Y}(x)=\left(\iota^{\M}\right)^{\ast}(\mathcal{Z}(x_0\cdot x))$ for all elements $x\in\B$. Hence $[\rO_{\mathcal{Y}(x)}\otimes^{\mathbb{L}}_{\rO_{\M}}\rO_{\mathcal{Y}(y)}]=\left(\iota^{\M}\right)^{\ast}[\rO_{\mathcal{Z}(x_0\cdot x)}\otimes^{\mathbb{L}}_{\rO_{\M}}\rO_{\mathcal{Z}(x_0\cdot y)}]$ also depends on $M$ only.
\end{proof}
\begin{lemma}
    Let $x,y\in\B$ be two elements. Then
    \begin{equation*}
        \nu_p((x,y))\geq\min\{q(x),q(y)\}.
    \end{equation*}
    \label{inner-product-B}
\end{lemma}
\begin{proof}
    The statement actually works for all the anisotropic quadratic spaces. Suppose that $\nu_p(q(y))\geq\nu_p(q(x))$. Let's assume the contrary that $\nu_p((x,y))<\nu_p(q(x))$. Then the following equation over $\zp$
    \begin{equation*}
        \frac{q(x)}{(x,y)}\cdot X^{2}+X+\frac{q(y)}{(x,y)}=0
    \end{equation*}
    must has a solution $a$ in $p\zp$ by the Hensel's lemma. Then the vector $y^{\prime}=y+ax$ is isotropic, which is a contradiction. Therefore $
    \nu_p((x,y))\geq\nu_p(q(x))=\min\{q(x),q(y)\}$.
\end{proof}
Let $L\subset\B$ be a $\zp$-lattice of rank $r$ where $1\leq r\leq 3$. Let $\boldsymbol{x}=\{x_1,\cdots,x_r\}$ be a basis of $L$. Define
\begin{equation*}
        {^{\mathbb{L}}\mathcal{Z}}(\boldsymbol{x})\coloneqq[\rO_{\mathcal{Z}(x_1)}\otimes^{\mathbb{L}}_{\rO_{\M}}\cdots\otimes^{\mathbb{L}}_{\rO_{\M}}\rO_{\mathcal{Z}(x_r)}]\in\textup{Gr}^{r}K_0^{\mathcal{Z}(L)}(M).
    \end{equation*}
    and \begin{equation*}
        {^{\mathbb{L}}\mathcal{Y}}(\boldsymbol{x})\coloneqq[\rO_{\mathcal{Y}(x_1)}\otimes^{\mathbb{L}}_{\rO_{\M}}\cdots\otimes^{\mathbb{L}}_{\rO_{\M}}\rO_{\mathcal{Y}(x_r)}]\in\textup{Gr}^{r}K_0^{\mathcal{Y}(L)}(M).
    \end{equation*}
Notice that the elements ${^{\mathbb{L}}\mathcal{Z}}(\boldsymbol{x})$ and ${^{\mathbb{L}}\mathcal{Y}}(\boldsymbol{x})$ are invariant under permutations and operations of the form $x_i\rightarrow a_ix_i+a_jx_j$ for some $a_i\in\zp^{\times}$ and $a_j\in\zp$ by Corollary \ref{linear-invariance-2}. We can also transform the basis $\boldsymbol{x}$ by permutations and operations of the form $x_i\rightarrow a_ix_i+a_jx_j$ for some $a_i\in\zp^{\times}$ and $a_j\in\zp$ to get any another basis $\boldsymbol{x}^{\prime}=\{x_1^{\prime},\cdots,x_r^{\prime}\}$ of $L$. Therefore the elements ${^{\mathbb{L}}\mathcal{Z}}(\boldsymbol{x})$ and ${^{\mathbb{L}}\mathcal{Y}}(\boldsymbol{x})$ only depend on $L$.
\begin{definition}
    Let $L\subset\B$ be a $\zp$-lattice of rank $r$ where $1\leq r\leq 3$. Let $\{x_1,\cdots,x_r\}$ be a basis of $L$. Define the derived special cycle 
    \begin{equation*}
        {^{\mathbb{L}}\mathcal{Z}}(L)\coloneqq[\rO_{\mathcal{Z}(x_1)}\otimes^{\mathbb{L}}_{\rO_{\M}}\cdots\otimes^{\mathbb{L}}_{\rO_{\M}}\rO_{\mathcal{Z}(x_r)}]\in\textup{Gr}^{r}K_0^{\mathcal{Z}(L)}(M).
    \end{equation*}
    \begin{equation*}
        {^{\mathbb{L}}\mathcal{Y}}(L)\coloneqq[\rO_{\mathcal{Y}(x_1)}\otimes^{\mathbb{L}}_{\rO_{\M}}\cdots\otimes^{\mathbb{L}}_{\rO_{\M}}\rO_{\mathcal{Y}(x_r)}]\in\textup{Gr}^{r}K_0^{\mathcal{Y}(L)}(M).
    \end{equation*}
\end{definition}
\begin{definition}
    Let $L\subset\B$ be a $\zp$-lattice of rank 3. Define the arithmetic intersection numbers
    \begin{equation*}
        \Int^{\mathcal{Z}}(L)\coloneqq\chi(\M,{^{\mathbb{L}}\mathcal{Z}}(L)),\,\,\,\,\Int^{\mathcal{Y}}(L)\coloneqq\chi(\M,{^{\mathbb{L}}\mathcal{Y}}(L))
    \end{equation*}
    Here $\chi$ denotes the Euler--Poincare characteristic.
\end{definition}
Now we are able to state to state the main theorem of the article.
\begin{theorem}
    Let $L\subset\B$ be a $\zp$-lattice of rank 3. Then
    \begin{equation*}
        \Int^{\mathcal{Z}}(L)=\partial\den\left(H_0(p),L\right),
    \end{equation*}
    and 
    \begin{equation*}
        \Int^{\mathcal{Y}}(L)=\partial\den\left(H_0(p)^{\vee},L\right)-1=\partial\den\left(H_0(p)^{\vee},L\right)-\frac{p^{7}}{2(p+1)^{2}}\cdot\den(\mathcal{O}_{\B}^{\vee},L).
    \end{equation*}
    \label{main-theorem}
\end{theorem}

\section{Difference formula on the geometric side}
\label{diff-geo-part}
\subsection{Difference formula: the hyperspecial case}
Let $L\subset\B$ be a $\zp$-lattice of rank $r$ where $1\leq r\leq3$. Let $\{x_1,\cdots,x_r\}$ be a basis of $L$. Then the element $\rO_{\mathcal{Z}_{\N}(x_1)}\otimes^{\mathbb{L}}_{\rO_{\N}}\cdots\otimes^{\mathbb{L}}_{\rO_{\N}}\rO_{\mathcal{Z}_{\N}(x_r)}\in\textup{F}^{r}K_0^{\mathcal{Z}(L)}(\N)$ only depends on the lattice $L$ by \cite[Corollary 4.11.2]{LZ22b}. Denote by ${^{\mathbb{L}}\mathcal{Z}}_{\N}(L)$ the image of this element in $\textup{Gr}^{r}K_0^{\mathcal{Z}(L)}(\N)$. If the rank of $L$ is $3$, i.e., $r=3$. Define the arithmetic intersection number on $\N$
\begin{equation*}
    \Int_{\N}(L)\coloneqq\chi(\N,{^{\mathbb{L}}\mathcal{Z}}_{\N}(L)).
\end{equation*}
The works of Gross--Keating \cite{GK93} and the ARGOS volume \cite{ARGOS} together imply the following theorem for all the prime number $p$ (including $p=2$):
\begin{theorem}
    Let $L\subset\B$ be a $\zp$-lattice of rank $3$. Then
    \begin{equation*}
        \Int_{\N}(L)=\partial\den(H,L).
    \end{equation*}\label{GK-main}
\end{theorem}
Let $\N_{\F}\coloneqq\N\times_{W}\F$. A key induction step in the calculations of the number $\Int_{\N}(L)$ is the following lemma, which is proved in \cite[Lemma 5.6]{GK93} and \cite[Proposition 1.6]{R}.
\begin{lemma}
    Let $L^{\flat}\subset\B$ be a $\zp$-lattice of rank $2$. Let $x\in\B$ be another element such that $\nu_p(q(x))\geq\max\{\max(L^{\flat}),2\}$ and $x\notin L^{\flat}\otimes_{\mathbb{Z}}\mathbb{Q}$, then
    \begin{align*}
        \Int_{\N}(L^{\flat}\oplus\langle x\rangle)-\Int_\N(L^{\flat}\oplus\langle p^{-1}x\rangle)&=\chi(\N,{^{\mathbb{L}}\mathcal{Z}_\N}(L^{\flat})\otimes^{\mathbb{L}}_{\rO_{\N}}\rO_{\N_{\F}})\\
    \end{align*} \label{def-px-x}
\end{lemma}
\begin{corollary}
    Let $L^{\flat}\subset\B$ be a $\zp$-lattice of rank $2$. Let $x\in\B$ be another element such that $\nu_p(q(x))\geq\max\{\max(L^{\flat}),2\}$ and $x\perp L^{\flat}$, then
    \begin{equation*}
        \chi(\N,{^{\mathbb{L}}\mathcal{Z}_\N}(L^{\flat})\otimes^{\mathbb{L}}_{\rO_{\N}}\rO_{\N_{\F}})=\partial\den\left(\lx[-1]\obot H_2^{+},L^{\flat}\right).
    \end{equation*}
    \label{diff-to-p-hpe}
\end{corollary}
\begin{proof}
    By \cite[Theorem 7.2.6]{Zhu23}, we have
    \begin{equation*}
        \partial\den(H,L^{\flat}\obot\lx)-\partial\den(H,L^{\flat}\obot\langle p^{-1}x\rangle)=\partial\den\left(\lx[-1]\obot H_2^{+},L^{\flat}\right).
    \end{equation*}
    Then the corollary follows from Theorem \ref{GK-main} and Lemma \ref{def-px-x}.
\end{proof}
\subsection{Blow up: the hyperspecial case}
Recall that $\N\simeq\textup{Spf}\,W[[t,t^{\prime}]]$, here the elements $t$ and $t^{\prime}$ are chosen so that under the identification $\N=\N_0\times\N_0$, the first $\N_0\simeq\textup{Spf}\,W[[t]]$ and the second $\N_0\simeq\textup{Spf}\,W[[t^{\prime}]]$. Hence $\N_{\F}\simeq\textup{Spf}\,\F[[t,t^{\prime}]]$. Let $\pi_{\F}:\widetilde{\N}_{\F}\rightarrow\N_{\F}$ be the blow up morphism along the unique closed point of $\N_\F$. There is an open cover $\{\widetilde{\N}_{\F}^{\circ},\widetilde{\N}_{\F}^{\bullet}\}$ of $\widetilde{\N}_{\F}$ given as follows,
\begin{itemize}
    \item [$\circ$] Let $t^{\prime}=xt$. Define $\widetilde{\N}_{\F}^{\circ}=\textup{Spf}\,\F[x][[t]]$.
    \item[$\bullet$] Let $t=x^{\prime}t^{\prime}$. Define $\widetilde{\N}_{\F}^{\bullet}=\textup{Spf}\,\F[x^{\prime}][[t^{\prime}]]$.
\end{itemize}
Over the intersection $\widetilde{\N}_{\F}^{\circ}\cap\widetilde{\N}_{\F}^{\bullet}$, we have $x^{\prime}x=1$. Let $\exc_{\widetilde{\N}_{\F}}$ be the exceptional divisor on $\widetilde{\N}_{\F}$. It is glued by $\textup{Spec}\,\F[x]$ and $\textup{Spec}\,\F[x^{\prime}]$ with the condition $x^{\prime}x=1$. Hence $\exc_{\widetilde{\N}_{\F}}\simeq\bP_{\F}^{1}$.
\begin{lemma}
    We have the following identity in $\textup{Gr}^{2}K_0^{\exc_{\widetilde{\N}_{\F}}}(\widetilde{\N}_{\F})\simeq\textup{Pic}(\exc_{\widetilde{\N}_{\F}})$:
    \begin{equation}
        [\rO_{\exc_{\widetilde{\N}_{\F}}}\otimes^{\mathbb{L}}_{\rO_{\widetilde{\N}_{\F}}}\rO_{\exc_{\widetilde{\N}_{\F}}}]=\rO(-1).
        \label{exc-N0-F-blow-formula}
    \end{equation}
    \label{exc-N0-F-blow}
\end{lemma}
\begin{proof}
    Similar to Lemma \ref{self-int-exc}, the derived tensor product $[\rO_{\exc_{\widetilde{\N}_{\F}}}\otimes^{\mathbb{L}}_{\rO_{\widetilde{\N}_{\F}}}\rO_{\exc_{\widetilde{\N}_{\F}}}]$ can be viewed as the restriction of the line bundle corresponding to the exceptional divisor $\exc_{\widetilde{\N}_{\F}}$ on $\widetilde{\N}_{\F}$ to the exceptional divisor $\exc_{\widetilde{\N}_\F}$ itself. Under the open cover $\widetilde{\N}_{\F}^{\circ}$ and $\widetilde{\N}_{\F}^{\bullet}$ of $\widetilde{\N}_{\F}$, the divisor $\exc_{\widetilde{\N}_{\F}}$ is given by the following equations and transformation rules:
 \begin{equation*}
     \begin{tikzcd}
     \widetilde{\N}_{\F}^{\circ} & \widetilde{\N}_{\F}^{\bullet} 
    \\
t \arrow[r,shift left,"\times x"] \arrow[Subseteq]{u}{}
    &
t^{\prime} \arrow[l,shift left,"\times x^{\prime}"] \arrow[Subseteq]{u}{}
\end{tikzcd}
 \end{equation*}
    The same transformation rule also applies to the corresponding open cover of $\exc_{\widetilde{\N}_{\F}}\simeq\mathbb{P}_{\mathbb{F}}^{1}$. Therefore the corresponding line bundle is $\mathcal{O}(-1)$.
\end{proof}

\subsection{Invariance of the intersection number}
\label{invariance-int}
Let $x\in\B$ be a nonzero element. Define $\mathcal{Z}_{\N_\F}(x)=\mathcal{Z}_{\N}(x)\cap\N_{\F}$. The intersection is proper by the flatness over $W$ of the divisor $\mathcal{Z}_{\N}(x)$ (Lemma \ref{proper-hpe}), hence $\mathcal{Z}_{\N_\F}(x)$ is an effective Cartier divisor on $\N_{\F}$. Define $\mathcal{Z}_{\widetilde{\N}_\F}(x)=\pi_{\F}^{\ast}\mathcal{Z}_{\N_\F}(x)=\widetilde{\N}_{\F}\times_{\N_{\F}}\mathcal{Z}_{\N_\F}(x)$ to be the pull back of the divisor $\mathcal{Z}_{\N_\F}(x)$ on $\N_{\F}$ to the blow up $\widetilde{\N}_{\F}$.
\begin{lemma}
    Let $x\in\B$ be a nonzero element such that $n\coloneqq\nu_p(q(x))\geq0$. Then
    \begin{itemize}
        \item [(a)] An irreducible component of $\mathcal{Z}_{\N_\F}(x)$ is of the form
        \begin{equation*}
            \mathcal{C}_a(x)=(\nu_x^{\prime}t^{\prime}-t^{p^{a}})\,\,\,\,\textup{or}\,\,\,\,\mathcal{C}_{-a}(x)=(\nu_xt-t^{\prime p^{a}}),
        \end{equation*}
        where $0\leq a\leq n$ and $a\equiv n\,\,\textup{mod}\,2$, $\nu_x,\nu_x^{\prime}\in\rO_{\N_{\F}}^{\times}$. Moreover, we have the following identity of effective Cartier divisors on $\N_\F$, 
        \begin{equation*}
            \mathcal{Z}_{\N_\F}(x)=\sum\limits_{\substack{-n\leq a\leq n\\a\equiv n\,\textup{mod}\,2}}p^{(n-\vert a \vert)/2}\cdot\mathcal{C}_a(x).
        \end{equation*}
        \item [(b)] We have $\exc_{\widetilde{\N}_{\F}}\simeq\mathbb{P}_{\F}^{1}$ and the following equality in the group $\textup{Pic}(\exc_{\widetilde{\N}_{\F}})$:
        \begin{equation}
            [\rO_{\mathcal{Z}_{\widetilde{\N}_\F}(x)}\otimes_{\rO_{\widetilde{\N}_\F}}^{\mathbb{L}}\mathcal{O}_{\exc_{\widetilde{\N}_{\F}}}]=\mathcal{O}(0).
            \label{exc-blow-up-N_0-F-int}
        \end{equation}
        \item[(c)] Let $y\in\B$ be another element such that $x,y$ are linearly independent. Let $\mathcal{C}_1$ and $\mathcal{C}_2$ be an irreducible component of $\mathcal{Z}_{\widetilde{\N}_\F}(x)$ and $\mathcal{Z}_{\widetilde{\N}_\F}(y)$ respectively. Then
        \begin{equation}
            \chi(\widetilde{\N}_{\F},\rO_{\pi_{\F}^{\ast}\mathcal{C}_1}\otimes^{\mathbb{L}}_{\rO_{\widetilde{\N}_{\F}}}\rO_{\pi_{\F}^{\ast}\mathcal{C}_2})=\chi(\N_{\F},\rO_{\mathcal{C}_1}\otimes^{\mathbb{L}}_{\rO_{\N_{\F}}}\rO_{\mathcal{C}_2}).
            \label{blow-up-N_0-spe-inv}
        \end{equation}
    \end{itemize}
    \label{inv-blow-up-NF}
\end{lemma}
\begin{proof}
    Part (a) is a consequence of Proposition \ref{cyclic-diff} (b2). Let $\mathcal{C}$ be an irreducible component of a special divisor $\mathcal{Z}_{\N_{\F}}(x)$, we have the following equality by (a):
    \begin{equation*}
        \pi_{\F}^{\ast}\mathcal{C}=\exc_{\widetilde{\N}_\F}+\widetilde{\mathcal{C}},
    \end{equation*}
    where $\widetilde{\mathcal{C}}$ is the strict transform of the divisor $\mathcal{C}$ under the blow up morphism $\pi_{\F}$. It is isomorphic to $\mathcal{C}$ under the morphism $\pi_{\F}$ since it is represented by a 1-dimensional regular local ring. It's easy to see that $\widetilde{\mathcal{C}}\cap\exc_{\widetilde{\N}_\F}$ consists of only $1$ point, hence we have the equality $[\mathcal{O}_{\widetilde{\mathcal{C}}\cap\exc_{\widetilde{\N}_\F}}]=\mathcal{O}(1)$, hence $[\mathcal{O}_{\pi_\F^{\ast}\mathcal{C}\cap\exc_{\widetilde{\N}_\F}}]=\mathcal{O}(0)$. Therefore (b) is true since the special divisor $\mathcal{Z}_{\widetilde{\N}_\F}(x)$ is a summation of divisors of the form in (a). Moreover,
    \begin{align*}
        \chi(\widetilde{\N}_{\F},\rO_{\pi_{\F}^{\ast}\mathcal{C}_1}\otimes^{\mathbb{L}}_{\rO_{\widetilde{\N}_{\F}}}\rO_{\pi_{\F}^{\ast}\mathcal{C}_2})&=\left(\pi_\F^{\ast}\mathcal{C}_1\cdot\pi_\F^{\ast}\mathcal{C}_2\right)_{\widetilde{\N}_\F}=\left((\widetilde{\mathcal{C}}_1+\exc_{\widetilde{\N}_\F})\cdot\pi_\F^{\ast}\mathcal{C}_2\right)_{\widetilde{\N}_\F}\\
        &=\left(\widetilde{\mathcal{C}}_1\cdot\pi_\F^{\ast}\mathcal{C}_2\right)_{\widetilde{\N}_\F}\overset{\textup{projection formula}}{=}\left(\mathcal{C}_1\cdot\mathcal{C}_2\right)_{\N_\F}\\
        &=\chi(\N_\F,\mathcal{O}_{\mathcal{C}_1}\otimes_{\rO_{\N_\F}}^{\mathbb{L}}\mathcal{O}_{\mathcal{C}_2}).
    \end{align*}
\end{proof}
\begin{corollary}
    Let $x,y\in\B$ be two linearly independent elements. Then
    \begin{equation}
        \chi(\widetilde{\N}_{\F},\rO_{\mathcal{Z}_{\widetilde{\N}_\F}(x)}\otimes^{\mathbb{L}}_{\rO_{\widetilde{\N}_{\F}}}\rO_{\mathcal{Z}_{\widetilde{\N}_\F}(y)})=\chi(\N_{\F},\rO_{\mathcal{Z}_{\N_\F}(x)}\otimes^{\mathbb{L}}_{\rO_{\N_{\F}}}\rO_{\mathcal{Z}_{\N_\F}(y)}).
        \label{inv-hpe-spe}
    \end{equation}
    \label{invariance-hpe}
\end{corollary}
\begin{proof}
    Since $\mathcal{Z}_{\N_\F}(x)$ is a summation of regular (smooth) divisors $\mathcal{C}$. The equality follows from Lemma \ref{inv-blow-up-NF} (b).
\end{proof}
\subsection{Geometric difference formula on $\M$}
\begin{lemma}
    Let $L^{\flat}\subset\B$ be a $\zp$-lattice of rank $2$. Let $x\in\B$ be another element such that $\nu_p(q(x))\geq\max\{\textup{val}(L^{\flat}),2\}$ and $x\perp L^{\flat}$, then
    \begin{align*}
        \Int^{\mathcal{Z}}(L^{\flat}\obot\langle x\rangle)-\Int^{\mathcal{Z}}(L^{\flat}\obot\langle p^{-1}x\rangle)&=\chi(\M,{^{\mathbb{L}}\mathcal{Z}}(L^{\flat})\otimes^{\mathbb{L}}_{\rO_{\M}}\rO_{\M_{\F}})\\
        &=\chi(\M,{^{\mathbb{L}}\mathcal{Z}}(L^{\flat})\otimes^{\mathbb{L}}_{\rO_{\M}}\rO_{\MFF})+\chi(\M,{^{\mathbb{L}}\mathcal{Z}}(L^{\flat})\otimes^{\mathbb{L}}_{\rO_{\M}}\rO_{\MVV})\\
        &+\chi(\M,{^{\mathbb{L}}\mathcal{Z}}(L^{\flat})\otimes^{\mathbb{L}}_{\rO_{\M}}\rO_{\MFV})+\chi(\M,{^{\mathbb{L}}\mathcal{Z}}(L^{\flat})\otimes^{\mathbb{L}}_{\rO_{\M}}\rO_{\MVF}).
    \end{align*}
    \label{geo-diff-M}
\end{lemma}
\begin{proof}
    By Lemma \ref{linearization}, we know that
    \begin{equation*}
        \Int^{\mathcal{Z}}(L^{\flat}\obot\langle x\rangle)-\Int^{\mathcal{Z}}(L^{\flat}\obot\langle p^{-1}x\rangle)=\chi(\M,{^{\mathbb{L}}\mathcal{Z}}(L^{\flat})\otimes^{\mathbb{L}}_{\rO_{\M}}\rO_{\mathcal{D}(x)}).
    \end{equation*}
    Therefore the first equality is equivalent to
    \begin{equation}
        \chi(\M,{^{\mathbb{L}}\mathcal{Z}}(L^{\flat})\otimes^{\mathbb{L}}_{\rO_{\M}}\rO_{\mathcal{D}(x)})=\chi(\M,{^{\mathbb{L}}\mathcal{Z}}(L^{\flat})\otimes^{\mathbb{L}}_{\rO_{\M}}\rO_{\M_{\F}}).
        \label{int-spe-fiber}
    \end{equation}
    By Lemma \ref{proper-with-strict}, 
    \begin{equation*}
        {^{\mathbb{L}}\mathcal{Z}}(L^{\flat})=\textup{linear combinations of}\,\,[\mathcal{O}(1,0)],\,\,[\mathcal{O}(0,1)]\,\,\textup{and}\,\,[\mathcal{O}_{W_s}].
    \end{equation*}
    For elements of the form $[\mathcal{O}_{W_s}]$, by the moduli interpretation of the special divisor $\mathcal{Z}(x)$ and \cite[Lemma 5.11]{GK93}, we have
    \begin{equation*}
        \chi(\M,[\mathcal{O}_{W_s}]\otimes^{\mathbb{L}}_{\rO_{\M}}\rO_{\mathcal{D}(x)})=\chi(\M,[\mathcal{O}_{W_s}]\otimes^{\mathbb{L}}_{\rO_{\M}}\rO_{\M_{\F}}).
    \end{equation*}
    For the elements $[\mathcal{O}(1,0)]$ and $[\mathcal{O}(0,1)]$, we have 
    \begin{equation*}
        \chi(\M,[\mathcal{O}(1,0)]\otimes^{\mathbb{L}}_{\rO_{\M}}\rO_{\mathcal{D}(x)})=\chi(\M,[\mathcal{O}(0,1)]\otimes^{\mathbb{L}}_{\rO_{\M}}\rO_{\mathcal{D}(x)})=0
    \end{equation*}
    by Proposition \ref{difference-decomposition}. On the other hand, by Proposition \ref{geometry-M} (iii), we have
    \begin{equation*}
        \chi(\M,[\mathcal{O}(1,0)]\otimes^{\mathbb{L}}_{\rO_{\M}}\rO_{\M_{\F}})=\chi(\M,[\mathcal{O}(0,1)]\otimes^{\mathbb{L}}_{\rO_{\M}}\rO_{\M_{\F}})=0.
    \end{equation*}
    Therefore (\ref{int-spe-fiber}) is true and hence the first equality in the Lemma.
    \par
    For the second equality: Let $x_1,x_2$ be a basis of $L$, we have
    \begin{align*}
        {^{\mathbb{L}}\mathcal{Z}}(L^{\flat})\otimes^{\mathbb{L}}_{\rO_{\M}}\rO_{\exc_\M}&=[\rO_{\mathcal{Z}(x_1)}\otimes_{\rO_\M}^{\mathbb{L}}\rO_{\mathcal{Z}(x_2)}\otimes_{\rO_\M}^{\mathbb{L}}\rO_{\exc_\M}]\\
        &=\left([\rO_{\mathcal{Z}(x_1)}\otimes_{\rO_\M}^{\mathbb{L}}\rO_{\exc_\M}]\right)\otimes_{\rO_{\exc_\M}}^{\mathbb{L}}\left([\rO_{\mathcal{Z}(x_2)}\otimes_{\rO_\M}^{\mathbb{L}}\rO_{\exc_\M}]\right)\\
        &=\rO(0,-1)\otimes_{\rO_\M}^{\mathbb{L}}\rO(0,-1)=0.
    \end{align*}
    Therefore
    \begin{align*}
        {^{\mathbb{L}}\mathcal{Z}}(L^{\flat})\otimes^{\mathbb{L}}_{\rO_{\M}}\rO_{\M_{\F}}&={^{\mathbb{L}}\mathcal{Z}}(L^{\flat})\otimes^{\mathbb{L}}_{\rO_{\M}}\left(2\rO_{\exc_{\M}}+\rO_{\MFF}+\rO_{\MVV}+\rO_{\MFV}+\rO_{\MFV}\right)\\
        &={^{\mathbb{L}}\mathcal{Z}}(L^{\flat})\otimes^{\mathbb{L}}_{\rO_{\M}}\left(\rO_{\MFF}+\rO_{\MVV}+\rO_{\MFV}+\rO_{\MFV}\right).
    \end{align*}
    Therefore the second equality is true.
\end{proof}
\par

\subsection{Intersections on $\MFF$}
We defined $p_{+}:\M^{+}\rightarrow\N$ as a composition $\M^{+}\rightarrow\M\stackrel{\pi}\rightarrow\N(x_0)\stackrel{s_{+}}\rightarrow\N$. Notice that $\MFF$ is a closed formal subscheme of $\M^{+}_{\F}$. Let $\pi^{\textup{FF}}:\MFF\rightarrow\N_{\F}$ be the composition $\MFF\hookrightarrow\M^{+}_{\F}\xrightarrow{(p_{+})_{\F}}\N_{\F}$. Using the open cover we fixed in $\S$\ref{opencover} and Lemma \ref{coordinate-change}, we have the following isomorphism of formal schemes over $\N_\F$:
\begin{equation}
    \iota^{\textup{FF}}:\MFF\xrightarrow{\sim}\widetilde{\N}_\F.
    \label{MFF-id}
\end{equation}
\par
Let $x\in\B$ be a nonzero element. By the moduli interpretations of the divisor $\mathcal{Z}^{+}(x)$ on $\M^{+}$ in Lemma \ref{divisor} and the isomorphism $\iota^{\textup{FF}}:\MFF\rightarrow\widetilde{\N}_{\F}$, we have
\begin{equation}
    \mathcal{Z}(x)\cap\MFF=\left(\iota^{\textup{FF}}\right)^{\ast}\mathcal{Z}_{\widetilde{\N}_{\F}}(x)\coloneqq\mathcal{Z}_{\widetilde{\N}_{\F}}(x)\times_{\widetilde{\N}_{\F},\iota^{\textup{FF}}}\MFF=\mathcal{Z}_{\N_{\F}}(x)\times_{\N_{\F},\pi^{\FF}}\MFF.
    \label{spe-MFF}
\end{equation}
\begin{lemma}
    Let $L^{\flat}\subset\B$ be a $\zp$-lattice of rank $2$. Then
    \begin{equation*}
        \chi(\M,{^{\mathbb{L}}\mathcal{Z}}(L^{\flat})\otimes^{\mathbb{L}}_{\rO_{\M}}\rO_{\MFF})=\chi(\N,{^{\mathbb{L}}\mathcal{Z}_\N}(L^{\flat})\otimes^{\mathbb{L}}_{\rO_{\N}}\rO_{\N_{\F}}).
    \end{equation*}
    \label{int-MFF}
\end{lemma}
\begin{proof}
    Let $x,y$ be a $\zp$-basis of the lattice $L^{\flat}$. Then
    \begin{align*}
        \chi(\M,{^{\mathbb{L}}\mathcal{Z}}(L^{\flat})\otimes^{\mathbb{L}}_{\rO_{\M}}\rO_{\MFF})=&\chi(\MFF,\rO_{\mathcal{Z}(x)\cap\MFF}\otimes_{\rO_{\MFF}}^{\mathbb{L}}\rO_{\mathcal{Z}(y)\cap\MFF})\\
        \overset{(\ref{spe-MFF})}{=}&\chi(\widetilde{\N}_{\F},\rO_{\mathcal{Z}_{\widetilde{\N}_\F}(x)}\otimes^{\mathbb{L}}_{\rO_{\widetilde{\N}_{\F}}}\rO_{\mathcal{Z}_{\widetilde{\N}_\F}(y)})\\
        \overset{(\ref{inv-hpe-spe})}{=}&\chi(\N_{\F},\rO_{\mathcal{Z}_{\N_\F}(x)}\otimes^{\mathbb{L}}_{\rO_{\N_{\F}}}\rO_{\mathcal{Z}_{\N_\F}(y)})\\
        =&\chi(\N,{^{\mathbb{L}}\mathcal{Z}_\N}(L^{\flat})\otimes^{\mathbb{L}}_{\rO_{\N}}\rO_{\N_{\F}}).
    \end{align*}
\end{proof}
\label{MFF-CAL}

\subsection{Intersections on $\MVV$}
The ideas of the computations in this part is similar to $\S$\ref{MFF-CAL}. We have defined $p_{-}:\M^{-}\rightarrow\N$ as a composition $\M^{-}\rightarrow\M\stackrel{\pi}\rightarrow\N(x_0)\stackrel{s_{-}}\rightarrow\N$. Notice that $\MVV$ is a closed formal subscheme of $\M^{-}_{\F}$. Let $\pi^{\textup{VV}}:\MVV\rightarrow\N_{\F}$ be the composition $\MVV\hookrightarrow\M^{-}_{\F}\xrightarrow{(p_{-})_{\F}}\N_{\F}$. Using the open cover we fixed in $\S$\ref{opencover} and Lemma \ref{coordinate-change}, we have the following isomorphism of formal schemes over $\N_\F$:
\begin{equation}
    \iota^{\textup{VV}}:\MVV\xrightarrow{\sim}\widetilde{\N}_\F.
    \label{MVV-id}
\end{equation}
\par
Let $x\in\B$ be a nonzero element. By the moduli interpretations of the divisor $\mathcal{Z}^{-}(x)$ on $\M^{-}$ in Lemma \ref{divisor} and the isomorphism $\iota^{\textup{VV}}:\MVV\rightarrow\widetilde{\N}_{\F}$, we have
\begin{equation}
    \mathcal{Z}(x)\cap\MVV=\left(\iota^{\textup{VV}}\right)^{\ast}\mathcal{Z}_{\widetilde{\N}_{\F}}(x^{\prime})\coloneqq\mathcal{Z}_{\widetilde{\N}_{\F}}(x^{\prime})\times_{\widetilde{\N}_{\F},\iota^{\textup{VV}}}\MFF=\mathcal{Z}_{\N_{\F}}(x^{\prime})\times_{\N_{\F},\pi^{\VV}}\MVV.
    \label{spe-MVV}
\end{equation}
\begin{lemma}
    Let $L^{\flat}\subset\B$ be a $\zp$-lattice of rank $2$. Let $L^{\flat\prime}\subset\B$ be the image of $L^{\flat}$ under the isometric homomorphism $(\cdot)^{\prime}:\B\rightarrow\B$. Then
    \begin{equation*}
        \chi(\M,{^{\mathbb{L}}\mathcal{Z}}(L^{\flat})\otimes^{\mathbb{L}}_{\rO_{\M}}\rO_{\MVV})=\chi(\N,{^{\mathbb{L}}\mathcal{Z}_\N}(L^{\flat\prime})\otimes^{\mathbb{L}}_{\rO_{\N}}\rO_{\N_{\F}}).
    \end{equation*}
    \label{int-MVV}
\end{lemma}
\begin{proof}
    The proof is the same as that of Lemma \ref{int-MFF}, we just need to replace $\MFF$ there by $\MVV$. 
\end{proof}
\subsection{Intersections on $\MFV$}
\label{MFV-CAL}
Let 
\begin{equation}
   \left(X_1^{\textup{FV}}\stackrel{x_{0}^{\textup{F}}}\longrightarrow X_{1}^{\prime\FV},\left(\rho_1^{\FV},\rho_1^{\prime\FV}\right)\right),\,\,\,\,\left(X_2^{\FV}\stackrel{x_{0}^{\textup{V}}}\longrightarrow X_{2}^{\prime\textup{\FV}},\left(\rho_2^{\FV},\rho_2^{\prime\FV}\right)\right)
   \label{MFV-univ}
\end{equation}
be the base change of the universal object (\ref{univ-obj-M}) over $\M$ to $\MFV$. By the equations of $\MFV$ in Proposition \ref{geometry-M}, the morphism $x_0^{\textup{F}}$ is isomorphic to the Frobenius morphism, while $x_0^{\textup{V}}$ is isomorphic to the Verschiebung morphism. Hence there exist two isomorphisms $\iota_1^{\textup{F}}:X_1^{\FV,(p)}\rightarrow X_1^{\prime\FV}$ and $\iota_2^{\textup{V}}:X_2^{\prime\FV,(p)}\rightarrow X_2^{\FV}$ over $\MFV$ such that the following diagrams commute,
\begin{equation*}
    \begin{tikzcd}
& X_1^{\FV,(p)} \ar[dd, "\iota_1^{\textup{F}}"]
\\
X_1^{\textup{FV}} \ar[ur, "\textup{F}"] \ar[dr, "x_0^{\textup{F}}"']
\\
& X_1^{\prime\FV} 
\end{tikzcd}
\begin{tikzcd}
& X_2^{\prime\FV,(p)} \ar[dd, "\iota_2^{\textup{V}}"'] \ar[dr, "\textup{V}"]
\\
& &X_2^{\prime\textup{FV}}  
\\
& X_2^{\FV} \ar[ur, "x_0^{\textup{V}}"']
\end{tikzcd}
\end{equation*}
here we use $\textup{F},\textup{V}$ to represent the relative Frobenius and Verschiebung morphism.
\par
For a morphism $f:S\rightarrow\MFV$. Let $(X_{1,S}^{\FV},\rho_{1,S}^{\FV})$ and $(X_{2,S}^{\prime\FV},\rho_{2,S}^{\prime\FV})$ be the base change of the $p$-divisible groups $(X_{1}^{\FV},\rho_{1}^{\FV})$ and $(X_{2}^{\prime\FV},\rho_{2,S}^{\prime\FV})$ to $S$ through the morphism $f$. Then we get a map $\MFV(S)\rightarrow\N_{\F}(S):f\mapsto\left(\left(X_{1,S}^{\FV},\rho_{1,S}^{\FV}\right),\left(X_{2}^{\prime\FV},\rho_{2,S}^{\prime\FV}\right)\right)$. Therefore we obtain a morphism between formal schemes $\pi^{\FV}:\MFV\rightarrow\N_{\F}$. Using the open cover we fixed in $\S$\ref{opencover} and Lemma \ref{coordinate-change}, we have the following isomorphism of formal schemes over $\N_\F$:
\begin{equation}
    \iota^{\textup{FV}}:\MFV\xrightarrow{\sim}\widetilde{\N}_\F.
    \label{MFV-id}
\end{equation}
Denote by $\MFV_{\circ}=\left(\iota^{\FV}\right)^{\ast}\left(\widetilde{\N}_\F^{\circ}\right)$, $\MVF_{\bullet}=\left(\iota^{\FV}\right)^{\ast}\left(\widetilde{\N}_\F^{\bullet}\right)$. Then $\MFV_{\circ}\subset\M^{+}$, $\MFV_{\bullet}\subset\M^{-}$ and $\{\MFV_{\circ},\MFV_{\bullet}\}$ gives an open cover of $\MFV$.
\par
Let $\N_{0,\F}=\N_{0}\times_{W}\F$. Denote by $\textup{Fr}:\N_{0,\F}\rightarrow\N_{0,\F}$ the relative Frobenius morphism of $\N_{0,\F}$. The morphism sends a pair $(X,\rho)\in\N_{0,\F}(S)$ to the pair $(X^{(p)},\rho^{(p)})$, here $X^{(p)}=X\times_{S,\textup{Fr}_{S}}S$ where $\textup{Fr}_{S}:S\rightarrow S$ is the absolute Frobenius morphism. Denote by $p^{\FV}_{\circ}:\MFV_{\circ}\rightarrow\N_{\F}$ the restriction of the morphism $p_{+}:\M^{+}\rightarrow\N$ (cf. (\ref{pMtoN})) to $\MFV_{\circ}$. Denote by $p^{\FV}_{\bullet}:\MFV_{\bullet}\rightarrow\N_{\F}$ the restriction of the morphism $p_{-}:\M^{-}\rightarrow\N$ (cf. (\ref{pMtoN})) to $\MFV_{\bullet}$. Then we have
\begin{align}
    p_{\circ}^{\FV}=\left(\textup{Id}\times\textup{Fr}\right)\circ\pi^{\FV},\,\,\,\,
    p_{\bullet}^{\FV}=\left(\textup{Fr}\times\textup{Id}\right)\circ\pi^{\FV}.
    \label{MFVtoN}
\end{align}
\par
Let $x\in\B$ be a nonzero element such that $\nu_p(q(x))\geq0$. By the moduli interpretation of the divisor $\mathcal{Z}(x)$ in Lemma \ref{divisor}, we have
\begin{equation}
   \mathcal{Z}(x)\cap\MFV_{\circ}=\left(\pi^{\FV}\right)^{\ast}\left(\textup{Id}\times\textup{Fr}\right)^{\ast}\mathcal{Z}_{\N_{\F}}(x),\,\, \mathcal{Z}(x)\cap\MFV_{\bullet}=\left(\pi^{\FV}\right)^{\ast}\left(\textup{Fr}\times\textup{Id}\right)^{\ast}\mathcal{Z}_{\N_{\F}}(x^{\prime}),
   \label{spe-mfv}
\end{equation}
Notice that the quasi-isogeny $x_0\cdot x=x^{\prime}\cdot x_0:X_1^{\FV}\dashrightarrow X_2^{\prime\FV}$ lifts to an isogeny over $\mathcal{Z}(x)\cap\MFV$. Therefore
\begin{equation}
    \mathcal{Z}(x)\cap\MFV_{\circ}\subset\left(\pi^{\FV}\right)^{\ast}\mathcal{Z}_{\N_{\F}}(x_0\cdot x),\,\, \mathcal{Z}(x)\cap\MFV_{\bullet}=\left(\pi^{\FV}\right)^{\ast}\mathcal{Z}_{\N_{\F}}(x^{\prime}\cdot x_0).
    \label{spe-MFV}
\end{equation}
\begin{lemma}
    Let $x\in\B$ be a nonzero element such that $\nu_p(q(x))\geq0$. Then we have the following identity of effective Cartier divisors on $\MFV$:
    \begin{equation}
        \mathcal{Z}(x)\cap\MFV=\left(\iota^{\textup{FV}}\right)^{\ast}\left(p\cdot\mathcal{Z}_{\widetilde{\N}_{\F}}(\overline{x_0}^{-1}\cdot x)+\exc_{\widetilde{\N}_{\F}}\right).
        \label{special-cycle-MFV-decom}
    \end{equation}
    \label{spe-cyc-MFV0dec}
\end{lemma}
\begin{proof}
    Let $n=\nu_p(q(x))\geq0$. Since $x_0^{\textup{F}}$ (resp. $x_0^{\textup{V}}$) is isomorphic to the relative Frobenius (resp. Verschiebung) morphism, we have the following equality:
    \begin{align*}
        \left(\textup{Id}\times\textup{Fr}\right)^{\ast}\mathcal{Z}_{\N_{\F}}(x)&=\mathcal{Z}_{\N_\F}(x_0\cdot x)-\mathcal{C}_{n+1}(x_0\cdot x)=p\cdot\mathcal{Z}_{\N_{\F}}(\overline{x_0}^{-1}\cdot x)+\mathcal{C}_{-n-1}(x_0\cdot x),\\
        \left(\textup{Fr}\times\textup{Id}\right)^{\ast}\mathcal{Z}_{\N_{\F}}(x)&=\mathcal{Z}_{\N_\F}( x^{\prime}\cdot x_0)-\mathcal{C}_{-n-1}(x^{\prime}\cdot x_0)=p\cdot\mathcal{Z}_{\N_{\F}}(\overline{x_0}^{-1}\cdot x)+\mathcal{C}_{n+1}(x^{\prime}\cdot x_0).
    \end{align*}
    Here we use the equality $p^{-1}x_0\cdot x=\overline{x_0}^{-1}\cdot x$.
    By (\ref{spe-mfv}) and (\ref{spe-MFV}) and the equation of $\mathcal{C}_{-n-1}(x_0\cdot x)$ in Lemma \ref{inv-blow-up-NF}(a), we have that over the open formal subscheme $\MFV_{\circ}$, 
    \begin{align*}
        \mathcal{Z}(x)\cap\MFV_{\circ}&=\left(\pi^{\FV}\right)^{\ast}\left(p\cdot\mathcal{Z}_{\N_{\F}}(\overline{x_0}^{-1}\cdot x)+\mathcal{C}_{-n-1}(x_0\cdot x)\right)\\&=\left(\iota^{\textup{FV}}\right)^{\ast}\left(p\cdot\mathcal{Z}_{\widetilde{\N}_{\F}}(\overline{x_0}^{-1}\cdot x)+\exc_{\widetilde{\N}_{\F}}\right).
    \end{align*}
    While over the open formal subscheme $\MFV_{\bullet}$, 
    \begin{align*}
        \mathcal{Z}(x)\cap\MFV_{\bullet}&=\left(\pi^{\FV}\right)^{\ast}\left(p\cdot\mathcal{Z}_{\N_{\F}}(\overline{x_0}^{-1}\cdot x)+\mathcal{C}_{n+1}(x^{\prime}\cdot x_0)\right)\\&=\left(\iota^{\textup{FV}}\right)^{\ast}\left(p\cdot\mathcal{Z}_{\widetilde{\N}_{\F}}(\overline{x_0}^{-1}\cdot x)+\exc_{\widetilde{\N}_{\F}}\right).
    \end{align*}
    The formula (\ref{special-cycle-MFV-decom}) is true on an open cover of $\MFV$, hence is true over all $\MFV$.
\end{proof}
\begin{corollary}
    Let $L^{\flat}\subset\B$ be a $\zp$-lattice of rank $2$. Then
    \begin{equation*}
        \chi(\M,{^{\mathbb{L}}\mathcal{Z}}(L^{\flat})\otimes^{\mathbb{L}}_{\rO_{\M}}\rO_{\MFV})=p^{2}\cdot\chi\left(\N,{^{\mathbb{L}}\mathcal{Z}_\N}\left(\overline{x_0}^{-1}\cdot L^{\flat}\right)\otimes^{\mathbb{L}}_{\rO_{\N}}\rO_{\N_{\F}}\right)-1.
    \end{equation*}
    \label{int-MFV}
\end{corollary}
\begin{proof}
    Let $x,y$ be a $\zp$-basis of the lattice $L^{\flat}$. We have
    \begin{align*}
        &\chi(\M,{^{\mathbb{L}}\mathcal{Z}}(L^{\flat})\otimes^{\mathbb{L}}_{\rO_{\M}}\rO_{\MFV})=\chi(\MFV,\rO_{\mathcal{Z}(x)\cap\MFV}\otimes^{\mathbb{L}}_{\rO_{\MFV}}\rO_{\mathcal{Z}(x)\cap\MFV})\\
        &\overset{(\ref{exc-blow-up-N_0-F-int}),(\ref{special-cycle-MFV-decom})}{=}p^{2}\cdot\chi(\widetilde{\N}_\F,\rO_{\mathcal{Z}_{\widetilde{\N}_\F}}(x)\otimes^{\mathbb{L}}_{\rO_{\widetilde{\N}_\F}}\rO_{\mathcal{Z}_{\widetilde{\N}_\F}}(y))+\chi(\widetilde{\N}_\F,\rO_{\exc_{\widetilde{\N}_\F}}\otimes^{\mathbb{L}}_{\rO_{\widetilde{\N}_\F}}\rO_{\exc_{\widetilde{\N}_\F}})\\
        &\overset{(\ref{exc-N0-F-blow-formula}),(\ref{inv-hpe-spe})}{=}p^{2}\cdot\chi\left(\N,{^{\mathbb{L}}\mathcal{Z}_\N}\left(\overline{x_0}^{-1}\cdot L^{\flat}\right)\otimes^{\mathbb{L}}_{\rO_{\N}}\rO_{\N_{\F}}\right)-1.
    \end{align*}
\end{proof}

\subsection{Intersections on $\MVF$}
The ideas of the computations in this part is similar to $\S$\ref{MFV-CAL}. Let 
\begin{equation}
   \left(X_1^{\textup{VF}}\stackrel{x_{0}^{\textup{V}}}\longrightarrow X_{1}^{\prime\VF},\left(\rho_1^{\VF},\rho_1^{\prime\VF}\right)\right),\,\,\,\,\left(X_2^{\VF}\stackrel{x_{0}^{\textup{F}}}\longrightarrow X_{2}^{\prime\textup{\VF}},\left(\rho_2^{\VF},\rho_2^{\prime\VF}\right)\right)
   \label{MVF-univ}
\end{equation}
be the base change of the universal object (\ref{univ-obj-M}) over $\M$ to $\MVF$. By the equations of $\MVF$ in Proposition \ref{geometry-M}, the morphism $x_0^{\textup{F}}$ is isomorphic to the Frobenius morphism, while $x_0^{\textup{V}}$ is isomorphic to the Verschiebung morphism. Hence there exist two isomorphisms $\iota_2^{\textup{F}}:X_2^{\VF,(p)}\rightarrow X_2^{\prime\VF}$ and $\iota_1^{\textup{V}}:X_1^{\prime\VF,(p)}\rightarrow X_1^{\VF}$ over $\MVF$ such that the following diagrams commute,
\begin{equation*}
    \begin{tikzcd}
& X_2^{\VF,(p)} \ar[dd, "\iota_2^{\textup{F}}"]
\\
X_2^{\textup{VF}} \ar[ur, "\textup{F}"] \ar[dr, "x_0^{\textup{F}}"']
\\
& X_2^{\prime\VF} 
\end{tikzcd}
\begin{tikzcd}
& X_1^{\prime\VF,(p)} \ar[dd, "\iota_1^{\textup{V}}"'] \ar[dr, "\textup{V}"]
\\
& &X_1^{\prime\textup{VF}}  
\\
& X_1^{\VF} \ar[ur, "x_0^{\textup{V}}"']
\end{tikzcd}
\end{equation*}
here we use $\textup{F},\textup{V}$ to represent the relative Frobenius and Verschiebung morphism.
\par
For a morphism $g:S\rightarrow\MVF$. Let $(X_{1,S}^{\prime\VF},\rho_{1,S}^{\prime\VF})$ and $(X_{2,S}^{\VF},\rho_{2,S}^{\VF})$ be the base change of the $p$-divisible groups $(X_{1}^{\prime\VF},\rho_{1}^{\prime\VF})$ and $(X_{2}^{\VF},\rho_{2,S}^{\VF})$ to $S$ through the morphism $g$. Then we get a map $\MVF(S)\rightarrow\N_{\F}(S):g\mapsto\left(\left(X_{1,S}^{\prime\VF},\rho_{1,S}^{\prime\VF}\right),\left(X_{2}^{\VF},\rho_{2,S}^{\VF}\right)\right)$. Therefore we obtain a morphism between formal schemes $\pi^{\VF}:\MVF\rightarrow\N_{\F}$. Using the open cover we fixed in $\S$\ref{opencover} and Lemma \ref{coordinate-change}, we have the following isomorphism of formal schemes over $\N_\F$:
\begin{equation}
    \iota^{\textup{VF}}:\MVF\xrightarrow{\sim}\widetilde{\N}_\F.
    \label{MVF-id}
\end{equation}
Denote by $\MVF_{\circ}=\left(\iota^{\VF}\right)^{\ast}\left(\widetilde{\N}_\F^{\circ}\right)$, $\MFV_{\bullet}=\left(\iota^{\VF}\right)^{\ast}\left(\widetilde{\N}_\F^{\bullet}\right)$. Then $\MVF_{\circ}\subset\M^{-}$, $\MVF_{\bullet}\subset\M^{+}$ and $\{\MVF_{\circ},\MVF_{\bullet}\}$ gives an open cover of $\MVF$.
\par
Recall that we defined $\textup{Fr}:\N_{0,\F}\rightarrow\N_{0,\F}$ as the relative Frobenius morphism of $\N_{0,\F}$ in $\S$\ref{hyperspecial0}. Denote by $p^{\VF}_{\circ}:\MVF_{\circ}\rightarrow\N_{\F}$ the restriction of the morphism $p_{-}:\M^{-}\rightarrow\N$ (cf. (\ref{pMtoN})) to $\MVF_{\circ}$. Denote by $p^{\VF}_{\bullet}:\MVF_{\bullet}\rightarrow\N_{\F}$ the restriction of the morphism $p_{+}:\M^{+}\rightarrow\N$ (cf. (\ref{pMtoN})) to $\MVF_{\bullet}$. Then we have
\begin{align}
    p_{\circ}^{\VF}=\left(\textup{Id}\times\textup{Fr}\right)\circ\pi^{\VF},\,\,\,\,
    p_{\bullet}^{\VF}=\left(\textup{Fr}\times\textup{Id}\right)\circ\pi^{\VF}.
    \label{MVFtoN}
\end{align}
\par
Let $x\in\B$ be a nonzero element such that $\nu_p(q(x))\geq0$. By the moduli interpretation of the divisor $\mathcal{Z}(x)$ in Lemma \ref{divisor}, we have
\begin{equation}
   \mathcal{Z}(x)\cap\MVF_{\circ}=\left(\pi^{\VF}\right)^{\ast}\left(\textup{Id}\times\textup{Fr}\right)^{\ast}\mathcal{Z}_{\N_{\F}}(x^{\prime}),\,\, \mathcal{Z}(x)\cap\MFV_{\bullet}=\left(\pi^{\FV}\right)^{\ast}\left(\textup{Fr}\times\textup{Id}\right)^{\ast}\mathcal{Z}_{\N_{\F}}(x),
   \label{spe-mvf}
\end{equation}
Notice that the quasi-isogeny $\overline{x}_0\cdot x^{\prime}=x\cdot \overline{x}_0:X_1^{\prime\FV}\dashrightarrow X_2^{\FV}$ lifts to an isogeny over $\mathcal{Z}(x)\cap\MVF$. Therefore
\begin{equation}
    \mathcal{Z}(x)\cap\MFV_{\circ}\subset\left(\pi^{\VF}\right)^{\ast}\mathcal{Z}_{\N_{\F}}(\overline{x}_0\cdot x^{\prime}),\,\, \mathcal{Z}(x)\cap\MFV_{\bullet}=\left(\pi^{\VF}\right)^{\ast}\mathcal{Z}_{\N_{\F}}(x\cdot \overline{x}_0).
    \label{spe-MVF}
\end{equation}
\begin{lemma}
    Let $x\in\B$ be a nonzero element such that $\nu_p(q(x))\geq0$. Then we have the following identity of effective Cartier divisors on $\MVF$:
    \begin{equation*}
        \mathcal{Z}(x)\cap\MVF=\left(\iota^{\textup{VF}}\right)^{\ast}\left(p\cdot\mathcal{Z}_{\widetilde{\N}_{\F}}(x_0^{-1}\cdot x)+\exc_{\widetilde{\N}_{\F}}\right).
    \end{equation*}
\end{lemma}
\begin{proof}
    The proof is almost identical to that of Lemma \ref{spe-cyc-MFV0dec}, so we omit it.
\end{proof}
\begin{corollary}
    Let $L^{\flat}\subset\B$ be a $\zp$-lattice of rank $2$. Then
    \begin{equation*}
        \chi(\M,{^{\mathbb{L}}\mathcal{Z}}(L^{\flat})\otimes^{\mathbb{L}}_{\rO_{\M}}\rO_{\MVF})=p^{2}\cdot\chi\left(\N,{^{\mathbb{L}}\mathcal{Z}_\N}\left(x_0^{-1}\cdot L^{\flat}\right)\otimes^{\mathbb{L}}_{\rO_{\N}}\rO_{\N_{\F}}\right)-1.
    \end{equation*}
    \label{int-MVF}
\end{corollary}
\begin{proof}
The proof is almost identical to that of Corollary \ref{int-MFV}, so we omit it.    
\end{proof}

\section{Proof of the main theorem}
\label{main-proof-part}
\subsection{Difference formulas combined}
\begin{lemma}
Let $L^{\flat}\subset\B$ be a $\zp$-lattice of rank $2$. Let $x\in\B$ be a nonzero element such that $x\,\bot\,L^{\flat}$ and $\nu_p(q(x))\geq\max\{\max(L^{\flat}),2\}$. Then we have
    \begin{align}
    &\chi(\M,{^{\mathbb{L}}\mathcal{Z}}(L^{\flat})\otimes^{\mathbb{L}}_{\rO_{\M}}\rO_{\MFF})=\chi(\M,{^{\mathbb{L}}\mathcal{Z}}(L^{\flat})\otimes^{\mathbb{L}}_{\rO_{\M}}\rO_{\MVV})=\partial\den\left(\lx[-1]\obot H_2^{+}, L^{\flat}\right),\label{MFF-MVV-INT}\\
    &\chi(\M,{^{\mathbb{L}}\mathcal{Z}}(L^{\flat})\otimes^{\mathbb{L}}_{\rO_{\M}}\rO_{\MFV})=\chi(\M,{^{\mathbb{L}}\mathcal{Z}}(L^{\flat})\otimes^{\mathbb{L}}_{\rO_{\M}}\rO_{\MVF})=\partial\den\left(\lx[-1]\obot H_2^{+}[p], L^{\flat}\right).\label{MFV-MVF-INT}
\end{align}
\label{term-identity}
\end{lemma}
\begin{proof}
    The formula (\ref{MFF-MVV-INT}) is proved by combining Lemma \ref{int-MFF}, Lemma \ref{int-MVV} and Corollary \ref{diff-to-p-hpe}. Let's now give the proof of the formula (\ref{MFV-MVF-INT}). By Corollaries \ref{int-MFV},  \ref{int-MVF} and   \ref{diff-to-p-hpe}, we have
    \begin{equation*}
        \chi(\M,{^{\mathbb{L}}\mathcal{Z}}(L^{\flat})\otimes^{\mathbb{L}}_{\rO_{\M}}\rO_{\MFV})=\chi(\M,{^{\mathbb{L}}\mathcal{Z}}(L^{\flat})\otimes^{\mathbb{L}}_{\rO_{\M}}\rO_{\MVF})=p^{2}\cdot\partial\den\left(\lx[-p^{-1}]\obot H_2^{+},L^{\flat}[p^{-1}]\right)-1.
    \end{equation*}
    By Lemma \ref{analytic-calculations-1}, we have
    \begin{equation*}
        \partial\den\left(\lx[-1]\obot H_2^{+}[p],L^{\flat}\right)=p^{2}\cdot\partial\den\left(\lx[-p^{-1}]\obot H_2^{+},L^{\flat}[p^{-1}]\right)-1.
    \end{equation*}
    Therefore (\ref{MFV-MVF-INT}) is true.
\end{proof}
\begin{corollary}
    Let $L^{\flat}\subset\B$ be a $\zp$-lattice of rank $2$. Let $x\in\B$ be a nonzero element such that $x\,\bot\,L^{\flat}$ and $\nu_p(q(x))\geq\max\{\max(L^{\flat}),2\}$. Then we have
    \begin{equation*}
        \Int^{\mathcal{Z}}\left(L^{\flat}\obot\lx\right)-\Int^{\mathcal{Z}}\left(L^{\flat}\obot\langle p^{-1}x\rangle\right)=\partial\den\left(H_0(p),L^{\flat}\obot\lx\right)-\partial\den\left(H_0(p),L^{\flat}\obot\langle p^{-1}x\rangle\right).
    \end{equation*}
    \begin{equation*}
        \Int^{\mathcal{Y}}\left(L^{\flat}\obot\lx\right)-\Int^{\mathcal{Y}}\left(L^{\flat}\obot\langle p^{-1}x\rangle\right)=\partial\den\left(H_0(p)^{\vee},L^{\flat}\obot\lx\right)-\partial\den\left(H_0(p)^{\vee},L^{\flat}\obot\langle p^{-1}x\rangle\right).
    \end{equation*}
    \label{diff-com}
\end{corollary}
\begin{proof}
    For the intersection number of $\mathcal{Z}$-cycles, we have
    \begin{align*}
        \Int^{\mathcal{Z}}\left(L^{\flat}\obot\lx\right)-\Int^{\mathcal{Z}}&\left(L^{\flat}\obot\langle p^{-1}x\rangle\right)=2\cdot\chi(\M,{^{\mathbb{L}}\mathcal{Z}}(L^{\flat})\otimes^{\mathbb{L}}_{\rO_{\M}}\rO_{\MFF})+2\cdot\chi(\M,{^{\mathbb{L}}\mathcal{Z}}(L^{\flat})\otimes^{\mathbb{L}}_{\rO_{\M}}\rO_{\MFV})\\
        &\overset{(\ref{MFF-MVV-INT}),(\ref{MFV-MVF-INT})}{=}2\cdot\partial\den\left(\lx[-1]\obot H_2^{+}, L^{\flat}\right)+2\cdot\partial\den\left(\lx[-1]\obot H_2^{+}[p], L^{\flat}\right).
    \end{align*}
    On the analytic side, we have the following formula by Lemma \ref{der-diff},
    \begin{align*}
        \partial\den\left(H_0(p),L^{\flat}\obot\lx\right)-&\partial\den\left(H_0(p),L^{\flat}\obot\langle p^{-1}x\rangle\right)\\
        &=2\cdot\partial\den\left(\lx[-1]\obot H_2^{+}, L^{\flat}\right)+2\cdot\partial\den\left(\lx[-1]\obot H_2^{+}[p], L^{\flat}\right).
    \end{align*}
    Hence the first formula follows from the above two identities.
    \par
    For the intersection number of $\mathcal{Y}$-cycles. By the identity $\mathcal{Y}(H)=\left(\iota_{\M}\right)^{\ast}\mathcal{Z}(x_0\cdot H)$, we have
    \begin{align*}
        \Int^{\mathcal{Y}}&\left(L^{\flat}\obot\lx\right)-\Int^{\mathcal{Y}}\left(L^{\flat}\obot\langle p^{-1}x\rangle\right)\\
        &=\Int^{\mathcal{Z}}\left(L^{\flat}[q(x_0)]\obot\lx[q(x_0)]\right)-\Int^{\mathcal{Z}}\left(L^{\flat}[q(x_0)]\obot\langle p^{-1}x\rangle[q(x_0)]\right)\\
        &=\partial\den\left(H_0(p),L^{\flat}[q(x_0)]\obot\lx[q(x_0)]\right)-\partial\den\left(H_0(p),L^{\flat}[q(x_0)]\obot\langle p^{-1}x\rangle[q(x_0)]\right)\\
        &=\partial\den\left(H_0(p)^{\vee},L^{\flat}\obot\lx\right)-\partial\den\left(H_0(p)^{\vee},L^{\flat}\obot\langle p^{-1}x\rangle\right).
    \end{align*}
    Here the last identity is proved by Lemma \ref{analytic-calculations-2}.
\end{proof}

\subsection{Base cases: the geometric side}
\label{base-geo}
Let $L\subset\B$ be a quadratic lattice of rank $3$ such that $\min(L)=0$. Let $x\in L$ be an element such that $\nu_p(q(x))=0$. By Lemma \ref{dec-diff-div}, we have the following equality of Cartier divisors on $\M$,
\begin{equation*}
    \mathcal{Z}(x)=\exc_{\M}+\N_0^{\textup{I}+}(x_0\cdot x),
\end{equation*}
where the divisor $\N_0^{\textup{I}+}(x_0\cdot x)$ is isomorphic to the blow up $\widetilde{\N}_0(x_0\cdot x)$ of the formal scheme $\N_0(x_0\cdot x)$ along its unique closed $\F$-point. For simplicity, let $z=x_0\cdot x$. We fix an isomorphism $\iota_{z}:\N_0^{\textup{I}+}(z)\xrightarrow{\sim}\widetilde{\N}_0(z)$. Let $p_z$ be the following composition morphism
\begin{equation*}
    p_z:\N_0^{\textup{I}+}(z)\xrightarrow{\iota_z}\widetilde{\N}_0(z)\xrightarrow{\pi_z}\N_0(z)\hookrightarrow\N.
\end{equation*}
Recall that we use $\exc_{z}^{\textup{I}+}$ to denote the exceptional divisor on $\N_0^{\textup{I}+}(z)$.
\begin{lemma}
    Let $y\in\B$ be an element such that $x,y$ are linearly independent. Then we have
    \begin{equation}
        \chi(\N_{0}^{\textup{I}+}(z),\mathcal{O}_{\N_{0}^{\textup{I}+}(z)\cap\mathcal{Z}(y)}\otimes^{\mathbb{L}}_{\rO_{\N_{0}^{\textup{I}+}(z)}}\rO_{\exc_{z}^{\textup{I}+}})=-1.
    \end{equation}\label{N0Z-EXC-SPE}
\end{lemma}
\begin{proof}
    By definition we know that $\exc_{z}^{\textup{I}+}=\exc_\M\cap\N_{0}^{\textup{I}+}(z)$. As a divisor on $\exc_{\M}$, we have
    \begin{equation*}
        [\rO_{\exc_{\M}}\otimes^{\mathbb{L}}_{\rO_{\M}}\rO_{\N_0^{\textup{I}+}(z)}]=[\rO_{\exc_z^{\textup{I}+}}]=\rO(1,0)
    \end{equation*}
    by Corollary \ref{spe-decom} (b). Then we have
    \begin{align*}
        \chi(\N_{0}^{\textup{I}+}(z),\mathcal{O}_{\N_{0}^{\textup{I}+}(z)\cap\mathcal{Z}(y)}\otimes^{\mathbb{L}}_{\rO_{\N_{0}^{\textup{I}+}(z)}}\rO_{\exc_{z}^{\textup{I}+}})&=\chi(\mathcal{M},\rO_{\N_0^{\textup{I}+}(z)}\otimes^{\mathbb{L}}_{\rO_\M}\rO_{\mathcal{Z}(y)}\otimes^{\mathbb{L}}_{\rO_\M}\rO_{\exc_\M})\\
        &=\chi(\exc_\M,\rO_{\exc_z^{\textup{I}+}}\otimes^{\mathbb{L}}_{\rO_{\exc_\M}}\rO_{\mathcal{Z}(y)\cap\exc_{\M}})\\
        &=\rO(1,0)\cdot\rO(-1,-1)=-1.
    \end{align*}
\end{proof}
\par
Let $y\in\B$ be an element such that it is linearly independent from $x$ and $\nu_p(q(y))=1$. Denote by $z^{\prime}=x_0\cdot y$.
By the moduli interpretation of the special divisor $\mathcal{Z}(y)$, we have 
\begin{equation*}
    \left(p_z\right)^{\ast}\left(\mathcal{Z}_{\N}(p^{-1}z^{\prime})\right)=\left(\pi_z\iota_z\right)^{\ast}\left(\N_0(z)\cap\mathcal{Z}_\N(p^{-1}z^{\prime})\right)\subset\N_0^{\textup{I}+}(z)\cap\mathcal{Z}(y).
\end{equation*}
Notice that $\nu_p(q(p^{-1}z^{\prime}))=0$, hence the intersection $\N_0(z)\cap\mathcal{Z}_{\N}(p^{-1}z^{\prime})$ is a regular divisor on $\mathcal{N}_0(z)$ which is isomorphic to $\textup{Spf}\,W_0$ by \cite[(5.10)]{GK93}, where $W_0$ is the integer ring of a ramified quadratic extension of $\qp$. Denote this divisor by $\mathcal{C}(y)$. Let $\widetilde{\mathcal{C}}(y)$ be the strict transform of the divisor $\mathcal{C}(y)$ under the blow up morphism $\pi_z:\widetilde{\N}_0(z)\rightarrow\N_0(z)$. Let $p_{z,1}, p_{z,2}:\N_0^{\textup{I}+}(z)\rightarrow\N_0$ be the composition of $p_z$ with the two projections from $\N\simeq\N_0\times_W\N_0$ to $\N_0$.
\begin{lemma}
    Let $y\in\B$ be an element such that it is perpendicular to $x$ and $\nu_p(q(y))=0$ or $1$. Then
    \begin{equation*}
        \N_0^{\textup{I}+}(z)\cap\mathcal{Z}(y)=\begin{cases}
            \iota_{z}^{\ast}\left(\exc_{\widetilde{\N}_0(z)}\right), &\textup{if $\nu_p(q(y))=0$;}\\
            \iota_{z}^{\ast}\left(2\cdot\exc_{\widetilde{\N}_0(z)}+\widetilde{\mathcal{C}}(y)\right), &\textup{if $\nu_p(q(y))=1$.}
        \end{cases}
    \end{equation*}
    \label{pull-back-0-strata}
\end{lemma}
\begin{proof}
Using Lemma \ref{cyclic-diff} (b), we can construct the following open cover of the formal scheme $\widetilde{\N}_0(z)$: 
\begin{equation}
    \widetilde{\N}_0(z)^{\circ}=\textup{Spf}\,W[u][[t]]\big{/}\left(\nu p+t^{2}(t^{p-1}-u)\right),\,\,\widetilde{\N}_0(z)^{\bullet}=\textup{Spf}\,W[v][[t^{\prime}]]\big{/}\left(\nu^{\prime} p+t^{\prime2}(t^{p-1}-v)\right),
    \label{exp-open-cover-N0z}
\end{equation}
where $\nu,\nu^{\prime}$ are invertible elements in the corresponding rings and $uv=1$. Using Lemma \ref{coordinate-change}, we can see easily that $\widetilde{\N}_0(z)^{\circ}\subset\M^{+}$ and $\widetilde{\N}_0(z)^{\bullet}\subset\M^{-}$. Let $y_1=x^{-1}y$. Then we have $y\in\B^{0}$. By the moduli interpretation of $\mathcal{Z}(y)$ in Lemma \ref{divisor}, we have
\begin{equation*}
    \mathcal{Z}(y)\cap\widetilde{\N}_0(z)^{\circ}=\left(p_{z,1}\right)^{\ast}\mathcal{Z}_{\N_0}(y_1),\,\,\,\,\mathcal{Z}(y)\cap\widetilde{\N}_0(z)^{\bullet}=\left(p_{z,2}\right)^{\ast}\mathcal{Z}_{\N_0}(y_1^{\prime}).
\end{equation*}
\begin{itemize}
    \item If $\nu_p(q(y))=0$, we have $\nu_p(q(y_1))=\nu_p(q(y_1^{\prime}))=0$. The equation cutting out $\mathcal{Z}_{\N_0}(y_1)$ (resp. $\mathcal{Z}_{\N_0}(y_1^{\prime})$) has the form $t+a\cdot p$ (resp. $t^{\prime}+a^{\prime}\cdot p$) for some $a\in\mathcal{O}_{\N_0}$ (resp. $a^{\prime}\in\mathcal{O}_{\N_0}$) since $\mathcal{Z}_{\N_0}(y_1)$ (resp. $\mathcal{Z}_{\N_0}(y_1^{\prime})$) is isomorphic to $\textup{Spf}\,W$ by Remark \ref{regularity-diff-smooth}. Using the open cover (\ref{exp-open-cover-N0z}), the equation of $\mathcal{Z}(y)$ on $\widetilde{\N}_0(z)^{\circ}$ (resp. $\widetilde{\N}_0(z)^{\bullet}$) is given by $t\times$(a unit element in $\mathcal{O}_{\widetilde{\N}_0(z)^{\circ}}$) (resp. $t^{\prime}\times$(a unit element in $\mathcal{O}_{\widetilde{\N}_0(z)^{\bullet}}$)). Therefore we have 
\begin{equation*}
    \N_0^{\textup{I}+}(z)\cap\mathcal{Z}(y)=\iota_{z}^{\ast}\left(\exc_{\widetilde{\N}_0(z)}\right).
\end{equation*}
\item If $\nu_p(q(y))=1$, we have $\nu_p(q(y_1))=\nu_p(q(y_1^{\prime}))=1$. The equation cutting out $\mathcal{Z}_{\N_0}(y_1)$ (resp. $\mathcal{Z}_{\N_0}(y_1^{\prime})$) has the form $p+t^{2}\cdot f(t)$ where $f(t)\in W[[t]]$ (resp. $p+t^{\prime2}g(t^{\prime})$ where $g(t^{\prime})\in W[[t^{\prime}]]$) since $\mathcal{Z}_{\N_0}(y_1)$ (resp. $\mathcal{Z}_{\N_0}(y_1^{\prime})$) is a regular formal scheme but not formally smooth over $W$ by Remark \ref{regularity-diff-smooth}. Using the open cover (\ref{exp-open-cover-N0z}), the equation of $\mathcal{Z}(y)$ on $\widetilde{\N}_0(z)^{\circ}$ (resp. $\widetilde{\N}_0(z)^{\bullet}$) is given by $t^{2}\cdot\left(\nu^{-1}(u-t^{p-1})+f(t)\right)$ (resp. $t^{\prime2}\cdot\left(\nu^{\prime-1}(v-t^{\prime p-1})+g(t^{\prime})\right)$). Therefore the multiplicity of $\exc_{z}^{\textup{I}+}$ in the divisor $\N_0^{\textup{I}+}(z)\cap\mathcal{Z}(y)$ is exactly $2$. Since every horizontal divisor in $\N_0^{\textup{I}+}(z)\cap\mathcal{Z}(y)$ has intersection number $1$ with $\exc_{z}^{\textup{I}+}$, we conclude that there is only one horizontal divisor by Lemma \ref{N0Z-EXC-SPE}. Notice that $\widetilde{\mathcal{C}}(y)\subset\N_0^{\textup{I}+}(z)\cap\mathcal{Z}(y)$ is a horizontal divisor. Therefore
\begin{equation*}
    \N_0^{\textup{I}+}(z)\cap\mathcal{Z}(y)=\iota_{z}^{\ast}\left(2\cdot\exc_{\widetilde{\N}_0(z)}+\widetilde{\mathcal{C}}(y)\right).
\end{equation*}
\end{itemize}
\end{proof}
\begin{corollary}
    Let $L\subset\B$ be a lattice of rank $3$ whose Gross-Keating invariant $\textup{GK}(L)=(0,0,1)$ or $(0,1,1)$. Then
    \begin{equation*}
        \Int^{\mathcal{Z}}(L)=\begin{cases}
            -1, &\textup{if $\textup{GK}(L)=(0,0,1)$;}\\
            0, &\textup{if $\textup{GK}(L)=(0,1,1)$.}
        \end{cases}
    \end{equation*}
    \label{base:geometric}
\end{corollary}
\begin{proof}
    Let $x,y_1,y_2$ be an orthogonal basis of the lattice $L$ where $\nu_p(q(x))=0$ and $\nu_p(q(y_2))=1$. Let $z=x_0\cdot x$. We have
    \begin{align*}
        \Int^{\mathcal{Z}}(L)=\chi\bigg{(}\M,\rO_{\mathcal{Z}(x)}\otimes^{\mathbb{L}}_{\rO_\M}\rO_{\mathcal{Z}(y_1)}\otimes^{\mathbb{L}}_{\rO_\M}&\rO_{\mathcal{Z}(y_2)}\bigg{)}=\chi\left(\M,\rO_{\N^{\textup{I}+}_0(z)}\otimes^{\mathbb{L}}_{\rO_\M}\rO_{\mathcal{Z}(y_1)}\otimes^{\mathbb{L}}_{\rO_\M}\rO_{\mathcal{Z}(y_2)}\right)\\
        &=\chi\left(\N_0^{\textup{I}+}(z),\rO_{\mathcal{Z}(y_1)\cap\N_0^{\textup{I}+}(z)}\otimes^{\mathbb{L}}_{\rO_\M}\rO_{\mathcal{Z}(y_2)\cap\N_0^{\textup{I}+}(z)}\right).
    \end{align*}
    In the following, we use $(\cdot,\cdot)_{\widetilde{\N}_0(z)}$ (resp. $(\cdot,\cdot)_{\N_0(z)}$) to denote the intersection pairing of divisors on $\widetilde{\N}_0(z)$ (resp. $\N_0(z)$).
    \begin{itemize}
        \item $\nu_p(q(y_1))=0$. By Lemma \ref{pull-back-0-strata}, we have
        \begin{align*}
            \Int^{\mathcal{Z}}(L)=\left(\exc_{\widetilde{\N}_0(z)},2\cdot\exc_{\widetilde{\N}_0(z)}+\widetilde{\mathcal{C}}(y_2)\right)_{\widetilde{\N}_0(z)}=-2+1=-1.
        \end{align*}
        \item $\nu_p(q(y_1))=1$. By the construction of $\widetilde{\mathcal{C}}(y_i)$ where $i=1,2$, we have
        \begin{equation}
            \mathcal{C}(y_i)\simeq\mathcal{Z}_{\N_0}(p^{-1}x_0\cdot y_i).
            \label{rel-C-y-i}
        \end{equation}
        By Lemma \ref{pull-back-0-strata}, we have
        \begin{equation*}
            \N_0^{\textup{I}+}(z)\cap\mathcal{Z}(y_i)=\iota_{z}^{\ast}\left(\exc_{\widetilde{\N}_0(z)}+\pi_z^{\ast}\mathcal{C}(y_i)\right).
        \end{equation*}
        By Lemma \ref{blow-up-N0(x)}, we have
        \begin{align*}
             \Int^{\mathcal{Z}}(L)&=\left(\exc_{\widetilde{\N}_0(z)}+\pi^{\ast}\mathcal{C}(y_1),\exc_{\widetilde{\N}_0(z)}+\pi^{\ast}\mathcal{C}(y_2)\right)_{\widetilde{\N}_0(z)}\\
             &=-1+0+0+\left(\pi^{\ast}\mathcal{C}(y_1),\pi^{\ast}\mathcal{C}(y_2)\right)_{\widetilde{\N}_0(z)}=-1+\left(\mathcal{C}(y_1),\mathcal{C}(y_2)\right)_{\N_0(z)}.
        \end{align*}
        Notice that $p^{-1}x_0\cdot y_1,p^{-1}x_0\cdot y_2,z$ span an orthogonal basis of a rank 3 lattice in $\B$ whose Gross-Keating invariant is $(0,0,1)$. By (\ref{rel-C-y-i}) and the fact that $\N_0(z)=\mathcal{Z}_{\N}(z)$, we have $\left(\mathcal{C}(y_1),\mathcal{C}(y_2)\right)_{\N_0(z)}=1$ by \cite[Proposition 5.4]{GK93}. Therefore $\Int^{\mathcal{Z}}(L)=0$.
    \end{itemize} 
\end{proof}

\subsection{Proof of Theorem \ref{main-theorem}}
Theorem \ref{main-theorem} states that for a $\zp$-lattice $L\subset\B$ of rank $3$, we have
\begin{equation}
    \Int^{\mathcal{Z}}(L)=\partial\den(H_0(p),L).
    \label{proof-main}
\end{equation}
and 
\begin{equation}
    \Int^{\mathcal{Y}}(L)=\partial\den(H_0(p)^{\vee},L)-1\,\,\,\textup{if $L[p]$ is integral}.
    \label{proof-main-2}
\end{equation}
\par
We first prove (\ref{proof-main}). Notice that this is automatically true if $L$ is not integral because both sides are $0$. Now we assume that the quadratic lattice $L$ is integral. Let $(a_1,a_2,a_3)$ be the Gross-Keating invariant of $L$ where $0\leq a_1\leq a_2\leq a_3$ are integers. Let $n(L)=a_1+a_2+a_3$. We prove the equality (\ref{proof-main}) by induction on $n$.
\begin{itemize}
    \item [1.] Base cases: $n(L)=1$ or $2$. The equality (\ref{proof-main}) is true by combining Lemma \ref{base:analytic} and Corollary \ref{base:geometric}.
    \item[2.] Induction step: If the equality (\ref{proof-main}) is true for all quadratic lattices $L$ such that $n(L)\leq k$ where $k\geq2$ is an integer. Let $L^{\prime}$ be a quadratic lattice such that $n(L^{\prime})=k+1$. Let $x_1,x_2,x_3$ be an orthogonal basis of $L^{\prime}$ such that $\nu_p(q(x_1))\leq\nu_p(q(x_2))\leq\nu_p(q(x_3))$. Then $\nu_p(q(x_3))\geq2$. Let $L^{\flat}=\zp\cdot x_1\obot\zp\cdot x_2$ and $L=L^{\flat}\obot\zp\cdot p^{-1}x_3$. Then we have $n(L)=k-1<k$. The equality (\ref{proof-main}) holds by combining $\Int^{\mathcal{Z}}(L)=\partial\den(H_0(p),L)$ and Corollary \ref{diff-com}.   
\end{itemize}
Therefore we can conclude that the equality (\ref{proof-main}) is true for all lattices $L\subset\B$ of rank $3$. The proof of (\ref{proof-main-2}) is similar so we omit it. 
$\hfill\qed$

\part{Applications}
\label{app-part}
\section{A conjecture of Kudla--Rapoport}
\label{rapoport-conj}
\subsection{CM cycles on $\N_0(x_0)$}
Roughly speaking, the formal scheme $\N_0(x_0)$ parameterizes cyclic isogenies deforming $x_0$. Since $\nu_p(q(x_0))=1$, we have $\N_0(x_0)=\mathcal{Z}_\N(x_0)$ by Lemma \ref{cyclic-diff}. Let
\begin{equation*}
    \left(X^{\univ}\xrightarrow{\pi} X^{\prime\univ},(\rho^{\univ},\rho^{\prime\univ})\right)
\end{equation*}
be the universal deformation of the isogeny $x_0$ over $\N_0(x_0)$. Let $\B^{0}$ be the subset of $\B$ consisting of trace $0$ elements. Let $S$ be the subset of $H_0(p)$ consisting of trace $0$ elements.
\begin{definition}
    Let $H\subset\B^0$ be a subset. Define the special cycle $\mathcal{Z}_{\N_0(x_0)}^{\textup{CM}}(H)\subset\N_0(x_0)$ to be the closed formal subscheme cut out by the conditions,
\begin{equation*}
    \rho^{\textup{univ}}\circ x \circ (\rho^{\textup{univ}})^{-1}\in \textup{End}(X^{\textup{univ}});
\end{equation*}
\begin{equation*}
    \rho^{\prime\textup{univ}}\circ x^{\prime} \circ (\rho^{\prime\textup{univ}})^{-1} \in \textup{End}(X^{\prime\textup{univ}}).
\end{equation*}
for all $x\in H$.
\end{definition}
For a given element $x\in\mathbb{B}^{0}$, the special cycle $\mathcal{Z}^{\textup{CM}}_{\N_0(x_0)}(x)$ is not a divisor and has embedded components. Let $\mathcal{Z}^{\textup{CM}}_{\N_0(x_0)}(x)^{\flat}$ be the associated divisor of $\mathcal{Z}^{\textup{CM}}_{\N_0(x_0)}(x)$ following the notion of Kudla and Rapoport \cite[$\S$4]{KRshimuracurve}. In the current situation, let $f_x,f_{x^{\prime}}\in\mathcal{O}_{\N_0(x_0)}$ be the two elements cutting out the closed formal subschemes over which $x$ and $x^{\prime}$ are isogenies respectively. Let $g_x$ be the greatest common divisor of $f_x$ and $f_{x^{\prime}}$, then the associated divisor of $\mathcal{Z}^{\textup{CM}}_{\N_0(x_0)}(x)$ is defined by the elemeng $g_x$.
\par
In 2006, Michael Rapoport gave on talk titled ``Some remarks on special cycles on Shimura curves". In that talk, he explained his conjecture with Steven Kudla:
\begin{conjecture}
    Let $x,y\in\B^0$ be two linearly independent vectors. Let $M$ be the $\zp$-lattice spanned by $x$ and $y$. Then
    \begin{equation*}
        \chi\left(\N_0(x_0),\mathcal{O}_{\mathcal{Z}^{\textup{CM}}_{\N_0(x_0)}(x)^{\flat}}\otimes^{\mathbb{L}}_{\mathcal{O}_{\N_0(x_0)}}\mathcal{Z}^{\textup{CM}}_{\N_0(x_0)}(y)^{\flat}\right)=\partial\den(S,M)+1.
    \end{equation*}\label{conj-of-KR}
\end{conjecture}
In the following paragraphs, we will reformulate and then confirm this conjecture.

\subsection{Geometric cancellation law}
\label{geo-cancel}
Let $1\in\B$ be the identity element. By the definition of the special cycles on $\N(x_0)$, there exists an isomorphism $\iota_0:\N_0(x_0)\simeq\mathcal{Z}_{\N(x_0)}(1)$. For a subset $H\subset\B^0$, let $H^{\prime}=\{1\}\cup H$. We have
\begin{equation*}
    \mathcal{Z}_{\N_0(x_0)}^{\textup{CM}}(H)=(\iota_0)^{\ast}(\mathcal{Z}_{\N(x_0)}(H^{\prime}))
\end{equation*}
Recall that we denote by $\widetilde{\N}_0(x_0)$ the blow up of the formal scheme $\N_0(x_0)$ along its unique closed $\F$-point. The isomorphism $\iota_0$ induces a closed immersion $\widetilde{\N}_0(x_0)\rightarrow\mathcal{Z}(1)\subset\M$. By Lemma \ref{dec-diff-div}, this closed immersion induces an isomorphism $\widetilde{\N}_0(x_0)\xrightarrow{\sim}\N_0^{\textup{I}+}(1)=\widetilde{\mathcal{D}}(1)=\widetilde{\mathcal{Z}}(1)$. We still denote this isomorphism by $\iota_0$.
\begin{definition}
    Let $H\subset\B^0$ be a subset. Define the special cycle $\mathcal{Z}^{\textup{CM}}(H)\subset\widetilde{\N}_0(x_0)$ to be 
    \begin{equation*}
        \mathcal{Z}^{\textup{CM}}(H)\coloneqq\mathcal{Z}_{\N_0(x_0)}^{\textup{CM}}(H)\times_{\N_0(x_0)}\widetilde{\N}_0(x_0).
    \end{equation*}
    For a single element $x\in\B^0$, denote by $\mathcal{Z}^{\textup{CM}}(x)$ the cycle $\mathcal{Z}^{\textup{CM}}(\{x\})$.
    \label{special-cycle-CM}
\end{definition}
Recall that we use $\left(X_{1,\M}\stackrel{(x_0)_{1,\M}}\longrightarrow X_{1,\M}^{\prime},\left(\rho_{1,\M},\rho_{1,\M}^{\prime}\right)\right),\,\,\,\,\left(X_{2,\M}\stackrel{(x_0)_{2,\M}}\longrightarrow X_{2,\M}^{\prime},\left(\rho_{2,\M},\rho_{2,\M}^{\prime}\right)\right)$ to denote the universal object over $\M$. Over the formal scheme $\widetilde{\N}_0(x_0)\simeq\N_0^{\textup{I}+}(1)\subset\M$, we have the following commutative diagram
\begin{equation}
        \begin{tikzcd}
    {X_{1,\M}}
    \arrow[d,"{x_{0,1}^{\univ}}"swap] \arrow[r, "{1}","{\sim}"']
    & {X_{2,\M}}
    \arrow[d,"{x_{0,2}^{\univ}}"]
    \\
    {X_{1,\M}^{\prime}}
    \arrow[r, "{1}","{\sim}"']
    &{X_{2,\M}^{\prime}},
\end{tikzcd}\label{over-Z1}
\end{equation}
We can identify $X_{1,\M}$ with $X_{2,\M}$ (resp.  $X_{1,\M}^{\prime}$ with $X_{2,\M}^{\prime}$) using the diagram (\ref{over-Z1}). Comparing the definition of $\mathcal{Z}(H)$ in Definition \ref{special-cycle-on-M} and the definition of $\mathcal{Z}^{\textup{CM}}(H)$ in Definition \ref{special-cycle-CM}, we get
\begin{lemma}
    Let $H\subset\B^0$ be a subset. Then
    \begin{equation*}
        \mathcal{Z}^{\textup{CM}}(H)=(\iota_0)^{\ast}\left(\N_0^{\textup{I}+}(x_0)\cap\mathcal{Z}(H)\right).
    \end{equation*} \label{id-CM-Z}
\end{lemma}
\begin{corollary}
    Let $x\in\B^0$. The CM cycle $\mathcal{Z}^{\textup{CM}}(x)$ is an effective Cartier divisor on the formal scheme $\widetilde{\N}_0(x_0)$.
\end{corollary}
\begin{proof}
    By Lemma \ref{id-CM-Z}, we know that $\mathcal{Z}^{\textup{CM}}(x)=\left(\mathcal{Z}(x)\cap\N_0^{\textup{I}+}(1)\right)\times_{\N_0^{\textup{I}+}(1),\iota_0}\widetilde{\N}_0(x_0)$. Lemma \ref{divisor} shows that $\mathcal{Z}(x)$ is a divisor on $\M$, hence $\mathcal{Z}(x)\cap\N_0^{\textup{I}+}(1)$ is a divisor on $\N_0^{\textup{I}+}(1)$. Therefore $\mathcal{Z}^{\textup{CM}}(x)$ is an effective Cartier divisor on the formal scheme $\widetilde{\N}_0(x_0)$.
\end{proof}
\begin{lemma}
    Let $x\in\B^0$. Let $\pi_{x_0}:\widetilde{\N}_0(x_0)\rightarrow\N_0(x_0)$ be the blow up morphism with exceptional divisor $\exc_{x_0}$. Then
    \begin{equation*}
        \mathcal{Z}^{\textup{CM}}(x)=\pi_{x_0}^{\ast}\left(\mathcal{Z}^{\textup{CM}}_{\N_0(x_0)}(x)^{\flat}\right)+\exc_{x_0}.
    \end{equation*}\label{rel-cm-flat}
\end{lemma}
\begin{proof}
    Let $f_x,f_{x^{\prime}}\in\mathcal{O}_{\N_0(x_0)}$ be the two elements cutting out the closed formal subschemes over which $x$ and $x^{\prime}$ are isogenies respectively. Let $g_x$ be the greatest common divisor of $f_x$ and $f_{x^{\prime}}$. Then $(f_x,f_{x^{\prime}})=(g_x)\cdot\mathfrak{m}_x$ where $\mathfrak{m}_x\subset\mathcal{O}_{\N_0(x_0)}$ is a primary ideal whose radial is the maximal ideal by the proof of \cite[Lemma 4.2]{KRshimuracurve}. Therefore $\pi_{x_0}^{\ast}\mathfrak{m}_{x}$ is a multiple of $\exc_{x_0}$, i.e.,
    \begin{equation}
        \mathcal{Z}^{\textup{CM}}(x)=\pi_{x_0}^{\ast}\left(\mathcal{Z}^{\textup{CM}}_{\N_0(x_0)}(x)^{\flat}\right)+m_x\cdot\exc_{x_0}.\label{mx}
    \end{equation}
    By Lemma \ref{N0Z-EXC-SPE}, we have
    \begin{equation*}
        \chi\left(\widetilde{\N}_0(x_0),\mathcal{O}_{\mathcal{Z}^{\textup{CM}}(x)}\otimes^{\mathbb{L}}_{\mathcal{O}_{\widetilde{\N}_0(x_0)}}\mathcal{O}_{\exc_{x}}\right)=-1.
    \end{equation*}
    Notice that $\chi\left(\widetilde{\N}_0(x_0),\mathcal{O}_{\pi_{x_0}^{\ast}\left(\mathcal{Z}^{\textup{CM}}_{\N_0(x_0)}(x)^{\flat}\right)}\otimes^{\mathbb{L}}_{\mathcal{O}_{\widetilde{\N}_0(x_0)}}\mathcal{O}_{\exc_{x}}\right)=0$, hence by (\ref{mx}), we get $-m_x=-1$, therefore $m_x=1$.
\end{proof}
By the above lemma and Lemma \ref{blow-up-N0(x)} (d), we obtain the following:
\begin{corollary}
    Let $x,y\in\B^0$ be two linearly independent vectors. Then
    \begin{equation*}
        \chi\left(\N_0(x_0),\mathcal{O}_{\mathcal{Z}^{\textup{CM}}_{\N_0(x_0)}(x)^{\flat}}\otimes^{\mathbb{L}}_{\mathcal{O}_{\N_0(x_0)}}\mathcal{Z}^{\textup{CM}}_{\N_0(x_0)}(y)^{\flat}\right)=\chi\left(\widetilde{\N}_0(x_0),\rO_{\mathcal{Z}^{\textup{CM}}(x)}\otimes^{\mathbb{L}}_{\rO_{\widetilde{\N}_0(x_0)}}\rO_{\mathcal{Z}^{\textup{CM}}(y)}\right)+1.
    \end{equation*}\label{rel-int-num}
\end{corollary}

\subsection{Analytic cancellation law}
\label{ana-cancel}
\begin{lemma}
    Let $M\subset\B^0$ be a $\zp$-lattice of rank $2$. Then
    \begin{equation*}
        \partial\den(S,M)=\partial\den(H_0(p),\zp\cdot1\obot M).
    \end{equation*}
    \label{analytic-cancel}
\end{lemma}
\begin{proof}
    In Lemma \ref{der-diff}, we take $L^{\flat}=M$ and $x=1$. Then $\nu_p(q(x))=0$. The formula in this lemma corresponds to the $``n=0"$ case of Lemma \ref{der-diff}.
\end{proof}

\subsection{Derived CM cycles ${^{\mathbb{L}}\mathcal{Z}^{\textup{CM}}(M)}$}
\begin{lemma}
    Let $M\subset\B^0$ be a $\zp$-lattice of rank $2$. Let $\{x,y\}$ be a basis of $M$, then the element $[\rO_{\mathcal{Z}^{\textup{CM}}(x)}\otimes^{\mathbb{L}}_{\rO_{\widetilde{\N}_0(x_0)}}\rO_{\mathcal{Z}^{\textup{CM}}(y)}]\in\textup{Gr}^{2}K_0^{\mathcal{Z}^{\textup{CM}}(M)}(\widetilde{\N}_0(x_0))$ only depends on the $\zp$-lattice $M$.
    \label{linear-invariance-CM}
\end{lemma}
\begin{proof}
    Let $L=M\obot\langle x_0\rangle$. Notice that
    \begin{align*}
        {^{\mathbb{L}}\mathcal{Z}}(L)&=[\rO_{\mathcal{Z}(x_0)}\otimes^{\mathbb{L}}_{\rO_{\M}}\rO_{\mathcal{Z}(x_1)}\otimes^{\mathbb{L}}_{\rO_{\M}}\rO_{\mathcal{Z}(x_2)}]\\
        &=[\rO_{\exc_{\M}}\otimes^{\mathbb{L}}_{\rO_{\M}}\rO_{\mathcal{Z}(x_1)}\otimes^{\mathbb{L}}_{\rO_{\M}}\rO_{\mathcal{Z}(x_2)}]+[\rO_{\N_0^{\textup{I}+}(x_0)}\otimes^{\mathbb{L}}_{\rO_{\M}}\rO_{\mathcal{Z}(x)}\otimes^{\mathbb{L}}_{\rO_{\M}}\rO_{\mathcal{Z}(y)}]\\
        &=[\mathcal{O}(0,-1)\otimes^{\mathbb{L}}_{\rO_{\exc_{\M}}}\mathcal{O}(0,-1)]+[\rO_{\N_0^{\textup{I}+}(x_0)\cap\mathcal{Z}(x)}\otimes^{\mathbb{L}}_{\rO_{\N_0^{\textup{I}+}(x_0)}}\rO_{\N_0^{\textup{I}+}(x_0)\cap\mathcal{Z}(y)}].\\
    \end{align*}
    \par
    Since ${^{\mathbb{L}}\mathcal{Z}}(L)$ is independent of the basis of $L$ by Lemma \ref{linear-invariance-2}, the element $[\rO_{\N_0^{\textup{I}+}(x_0)\cap\mathcal{Z}(x)}\otimes^{\mathbb{L}}_{\rO_{\N_0^{\textup{I}+}(x_0)}}\rO_{\N_0^{\textup{I}+}(x_0)\cap\mathcal{Z}(y)}]\in\textup{Gr}^{2}K_0^{\mathcal{Z}(L)}(\N^{\textup{I}+}_0(x_0))$ is independent of basis $x,y$ of $M$. Notice that we have an isomorphism $\iota_0:\widetilde{\N}_0(x_0)\xrightarrow{\sim}\N_0^{\textup{I}+}(x_0)$. By Lemma \ref{id-CM-Z}, we have
    \begin{equation*}
        [\rO_{\mathcal{Z}^{\textup{CM}}(x)}\otimes^{\mathbb{L}}_{\rO_{\widetilde{\N}_0(x_0)}}\rO_{\mathcal{Z}^{\textup{CM}}(y)}]=\left(\iota_0\right)^{\ast}[\rO_{\N_0^{\textup{I}+}(x_0)\cap\mathcal{Z}(x)}\otimes^{\mathbb{L}}_{\rO_{\N_0^{\textup{I}+}(x_0)}}\rO_{\N_0^{\textup{I}+}(x_0)\cap\mathcal{Z}(y)}]
    \end{equation*}
    Therefore $[\rO_{\mathcal{Z}^{\textup{CM}}(x)}\otimes^{\mathbb{L}}_{\rO_{\widetilde{\N}_0(x_0)}}\rO_{\mathcal{Z}^{\textup{CM}}(y)}]\in\textup{Gr}^{2}K_0^{\mathcal{Z}^{\textup{CM}}(M)}(\widetilde{\N}_0(x_0))$ only depends on the $\zp$-lattice $M$.
\end{proof}
Based on Lemma \ref{linear-invariance-CM}, the following definition is reasonable.
\begin{definition}
    Let $M\subset\B^0$ be a $\zp$-lattice of rank $2$. Let $\{x,y\}$ be a basis of $M$. Define the derived CM cycle to be
    \begin{equation*}
        {^{\mathbb{L}}\mathcal{Z}^{\textup{CM}}(M)}=[\rO_{\mathcal{Z}^{\textup{CM}}(x)}\otimes^{\mathbb{L}}_{\rO_{\M}}\rO_{\mathcal{Z}^{\textup{CM}}(y)}]\in\textup{Gr}^{2}K_0^{\mathcal{Z}^{\textup{CM}}(M)}(\widetilde{\N}_0(x_0)).
    \end{equation*}
\end{definition}

\subsection{Intersection of CM cycles}
\begin{definition}
    Let $M\subset\B$ be a $\zp$-lattice of rank 2. Define the arithmetic intersection numbers
    \begin{equation*}
        \Int^{\textup{CM}}(M)\coloneqq\chi\left(\widetilde{\N}_0(x_0),{^{\mathbb{L}}\mathcal{Z}^{\textup{CM}}}(M)\right).
    \end{equation*}
    Here $\chi$ denotes the Euler-Poincare characteristic.
\end{definition}
\begin{theorem}
    Let $M\subset\B$ be a $\zp$-lattice of rank 2. Then
    \begin{equation*}
        \Int^{\textup{CM}}(M)=\partial\den(S,M).
    \end{equation*}
    \label{int-CM}
\end{theorem}
\begin{proof}
    Let $L=M\obot\zp\cdot1$. By the proof of Lemma \ref{linear-invariance-CM}, we have
    \begin{equation*}
        {^{\mathbb{L}}\mathcal{Z}}(L)=[\mathcal{O}(0,-1)\otimes^{\mathbb{L}}_{\rO_{\exc_{\M}}}\mathcal{O}(0,-1)]+\left(\iota_0\right)_{\ast}{^{\mathbb{L}}\mathcal{Z}^{\textup{CM}}}(M).
    \end{equation*}
    Therefore
    \begin{align*}
        \Int^{\textup{CM}}(M)=\Int^{\mathcal{Z}}(L)=\partial\den(H_0(p),L)\overset{\textup{Lemma}\,\ref{analytic-cancel}}{=}\partial\den(S,M).
    \end{align*}
\end{proof}
\begin{remark}
    Now the original Conjecture \ref{conj-of-KR} can be proved by combining Theorem \ref{int-CM} and Corollary \ref{rel-int-num}.
\end{remark}

\section{Arithmetic intersection of Hecke correspondences}
\label{ari-int-hecke-app}
\subsection{Modular curves}
\label{modu-curve}
Let $\mathbb{A}_f$ be the ring of finite ad$\grave{\textup{e}}$les over $\mathbb{Q}$. Let $U\subset\textup{GL}_2(\mathbb{A}_f)$ be a sufficiently small compact open subgroup, the quotient space $\textup{GL}_2(\mathbb{Q})\backslash\mathbb{H}\times\textup{GL}_2(\mathbb{A}_f)/U$ admits a canonical model, i.e., there exists a finite extension $\mathbb{Q}_U$ of $\mathbb{Q}$ and a variety $Y_U$ over $\mathbb{Q}_U$ such that
\begin{equation*}
    Y_U(\mathbb{C})\coloneqq Y_U\times_{\mathbb{Q}}\mathbb{C}=\textup{GL}_2(\mathbb{Q})\backslash\mathbb{H}\times\textup{GL}_2(\mathbb{A}_f)/U\simeq\textup{GL}_2(\mathbb{Q})^{+}\backslash\mathbb{H}^{+}\times\textup{GL}_2(\mathbb{A}_f)/U,
\end{equation*}
where $\textup{GL}_2(\mathbb{Q})^{+}$ consists of elements with positive determinant and $\mathbb{H}^{+}$ is the upper half plane. The curve $Y_U$ has a canonical compactification $X_U$ which is a proper smooth curve over $\mathbb{Q}_U$. The curve $X_U$ has a Hodge class $\mathscr{L}_U\in\textup{Pic}(X_U)_{\mathbb{Q}}$.

\subsection{Cyclic isogeny}
Let $\pi:E\rightarrow E^{\prime}$ is an isogeny of two elliptic curves over a base scheme $S$. For a prime $p\vert\textup{deg}(\pi)$, define the $p$-part $\pi_p:E\rightarrow E_1$ of $\pi$ to be an isogeny which factors $\pi$ in the way $\pi=\pi^{p}\circ\pi_p$ where $\pi^{p}:E_1\rightarrow E^{\prime}$ is another isogeny such that $p\nmid\textup{deg}(\pi^{p})$.
\begin{definition}
    An isogeny $\pi:E\rightarrow E^{\prime}$ between two elliptic curves over $S$ is cyclic if $\textup{ker}(\pi)$ is a cyclic group scheme over $S$ in the sense of \cite[$\S$6.1]{KM85}. For a prime number $p\vert\textup{deg}(\pi)$, we say $\pi$ is $p$-cyclic if the $p$-part $\pi_p:E\rightarrow E_1$ of $\pi$ is cyclic.
    \label{def-cyclic}
\end{definition}
\par
For a cyclic isogeny $\pi:E\rightarrow E^{\prime}$ of degree $p^{n}$ for some integer $n\geq0$ between elliptic curves. There is a standard decomposition of $\pi$ into the composition of $n$ degree $p$ isogenies:
\begin{equation*}
    \pi:E=E_0\xrightarrow{\pi_1}E_1\xrightarrow{\pi_2}\cdots\xrightarrow{\pi_n}E_n\simeq E^{\prime}.
\end{equation*}
We refer to \cite[$\S$6.7]{KM85} for the details on the notions of standard decomposition.

\subsection{Global Hecke correspondences}
Let $U$ be the same compact open subgroup in $\textup{GL}_2(\mathbb{A}_f)$ as $\S$\ref{modu-curve}. For a double coset $UxU$ of $U\backslash\textup{GL}_2(\mathbb{A}_f)/U$ where $x\in\textup{GL}_2(\mathbb{A}_f)$, we have a Hecke correspondence defined as the image of the morphism:
\begin{equation*}
    (\pi_{U\cap xUx^{-1},U},\pi_{U\cap x^{-1}U,U}\circ T_x): \mathsf{T}(x)_U\coloneqq X_{U\cap xUx^{-1}}(\mathbb{C})\rightarrow X_{U}(\mathbb{C})^{2}.
\end{equation*}
Notice that the geometrically connected components $\pi_0(X_{U})$ of $X_U$ is a principal homogeneous space over $\mathbb{Q}_{+}^{\times}\backslash\mathbb{A}_f^{\times}/\textup{det}(U)$. For an element $a\in\pi_0(X_{U})$, define $X_{U,a}$ to be the connected component of $X_U(\mathbb{C})$ corresponding to $a$. Define $M_{U,\alpha}=\bigsqcup\limits_{a\in\pi_0(X_{U})} X_{U,a}\times X_{U,\alpha a}$. Then
\begin{equation*}
    X_{U}(\mathbb{C})^{2}=\bigsqcup\limits_{\alpha\in\mathbb{Q}_{+}^{\times}\backslash\mathbb{A}_f^{\times}/\textup{det}(U)}M_{U,\alpha}.
\end{equation*}
The Hecke correspondence $\mathsf{T}(x)_U$ is supported on the component $M_{U,\textup{det}(x)^{-1}}$.

\subsection{Integral models of modular curves}
Let $p$ be a finite prime. Let $n$ be a positive integer. Let $U^{p}\subset\textup{GL}_2(\mathbb{A}_f^{p})$ is a sufficiently small compact open subgroup. Let $\Gamma_0(p^{n})$ be the following subgroup of $\textup{GL}_2(\zp)$:
\begin{equation*}
    \Gamma_0(p^{n})=\left\{g=\begin{pmatrix}
        a & b\\
        p^{n}c & d
    \end{pmatrix}:\,a,b,c,d\in\zp,\,\,\textup{det}(g)\in\zp^{\times}.\right\}
\end{equation*}
\par
Let $U=\Gamma_0(p^{n})\cdot U^{p}$, we are going to construct an integral model $\mathcal{Y}_U$ for the modular curve $Y_{U}$. The model $\mathcal{Y}_{U^{p}}(p^{n})$ is a Deligne-Mumford stack over $\mathbb{Z}_{(p)}$. Let $S$ be a connected $\mathbb{Z}_{(p)}$-scheme. Let $\overline{s}$ be a geometric point of $S$. 
The groupoid $\mathcal{Y}_{U^{p}}(p^{n})(S)$ consists of objects $(E\stackrel{\pi}\rightarrow E^{\prime},\overline{\eta^{p}})$, where $E\stackrel{\pi}\rightarrow E^{\prime}$ is a cyclic isogeny of degree $p^{n}$ and  $\overline{\eta^{p}}$ is a $\pi_{1}(S,\overline{s})$-invariant $U^p$-equivalence class of isomorphisms
\begin{equation*}
    \eta^{p}:\left(\mathbb{A}_{f}^{p}\right)^{2}\stackrel{\simeq}\longrightarrow  V^{p}(E_{\overline{s}}).
\end{equation*}
A morphism from $(E_{1}\stackrel{\pi_{1}}\rightarrow E_{1}^{\prime},\overline{\eta_{1}^{p}})$ to $(E_{2}\stackrel{\pi_{2}}\rightarrow E_{2}^{\prime},\overline{\eta_{2}^{p}})$ is a pair $(f,f^{\prime})$ of isomorphisms $f:E_{1}\rightarrow E_{2}$ and $f^{\prime}:E_{1}^{\prime}\rightarrow E_{2}^{\prime}$ such that $f^{\prime}\circ \pi_{1}=\pi_{2}\circ f$ and $\overline{\eta_{1}^{p}} =\overline{V^{p}(f)\circ\eta_{2}^{p}}$ as $U^p$-orbits. 
\par
For a point $\tau\in\mathbb{H}^{+}$, denote by $E_{\tau}$ the elliptic curve $\mathbb{C}\big{/}\mathbb{Z}+\mathbb{Z}\tau$. For the elliptic curve $E_{\tau}$, we have $E[m]=\left(\frac{1}{m}\mathbb{Z}+\frac{1}{m}\mathbb{Z}\tau\right)\big{/}\mathbb{Z}+\mathbb{Z}\tau\simeq\left(\mathbb{Z}/m\mathbb{Z}\right)^{2}$, there is naturally defined $a_{\tau}^{p}:V^{p}(E_{\tau})\rightarrow \left(\mathbb{A}_f^{p}\right)^{2}$ by choosing the isomorphism $E[m]\xrightarrow{\sim}\left(\mathbb{Z}/m\mathbb{Z}\right)^{2}$ compatibly using the basis $1,\tau$. There is a bijection between the set associated to the groupoid $\mathcal{Y}_U(\mathbb{C})$ and $Y_U(\mathbb{C})$ given an follows: Let $[\tau,g]\in Y_U(\mathbb{C})\simeq\textup{GL}_2(\mathbb{Q})^{+}\backslash\mathbb{H}^{+}\times\textup{GL}_2(\mathbb{A}_f)/U$ be a point. It is mapped to the object 
\begin{equation}
    \left(E_{\tau}\xrightarrow{\times p^{n}} E_{p^{n}\tau},\overline{\left(a_\tau^{p}\right)^{-1}\circ g}\right)\in\mathcal{Y}_U(\mathbb{C}).
    \label{map1}
\end{equation}
Given an object $(E\stackrel{\pi}\rightarrow E^{\prime},\overline{\eta^{p}})\in\mathcal{Y}_U(\mathbb{C})$, there exists an element $\tau\in\mathbb{H}^{+}$ such that there exists two isomorphisms $f:E\rightarrow E_{\tau}$ and $f^{\prime}:E^{\prime}\rightarrow E_{p^{n}\tau}$ such that $f^{\prime}\circ\pi\circ f^{-1}:E_{\tau}\rightarrow E_{p^{n}\tau}$ is given by multiplication by $p^{n}$. This object is mapped to the following element in $Y_U(\mathbb{C})$:
\begin{equation}
    [\tau,a_{\tau}^{p}\circ V^{p}(f)\circ\eta^{p}]\in\textup{GL}_2(\mathbb{Q})^{+}\backslash\mathbb{H}^{+}\times\textup{GL}_2(\mathbb{A}_f)/U.
    \label{map2}
\end{equation}
The above two maps are inverse to each other.
\par
The stacks $\mathcal{Y}_{U}$ is a 2-dimensional Deligne-Mumford stack. There exists another 2-dimensional Deligne-Mumford stack $\mathcal{X}_{U}$ which can be viewed as a compactification of $\mathcal{Y}_{U}$ and $\mathcal{Y}_{U}$. The stack $\mathcal{Y}_{U}$ is an open substack of $\mathcal{X}_{U}$ and $\mathcal{X}_{U}(\mathbb{C})\simeq X_{U}(\mathbb{C})$.

\subsection{Moduli interpretations of Hecke correspondences}
\begin{lemma}
    The following elements form a set of coset representative for $\Gamma_0(p)\backslash\textup{GL}_2(\mathbb{Q}_p)/\Gamma_0(p)$:
    \begin{equation*}
        \begin{pmatrix}
            p^{a} & 0\\
            0 & p^{b}
        \end{pmatrix},\,\,\,\,\begin{pmatrix}
            0 & p^{a}\\
            p^{b} & 0
        \end{pmatrix},
    \end{equation*}
    where $a,b\in\mathbb{Z}$.
    \label{double-coset}
\end{lemma}
\begin{proof}
We refer to the proof of Bunshnell and Henniart \cite[$\S$17.1]{LLCforGL(2)}.
\end{proof}
\begin{definition}
    We divide the group $\textup{GL}_2(\mathbb{Q}_p)$ into the following four subsets
    \begin{align*}
        &\mathbb{V}_p^{\textup{I}+}=\bigcup\limits_{a\leq b}\Gamma_0(p)\begin{pmatrix}
            p^{a} & 0\\
            0 & p^{b}
        \end{pmatrix}\Gamma_0(p),\,\,\,\,\,\,\,\mathbb{V}_p^{\textup{I}-}=\bigcup\limits_{b< a}\Gamma_0(p)\begin{pmatrix}
            p^{a} & 0\\
            0 & p^{b}
        \end{pmatrix}\Gamma_0(p),\\
        &\mathbb{V}_p^{\textup{II}+}=\bigcup\limits_{a< b}\Gamma_0(p)\begin{pmatrix}
            0 & p^{a}\\
            p^{b} & 0
        \end{pmatrix}\Gamma_0(p),\,\,\,\,\,\,\,\mathbb{V}_p^{\textup{II}-}=\bigcup\limits_{b\leq a}\Gamma_0(p)\begin{pmatrix}
            0 & p^{a}\\
            p^{b} & 0
        \end{pmatrix}\Gamma_0(p).
    \end{align*}
\end{definition}
\begin{lemma}
    Let $x\in\textup{M}_2(\mathbb{A}_f)$ such that $\textup{det}(x)\in\mathbb{Q}_{>0}^{\times}$ and $x_p$ is a primitive element in the lattice $H_0(p)$. Let $x^{\ast}\in\textup{M}_2(\mathbb{A}_f)$ be the contragradient of $x$. Let $n=\nu_p(\textup{det}(x_p))\geq0$. Let $U=\Gamma_0(p)\cdot U^{p}$ be a compact open subgroup of $\textup{GL}_2(\mathbb{A}_f)$. Then
    \begin{itemize}
        \item [$\textup{I}+$:] If $x_p\in\mathbb{V}_p^{\textup{I}+}$, there exists an isomorphism $\mathsf{Z}(x)_U\simeq \mathcal{X}_{\Gamma_0(p^{n+1})\cdot U^p\cap xU^{p}x^{-1}}(\mathbb{C})$ under which the morphism $\mathsf{Z}(x)_U\rightarrow X_U^{2}\simeq\mathcal{X}_{U}^{2}(\mathbb{C})$ on the open subvariety $\mathcal{Y}_{\Gamma_0(p^{n+1})\cdot U^p\cap xU^{p}x^{-1}}(\mathbb{C})$ is given by
        \begin{equation*}
            \iota_{x}^{\textup{I}+}:\left(E\xrightarrow{\pi} E^{\prime},\overline{\eta^{p}}\right)\mapsto\left(\left(E\xrightarrow{\pi_1} E_1,\overline{\eta^{p}}\right),\left(E_{n}\xrightarrow{\pi_{n+1}} E^{\prime},\overline{V^{p}(\pi_{n}\circ\cdots\circ\pi_1)\circ\eta^{p}\circ \left(x^{\ast}\right)^{-1}}\right)\right).
        \end{equation*}
        \item [$\textup{I}-$:] If $x_p\in\mathbb{V}_p^{\textup{I}-}$, there exists an isomorphism $\mathsf{Z}(x)_U\simeq \mathcal{X}_{\Gamma_0(p^{n+1})\cdot U^p\cap x^{-1}U^{p}x}(\mathbb{C})$ under which the morphism $\mathsf{Z}(x)_U\rightarrow X_U^{2}\simeq\mathcal{X}_{U}^{2}(\mathbb{C})$ on the open subvariety $\mathcal{Y}_{\Gamma_0(p^{n+1})\cdot U^p\cap x^{-1}U^{p}x}(\mathbb{C})$ is given by
        \begin{equation*}
            \iota_{x}^{\textup{I}-}:\left(E\xrightarrow{\pi} E^{\prime},\overline{\eta^{p}}\right)\mapsto\left(\left(E_n\xrightarrow{\pi_{n+1}} E^{\prime},\overline{V^{p}(\pi_{n}\circ\cdots\circ\pi_1)\circ\eta^{p}\circ x^{-1}}\right),\left(E\xrightarrow{\pi_{1}} E_1,\overline{\eta^{p}}\right)\right).
        \end{equation*}
        \item [$\textup{II}+$:] If $x_p\in\mathbb{V}_p^{\textup{II}+}$, there exists an isomorphism $\mathsf{Z}(x)_U\simeq \mathcal{X}_{\Gamma_0(p^{n})\cdot U^p\cap xU^{p}x^{-1}}(\mathbb{C})$ under which the morphism $\mathsf{Z}(x)_U\rightarrow X_U^{2}\simeq\mathcal{X}_{U}^{2}(\mathbb{C})$ on the open subvariety $\mathcal{Y}_{\Gamma_0(p^{n})\cdot U^p\cap xU^{p}x^{-1}}(\mathbb{C})$ is given by
        \begin{equation*}
            \iota_{x}^{\textup{II}+}:\left(E\xrightarrow{\pi} E^{\prime},\overline{\eta^{p}}\right)\mapsto\left(\left(E\xrightarrow{\pi_{1}} E_1,\overline{\eta^{p}}\right),\left(E^{\prime}\xrightarrow{\pi_{n}^{\vee}} E_{n-1},\overline{V^{p}(\pi)\circ\eta^{p}\circ \left(x^{\ast}\right)^{-1}}\right)\right).
        \end{equation*}
        \item [$\textup{II}-$:] If $x_p\in\mathbb{V}_p^{\textup{II}-}$, there exists an isomorphism $\mathsf{Z}(x)_U\simeq \mathcal{X}_{\Gamma_0(p^{n})\cdot U^p\cap x^{-1}U^{p}x}(\mathbb{C})$ under which the morphism $\mathsf{Z}(x)_U\rightarrow X_U^{2}\simeq\mathcal{X}_{U}^{2}(\mathbb{C})$ on the open subvariety $\mathcal{Y}_{\Gamma_0(p^{n})\cdot U^p\cap x^{-1}U^{p}x}(\mathbb{C})$ is given by
        \begin{equation*}
            \iota_{x}^{\textup{II}-}:\left(E\xrightarrow{\pi} E^{\prime},\overline{\eta^{p}}\right)\mapsto\left(\left(E_1\xrightarrow{\pi_1^{\vee}} E,\overline{V^{p}(\pi_1)\circ \eta^{p}\circ x^{-1}}\right),\left(E_{n-1}\xrightarrow{\pi_{n}} E^{\prime},\overline{V^{p}(\pi_{n-1}\circ\cdots\circ\pi_1)\circ\eta^{p}}\right)\right).
        \end{equation*}
    \end{itemize}
    \label{moduli-hecke}
\end{lemma}
\begin{proof}
    We only give the proof for the case $\textup{I}+$. For simplicity, denote by $x_n$ the following element 
    \begin{equation*}
        x_n=\begin{pmatrix}
            p^{n} & 0\\
            0 & 1
        \end{pmatrix}\in\textup{GL}_2(\mathbb{Q})^{+}.
    \end{equation*}
    We have the following diagram for the correspondence $\mathsf{T}(x)_U:$
    \begin{equation*}
        \begin{tikzcd}
&  & \mathsf{T}(x)_U \ar[ddl, "\pi"'] \ar[ddr, "\pi\circ T(x)"]
\\
& & &
\\
& X_{\Gamma_0(p)\cdot U^p} & & X_{\Gamma_0(p)\cdot U^p}&
\end{tikzcd}
    \end{equation*}
    By the definition of Hecke correspondence $\mathsf{T}(x)_U$, we know that it is isomorphic to compactification of the modular curve $Y_{\Gamma_0(p^{n+1})\cdot U^{p}\cap xU^{p}x^{-1}}(\mathbb{C})\simeq\textup{GL}_2(\mathbb{Q})^{+}\backslash\mathbb{H}^{+}\times\textup{GL}_2(\mathbb{A}_f)/U^{p}\cap xU^{p}x^{-1}$. Let $[\tau,g]$ be an element of the set $Y_{\Gamma_0(p^{n+1})\cdot U^{p}\cap xU^{p}x^{-1}}(\mathbb{C})$, it is mapped to $[\tau,g]\in X_{\Gamma_0(p)\cdot U^{p}}(\mathbb{C})$ under the morphism $\pi$, and $[\tau,gx]=[x_n\tau,x_ngx]=[\textup{det}(x)^{-1}x_nz,x_ng\left(x^{\ast}\right)^{-1}]\in X_{\Gamma_0(p)\cdot U^{p}}(\mathbb{C})$ under the morphism $\pi\circ T(x)$.
    \par
    By the bijection map (\ref{map1}), the element $[\tau,g]\in Y_{\Gamma_0(p^{n+1})\cdot U^{p}\cap xU^{p}x^{-1}}(\mathbb{C})$ corresponds to the object $\left(E_{\tau}\xrightarrow{\times p^{n+1}} E_{p^{n+1}\tau},\overline{\left(a_\tau^{p}\right)^{-1}\circ g}\right)\in\mathcal{Y}_{\Gamma_0(p^{n+1})\cdot U^{p}\cap xU^{p}x^{-1}}(\mathbb{C})$ in the groupoid $\mathcal{Y}_{\Gamma_0(p^{n+1})\cdot U^{p}\cap xU^{p}x^{-1}}(\mathbb{C})$. The element $[\tau,g]\in Y_{\Gamma_0(p)\cdot U^{p}}(\mathbb{C})$ corresponds to the object $\left(E_{\tau}\xrightarrow{\times p} E_{p\tau},\overline{\left(a_\tau^{p}\right)^{-1}\circ g}\right)\in\mathcal{Y}_{\Gamma_0(p)\cdot U^{p}}(\mathbb{C})$ in the groupoid $\mathcal{Y}_{\Gamma_0(p)\cdot U^{p}}(\mathbb{C})$. The element $[\textup{det}(x)^{-1}x_n\tau,x_ng\left(x^{\ast}\right)^{-1}]\in Y_{\Gamma_0(p)\cdot U^{p}}(\mathbb{C})$ corresponds to the object $\left(E_{p^{n}\tau}\xrightarrow{\times p} E_{p^{n+1}\tau},\overline{\left(a_{p^{n}\tau}^{p}\right)^{-1}\circ x_ng\left(x^{\ast}\right)^{-1}}\right)\in\mathcal{Y}_{\Gamma_0(p)\cdot U^{p}}(\mathbb{C})$ in the groupoid $\mathcal{Y}_{\Gamma_0(p)\cdot U^{p}}(\mathbb{C})$. Notice that $E_{\tau}\xrightarrow{\times p^{n}} E_{p^{n}\tau}$ is the composition of the first $n$ degree $p$ isogenies of the standard decomposition of the degree $p^{n+1}$ cyclic isogeny $E_{\tau}\xrightarrow{\times p^{n+1}}E_{p^{n+1}\tau}$. We have the following commutative diagram
    \begin{equation*}
        \begin{tikzcd}
&  \left(\mathbb{A}_f^{p}\right)^{2} \ar[rrr,"\eta^{p}"] \ar[d, "{x^{\ast}}"'] & & & V^{p}(E_\tau) \ar[r, "a_{\tau}^{p}"] \ar[d, "V^{p}(\times p^{n})"]& \left(\mathbb{A}_f^{p}\right)^{2} \ar[d, "x_n"]
\\
&  \left(\mathbb{A}_f^{p}\right)^{2} \ar[rrr,"V^{p}(\times p^{n})\circ\eta^{p}\circ \left(x^{\ast}\right)^{-1}"] & & & V^{p}(E_{p^{n}\tau}) \ar[r, "a_{p^{n}\tau}^{p}"] & \left(\mathbb{A}_f^{p}\right)^{2}.
\end{tikzcd}
    \end{equation*}
    Therefore $\left(a_{p^{n}\tau}^{p}\right)^{-1}\circ x_ng\left(x^{\ast}\right)^{-1}=V^{p}(\times p^{n})\circ\eta^{p}\circ\left(x^{\ast}\right)^{-1}$. Hence the object $\left(E\xrightarrow{\pi} E^{\prime},\overline{\eta^{p}}\right)\in\mathcal{Y}_{\Gamma_0(p^{n+1})\cdot U^{p}\cap xU^{p}x^{-1}}(\mathbb{C})$ is mapped to $ \left(E\xrightarrow{\pi_1} E_1,\overline{\eta^{p}}\right),\left(E_{n}\xrightarrow{\pi_{n+1}} E^{\prime},\overline{V^{p}(\pi_{n}\circ\cdots\circ\pi_1)\circ\eta^{p}\circ \left(x^{\ast}\right)^{-1}}\right)\in\mathcal{Y}_{\Gamma_0(p)\cdot U^{p}}(\mathbb{C})$, i.e., the case $\textup{I}+$ is true.
\end{proof}

\subsection{Generating series}
Let $v$ be a place of $\mathbb{Q}$. Let $\mathbb{A}$ be the ring of ad$\grave{\textup{e}}$les over $\mathbb{Q}$. Let $V=\textup{M}_2(\mathbb{Q})$ be equipped with the quadratic form given by determinant. Recall that $\mathbb{V}=\{\mathbb{V}_{v}\}$ is the following incoherent collection of quadratic spaces of $\mathbb{A}$ of rank 4 (cf. (\ref{incoherent-intro})),
\begin{equation}
    \mathbb{V}_{v}=V_v=\textup{M}_2(\mathbb{Q}_v)\,\,\textup{if $v<\infty$, and $\mathbb{V}_{\infty}$ is positive definite.}
    \label{incoherent}
\end{equation}
Let $q_{\mathbb{V}}:\mathbb{V}\rightarrow\mathbb{A}$ be the quadratic form on $\mathbb{V}$. Let $\mathbb{V}_f=\mathbb{V}\otimes_{\mathbb{A}}\mathbb{A}_f$. Let $H=\textup{GSpin}(V)$, we have $H\simeq\textup{GL}_2\times_{\mathbb{G}_m}\textup{GL}_2$. We still use $\textup{det}$ to denote the morphism $H\rightarrow\mathbb{G}_m$.
\par
Let $U\subset\textup{GL}_2(\mathbb{A}_f)$ be a sufficiently small compact open subgroup. Let $K=(U\times U)\cap H(\mathbb{A}_f)\simeq U\times_{\mathbb{G}_m}U$. Let $M_K$ be the compactification of the Shimura variety $\textup{Sh}(H,K)$. It's easy to see that $M_K(\mathbb{C})=M_{U,1}$.
\par
For $i=1,2$, let $p_{i}:M_K\rightarrow X_U$ be the two projection morphisms. Let $x=x_{\infty}\otimes x_f\in\mathbb{V}$ be an element. Following \cite[$\S$3.3]{YZZ}, define the following cycle in $M_K$:
\begin{equation*}
    \mathsf{Z}(x)_{K}=\begin{cases}
        \mathsf{Z}(x_f)_U,&\textup{if $q_{\mathbb{V}}(x)\in\mathbb{Q}^{\times}$;}\\
        \frac{1}{2}\left(p_1^{\ast}\mathscr{L}_{U}+p_2^{\ast}\mathscr{L}_U\right),&\textup{if $x=0$;}\\
        0, &\textup{if $q_{\mathbb{V}}(x)\notin\mathbb{Q}^{\times}$ and $x\neq0$.}
    \end{cases}
\end{equation*}
\par
Let $\mathscr{S}(\mathbb{V})=\mathscr{S}(\mathbb{V}_{\infty})\otimes\mathscr{S}(\mathbb{V}_f)$ be the Schwartz function space on $\mathbb{V}$.
Let $\phi_{\infty}\in\mathscr{S}(\mathbb{V}_{\infty})$ be the standard Gaussian function. Let $\phi_{f}\in\mathscr{S}(\mathbb{V}_f)$ be a $K$-invariant function. Let $\phi=\phi_{\infty}\otimes\phi_{f}$. Let $\widetilde{K}=\textup{O}(\mathbb{V}_{\infty})\otimes K$, it acts on $\mathbb{V}$. Define the following generating series
\begin{equation}
    \mathsf{T}(\phi)=\sum\limits_{x\in\widetilde{K}\backslash\mathbb{V}}\phi(x)\mathsf{T}(x)_K.
    \label{gen-ser}
\end{equation}
Notice that there exists an extended Weil representation $r$ of $\textup{GL}_2(\mathbb{A})$ on the space $\mathscr{S}(\mathbb{V})$ (cf. \cite[$\S$2.1]{YZZ}). Define
\begin{equation}
    \mathsf{T}(g,\phi)=\sum\limits_{x\in\widetilde{K}\backslash\mathbb{V}}\left(r(g)\phi\right)(x)\cdot\mathsf{T}(x)_K.
    \label{gen-ser-g}
\end{equation}

\subsection{A regular integral model for $M_K$ and supersingular uniformization}
Let $U=\Gamma_0(p)\cdot U^{p}$, where $U^{p}\subset\textup{GL}_2(\mathbb{A}_f^{p})$ is a sufficiently small compact open subgroup. Let $K=U\times_{\mathbb{G}_m}U$, we will construct a regular integral model $\M_{K,(p)}$ for the compactified Shimura variety $M_K$ over $\Z_{(p)}$.
\par
Let $\mathcal{O}_U$ be the integer ring of $\mathbb{Q}_U$. The morphism $\mathcal{X}_U\rightarrow\textup{Spec}\,\Z_{(p)}$ factors through $\mathcal{O}_{U,(p)}\coloneqq\mathcal{O}_{U}\otimes_{\mathbb{Z}}\Z_{(p)}$. By the class field theory, we have $\mathcal{O}_{U,(p)}\simeq\prod\limits_{a\in\pi_0(X_U)}\Z_{(p)}$. Let $\mathcal{X}_{U,a}=\mathcal{X}_{U}\times_{\mathcal{O}_{U,(p)},a}\Z_{(p)}$. The Deligne-Mumford stack $\mathcal{X}_{U}=\coprod\limits_{a\in\pi_0(X_U)}\mathcal{X}_{U,a}$. Define
\begin{equation*}
    \mathcal{H}_{K,(p)}=\coprod\limits_{a\in\pi_0(X_U)}\mathcal{X}_{U,a}\times_{\Z_{(p)}}\mathcal{X}_{U,a}.
\end{equation*}
\par
Denote by $\mathcal{H}_{K,(p)}^{\textup{ss}}$ the supersingular locus of the Deligne-Mumford stack $\mathcal{H}_{K,(p)}$, it is a union of finitely many $\F$-points. Let $\mathbb{E}$ be a supersingular elliptic curve over $\mathbb{F}$ such that $\mathbb{E}[p^{\infty}]\simeq\mathbb{X}$. Let $B$ be the unique quaternion division algebra over $\mathbb{Q}$ which ramifies at $p$ and $\infty$. Let
\begin{equation*}
    \left(\left(\mathbb{E}_1\xrightarrow{f_1} \mathbb{E}_1^{\prime},\overline{\eta_{\mathbb{E},1}^p}\right), \left(\mathbb{E}_2\xrightarrow{f_2} \mathbb{E}_2^{\prime},\overline{\eta_{\mathbb{E},2}^p}\right)\right)
\end{equation*}
be one of the points in $\mathcal{H}_{K,(p)}^{\textup{ss}}(\mathbb{F})$. Then there exists four quasi-isogenies $\rho_{\mathbb{E},i}:\mathbb{E}\rightarrow\mathbb{E}_i$, $\rho_{\mathbb{E},i}^{\prime}:\mathbb{E}\rightarrow\mathbb{E}_i^{\prime}$ ($i=1,2$) of degrees 1 such that $\rho_{\mathbb{E},i}^{\prime-1}\circ f_i\circ\rho_{\mathbb{E},i}=x_0$. We fix a choice of representatives $\eta_{\mathbb{E},1}^p$ and $\eta_{\mathbb{E},2}^p$ from now on.
\par
Let $\widehat{\mathcal{H}}_K^{\textup{ss}}$ be the completion of $\mathcal{H}_{K,(p)}$ along the closed substack $\mathcal{H}_{K,(p)}^{\textup{ss}}$. Then we have the following supersingular uniformization map:
\begin{equation}
    \Theta_{\mathcal{H}}:\widehat{\mathcal{H}}_{K,(p)}^{\textup{ss}}\xrightarrow{\sim} H^{\prime}(\mathbb{Q})_{0}\backslash\mathcal{N}(x_0)\times H(\mathbb{A}_f^{p})/K,
\end{equation}
where $H^{\prime}(\mathbb{Q})_0=\{g\in B^{\times}(\mathbb{Q})\times_{\mathbb{Q}^{\times}}B^{\times}(\mathbb{Q}):\nu_p(\nu(g))=0\}$. 
The map $\Theta_{\mathcal{H}}$ is given as follows: Let $S$ be an object in $\textup{Nilp}_W$. Let
\begin{equation}
    \left(\left(E_1\xrightarrow{\pi_1} E_1^{\prime},\overline{\eta_1^p}\right), \left(E_2\xrightarrow{\pi_2} E_2^{\prime},\overline{\eta_2^p}\right)\right)
    \label{obj-naive}
\end{equation}
be an object in $\widehat{\mathcal{H}}_{K,(p)}^{\textup{ss}}(S)$. Then there exists four quasi-isogenies $\rho_{E,i}:\mathbb{E}_i\times_{\F}\overline{S}\rightarrow E_i\times_{S}\overline{S}$, $\rho_{E,i}^{\prime}:\mathbb{E}_i^{\prime}\times_{\F}\overline{S}\rightarrow E_i^{\prime}\times_{S}\overline{S}$ ($i=1,2$) of degrees 1 such that $\pi_i\circ\rho_{E,i}=\rho_{E,i}^{\prime}\circ f_i$ for $i=1,2$. For $i=1,2$, let $g_i\in\textup{GL}_2(\mathbb{A}_f^{p})$ be the element such that 
\begin{equation}
    g_i= \left(\eta_{\mathbb{E},i}^{p}\right)^{-1}\circ V^p(\rho_{E,i})^{-1}\circ \eta_{i}^{p}.
    \label{fram-g}
\end{equation}
Then $g_i$ is determined up to right multiplication by an element in $U^p$. By the definition of $\mathcal{H}_{K,(p)}$, we must have $\textup{det}(g_1)=\textup{det}(g_2)$. Let $\rho_i=\left(\rho_{E,i}\circ\rho_{\mathbb{E},i}\right)[p^{\infty}]$ and $\rho_i^{\prime}=\left(\rho_{E,i}^{\infty}\circ\rho_{\mathbb{E},i}^{\infty}\right)[p^{\infty}]$. Then the object (\ref{obj-naive}) is mapped to 
\begin{align*}
    \left(\left(E_1[p^{\infty}]\xrightarrow{\pi_1[p^{\infty}]} E_1^{\prime}[p^{\infty}],(\rho_1,\rho_1^{\prime})\right), \left(E_2[p^{\infty}]\xrightarrow{\pi_2[p^{\infty}]} E_2^{\prime}[p^{\infty}],(\rho_2,\rho_2^{\prime})\right)\right)\times(g_1,g_2)\\
    \in H^{\prime}(\mathbb{Q})_{0}\backslash\mathcal{N}(x_0)(S)\times H(\mathbb{A}_f^{p})/K.
\end{align*}
\par
We also want to make the action of $H^{\prime}(\mathbb{Q})_{0}$ on the formal scheme $\N(x_0)$ clear. Let 
    \begin{equation}
        \left(X_1\stackrel{\pi_1}\rightarrow X_1^{\prime}, (\rho_1,\rho_1^{\prime})\right), \left(X_2\stackrel{\pi_2}\rightarrow X_2^{\prime}, (\rho_2,\rho_2^{\prime})\right)
    \end{equation}
    be an object in the set $\N(x_0)(S)$. Let $(b_1,b_2)\in H^{\prime}(\mathbb{Q})_{0}$. Then
    \begin{align*}
        (b_1,b_2)\cdot&\left(\left(X_1\stackrel{\pi_1}\rightarrow X_1^{\prime}, (\rho_1,\rho_1^{\prime})\right), \left(X_2\stackrel{\pi_2}\rightarrow X_2^{\prime}, (\rho_2,\rho_2^{\prime})\right)\right)\\
        &=\left(\left(X_1\stackrel{\pi_1}\rightarrow X_1^{\prime}, (\rho_1 b_1^{-1},\rho_1^{\prime}b_1^{\prime-1})\right), \left(X_2\stackrel{\pi_2}\rightarrow X_2^{\prime}, (\rho_2 b_2^{-1},\rho_2^{\prime}b_2^{\prime-1})\right)\right).
    \end{align*}
\par
Let $\M_{K,(p)}$ be the blow up of the Deligne-Mumford stack $\mathcal{H}_{K,(p)}$ along the closed substack $\mathcal{H}_{K,(p)}^{\textup{ss}}$. Let $\widehat{\M}_{K,(p)}^{\textup{ss}}$ be the completion of $\M_{K,(p)}$ along its supersingular locus, i.e., the exceptional divisor $\exc_{K}$. Notice that by the universal property of blow up, the automorphism $g:\N(x_0)\rightarrow\N(x_0)$ extends to an automorphism $g:\M\rightarrow\M$, i.e., the group $H^{\prime}(\mathbb{Q})_0$ also acts on the formal scheme $\M$.
\begin{lemma}
    The stack $\M_{K,(p)}$ is a 3-dimensional regular Deligne-Mumford stack. There exists a supersingular uniformization map:
    \begin{equation}
        \Theta_{\M}:\widehat{\M}_{K,(p)}^{\textup{ss}}\xrightarrow{\sim} H^{\prime}(\mathbb{Q})_{0}\backslash\M\times H(\mathbb{A}_f^{p})/K
    \end{equation}
    which makes the following diagram commutative
        \begin{equation*}
\begin{tikzcd}
    {\widehat{\M}_{K,(p)}^{\textup{ss}}}
    \arrow[d] \arrow[r,"\Theta_{\M}"]
    & {H^{\prime}(\mathbb{Q})_{0}\backslash\M\times H(\mathbb{A}_f^{p})/K}
    \arrow[d]
    \\
    {\widehat{\mathcal{H}}_{K,(p)}^{\textup{ss}}}
    \arrow[r,"\Theta_{\mathcal{H}}"]
    &{H^{\prime}(\mathbb{Q})_{0}\backslash\mathcal{N}(x_0)\times H(\mathbb{A}_f^{p})/K}.
\end{tikzcd}
    \end{equation*}
\end{lemma}
\begin{proof}
    The formal scheme $\M$ and the integral model $\M_{K,(p)}$ are constructed by blowing up along the supersingular $\mathbb{F}$-point(s) of the formal scheme $\N(x_0)$ and the integral model $\mathcal{H}_{K,(p)}$. Since we have the uniformization morphism $\Theta_{\mathcal{H}}:\widehat{\mathcal{H}}_{K,(p)}^{\textup{ss}}\xrightarrow{\sim} H^{\prime}(\mathbb{Q})_{0}\backslash\mathcal{N}(x_0)\times H(\mathbb{A}_f^{p})/K$, the supersingular uniformization morphism $\Theta_{\M}$ for $\widehat{\M}_{K,(p)}^{\textup{ss}}$ follows easily.
\end{proof}

\subsection{Supersingular uniformization of Hecke correspondences}
We still assume that $U=\Gamma_0(p)\cdot U^{p}$, where $U^{p}\subset\textup{GL}_2(\mathbb{A}_f^{p})$ is a sufficiently small compact open subgroup. Let $x\in\mathbb{V}$ be an element such that $q_{\mathbb{V}}(x)\in\mathbb{Q}^{\times}$ and $x_p$ is a primitive element in the lattice $H_0(p)$. Let $n=\nu_p(q_{\mathbb{V}}(x))\geq0$. Let symbols $\ast\in\{\textup{I},\textup{II}\}$ and $?\in\{+,-\}$. If $x_p\in\mathbb{V}^{\ast?}$, we can extend the morphism $\mathsf{T}(x_f)_U\rightarrow X_U^{2}$ to $\mathcal{X}_{\Gamma_0(p^{n+1})\cdot U^{\prime p}}\rightarrow\mathcal{H}_{K}$ (if $\ast=\textup{I}$) or $\mathcal{X}_{\Gamma_0(p^{n})\cdot U^{\prime p}}\rightarrow\mathcal{H}_K$ (if $\ast=\textup{II}$) using the moduli interpretations of Lemma \ref{moduli-hecke}. Define
\begin{equation*}
    \widehat{\mathsf{T}}(x)_K=\begin{cases}
        \left(\textup{strict transform of $\mathcal{X}_{\Gamma_0(p^{n+1})\cdot U^{\prime p}}$}\right)+\exc_{K}, &\textup{if $\ast=\textup{I}$;}\\
        \textup{strict transform of $\mathcal{X}_{\Gamma_0(p^{n})\cdot U^{\prime p}}$}, &\textup{if $\ast=\textup{II}$.}
    \end{cases}
\end{equation*}
\begin{definition}
    Let $x\in\mathbb{V}$ be an element such that $q_{\mathbb{V}}(x)\in\mathbb{Q}^{\times}$. Then there exists a unique integer $r$ such that $p^{r}x_p$ is a primitive element in the lattice $H_0(p)$. Define
    \begin{equation*}
        \widehat{\mathsf{T}}(x)_K=\widehat{\mathsf{T}}(p^{r}x)_K.
    \end{equation*}
    For a $K$-invariant Schwartz function $\phi_{f}\in\mathscr{S}(\mathbb{V}_f)$. Let $\phi=\phi_{\infty}\otimes\phi_{f}\in\mathscr{S}(\mathbb{V})$ where $\phi_{\infty}\in\mathscr{S}(\mathbb{V}_{\infty})$ is the standard Gaussian function. Define
\begin{equation*}
    \widehat{\mathsf{T}}(g,\phi)=\sum\limits_{x\in\widetilde{K}\backslash\mathbb{V}}\left(r(g)\phi\right)(x)\cdot\widehat{\mathsf{T}}(x)_K,\,\,g\in\textup{GL}_2(\mathbb{A}).
\end{equation*}
This can be viewed as the integral version of the generating series $\mathsf{T}(g,\phi)$ as in (\ref{gen-ser-g}) by the following lemma.
\end{definition}
\begin{lemma}
    Let $x\in\mathbb{V}$ be an element such that $q_{\mathbb{V}}(x)\in\mathbb{Q}^{\times}$. Then $\widehat{\mathsf{T}}(x)_K$ is an effective Cartier divisor on the Deligne-Mumford stack $\M_{K}$. Moreover, we have
    \begin{equation*}
        \widehat{\mathsf{T}}(x)_K(\mathbb{C})=\mathsf{T}(x)_K(\mathbb{C}).
    \end{equation*}
    \label{C-points}
\end{lemma}
\begin{proof}
    The strict transform of the regular Deligne-Mumford stack $\mathcal{X}_{\Gamma_0(p^{n})\cdot U^{\prime p}}$ is still regular by Lemma \ref{blow-up-N0(x)}, hence it's still a divisor on $\M_{K,(p)}$. Therefore $\widehat{\mathsf{T}}(x)_K$ is still a divisor.
    \par
    The Hecke correspondence only depends on the coset $\mathbb{Q}^{\times}\backslash\mathbb{V}$. Let $r$ be an integer such that $p^{r}x_p$ is a primitive element in $H_0(p)$. Then we have $\mathsf{T}(x)_K=\mathsf{T}(p^{r}x)_K$. Since $\exc_{\M_K}$ is supported on the special fiber of $\M_{K,(p)}$, we have
    \begin{equation*}
        \widehat{\mathsf{T}}(x)_K(\mathbb{C})=\widehat{\mathsf{T}}(p^{r}x)_K(\mathbb{C})=\mathsf{T}(p^{r}x)_K(\mathbb{C})=\mathsf{T}(x)_K(\mathbb{C}).
    \end{equation*}
\end{proof}
Recall that $B$ is the unique quaternion division algebra over $\mathbb{Q}$ which ramifies at $p$ and $\infty$. For an element $y\in B$, let $H_{y}^{\prime}\subset H^{\prime}$ be the subgroup which stabilizes $y$. Let $H_{y}^{\prime}(\mathbb{Q})_0=H_{y}^{\prime}\cap H^{\prime}(\mathbb{Q})_0$.
\begin{lemma}
    Let $x\in\mathbb{V}$ be an element such that $q_{\mathbb{V}}(x)\in\mathbb{Q}^{\times}$ and $x_p$ is a primitive element in the lattice $H_0(p)$. Let $\widehat{\mathsf{T}}(x)_K^{\textup{ss}}$ be the completion of $\widehat{\mathsf{T}}(x)_K$ along its supersingular locus. Then we have the following identity in $\textup{K}_0^{\prime}(\widehat{\M}_K^{\textup{ss}})_{\mathbb{C}}$:
    \begin{equation*}
        \widehat{\mathsf{T}}(x)_K^{\textup{ss}}=\sum\limits_{\substack{y\in H^{\prime}(\mathbb{Q})_0\backslash B\\q_{B}(y)=q_{\mathbb{V}}(x)}}\sum\limits_{\substack{g\in H_{y}^{\prime}(\mathbb{Q})_0\backslash H(\mathbb{A}_f^{p})/K^p\\g^{-1}y=x}}\Theta_{\M}^{-1}\begin{cases}
            \left(\N_0^{\textup{I}+}(x_0y)+\exc_{\M},g\right),&\textup{if $x_p\in\mathbb{V}_p^{\textup{I}+}$;}\\
            \left(\N_0^{\textup{I}-}(x_0\overline{y})+\exc_{\M},g\right),&\textup{if $x_p\in\mathbb{V}_p^{\textup{I}-}$;}\\
            \left(\N_0^{\textup{II}+}(y),g\right),&\textup{if $x_p\in\mathbb{V}_p^{\textup{II}+}$;}\\
            \left(\N_0^{\textup{II}-}(y^{\prime}),g\right) ,&\textup{if $x_p\in\mathbb{V}_p^{\textup{II}-}$.}\\
        \end{cases}
    \end{equation*}
\end{lemma}
\begin{proof}
    We will only give the proof for the case $x_p\in\mathbb{V}_p^{\textup{I}+}$. Let $\mathsf{T}_{\mathcal{H}}(x)\simeq\mathcal{X}_{\Gamma_0(p^{n+1})\cdot U^{p}\cap xU^{p}x^{-1}}$ be the closure of the image of $\mathsf{T}(x)$ in the stack $\mathcal{H}_K$ under the morphism $\mathsf{T}(x)\rightarrow\mathcal{X}_U^{2}$ given in Lemma \ref{moduli-hecke} (\textup{I}+). Let $\widehat{\mathsf{T}}_{\mathcal{H}}(x)^{\textup{ss}}$ be the completion of $\mathsf{T}_{\mathcal{H}}(x)$ along its supersingular locus. We will prove that
    \begin{equation}
        \widehat{\mathsf{T}}_{\mathcal{H}}(x)^{\textup{ss}}=\bigsqcup\limits_{\substack{y\in H^{\prime}(\mathbb{Q})_0\backslash B\\q_{B}(y)=q_{\mathbb{V}}(x)}}\bigsqcup\limits_{\substack{g\in H_{y}^{\prime}(\mathbb{Q})_0\backslash H(\mathbb{A}_f^{p})/K^p\\g^{-1}y=x}}\Theta_{\mathcal{H}}^{-1}\left(\st_{x_0y}^{\textup{I}+}\left(\mathcal{N}_0(x_0y)\right),g\right).
        \label{unif-hecke-I+}
    \end{equation}
    We prove this by constructing explicit maps between the two sides of (\ref{unif-hecke-I+}).
    \par
    Let $S$ be a $W$-scheme such that $p$ is locally nilpotent on $S$. Let $(z,g)\in\widehat{\mathcal{H}}_K^{\textup{ss}}(S)\times H(\mathbb{A}_f^{p})/K$ be an element such that $[(z,g)]\in\widehat{\mathsf{T}}_{\mathcal{H}}(x)^{\textup{ss}}(S)$. It gives rise to the following object in $\widehat{\mathcal{H}}_{K,(p)}^{\textup{ss}}(S)$:
\begin{equation*}
    \left(\left(E_1\xrightarrow{\pi_1} E_1^{\prime},\overline{\eta_1^p}\right), \left(E_2\xrightarrow{\pi_2} E_2^{\prime},\overline{\eta_2^p}\right)\right).
\end{equation*}
There exists four quasi-isogenies $\rho_{E,i}:\mathbb{E}_i\times_{\F}\overline{S}\rightarrow E_i\times_{S}\overline{S}$, $\rho_{E,i}^{\prime}:\mathbb{E}_i^{\prime}\times_{\F}\overline{S}\rightarrow E_i^{\prime}\times_{S}\overline{S}$ ($i=1,2$) of degrees 1 such that $\pi_i\circ\rho_{E,i}=\rho_{E,i}^{\prime}\circ f_i$ for $i=1,2$. Since $[(z,g)]\in\widehat{\mathsf{T}}(x)^{\textup{ss}}(S)$, there exists a cyclic isogeny $\pi:E_1\rightarrow E_2^{\prime}$ of degree $p^{n+1}$ such that $\pi_1$ (resp. $\pi_2$) is the first (resp. last) degree $p$ isogeny in the standard decomposition of $\pi$. Let $\Tilde{y}=\left(\rho_{E,2}^{\prime}\circ\rho_{\mathbb{E},2}^{\prime}\right)^{-1}\circ\pi\circ\left(\rho_{E,1}^{\prime}\circ\rho_{\mathbb{E},1}^{\prime}\right)\in\textup{End}^{0}(\mathbb{E})\simeq B$ and $y=x_0^{-1}\cdot\Tilde{y}$. By (\ref{fram-g}), we have $q_B(y)=q_{\mathbb{V}}(x)$ and $g^{-1}y=x$. Moreover, the following element 
\begin{align*}
    \left(\left(E_1[p^{\infty}]\xrightarrow{\pi_1[p^{\infty}]} E_1^{\prime}[p^{\infty}],(\rho_1,\rho_1^{\prime})\right), \left(E_2[p^{\infty}]\xrightarrow{\pi_2[p^{\infty}]} E_2^{\prime}[p^{\infty}],(\rho_2,\rho_2^{\prime})\right)\right)
\end{align*}
lies in the closed formal subscheme $\st_{x_0y}^{\textup{I}+}\left(\mathcal{N}_0(x_0y)\right)\subset\N(x_0)$. The reverse direction can be proved easily, therefore (\ref{unif-hecke-I+}) is true.
\par
Let $\widehat{\mathsf{T}}^{\textup{bl}}(x)$ be the strict transform of $\widehat{\mathsf{T}}_{\mathcal{H}}(x)$ under the blow up morphism $\M_K\rightarrow\mathcal{H}_K$. By the definition of $\widehat{\mathsf{T}}(x)_K$, we have 
\begin{equation}
    \widehat{\mathsf{T}}(x)_K=\widehat{\mathsf{T}}^{\textup{bl}}(x)+\exc_K.
    \label{diffence2}
\end{equation}
Denote by $\widehat{\mathsf{T}}^{\textup{bl}}(x)^{\textup{ss}}$ the completion of $\widehat{\mathsf{T}}^{\textup{bl}}(x)$ along the supersingular locus. By (\ref{unif-hecke-I+}), we have
\begin{equation*}
        \widehat{\mathsf{T}}^{\textup{bl}}(x)^{\textup{ss}}=\bigsqcup\limits_{\substack{y\in H^{\prime}(\mathbb{Q})_0\backslash B\\q_{B}(y)=q_{\mathbb{V}}(x)}}\bigsqcup\limits_{\substack{g\in H_{y}^{\prime}(\mathbb{Q})_0\backslash H(\mathbb{A}_f^{p})/K^p\\g^{-1}y=x}}\Theta_{\mathcal{M}}^{-1}\left(\widetilde{\mathcal{N}}^{\textup{I}+}_0(x_0y),g\right).
    \end{equation*}
Combining with (\ref{diffence2}), the formula in the lemma of the case $\textup{I}+$ is true.
\end{proof}
\begin{corollary}
    Let $\widehat{\mathsf{T}}(\phi)_K^{\textup{ss}}$ be the completion of $\widehat{\mathsf{T}}(\phi)_K$ along its supersingular locus. Then we have the following identity in $\textup{K}_0^{\prime}(\widehat{\M}_K^{\textup{ss}})_{\mathbb{C}}$:
    \begin{equation*}
        \widehat{\mathsf{T}}(\phi)_K^{\textup{ss}}=\sum\limits_{y\in H^{\prime}(\mathbb{Q})_0\backslash B}\sum\limits_{g\in H_{y}^{\prime}(\mathbb{Q})_0\backslash H(\mathbb{A}_f^{p})/K^p}\phi^{p}(g^{-1}y)\cdot\Theta_{\M}^{-1}\begin{cases}
            \left(\mathcal{Z}(y),g\right),&\textup{if $\phi_p=1_{H_0(p)}$;}\\
            \left(\mathcal{Y}(y)-\exc_{\M},g\right) ,&\textup{if $\phi_p=1_{H_0(p)^{\vee}}$.}
        \end{cases}
    \end{equation*}
    \label{uni-hecke}
\end{corollary}
\begin{proof}
    Denote by $\mathbb{V}_p(m)$ the following subset of $H_0(p)$:
    \begin{equation*}
        \mathbb{V}_p(m)=\{x\in\mathbb{V}:\textup{det}(x)\in\mathbb{Q}^{\times}_{>0}, \nu_p(q_{\mathbb{V}}(x))=m\}.
    \end{equation*}
    For symbols $\ast\in\{\textup{I},\textup{II}\}$ and $?\in\{+,-\}$, define $\textup{sgn}:\{\textup{I},\textup{II}\}\rightarrow\{0,1\}$ to be $\textup{sgn}(\textup{I})=1$ and $\textup{sgn}(\textup{II})=0$. Define 
    \begin{equation*}
        \mathbb{V}_p^{\ast?}(m)=\{x\in\mathbb{V}_p(m):x_p\in\mathbb{V}_p^{\ast?}\cap H_0(p)\},\,\,\,\,\mathbb{V}_p^{\ast?}(m)^{\vee}=\{x\in\mathbb{V}_p(m):x_p\in\mathbb{V}_p^{\ast?}\cap H_0(p)^{\vee}\}.
    \end{equation*}
    Therefore we have
    \begin{align*}
        \widehat{\mathsf{T}}(\phi)_ K=\sum\limits_{x\in\widetilde{K}\backslash\mathbb{V}}\phi(x)\cdot\widehat{\mathsf{T}}(x)_K=\sum\limits_{m=-\infty}^{\infty}\sum\limits_{x\in\widetilde{K}\backslash\mathbb{V}(m)}\phi(x)\cdot\widehat{\mathsf{T}}(x)_K
    \end{align*}
    \par
     We first consider the case $\phi_p=1_{H_0(p)}$. By Corollary \ref{uni-hecke}, we have
     \begin{align*}
        \widehat{\mathsf{T}}(\phi)_K^{\textup{ss}}&=\sum\limits_{m=0}^{\infty}\sum\limits_{\ast\in\{\textup{I},\textup{II}\}}\sum\limits_{?\in\{+,-\}}\sum\limits_{x\in\widetilde{K}\backslash\mathbb{V}_p^{\ast?}(m)}\phi^{p}(x^{p})\cdot\widehat{\mathsf{T}}(x)_K^{\textup{ss}}.
    \end{align*}
    For $m=0$, we have $\ast=\textup{I}$ and $?=+$, hence 
    \begin{align*}
        &\sum\limits_{\ast\in\{\textup{I},\textup{II}\}}\sum\limits_{?\in\{+,-\}}\sum\limits_{x\in\widetilde{K}\backslash\mathbb{V}_p^{\ast?}(0)}\phi^{p}(x^{p})\cdot\widehat{\mathsf{T}}(x)_K^{\textup{ss}}=\sum\limits_{x\in\widetilde{K}\backslash\mathbb{V}_p^{\textup{I}+}(0)}\phi^{p}(x^{p})\cdot\widehat{\mathsf{T}}(x)_K^{\textup{ss}}\\
        &=\sum\limits_{x\in\widetilde{K}\backslash\mathbb{V}_p^{\textup{I}+}(0)}\sum\limits_{\substack{y\in H^{\prime}(\mathbb{Q})_0\backslash B\\q_{B}(y)=q_{\mathbb{V}}(x)}}\sum\limits_{\substack{g\in H_{y}^{\prime}(\mathbb{Q})_0\backslash H(\mathbb{A}_f^{p})/K^p\\g^{-1}y=x}}\phi^{p}(g^{-1}y)\cdot\Theta_{\M}^{-1}\left(\exc_{\M}+\N_0^{\textup{I}+}(x_0y),g\right)\\
        &\overset{\textup{Lemma}\,\ref{dec-diff-div}}{=}\sum\limits_{\substack{y\in H^{\prime}(\mathbb{Q})_0\backslash B\\\nu_p(q_{B}(y))=0}}\sum\limits_{g\in H_{y}^{\prime}(\mathbb{Q})_0\backslash H(\mathbb{A}_f^{p})/K^p}\phi^{p}(g^{-1}y)\cdot\Theta_{\M}^{-1}\left(\mathcal{Z}(y),g\right).
    \end{align*}
    Notice that $\mathcal{D}(y)=\mathcal{Z}(y)$ if $\nu_p(q_B(y))=0$ or $1$. Similar arguments apply to $m\geq1$, we get (using Lemma \ref{dec-diff-div} repeatedly)
    \begin{equation*}
        \sum\limits_{\ast\in\{\textup{I},\textup{II}\}}\sum\limits_{?\in\{+,-\}}\sum\limits_{x\in\widetilde{K}\backslash\mathbb{V}_p^{\ast?}(m)}\phi^{p}(x^{p})\cdot\widehat{\mathsf{T}}(x)_K^{\textup{ss}}=\sum\limits_{\substack{y\in H^{\prime}(\mathbb{Q})_0\backslash B\\\nu_p(q_{B}(y))=m}}\sum\limits_{g\in H_{y}^{\prime}(\mathbb{Q})_0\backslash H(\mathbb{A}_f^{p})/K^p}\phi^{p}(g^{-1}y)\cdot\Theta_{\M}^{-1}\left(\mathcal{D}(y),g\right).
    \end{equation*}
    Therefore
    \begin{align*}
        \widehat{\mathsf{T}}(\phi)_K^{\textup{ss}}&=\sum\limits_{m=0}^{\infty}\sum\limits_{\substack{y\in H^{\prime}(\mathbb{Q})_0\backslash B\\\nu_p(q_{B}(y))=m}}\sum\limits_{g\in H_{y}^{\prime}(\mathbb{Q})_0\backslash H(\mathbb{A}_f^{p})/K^p}\phi^{p}(g^{-1}y)\cdot\Theta_{\M}^{-1}\left(\mathcal{D}(y),g\right)\\
        &=\sum\limits_{y\in H^{\prime}(\mathbb{Q})_0\backslash B}\sum\limits_{g\in H_{y}^{\prime}(\mathbb{Q})_0\backslash H(\mathbb{A}_f^{p})/K^p}\phi^{p}(g^{-1}y)\cdot\Theta_{\M}^{-1}\left(\sum\limits_{m=0}^{\infty}\mathcal{D}(p^{-m}y),g\right)\\
        &=\sum\limits_{y\in H^{\prime}(\mathbb{Q})_0\backslash B}\sum\limits_{g\in H_{y}^{\prime}(\mathbb{Q})_0\backslash H(\mathbb{A}_f^{p})/K^p}\phi^{p}(g^{-1}y)\cdot\Theta_{\M}^{-1}\left(\mathcal{Z}(y),g\right).
    \end{align*}
    \par
    Using (\ref{local-id-aut-M-1}), (\ref{local-id-aut-M}) and (\ref{diff-Y}), the proof of the case $\phi_p=1_{H_0(p)^{\vee}}$ is similar. So we omit it.
\end{proof}

\subsection{Whittaker functions and Eisenstein series}
\label{eisen}
Let $\nu:\textup{GSp}_6\rightarrow\mathbb{G}_m$ be the homomorphism such that $\textup{ker}(\nu)=\textup{Sp}_6$. Let $P$ be the following Siegel parabolic subgroup of $\textup{GSp}_6$:
\begin{equation*}
    P=\left\{\begin{pmatrix}
        a & \ast\\
        0 & \nu\cdot{^{t}a^{-1}}
    \end{pmatrix}\in\textup{GSp}_6\,\,\bigg\vert\,\,a\in\textup{GL}_3,\,\,\nu\in\mathbb{G}_m\right\}.
\end{equation*}
Let $M,N\subset P$ be the following groups,
\begin{equation*}
    M=\left\{m(a)=\begin{pmatrix}
        a & 0\\
        0 & \cdot{^{t}a^{-1}}
    \end{pmatrix}\in\textup{GSp}_6\,\,\bigg\vert\,\,a\in\textup{GL}_3\right\},\,\,\,\,N=\left\{n(b)=\begin{pmatrix}
        \mathbf{1}_3 & b\\
        0 & \mathbf{1}_3
    \end{pmatrix}\in\textup{GSp}_6\,\,\bigg\vert\,\,b\in\textup{Sym}_3\right\}.
\end{equation*}
\par
For a complex number $s\in\mathbb{C}$, define the following character $\lambda_s$ of $P(\mathbb{A})$:
\begin{equation*}
    \lambda_s\left(\begin{pmatrix}
        a & \ast\\
        0 & \nu\cdot{^{t}a^{-1}}
    \end{pmatrix}\right)=\vert \nu\vert_{\mathbb{A}}^{-3s}\vert\textup{det}(a)\vert_{\mathbb{A}}^{2s}.
\end{equation*}
Let $I(s)=\textup{Ind}_{P(\mathbb{A})}^{\textup{GSp}_6(\mathbb{A})}\lambda_s$ be the degenerate principal series of $\textup{GSp}_6(\mathbb{A})$. It's easy to see that $I(s)=\otimes_v^{\prime} I_v(s)$ is the restricted tensor product of the local function space $I_v(s)$.
\par
There exists a Weil representation $r$ of the group $\textup{Sp}_6(\mathbb{A})$ on the Schwartz function space $\mathscr{S}(\mathbb{V}^3)$ (cf. \cite[$\S$2.1]{YZZ}, see also \cite[Definition 2.2.1]{Li21}). For an element $a\in\mathbb{A}^{\times}$, let
\begin{equation*}
    d(a)=\begin{pmatrix}
        \mathbf{1}_3 & 0\\
        0 & a\cdot\mathbf{1}_3
    \end{pmatrix}.
\end{equation*}
\par
Let $\Phi\in\mathscr{S}(\mathbb{V}^{3})$ be a Schwartz function. Define 
\begin{equation}
    f_{\Phi}(g,0)=\vert\nu(g)\vert_{\mathbb{A}}^{-3}r\left(d(\nu(g))^{-1}g\right)\Phi(0).
    \label{section-0}
\end{equation}
The function $f_{\Phi}(g,0)$ is an element of $I(s)$. We extend it to a standard section $f_{\Phi}(g,s)$ of $I(s)$. It satisfies
\begin{equation}
    f_{\Phi}(d(\nu)n(b)m(a)g,s)=\vert v\vert_{\mathbb{A}}^{-3s-3}\vert\textup{det}(a)\vert_{\mathbb{A}}^{2s+2}f_{\Phi}(g,s).
    \label{rela-siegel-weil}
\end{equation}
Notice that for a single place $v$ and a Schwartz function $\Phi_v\in\mathscr{S}(\mathbb{V}_v^{3})$, we can define a section $f_{\Phi_v}(g_v,s)\in I_v(s)$ using similar formulas as (\ref{section-0}) and (\ref{rela-siegel-weil}).
\par
Define the Siegel Eisenstein series associated to the function $\Phi\in\mathscr{S}(\mathbb{V}^{3})$ to be
\begin{equation*}
    E(g,s,\Phi)=\sum\limits_{\gamma\in P(\mathbb{Q})\backslash\textup{GSp}_6(\mathbb{Q})}f_{\Phi}(\gamma g,s).
\end{equation*}
This summation is absolutely convergent when $\textup{re}(s)>2$. It extends to a meromorphic function of $s\in\mathbb{C}$ and holomorphic at $s=0$ (cf. \cite[Theorem 2.2]{Kud97Ann}). In the following, we will always assume that $\Phi=\Phi_{\infty}\otimes\Phi_f$ where $\Phi_{\infty}$ is the standard Gaussian function on $\mathbb{V}_{\infty}^{3}$ and $\Phi_{f}\in\mathscr{\mathbb{V}^{3}}$.
\par
Let $\psi=\otimes_v\psi_v:\mathbb{Q}\backslash\mathbb{A}\rightarrow\mathbb{C}^{\times}$ be the standard character. For $T\in\textup{Sym}_3(\mathbb{Q})$, define its $T$-th Fourier coefficients to be
\begin{equation*}
    E_T(g,s,\Phi)=\int\limits_{\textup{Sym}_3(\mathbb{Q})\backslash\textup{Sym}_3(\mathbb{A})}E(n(b)g,s,\Phi)\psi(-Tb)\textup{d}b.
\end{equation*}
When $\Phi=\otimes_{v}\Phi_v$ is decomposable and $T$ is non-singular, we have the following decomposition
\begin{equation*}
    E_T(g,s,\Phi)=\prod\limits_{v}W_{T,v}(g_v,s,\Phi_v),
\end{equation*}
where the local Whittaker function is given by
\begin{equation*}
    W_{T,v}(g_v,s,\Phi_v)=\int\limits_{\textup{Sym}_3(\mathbb{Q}_v)}f_{\Phi_v}\left(w_3^{-1}n(b)g_v,s\right)\psi_v\left(-\textup{tr}\left(\frac{1}{2}Tb\right)\right)\textup{d}b,\,\,w_3=\begin{pmatrix}
        0 & \mathbf{1}_3\\
        -\mathbf{1}_3 & 0
    \end{pmatrix}.
\end{equation*}
Let $\mathbf{1}=(\mathbf{1}_v)\in\textup{GSp}(\mathbb{A})$ be the identity element. Kudla \cite[Proposition A.6]{Kud97Ann} proved the following:
\begin{lemma}
    Let $v$ be a finite prime. Let $\Lambda\subset\mathbb{V}_v$ be a $\mathbb{Z}_v$-lattice of rank $4$. Let $L$ be a non-degenerate quadratic $\mathbb{Z}_v$-lattice of rank $3$ with fundamental matrix $T$. Then for positive integers $k$,
    \begin{equation*}
        W_{T,v}(1_v,k,1_{\Lambda^{3}})=\vert \textup{det}(S)\vert_p^{3/2}\cdot\den(\Lambda\obot H_{2k}^{+},L),
    \end{equation*}
    where $S$ is the a fundamental matrix of the lattice $\Lambda$.
    \label{whit-rep}
\end{lemma}
\par
For a nonsingular matrix $T\in\textup{Sym}_3(\mathbb{Q})$, define
\begin{equation}
    \textup{Diff}(T)=\{v:\,\,T\,\,\textup{is not represented by the quadratic space $\mathbb{V}_v$}\}.
    \label{diff-set}
\end{equation}
The set $\textup{Diff}(T)$ has odd cardinality. By \cite[Proposition A.4]{Kud97Ann}, we also have
\begin{equation*}
    W_{T,v}(g_v,0,\Phi_v)=0,\,\,\textup{if $v\in\textup{Diff}(T)$}.
\end{equation*}
For a finite place $p$ and a Schwartz function $\Phi\in\mathscr{S}(\mathbb{V}^3)$, define
\begin{equation*}
    E_p^{\prime}(g,0,\Phi)=\sum\limits_{T:\textup{Diff}(T)=\{p\}}E_{T}^{\prime}(g,0,\Phi).
\end{equation*}

\subsection{Arithmetic intersection of Hecke correspondences}
\label{mimic-gk}
Recall that $N$ is an odd and squarefree positive integer. Take $U=\Gamma_0(N)(\widehat{\mathbb{Z}})$. Then $\mathcal{O}_U$ is $\mathbb{Z}$. Let $\mathcal{X}_0(N)$ be the proper regular integral model of $X_0(N)$ constructed by Katz, Mazur \cite{KM85} and \v{C}esnavi\v{c}ius \cite{Ces17}. Let $\M_0(N)$ be the blow up of the Deligne-Mumford stack $\mathcal{X}_0(N)\times_{\mathbb{Z}}\mathcal{X}_0(N)$ along its supersingular points with residue field characteristic $p\vert N$. 
\par
Let $H_0(N)$ be the following rank $4$ quadratic lattice over $\mathbb{Z}$:
\begin{equation*}
    H_0(N)=\left\{\begin{pmatrix}
        a & b\\
        N c & d
    \end{pmatrix}:a,b,c,d\in\Z\right\}.
\end{equation*}
Then $\mathbf{1}_{H_0(N)\otimes\widehat{\mathbb{Z}}}\coloneqq \otimes_{p<\infty}\mathbf{1}_{H_0(N)\otimes\zp}$ is a function in $\mathscr{S}(\mathbb{V}_f)$.
For an integer $m>0$, define
\begin{equation}\label{def-hecke}
    \widehat{\mathsf{T}}(m)=\sum\limits_{\substack{x\in\widetilde{K}\backslash\mathbb{V}\\ q_{\mathbb{V}}(x)=m}}\mathbf{1}_{H_0(N)\otimes\widehat{\mathbb{Z}}}(x)\cdot\widehat{\mathsf{T}}(x)_K=\sum\limits_{\substack{x\in\widetilde{K}\backslash H_0(N)\otimes\widehat{\mathbb{Z}}\\ q_{\mathbb{V}}(x)=m}}\widehat{\mathsf{T}}(x)_K.
\end{equation}
It is a divisor on $\M_0(N)$. For three positive integers $m_1,m_2,m_3$, define the intersection numbers of three divisors to be
\begin{equation}
    \left(\widehat{\mathsf{T}}(m_1)\cdot\widehat{\mathsf{T}}(m_2)\cdot\widehat{\mathsf{T}}(m_3)\right)=\chi\left(\M_0(N),\mathcal{O}_{\widehat{\mathsf{T}}(m_1)}\otimes^{\mathbb{L}}_{\mathcal{O}_{\M_0(N)}}\mathcal{O}_{\widehat{\mathsf{T}}(m_2)}\otimes^{\mathbb{L}}_{\mathcal{O}_{\M_0(N)}}\mathcal{O}_{\widehat{\mathsf{T}}(m_3)}\right)
    \label{integral-m-hecke}
\end{equation}
\begin{theorem}
    Let $m_1,m_2,m_3$ be three positive integers such that there is no positive definite binary quadratic form over $\mathbb{Z}$ which represents the three integers $m_1,m_2,m_3$. Then
    \begin{equation*}
        \left(\widehat{\mathsf{T}}(m_1)\cdot\widehat{\mathsf{T}}(m_2)\cdot\widehat{\mathsf{T}}(m_3)\right)=-2\cdot\sum\limits_{T}E_T^{\prime}(\mathbf{1},0,\Phi_{\infty}\otimes\mathbf{1}_{\left(H_0(N)\otimes\widehat{\mathbb{Z}}\right)^{3}}),
    \end{equation*}
    where the summation ranges over all the half-integral symmetric positive definite $3\times3$ matrices $T$ with diagonal elements $m_1,m_2,m_3$.
    \label{a-i-h-1}
\end{theorem}
\begin{proof}
    The condition that there is no positive definite binary quadratic form over $\mathbb{Z}$ which represents the three integers $m_1,m_2,m_3$ implies that the three divisors $\{\widehat{\mathsf{T}}(m_i)\}_{i=1}^{3}$ have no self-intersections on the generic fiber of $\M_0(N)$. For a prime number $p$ (not necessarily odd), let $B(p)$ be the unique quaternion algebra over $\mathbb{Q}$ which ramifies at $p$ and $\infty$. Let $\M_{(p)}$ be following formal scheme:
\begin{equation*}
    \M_{(p)}=\begin{cases}
        \N\simeq\textup{Spf}\,W[[t_1,t_2]], &\textup{if $p\nmid N$;}\\
        \M, &\textup{if $p\,\vert\, N$.}
    \end{cases}
\end{equation*}
Take $\Phi=\Phi_{\infty}\otimes\mathbf{1}_{\left(H_0(N)\otimes\widehat{\mathbb{Z}}\right)^{3}}\in\mathscr{S}(\mathbb{V}^{3})$, we have the following by Corollary \ref{uni-hecke} and (\ref{def-hecke})
\begin{align}
    &\left(\widehat{\mathsf{T}}(m_1)\cdot\widehat{\mathsf{T}}(m_2)\cdot\widehat{\mathsf{T}}(m_3)\right)=\sum\limits_{p<\infty}\chi\left(\M_0(N)_{(p)},\mathcal{O}_{\widehat{\mathsf{T}}(m_1)}\otimes^{\mathbb{L}}_{\mathcal{O}_{\M_0(N)_{(p)}}}\mathcal{O}_{\widehat{\mathsf{T}}(m_2)}\otimes^{\mathbb{L}}_{\mathcal{O}_{\M_0(N)_{(p)}}}\mathcal{O}_{\widehat{\mathsf{T}}(m_3)}\right)\label{lasstt}\\
    &=\sum\limits_{p<\infty}\sum\limits_{\substack{\boldsymbol{y}\in H^{\prime}(\mathbb{Q})_0\backslash B(p)^{3}\\q_{B(p)}(y_i)=m_i}}\sum\limits_{g\in H_{\boldsymbol{y}}^{\prime}(\mathbb{Q})_0\backslash H(\mathbb{A}_f^p)/K^p}\Phi^{p}(g^{-1}\boldsymbol{y})\cdot\chi\left(\M_{(p)},[{^{\mathbb{L}}\otimes_{i=1}^{3}}\rO_{\mathcal{Z}(y_i)}]\right)\cdot\log(p)\notag\\
    &=\sum\limits_{p<\infty}\sum\limits_{T}\sum\limits_{\substack{\boldsymbol{y}\in H^{\prime}(\mathbb{Q})_0\backslash B(p)^{3}\\T(\boldsymbol{y})=T}}\sum\limits_{g\in H_{\boldsymbol{y}}^{\prime}(\mathbb{Q})_0\backslash H(\mathbb{A}_f^p)/K^p}\Phi^{p}(g^{-1}\boldsymbol{y})\cdot\chi\left(\M_{(p)},[{^{\mathbb{L}}\otimes_{i=1}^{3}}\rO_{\mathcal{Z}(y_i)}]\right)\cdot\log(p).\notag
\end{align}
where the summation symbol for $T$ in the last line ranges over all the half-integral symmetric positive definite $3\times3$ matrices $T$ with diagonal elements $m_1,m_2,m_3$. 
\par
For $p\nmid N$, \cite[Theorem 5.4.3]{YZZ} shows that
\begin{align}
   \sum\limits_{\substack{\boldsymbol{y}\in H^{\prime}(\mathbb{Q})_0\backslash B(p)^{3}\\T(\boldsymbol{y})=T}}\sum\limits_{g\in H_{\boldsymbol{y}}^{\prime}(\mathbb{Q})_0\backslash H(\mathbb{A}_f^p)/K^p}&\Phi^{p}(g^{-1}\boldsymbol{y})\cdot\chi\left(\M_{(p)},[{^{\mathbb{L}}\otimes_{i=1}^{3}}\rO_{\mathcal{Z}(y_i)}]\right)\cdot\log(p)\label{lastone}\\&=-2\cdot E_T^{\prime}(\mathbf{1},0,\Phi_{\infty}\otimes\mathbf{1}_{\left(H_0(N)\otimes\widehat{\mathbb{Z}}\right)^{3}}).\notag
\end{align}
For $p\,\vert\,N$, we obtain the same formula as (\ref{lastone}) by combining the volume calculation (\ref{measure-SW}) and the intersection of $\mathcal{Z}$-cycles in Theorem \ref{main-theorem} (we refer to the proof of Theorem \ref{int-yzz} in the next section for more details). Then the theorem follows by combining (\ref{lasstt}) and (\ref{lastone}).
\end{proof}
\begin{remark}
    For $N=1$, this theorem is proved by Gross and Keating \cite[(1.19)]{GK93}.
\end{remark}

\subsection{Triple product formula: the minimal ramification case}
\label{yzz-work}
Let $p$ be an odd prime number. Let $U=\Gamma_0(p)\cdot U^{p}$, where $U^{p}\subset\textup{GL}_2(\mathbb{A}_f^{p})$ is a sufficiently small compact open subgroup. Let $K=U\times_{\mathbb{G}_m}U=K_p\cdot K^{p}$. For $i=1,2,3$, let $\phi_i=\phi_{i,p}\otimes\phi_i^{p}\in\mathscr{S}(\mathbb{V})$ be three Schwartz functions such that 
\begin{itemize}
    \item [(1)] $\phi_{i,\infty}$ is the standard Gaussian function on $\mathbb{V}_{\infty}$.
    \item[(2)] $\phi_{i,p}=1_{H_0(p)}$ or $1_{H_0(p)^{\vee}}$ and $\phi_i^{p}$ is invariant under the group $K^p$.
    \item[(3)] There exists a finite place $v$ prime to $p$ such that the Schwartz function $\phi_v=\phi_{1,v}\otimes\phi_{2,v}\otimes\phi_{3,v}\in\mathscr{S}(\mathbb{V}_v^{3})$ is regularly supported in the sense of \cite[Definition 4.4.1]{YZZ}.
\end{itemize}
\par
Notice that by condition (3), the three cycles $\widehat{\mathsf{Z}}(g_i,\phi_i)$ have no self-intersections on the generic fiber of $\M_{K,(p)}$. The existence of these functions with some additional non-vanishing properties were proved by Liu \cite{Liu}. Let $\mathbb{G}=\textup{GL}_2\times_{\mathbb{G}_m}\textup{GL}_2\times_{\mathbb{G}_m}\textup{GL}_2$. 
\begin{lemma}
    For an element $g=(g_1,g_2,g_3)\in\mathbb{G}(\mathbb{A})$ such that $g_p=\mathbf{1}_p$ is the identity element, let $z\in\widehat{\mathsf{T}}(g_1,\phi_1)\cap\widehat{\mathsf{T}}(g_2,\phi_2)\cap\widehat{\mathsf{T}}(g_3,\phi_3)(\mathbb{F})$. Then for $i=1,2,3$, we have
    \begin{equation*}
        z\in\widehat{\mathsf{T}}(g_i,\phi_i)^{\textup{ss}}(\mathbb{F}).
    \end{equation*}
    \label{int-ss}
\end{lemma}
\begin{proof}
    Let $\left(\left(E_1\xrightarrow{\pi_1} E_1^{\prime},\overline{(\eta_1^p,\eta_1^{\prime p})}\right), \left(E_2\xrightarrow{\pi_2} E_2^{\prime},\overline{(\eta_2^p,\eta_2^{\prime p})}\right)\right)$ be the pair of degree $p$ isogenies corresponding to the point $z$. Condition (3) implies that $\textup{dim}_{\mathbb{Q}}\,\textup{Hom}(E_1,E_2)\geq3$. Therefore $E_1$ and $E_2$ must be supersingular elliptic curves. Hence $z\in\widehat{\mathsf{T}}(g_i,\phi_i)^{\textup{ss}}(\mathbb{F})$ for $i=1,2,3$.
\end{proof}
Denote by $\Phi\in\mathscr{S}(\mathbb{V}^{3})$ by $\Phi=\phi_1\otimes\phi_2\otimes\phi_3$. For an element $g=(g_1,g_2,g_3)\in\mathbb{G}(\mathbb{A})$, define $[\mathcal{O}_{\widehat{\mathsf{T}}^{\textup{ss}}(g,\Phi)}]$ to be the image of ${^{\mathbb{L}}\otimes}_{i=1}^{3}\mathcal{O}_{\widehat{\mathsf{T}}^{\textup{ss}}(g_i,\phi_i)}$ in $\textup{Gr}^3\textup{K}_0^{\prime}(\widehat{\M}_{K,(p)}^{\textup{ss}})$. Define
\begin{equation*}
    \left(\widehat{\mathsf{T}}(g_1,\phi_1)\cdot\widehat{\mathsf{T}}(g_2,\phi_2)\cdot\widehat{\mathsf{T}}(g_3,\phi_3)\right)_p=\chi\left(\M_{K,(p)},[\mathcal{O}_{\widehat{\mathsf{T}}^{\textup{ss}}(g,\Phi)}]\right)\cdot\log(p).
\end{equation*}
\begin{theorem}
    Let $\phi_i\in\mathscr{S}(\mathbb{V})$ be three Schwartz functions satisfying (1), (2) and (3). Then for all elements $g=(g_1,g_2,g_3)\in\mathbb{G}(\mathbb{A})$ such that $g_{v}=\mathbf{1}_v$, we have
    \begin{equation}
        \left(\widehat{\mathsf{T}}(g_1,\phi_1)\cdot\widehat{\mathsf{T}}(g_2,\phi_2)\cdot\widehat{\mathsf{T}}(g_3,\phi_3)\right)_p=-2E_p^{\prime}(g,0,\Phi),\,\,\,\,\textup{if $g_p=\mathbf{1}_p$ and $\Phi_p=1_{H_0(p)^{3}}$ or $1_{H_0(p)^{\vee3}}$}.
        \label{int-hecke-main}
    \end{equation}
    \label{int-yzz}
\end{theorem}
\begin{proof}
    We first consider the case $\Phi_p=1_{H_0(p)^{3}}$. Let $z\in\widehat{\mathsf{T}}(g_i,\phi_i)^{\textup{ss}}(\mathbb{F})$, there exists three linearly independent vectors $y_1,y_2,y_3\in B$ such that $z\in\bigcap\limits_{i=1}^{3}\mathcal{Z}(y_i)$. Let $T$ be the fundamental matrix of the quadratic lattice $L$ spanned by $y_1,y_2$ and $y_3$. Then
        \begin{equation*}
            \textup{Diff}(T)=\{p\}.
        \end{equation*}
        Moreover, we have
        \begin{align*}
            \chi\left(\M,[{^{\mathbb{L}}\otimes_{i=1}^{3}}\rO_{\mathcal{Z}(y_i)}]\right)&=\Int^{\mathcal{Z}}(L)=\partial\den(H_0(p),L)\\
            &=p^{4}(p-1)^{-2}\cdot\log(p)^{-1}\cdot W_{T,p}^{\prime}(1,0,1_{H_0(p)^3}).
        \end{align*}
        By Lemma \ref{uni-hecke}, we have 
        \begin{align}
            [\mathcal{O}_{\widehat{\mathsf{T}}^{\textup{ss}}(g,\Phi)}]=\sum\limits_{\boldsymbol{y}\in H^{\prime}(\mathbb{Q})_0\backslash B^{3}}&\sum\limits_{g\in H_{\boldsymbol{y}}^{\prime}(\mathbb{Q})_0\backslash H(\mathbb{A}_f^p)/K^p}\Phi^{p}(g^{-1}\boldsymbol{y})\cdot\Theta_{\M}^{-1}\left([{^{\mathbb{L}}\otimes_{i=1}^{3}}\rO_{\mathcal{Z}(y_i)}]\right)
            \label{uniffffff}
        \end{align}
        Therefore
        \begin{align*}
            &\left(\widehat{\mathsf{T}}(g_1,\Phi_1)\cdot\widehat{\mathsf{T}}(g_2,\Phi_2)\cdot\widehat{\mathsf{T}}(g_3,\Phi_3)\right)_p\\&=\sum\limits_{\boldsymbol{y}\in H^{\prime}(\mathbb{Q})_0\backslash B^{3}}\sum\limits_{g\in H_{\boldsymbol{y}}^{\prime}(\mathbb{Q})_0\backslash H(\mathbb{A}_f^p)/K^p}\Phi^{p}(g^{-1}\boldsymbol{y})\cdot\chi\left(\M,[{^{\mathbb{L}}\otimes_{i=1}^{3}}\rO_{\mathcal{Z}(y_i)}]\right)\cdot\log(p)\\
            &=\sum_{T:\textup{Diff}(T)=\{p\}}\sum\limits_{\substack{\boldsymbol{y}\in H^{\prime}(\mathbb{Q})_0\backslash B^{3}\\T(\boldsymbol{y})=T}}\sum\limits_{g\in H_{\boldsymbol{y}}^{\prime}(\mathbb{Q})_0\backslash H(\mathbb{A}_f^p)/K^p}\Phi^{p}(g^{-1}\boldsymbol{y})\cdot \frac{p^{4}}{(p-1)^{2}}\cdot W_{T,p}^{\prime}(1,0,1_{H_0(p)^{3}})\\
            &=\frac{p^{4}}{(p-1)^{2}}\cdot\frac{1}{\textup{vol}(K^p)}\cdot\int\limits_{\textup{SO}(B)(\mathbb{A}_f^p)}\Phi^p(g^{-1}\boldsymbol{y})\textup{d}g\cdot W_{T,p}^{\prime}(1,0,1_{H_0(p)^{3}}).
        \end{align*}
    We use the same measures for orthogonal groups as fixed in \cite[$\S$1.4]{YZZ}. By the calculations of the volume factor in Theorem 5.4.3 of \textit{loc.cit.} and Siegel-Weil formula, we have
    \begin{equation}
       \frac{p^{4}}{(p-1)^{2}}\cdot \frac{1}{\textup{vol}(K^p)}\cdot\int\limits_{\textup{SO}(B)(\mathbb{A}_f^p)}\Phi^p(g^{-1}\boldsymbol{y})\textup{d}g=-2\cdot\prod\limits_{v\neq p}W_{T,v}(g_v,0,\Phi_v).
        \label{measure-SW}
    \end{equation}
    Therefore (\ref{int-hecke-main}) is true when $g_p=\mathbf{1}_p$.
    \par
    Now we consider the case $\Phi_p=1_{H_0(p)^{\vee3}}$. Let $z\in\widehat{\mathsf{T}}(g_i,\Phi_i)^{\textup{ss}}(\mathbb{F})$, there exists three linearly independent vectors $y_1,y_2,y_3\in B$ such that $z\in\bigcap\limits_{i=1}^{3}\left(\mathcal{Y}(y_i)-\exc_{\M}\right)$. Let $T$ be the fundamental matrix of the quadratic lattice $L$ spanned by $y_1,y_2$ and $y_3$. Then
        \begin{equation*}
            \textup{Diff}(T)=\{p\}.
        \end{equation*}
        Moreover, we have
        \begin{align*}
            \chi\left(\M,[{^{\mathbb{L}}\otimes_{i=1}^{3}}\rO_{\mathcal{Y}(y_i)-\exc_{\M}}]\right)&=\Int^{\mathcal{Y}}(L)+1=\partial\den(H_0(p)^{\vee},L)\\
            &=p^{4}(p-1)^{-2}\cdot\log(p)^{-1}\cdot W_{T,p}^{\prime}(1,0,1_{H_0(p)^{\vee3}}).
        \end{align*}
        By Lemma \ref{uni-hecke}, we have 
        \begin{align*}
            [\mathcal{O}_{\widehat{\mathsf{T}}^{\textup{ss}}(g,\Phi)}]=\sum\limits_{\boldsymbol{y}\in H^{\prime}(\mathbb{Q})_0\backslash B^{3}}&\sum\limits_{g\in H_{\boldsymbol{y}}^{\prime}(\mathbb{Q})_0\backslash H(\mathbb{A}_f^p)/K^p}\Phi^{p}(g^{-1}\boldsymbol{y})\cdot\Theta_{\M}^{-1}\left([{^{\mathbb{L}}\otimes_{i=1}^{3}}\rO_{\mathcal{Y}(y_i)-\exc_{\M}}]\right).
        \end{align*}
        The remaining parts are similar to the proof of the case $\Phi_p=1_{H_0(p)^{3}}$.
\end{proof}

\bibliographystyle{alpha}
\bibliography{reference}

\newcommand{\etalchar}[1]{$^{#1}$}
\begin{thebibliography}{VGW{\etalchar{+}}07}

\bibitem[BBM82]{BBM82}
Pierre Berthelot, Lawrence Breen, and William Messing.
\newblock {\em Th\'eorie de Dieudonn\'e cristalline II}.
\newblock Lecture Notes in Mathematics. Springer Berlin Heidelberg, 1982.

\bibitem[BGHZ08]{Modularforms123}
Jan Bruinier, Gerard Geer, Günter Harder, and Don Zagier.
\newblock {\em The 1-2-3 of Modular Forms}.
\newblock Universitext. Springer, 2008.

\bibitem[BH06]{LLCforGL(2)}
Colin~J. Bushnell and Guy Henniart.
\newblock {\em The {L}ocal {L}anglands {C}onjecture for {G}{L}(2)}, volume 335 of {\em Grundlehren der Mathematischen Wissenschafte}.
\newblock Springer, Heidelberg, 2006.

\bibitem[BH21]{BH}
Jab~Hendrik Bruinier and Ben Howard.
\newblock {Arithmetic volumes of unitary Shimura varieties }.
\newblock {\em arxiv: 2105.11274}, 2021.

\bibitem[BY20]{BY20}
Jan Bruinier and Tonghai Yang.
\newblock Arithmetic degrees of special cycles and derivatives of siegel eisenstein series.
\newblock {\em Journal of European Math. Soc}, page pp60, 2020.

\bibitem[Che24]{Chen}
Ryan Chen.
\newblock {Co-rank 1 Arithmetic Siegel–Weil }.
\newblock {\em preprint}, 2024.

\bibitem[Cho22]{Cho}
Sungyoon Cho.
\newblock Special cycles on unitary {S}himura varieties with minuscule parahoric level structure.
\newblock {\em Mathematische Annalen}, pages 1--67, 2022.

\bibitem[CHZ23]{CHZ}
Sungyoon Cho, Qiao He, and Zhiyu Zhang.
\newblock On the kudla-rapoport conjecture for unitary shimura varieties with maximal parahoric level structure at unramified primes.
\newblock {\em arXiv preprint arXiv:2312.16906}, 2023.

\bibitem[FGL08]{FGL08}
Laurent Fargues, Alain Genestier, and Vincent Lafforgue.
\newblock {\em L'isomorphisme entre les tours de Lubin-Tate et de Drinfeld}.
\newblock Number XXII, 406 in Progress in Mathematics. Birkhäuser Basel, 2008.

\bibitem[FYZ24]{FYZ}
Tony Feng, Zhiwei Yun, and Wei Zhang.
\newblock Higher siegel--weil formula for unitary groups: the non-singular terms.
\newblock {\em Inventiones mathematicae}, 235(2):569--668, 2024.

\bibitem[GK93]{GK93}
Benedict Gross and Kevin Keating.
\newblock On the intersection of modular correspondences.
\newblock {\em Inventiones mathematicae}, 112(2):225--246, 1993.

\bibitem[GS19]{GS19}
Luis~E. Garcia and Siddarth Sankaran.
\newblock Green forms and the arithmetic {S}iegel-{W}eil formula.
\newblock {\em Invent. Math.}, 215(3):863--975, 2019.

\bibitem[HLSY23]{HLSY}
Qiao He, Chao Li, Yousheng Shi, and Tonghai Yang.
\newblock A proof of the {K}udla--{R}apoport conjecture for {Kr\"a}mer models.
\newblock {\em Inventiones mathematicae}, 234:721--817, 2023.

\bibitem[HSY20]{HSY}
Qiao He, Yousheng Shi, and Tonghai Yang.
\newblock {The Kudla-Rapoport conjecture at a ramified prime for $\mathrm{U}(1,1)$}.
\newblock {\em Transactions of the AMS, to appear.}, 2020.

\bibitem[HSY23]{HSY3}
Qiao He, Yousheng Shi, and Tonghai Yang.
\newblock Kudla–{R}apoport conjecture for {K}rämer models.
\newblock {\em Compositio Mathematica}, 159(8):1673–1740, 2023.

\bibitem[KM85]{KM85}
Nicholas~M. Katz and Barry Mazur.
\newblock {\em Arithmetic Moduli of Elliptic Curves. (AM-108)}.
\newblock Princeton University Press, 1985.

\bibitem[KR00]{KRshimuracurve}
Stephen Kudla and Michael Rapoport.
\newblock Height pairings on {S}himura curves and {$p$}-adic uniformization.
\newblock {\em Invent. Math.}, 142(1):153--223, 2000.

\bibitem[KR11]{KR11}
Stephen Kudla and Michael Rapoport.
\newblock Special cycles on unitary {S}himura varieties {I}. {U}nramified local theory.
\newblock {\em Invent. Math.}, 184(3):629--682, 2011.

\bibitem[KRY99]{KRYtiny}
Stephen~S. Kudla, Michael Rapoport, and Tonghai Yang.
\newblock On the derivative of an {E}isenstein series of weight one.
\newblock {\em Internat. Math. Res. Notices}, 1999(7):347--385, 1999.

\bibitem[KRY04]{KRYcomp}
Stephen~S. Kudla, Michael Rapoport, and Tonghai Yang.
\newblock Derivatives of {E}isenstein series and {F}altings heights.
\newblock {\em Compos. Math.}, 140(4):887--951, 2004.

\bibitem[KRY06]{KRYbook}
Stephen Kudla, Michael Rapoport, and Tonghai Yang.
\newblock {\em Modular Forms and Special Cycles on Shimura Curves.(AM-161)}, volume 161.
\newblock Princeton university press, 2006.

\bibitem[Kud97]{Kud97Ann}
Stephen Kudla.
\newblock Central derivatives of {E}isenstein series and height pairings.
\newblock {\em Annals of mathematics}, 146(3):545--646, 1997.

\bibitem[Kud04]{Kudla2004}
Stephen Kudla.
\newblock Special cycles and derivatives of {E}isenstein series.
\newblock In {\em Heegner points and {R}ankin {$L$}-series}, volume~49 of {\em Math. Sci. Res. Inst. Publ.}, pages 243--270. Cambridge Univ. Press, Cambridge, 2004.

\bibitem[Li23]{Li21}
Chao Li.
\newblock From sum of two squares to arithmetic siegel-weil formulas.
\newblock {\em Bull. Amer. Math. Soc.}, 60(3):327--370, 2023.

\bibitem[Liu24]{Liu}
Yifeng Liu.
\newblock {T}est functions for trilinear zeta integrals with regular support.
\newblock {\em Acta Math. Sin. Chin. Ser.}, 40:1599--1644, 2024.

\bibitem[LL22]{LL2}
Chao Li and Yifeng Liu.
\newblock Chow groups and {$L$}-derivatives of automorphic motives for unitary groups, {II}.
\newblock In {\em Forum of Mathematics, Pi}, volume~10. Cambridge University Press, 2022.

\bibitem[LZ22a]{LZ22a}
Chao Li and Wei Zhang.
\newblock Kudla--{R}apoport cycles and derivatives of local densities.
\newblock {\em Journal of the American Mathematical Society}, 35(3):705--797, 2022.

\bibitem[LZ22b]{LZ22b}
Chao Li and Wei Zhang.
\newblock On the arithmetic {S}iegel--{W}eil formula for {GS}pin {S}himura varieties.
\newblock {\em Inventiones mathematicae}, 228(3):1353--1460, 2022.

\bibitem[Rap05]{R}
Michael Rapoport.
\newblock 13. deformations of isogenies of formal groups.
\newblock In {\em ARGOS Seminar on Intersections of Modular Correspondences}, pages 147--176, 2005.

\bibitem[Shi18]{Shi1}
Yousheng Shi.
\newblock Special cycles on the basic locus of unitary {S}himura varieties at ramified primes.
\newblock {\em arXiv preprint arXiv:1811.11227}, 2018.

\bibitem[Shi22]{ShiCM}
Yousheng Shi.
\newblock {Special cycles on modular curves }.
\newblock {\em Unpublished notes}, 2022.

\bibitem[SP21]{schulzepillot2021lecture}
Rainer Schulze-Pillot.
\newblock Lecture notes on quadratic forms and their arithmetic.
\newblock {\em arXiv e-prints}, 2021.

\bibitem[SSY23]{SSY}
Siddarth Sankaran, Yousheng Shi, and Tonghai Yang.
\newblock {A genus two arithmetic Siegel-Weil formula on $X_0(N)$}.
\newblock {\em Transactions of the American Mathematical Society}, 376(06):3995--4041, 2023.

\bibitem[Ter11]{Ter11}
Ulrich Terstiege.
\newblock Intersections of arithmetic {H}irzebruch–{Z}agier cycles.
\newblock {\em Mathematische Annalen}, 349:161--213, 2011.

\bibitem[Ter13a]{Ter13a}
Ulrich Terstiege.
\newblock Intersections of special cycles on the {S}himura variety for $\mathrm{GU}(1,2)$.
\newblock {\em Journal fur die Reine und Angewandte Mathematik}, 2013, 06 2013.

\bibitem[Ter13b]{Ter13b}
Ulrich Terstiege.
\newblock On the regularity of special difference divisors.
\newblock {\em Comptes Rendus Mathematique}, 351(3):107--109, 2013.

\bibitem[\v{C}17]{Ces17}
K\k{e}stutis \v{C}esnavi\v{c}ius.
\newblock A modular description of $\mathscr{X}_0(n)$.
\newblock {\em Algebra $\&$ Number theory}, 11(9):2001--2089, 2017.

\bibitem[VGW{\etalchar{+}}07]{ARGOS}
Gunther Vogel, Ulrich G\"ortz, Torsten Wedhorn, Volker Meusers, Eva Viehmann, Konstantin Ziegler, Stefan Wewers, Inken Vollaard, Irene~I. Bouw, and Michael Rapoport.
\newblock {\em Argos seminar on intersections of modular correspondences}.
\newblock Number 312 in Ast\'erisque. Soci\'et\'e math\'ematique de France, 2007.

\bibitem[Yan98]{yanglocaldensity}
Tonghai Yang.
\newblock An explicit formula for local densities of quadratic forms.
\newblock {\em J. number theory}, 72(2):309--356, 1998.

\bibitem[Yao24]{Yao}
Haodong Yao.
\newblock A kudla-rapoport formula for exotic smooth models of odd dimension.
\newblock {\em arXiv preprint arXiv:2404.14431}, 2024.

\bibitem[YZZ23]{YZZ}
Xinyi Yuan, Shou-Wu Zhang, and Wei Zhang.
\newblock Triple product {L}-series and {G}ross--{K}udla--{S}choen cycles.
\newblock {\em preprint}, 2023.

\bibitem[Zha21]{zhang2021AFL}
Wei Zhang.
\newblock Weil representation and arithmetic fundamental lemma.
\newblock {\em Annals of Mathematics}, 193(3):863--978, 2021.

\bibitem[Zhu23]{Zhu23}
Baiqing Zhu.
\newblock Arithmetic {S}iegel-{W}eil formula on $\mathcal {X} _ {0}({N}) $.
\newblock {\em arXiv preprint, arXiv:2304.10696}, 2023.

\bibitem[Zhu24]{Zhu23diff}
Baiqing Zhu.
\newblock The regularity of difference divisors.
\newblock {\em Math. Ann.}, 2024.

\end{thebibliography}

\end{document}